\documentclass[11pt]{amsart}
\usepackage[leqno]{amsmath}
\usepackage{amssymb}
\usepackage{enumerate}
\usepackage{subfigure}
\usepackage{graphicx}
\usepackage{hyperref}
\usepackage{mathtools}
\usepackage{mathrsfs}
\usepackage{mathabx}
\usepackage[all]{xy}
\usepackage[usenames,dvipsnames]{xcolor}

\usepackage{microtype}
\usepackage{stmaryrd}

\usepackage{comment}

\usepackage{bm}

\numberwithin{equation}{section}

\newtheorem{theorem}{Theorem}[section]

\newtheorem{lemma}[theorem]{Lemma}
\newtheorem{proposition}[theorem]{Proposition}

\newtheorem{definition}[theorem]{Definition}

\newcommand{\C}{\mathbf{C}}
\newcommand{\D}{\mathbf{D}}
\newcommand{\E}{\mathbf{E}}

\newcommand{\h}{\mathbf{H}}
\newcommand{\N}{\mathbf{N}}
\newcommand{\Z}{\mathbf{Z}}
\newcommand{\p}{\mathbf{P}}
\newcommand{\Q}{\mathbf{Q}}
\newcommand{\R}{\mathbf{R}}
\newcommand{\s}{\mathbf{S}}

\newcommand{\Fg}{\mathfrak {g}}
\newcommand{\Fh}{\mathfrak {h}}
\newcommand{\Fp}{\mathfrak {p}}

\newcommand{\CA}{\mathcal {A}}

\newcommand{\CD}{\mathcal {D}}

\newcommand{\CF}{\mathcal {F}}

\newcommand{\CN}{\mathcal {N}}

\newcommand{\CQ}{\mathcal {Q}}
\newcommand{\CR}{\mathcal {R}}
\newcommand{\CS}{\mathcal {S}}
\newcommand{\CT}{\mathcal {T}}

\newcommand{\CW}{\mathcal {W}}

\newcommand{\CZ}{\mathcal {Z}}

\newcommand{\CG}{\mathcal {G}}

\newcommand{\Ff}{\mathfrak{f}}

\newcommand{\SLE}{{\rm SLE}}

\newcommand{\dist}{\mathrm{dist}}

\newcommand{\diam}{\mathrm{diam}}
\newcommand{\var}{\mathrm{var}}
\newcommand{\im}{\mathrm{Im}}
\newcommand{\re}{\mathrm{Re}}

\newcommand{\loc}{\mathrm{loc}}

\newcommand{\Leb}{\mathcal{L}}
\newcommand{\cont}{\mathrm{Cont}}
\newcommand{\inrad}{\mathrm{inrad}}
\newcommand{\argmin}{\arg \, \min}
\newcommand{\cen}{\mathrm{cen}}

\newcommand{\confrad}{{\rm CR}}

\newcommand{\inn}{\mathrm{in}}

\newcommand{\pup}{\partial^\Up}
\newcommand{\pdown}{\partial^\Down}

\newcommand{\pright}{\partial^\Right}

\newcommand{\pleft}{\partial^\Left}

\newcommand{\one}{{\bf 1}}

\newcommand{\wt}{\widetilde}
\newcommand{\wh}{\widehat}
\newcommand{\wc}{\widecheck}
\newcommand{\ol}{\overline}
\newcommand{\ul}{\underline}

\newcommand{\giv}{\,|\,}

\newcommand{\BES}{\mathrm{BES}}

\newcommand{\strip}{\mathscr S}
\newcommand{\cyl}{\mathscr C}

\newcommand{\SH}{\mathrm{SH}}
\newcommand{\length}{\mathrm{length}}
\newcommand{\distqh}{\mathrm{dist}_{\mathrm{qh}}}
\newcommand{\disthyp}{\mathrm{dist}_{\mathrm{hyp}}}
\newcommand{\Bhyp}{B_{\mathrm{hyp}}}
\newcommand{\pin}{\partial^{\mathrm{in}}}
\newcommand{\pout}{\partial^{\mathrm{out}}}

\newcommand{\neigh}[2]{B(#1,#2)}

\newcommand{\qmeasure}[1]{\mu_{#1}}
\newcommand{\closure}[1]{\mathrm{cl}(#1)}
\newcommand{\nmeasure}[1]{\mu_{#1}^{\mathrm{nat}}}

\newcommand{\harm}[3]{\omega(#1,#2,#3)}

\newcommand{\Up}{\mathrm{U}}
\newcommand{\Down}{\mathrm{D}}
\newcommand{\Left}{\mathrm{L}}
\newcommand{\Right}{\mathrm{R}}
\newcommand{\Inner}{\mathrm{I}}

\newcommand{\bIn}{\partial^{\mathrm{in}}}
\newcommand{\bOut}{\partial^{\mathrm{out}}}

\newcommand{\kay}{{{\mathbf k}}}
\newcommand{\len}{{\mathrm {len}}}

\begin{document}

\title{Conformal removability of $\SLE_4$}

\author{Konstantinos Kavvadias, Jason Miller, Lukas Schoug}

\begin{abstract}
We consider the Schramm-Loewner evolution ($\SLE_\kappa$) with $\kappa=4$, the critical value of $\kappa > 0$ at or below which $\SLE_\kappa$ is a simple curve and above which it is self-intersecting.  We show that the range of an $\SLE_4$ curve is a.s.\ conformally removable, answering a question posed by Sheffield.  Such curves arise as the conformal welding of a pair of independent critical ($\gamma=2$) Liouville quantum gravity (LQG) surfaces along their boundaries and our result implies that this conformal welding is unique.  In order to establish this result, we give a new sufficient condition for a set $X \subseteq \C$ to be conformally removable which applies in the case that $X$ is not necessarily the boundary of a simply connected domain.
\end{abstract}

\date{\today}
\maketitle

\setcounter{tocdepth}{1}

\tableofcontents

\parindent 0 pt
\setlength{\parskip}{0.20cm plus1mm minus1mm}

\section{Introduction}
\label{sec:intro}

The Schramm-Loewner evolution ($\SLE_\kappa$) is a one parameter family ($\kappa > 0$) of curves which connect two boundary points of a simply connected domain.  It was introduced by Schramm in \cite{s2000sle} as a candidate to describe the scaling limit of the interfaces which arise in discrete models from statistical mechanics on planar lattices at criticality, such as loop-erased random walk and the percolation model.  Scaling limit results for such models towards $\SLE$ have now been proved in a number of cases \cite{s2001cardy,lsw2004lerw,ss2009contours,s2010ising} and similar results have also been proved where the underlying graph is a random planar map \cite{s2016inventory,kmsw2019bipolar,gm2017perc,lsw2017wood,gm2021saw}.  The value of $\kappa > 0$ determines the roughness of an $\SLE_\kappa$ curve.  In particular, the $\SLE_\kappa$ curves are simple if $\kappa \leq 4$, self-intersecting but not space-filling if $\kappa \in (4,8)$, and space-filling if $\kappa \geq 8$ \cite{rs2005basic}.  Moreover, the dimension of an $\SLE_\kappa$ curve is $\min(1+\kappa/8,2)$ \cite{rs2005basic,beffara2008dimension}.  The focus of this work is on the critical value $\kappa = 4$, which corresponds to a curve which is simple but just barely so.

In recent years, there has been a substantial amount of work centered around the relationship between $\SLE_\kappa$ and Liouville quantum gravity (LQG) surfaces.  Roughly speaking, LQG is the canonical model of a random two-dimensional Riemannian manifold and can formally be described by its metric tensor
\begin{equation}
\label{eqn:lqg_def}
e^{\gamma h(z)} (dx^2 + dy^2)
\end{equation}
where $h$ is an instance of (some form of) the Gaussian free field (GFF) on a planar domain $D$ and $\gamma \in (0,2]$ is a parameter.  Since the GFF is a distribution in the sense of Schwartz, it is non-trivial to make sense of~\eqref{eqn:lqg_def} as the exponential is a non-linear operation and this has been the focus of a tremendous amount of mathematical work in the last two decades.  For example, Kahane's theory of Gaussian multiplicative chaos \cite{k1985gmc} can be used to construct the volume form and boundary length measure associated with~\eqref{eqn:lqg_def}.  The work \cite{ds2011lqg} gives another construction of the volume form and boundary length measure.  Moreover, \cite{ds2011lqg} proves a rigorous version of the KPZ formula, which was previously used by physicists to give non-rigorous derivations of critical exponents for planar lattice models (and more recently) $\SLE_\kappa$.  At the critical value $\gamma=2$, the volume form and boundary length measure were constructed in \cite{drsv2014critical2,drsv2014critical1} and later shown to arise as properly renormalized limits of the volume form and boundary length measure as $\gamma \to 2$ in \cite{aps2019critical}.  The metric (two point distance function) associated with~\eqref{eqn:lqg_def} for $\gamma = \sqrt{8/3}$ was constructed in \cite{ms2020lqg1,ms2021lqg2} and for $\gamma \in (0,2)$ in \cite{dddf2020tightness,gm2021metric}.

The manner in which $\SLE_\kappa$ arises in the context of LQG is that it describes the interfaces which are formed when one takes two independently sampled LQG surfaces and glues them together along their boundaries.  The way that the gluing is constructed is using conformal welding, which we recall is defined as follows.  Suppose that $\D_1$, $\D_2$ are two copies of the unit disk and $\phi \colon \partial \D_1 \to \partial \D_2$ is a homeomorphism.  We say that a simple loop $\eta$ on the two-dimensional sphere $\s^2$ together with conformal maps $\psi_i$ for $i=1,2$ from $\D_i$ to the left and right sides of $\s^2 \setminus \eta$ is a conformal welding with welding homeomorphism $\phi$ if $\phi = \psi_2^{-1} \circ \psi_1|_{\partial \D_1}$.  For a given homeomorphism $\phi$, it is not obvious if such a conformal welding exists but Sheffield \cite{she2016zipper} showed that for each $\gamma \in (0,2)$ it does exist if one takes the welding homeomorphism to be the one which associates points along the boundaries of two independent $\gamma$-LQG surfaces according to boundary length using the boundary measure from~\eqref{eqn:lqg_def}.  In this case, the welding interface is an $\SLE_\kappa$ with $\kappa=\gamma^2 \in (0,4)$.  This was extended to the case $\kappa=4$ and $\gamma=2$ in \cite{hp2021welding} and we also remark that a number of welding-type results which include the case $\kappa > 4$ were established in \cite{dms2021mating}.  (Let us also mention the work \cite{ajks2011randomconformal} which studies the conformal welding of an LQG surface to a Euclidean surface.  In this case, curves different from $\SLE_\kappa$ arise.)

For a homeomorphism $\phi$ as above for which there exists a conformal welding, it is also a non-trivial question to determine whether the conformal welding is unique.  Recall that a set $X \subseteq \C$ is said to be \emph{conformally removable} if every homeomorphism $\varphi \colon \C \to \C$ which is conformal on $\C \setminus X$ is conformal on $\C$.  It is not difficult to see that the uniqueness of a conformal welding is equivalent to the welding interface being conformally removable.

In order to prove that a set $X \subseteq \C$ which is equal to the boundary of a simply connected domain~$D$ (e.g., a simple curve) is conformally removable, one often makes use of a condition due to Jones and Smirnov \cite{js2000removability}.  We will not describe the Jones-Smirnov condition here, but remark that one often checks that it holds by using one of the sufficient conditions established in \cite{js2000removability}.  For example, it is shown in \cite{js2000removability} that if the Riemann map $\D \to D$ is H\"older continuous up to $\partial \D$ (so that $D$ is a so-called H\"older domain) then $X = \partial D$ is conformally removable; in fact it is shown in \cite{js2000removability} that a much weaker modulus of continuity suffices (see also \cite{kn2005remove}).  Rohde and Schramm proved that the complementary components of an $\SLE_\kappa$ curve for $\kappa \neq 4$ are H\"older domains hence conformally removable.  The optimal H\"older exponent was later computed in \cite{gms2018multifractal} in which it was also shown that $\SLE_4$ does not form the boundary of a H\"older domain.  The precise modulus of continuity of the uniformizing map of the complementary component of an $\SLE_4$ was determined in \cite{kms2021regularity} and is given by $(\log \delta^{-1})^{-1/3+o(1)}$ as $\delta \to 0$.  The reason for the difference in behavior for $\kappa \neq 4$ is that $\SLE_4$ curves are barely non-self-intersecting and contain tight bottlenecks.  Another way of formulating this is that if $\eta$ is an $\SLE_\kappa$ for $\kappa \in (0,4)$ and $z \in \eta$ then the harmonic measure of $B(z,\epsilon)$ in the complement of $\eta$ decays as fast as a power of $\epsilon$ as $\epsilon \to 0$ \cite{gms2018multifractal} but for $\kappa=4$ can decay as fast as $\exp(-\epsilon^{-3+o(1)})$ as $\epsilon \to 0$ \cite{kms2021regularity}.  It was further shown in \cite{kms2021regularity} that the Jones-Smirnov condition itself does not hold for $\SLE_4$.

The main result of this paper is the conformal removability of $\SLE_4$, which implies the uniqueness of the welding problem for critical ($\gamma=2$) LQG. (We remark that a weaker version of the uniqueness of the welding problem for critical LQG was proved in \cite{mmq2021uniqueness}.)

\begin{theorem}
\label{thm:sle_4_removable}
Suppose that $\eta$ is an $\SLE_4$ in $\h$ from $0$ to $\infty$.  Suppose that $f \colon \h \to \h$ is a homeomorphism which is conformal on $\h \setminus \eta$.  Then $f$ is a.s.\ conformal on $\eta$.  In particular, the range of $\eta$ is a.s.\ conformally removable.
\end{theorem}

We note that the first assertion of Theorem~\ref{thm:sle_4_removable} implies that the range of $\eta$ is a.s.\ conformally removable because if $f \colon \C \to \C$ is a homeomorphism which is conformal on $\C \setminus \eta$ then by the Riemann mapping theorem we can post-compose its restriction to $\h$ with a conformal map so that we obtain a homeomorphism $\h \to \h$ which is conformal on $\h \setminus \eta$.  We also remark that Theorem~\ref{thm:sle_4_removable} applies if $\eta$ is an $\SLE_4$ in an arbitrary simply connected domain $D$ since one can always conformally map $D$ to $\h$.

One of the main steps in proving Theorem~\ref{thm:sle_4_removable} is Theorem~\ref{thm:removability_of_X}, which contains a new sufficient condition for a set $X \subseteq \C$ to be conformally removable.  One of the novelties of Theorem~\ref{thm:removability_of_X} is that it does not require $X = \partial D$ for $D \subseteq \C$ simply connected as in \cite{js2000removability}.  We also remark that in general the conformal removability for boundaries of domains which are not simply connected is less well-understood.  Perhaps the simplest example is the standard Sierpinski carpet which (without much difficulty) is seen to not be conformally removable; see, e.g., the discussion after \cite[Theorem~1.8]{n2019nonremove}.  See also the recent work \cite{n2021carpet} which proves that all topological Sierpinski carpets are not conformally removable.  It is a much more difficult problem to show that the Sierpinski gasket is not conformally removable \cite{n2019nonremove}.  Like the Sierpinski gasket, the complement of an $\SLE_\kappa$ curve with $\kappa \in (4,8)$ consists of a countable collection of open sets whose boundaries can intersect each other.  In another work \cite{kms2022sle48remov}, we will use Theorem~\ref{thm:removability_of_X} to prove the conformal removability of $\SLE_\kappa$ for $\kappa \in (4,8)$ whenever the adjacency graph of connected components of the complement of such a curve is connected.  This means that if $\eta$ is an $\SLE_\kappa$ in $\h$ from $0$ to $\infty$ and $U,V$ are components of $\h \setminus \eta$ then there exist components $U_1,\ldots,U_n$ of $\h \setminus \eta$ with $U_1 = U$, $U_n = V$, and $\partial U_i \cap \partial U_{i+1} \neq \emptyset$ for each $1 \leq i \leq n-1$.  We note that it was shown in \cite{gp2020adj} that there exists $\kappa_0 \in (4,8)$ so that this condition holds for all $\kappa \in (4,\kappa_0)$.

We refer the reader to Theorem~\ref{thm:removability_of_X} for the precise statement of the sufficient condition for conformal removability, but we will give a rough description of the setup and strategy here.  Suppose that $X \subseteq \C$ has zero Lebesgue measure and $f \colon \C \to \C$ is a homeomorphism which is conformal on $\C \setminus X$.  To show that $f$ is conformal on $\C$ it suffices to show that $f$ is absolutely continuous on lines (ACL), which we recall means that it is absolutely continuous on Lebesgue a.e.\ horizontal and vertical line.  To prove that $X$ is conformally removable, it therefore suffices to control the variation of $f$ on the places where a Lebesgue typical horizontal or vertical line intersects $X$.  We will fix $a > 0$ small, $M > 1$ large, and assume that $X$ has upper Minkowski dimension at most $2-5a$.  We will further assume that we have a family of sets $\CA = \cup_{n=1}^\infty \CA_n$ where each $A \in \CA_n$ has the topology of an annulus with $\diam(A) \leq M 2^{-n}$ such that for each $z \in X$ and $n \in \N$ there exists $(1-a^2) n \leq k \leq n$ and $A \in \CA_k$ so that $B(z,2^{-n})$ is contained in the bounded component of $\C \setminus A$ and the following additional condition holds.  There are finitely many open, simply connected, and pairwise disjoint sets $U_1,\ldots,U_m$ for $1 \leq m \leq M$ in $A \setminus X$ with $\partial U_i \cap \partial U_{i+1} \neq \emptyset$ for $1 \leq i \leq m$ (with $U_{m+1} = U_1$) which satisfy some additional assumptions on the geometry of their pairwise intersections.  This will allow us to construct a path $\gamma$ which disconnects the inner and outer boundaries of $A$ and whose image under $f$ has diameter at most $2^{-(1-3a)k} + 2^{(1-a)k} \int_A |f'(w)|^2 d\Leb_2(w)$ (where $\Leb_2$ denotes two-dimensional Lebesgue measure).  This gives an upper bound on $\diam(f([z-2^{-k},z+2^{-k}]))$ which suffices because for (a compact part of) a typical line $L$ the number of intervals of length $2^{-n}$ hit by $X$ is $O(2^{(1-5a)n})$ (as $X$ has upper Minkowski dimension at most $2-5a$) and the integral of $|f'|^2$ on the $2^{-(1-a^2)n}$-neighborhood of (a compact part of) $L$ is $O(2^{-(1-a^2)n})$.  In particular, the variation of $f$ on (a compact part of) $L$ in $X \cap L$ if $O(2^{-a n/2})$ which tends to $0$ as $n \to \infty$, hence $f$ is absolutely continuous on $L$.

\subsection*{Outline and strategy}

The remainder of this article is structured as follows.  First, in Section~\ref{sec:preliminaries} we will collect a number of preliminaries.  Next, in Section~\ref{sec:initial_estimates} we will prove two regularity estimates for $\SLE_4$.  We will next in Section~\ref{sec:first_crossing} prove several estimates regarding the regularity of the first crossing of an $\SLE_4$ across an annulus.  The purpose of Section~\ref{sec:rectangle_exploration} is to define and study the properties of an exploration of crossings of GFF level lines across a rectangle \cite{schramm2013contour,wang2017level}.  We will then extend this to the case of annuli in Section~\ref{sec:annulus_exploration}.  We next show in Section~\ref{sec:sle4_is_removable} that the hypotheses of Theorem~\ref{thm:removability_of_X} are satisfied for $\SLE_4$.  Finally, we will state and prove the removability theorem in Section~\ref{sec:removability_theorem}.  Appendix~\ref{app:whitney} contains a number of facts regarding the relationship between Whitney squares and the hyperbolic metric.

\subsection*{Notation}
For two quantities $a,b$ we write $a \lesssim b$ if there is a constant $C>0$ (referred to as the implicit constant), independent of any parameters of interest, such that $a \leq Cb$. Moreover, we write $a \gtrsim b$ if $b \lesssim a$, and $a \asymp b$ if $a \lesssim b$ and $a \gtrsim b$. We often specify what parameters the implicit constant depends on.

For a domain $D$, $E \subseteq \partial D$ and $z \in D$, we denote by $\harm{z}{E}{D}$ the harmonic measure of $E$ in $D$ seen from $z$. Moreover, for a given curve $\eta$ we denote by $\eta^\Left$ (resp.\ $\eta^\Right$) the left (resp.\ right) side of the curve in the sense of prime ends.

We let $\Z$ denote the set of integers, $\N$ the set of positive integers, $\N_0 = \N \cup \{ 0 \}$, $\R$ the real numbers, $\C$ the complex plane, $\D$ the unit disk and $\h$ the upper half-plane, and $\strip = \R \times (0,\pi)$ the infinite horizontal strip of height $\pi$. For a set $A \subseteq \R^n$, we let $\closure{A}$ denote its closure. Finally, we denote by $\Leb_n$ the $n$-dimensional Lebesgue measure.

\subsection*{Acknowledgements} 

K.K.\ was supported by the EPSRC grant EP/L016516/1 for the University of Cambridge CDT (CCA).  J.M.\ and L.S.\ were supported by ERC starting grant 804166 (SPRS).

\section{Preliminaries}
\label{sec:preliminaries}

We will now collect a number of preliminaries which are used throughout this work.  We begin by reviewing some basic facts about Whitney squares and the hyperbolic metric in Section~\ref{subsec:whitney_hyperbolic}, quasiconformal maps in Section~\ref{subsec:quasiconformal}, $\SLE_\kappa$ and its relevant variants in Section~\ref{subsec:sle}, and the natural parameterization of $\SLE$ in Section~\ref{subsec:natural_parameterization}. We next review the GFF in Section~\ref{subsec:gff}, LQG in Section~\ref{subsec:lqg}, and finally the level line coupling of $\SLE_4$ and its variants with the GFF in Section~\ref{subsec:level_lines}.

\subsection{Whitney squares, hyperbolic and quasihyperbolic distance}
\label{subsec:whitney_hyperbolic}

Recall that the \emph{hyperbolic metric} in $\D$ is defined by 
\begin{align*}
	\disthyp^{\D}(z_1,z_2) = \inf \int_{z_1}^{z_2} \frac{|dz|}{1-|z|^2} \quad\text{for}\quad z_1,z_2 \in \D,
\end{align*}
where the infimum is taken over all smooth curves in $\D$ connecting $z_1$ and $z_2$.  One of the basic properties of  $\disthyp^{\D}$ is that it is conformally invariant which means that for each conformal automorphism $\varphi$ of $\D$ we have that $\disthyp^{\D}(\varphi(z_1),\varphi(z_2)) = \disthyp^{\D}(z_1,z_2)$ (see~\cite[Chapter~I.4]{gm2005harmonic}).  This makes it possible to define the hyperbolic metric on any simply connected domain $D \subseteq \C$ by $\disthyp^D(z_1,z_2) \coloneqq \disthyp^{\D}(\varphi(z_1),\varphi(z_2))$ where $\varphi: D \to \D$ is a conformal map as $\disthyp^D$ does not depend on the choice of $\varphi$. It follows that the hyperbolic distance is conformally invariant which means that if $D,\wt{D} \subseteq \C$ are simply connected domains, $\varphi: D \to \wt{D}$ is a conformal map, and $z_1,z_2 \in D$, then we have that $\disthyp^D(z_1,z_2) = \disthyp^{\wt{D}}(\varphi(z_1),\varphi(z_2))$.

We note that for each $z \in \D$, $\disthyp^\D(z,\partial \D) = \infty$ and hence the same holds in any simply connected domain $D$. However, the notion of a geodesic between an interior point and a boundary point still makes sense. We define the hyperbolic geodesic from $0$ to $1$ in $\D$ to be the line segment $[0,1]$.  The reason for this definition is that for each $\epsilon \in (0,1)$ the curve with the shortest hyperbolic length connecting $0$ and $\partial B(1,\epsilon) \cap \D$ is the segment $[0,1-\epsilon]$.  Moreover, for a Jordan domain $D$, $z \in D$, $w \in \partial D$, we define the hyperbolic geodesic from $z$ to $w$ to be $\varphi([0,1])$, where $\varphi:\D \to D$ is the unique conformal map such that $\varphi(0) = z$ and $\varphi(1) = w$.  This definition makes sense more generally if $D \subseteq \C$ is a simply connected domain, $z \in D$, and $w$ is a prime end of $\partial D$. Throughout the paper, for $z,w \in \closure{D}$, we shall denote by $\gamma_{z,w}^D$ the hyperbolic geodesic from $z$ to $w$ in $D$, parameterized so that the hyperbolic length of $\gamma_{z,w}^D([s,t])$ is $t-s$, and extended to all times $t \in \R$.

For a simply connected domain $D \subseteq \C$ one can also consider the \emph{quasihyperbolic metric} $\distqh^D$ which is defined by
\begin{align*}
	\distqh^D(z_1,z_2) = \inf \int_{z_1}^{z_2} \frac{|dz|}{\dist(z,\partial D)} \quad\text{for}\quad z_1,z_2 \in D,
\end{align*}
where the infimum is taken over all smooth curves in $D$ connecting $z_1$ and $z_2$. It follows by the Koebe-$1/4$ theorem that for any simply connected domain $D$ we have
\begin{align}\label{eq:dist_hyp_qh_comparable}
	\disthyp^D(z,w) \leq \distqh^D(z,w) \leq 4\disthyp^D(z,w);
\end{align}
see~\cite[Chapter~I.4]{gm2005harmonic}.

For any open subset $U \subsetneq \C$, there exists a family $\CW = (Q_j)$ of closed squares with pairwise disjoint interiors and sides parallel to the axes, such that $Q_j$ has sidelength $\len(Q_j) = 2^{-n_j}$ for some $n_j \in \Z$, $U = \cup_j Q_j$ and such that
\begin{align*}
	\diam(Q_j) \leq \dist(Q_j,\partial U) < 4 \diam(Q_j).
\end{align*}
Such a family is referred to as a \emph{Whitney square} decomposition of $U$ and we let $\cen{Q_j}$ denote the center of $Q_j \in \CW$.  In what follows, we assume that $U = D$ is a simply connected domain. Consider a Whitney square decomposition $\CW = (Q_j)$ of $D$ and write $x_j = \cen(Q_j)$. Furthermore, let $G_\CW = (V_\CW,E_\CW)$ denote the graph with vertex set $V = (x_j)$ and edge set $E$ such that $\{x_i,x_j\} \in E$ if and only if $\partial Q_i \cap \partial Q_j$ contains an open interval. Denote by $d_\CW$ the graph distance in $G_\CW$, that is, $d_\CW(x_i,x_j)$ is the minimal number of edges of a path in $G_\CW$ from $x_i$ to $x_j$. In other words, $d_\CW(x_i,x_j)$ is the minimal number of adjacent Whitney squares one needs to traverse to travel from $x_i$ to $x_j$. Then it follows that
\begin{align*}
	d_\CW(x_i,x_j) \asymp \distqh(x_i,x_j)
\end{align*}
where the implicit constants are universal. Furthermore, we write $x(z) = x_j$ if and only if $z \in Q_j$ and we denote by $Q_j^\dagger$ the square with center $x_j$ such that $\len(Q_j^\dagger) = \len(Q_j)/2$. Then, noting that if $z,w \in \closure{Q_j}$ for some $j$, then $\distqh^D(z,w) \leq 1$ and $\distqh^D(z,x_j) \leq 1/2$, it follows that
\begin{align}\label{eq:dist_qh_cube_comparable}
	\distqh^D(z,w) \asymp d_\CW(x(z),x(w))
\end{align}
whenever $w \in Q_j^\dagger$ and $z \in Q_i$, $i \neq j$, where the implicit constants are universal. By~\eqref{eq:dist_hyp_qh_comparable}, we also have that
\begin{align}\label{eq:dist_hyp_cube_comparable}
	\disthyp^D(z,w) \asymp d_\CW(x(z),x(w))
\end{align}
holds for some universal implicit constants.

We wrap up this subsection with a few useful inequalities for Whitney square decompositions and conformal maps.

We note that if $\CW$ is a Whitney square decomposition and $Q,Q' \in \CW$ are adjacent then there exists a constant $M > 1$ such that
\begin{align}\label{eq:neighbor_squares_sidelength}
	\frac{1}{M} \len(Q') \leq \len(Q) \leq M \len(Q').
\end{align}

Let $D$ and $\wt{D}$ be simply connected domains, $f: D \to \wt{D}$ be a conformal transformation, and let $\CW$ and $\wt{\CW}$ be Whitney square decompositions of $D$ and $\wt{D}$, respectively. By~\cite[Theorem~3.21]{lawler2008conformally} there exist universal constants $c_1,c_2>0$ such that for all $Q \in \CW$ and $z \in Q$ we have
\begin{align}\label{eq:uniform_derivative_in_square}
	c_1 |f'(\cen(Q))| \leq |f'(z)| \leq c_2 |f'(\cen(Q))|.
\end{align}
Moreover, by the Koebe-$1/4$ theorem, letting $Q(z)$ (resp.\ $\wt{Q}(f(z))$) denote any square $Q \in \CW$ (resp.\ $\wt{Q} \in \wt{\CW}$) such that $z \in \closure{Q}$ (resp.\ $f(z) \in \closure{\wt{Q}}$), we have that
\begin{align}\label{eq:derivative_diameter_relation}
	|f'(z)| \asymp \frac{\dist(f(z),\partial \wt{D})}{\dist(z,\partial D)} \asymp \frac{\diam(\wt{Q}(f(z)))}{\diam(Q(z))} = \frac{\len(\wt{Q}(f(z)))}{\len(Q(z))},
\end{align}
where the implicit constants are universal.

\subsection{Quasiconformal maps}
\label{subsec:quasiconformal}
Let $D,\wt{D}$ be domains in $\wh{\C} = \C \cup \{\infty\}$ and let $f: D \to \wt{D}$ be an orientation preserving homeomorphism. We say that $f$ is \emph{ACL} (absolutely continuous on lines) if $f$ is absolutely continuous on Lebesgue a.e.\ line segment in $D$ which is parallel to one of the axes. For $M \geq 1$, we say that $f$ is an $M$-quasiconformal mapping if $f$ is ACL and
\begin{align}\label{eq:quasiconformal_inequality}
	\left| \frac{\partial f}{\partial \ol{z}} \right| \leq \frac{M-1}{M+1} \left| \frac{\partial f}{\partial z} \right| \quad\text{a.e.}
\end{align}
(with the understanding that $\partial f/\partial z$ and $\partial f/\partial \ol{z}$ are defined a.e.). The smallest constant $M \geq 1$ such that~\eqref{eq:quasiconformal_inequality} holds is called the dilation of $f$ on $D$. The dilation is the maximal factor with which infinitesimal circles can be dilated or extremal distance can increase or decrease, see~\cite[Chapter~VII.3]{gm2005harmonic}, when transformed by $f$. We say that $f$ is quasiconformal if it is $M$-quasiconformal for some $M \geq 1$. Moreover, it holds that $f$ is a conformal map if and only if $f$ is a $1$-quasiconformal map. For more material on quasiconformal maps as well as other equivalent definitions, see for example~\cite{ahlfors1966quasiconformal,gm2005harmonic,pommerenke1992boundary}.

The central notion in this paper will be that of (quasi)conformal removability. We say that a set $K \subseteq D$ is \emph{(quasi)conformally removable inside a domain} $D$ if any homeomorphism $f$ which is (quasi)conformal on $D \setminus K$ is (quasi)conformal on $D$. One can apply the measurable Riemann mapping theorem (see~\cite{ahlfors1966quasiconformal}) to see that any planar set is quasiconformally removable if and only if it is conformally removable.

To show that a set $K \subseteq D$ of zero (two-dimensional) Lebesgue measure is conformally removable in $D$, then it suffices to show that every homeomorphism $f$ on $D$ which is conformal on $D \setminus K$ is ACL. Indeed, the fact that $f$ is conformal on $D \setminus K$ implies that~\eqref{eq:quasiconformal_inequality} holds a.e.\ with $M = 1$. If in addition, $f$ is ACL, then $f$ is $1$-quasiconformal and hence conformal.

\subsection{$\SLE_\kappa$ and $\SLE_{\kappa}(\underline{\rho})$ processes}
\label{subsec:sle}

\subsubsection{Chordal $\SLE$}
As we mentioned in the introduction, $\SLE_{\kappa}$ is a one-parameter family of conformally invariant random curves indexed by $\kappa > 0$ introduced by Schramm in~\cite{s2000sle} as a candidate for the scaling limit of the interfaces in discrete lattice models in two dimensions at criticality.  $\SLE_{\kappa}$ in $\h$ from $0$ to $\infty$ is defined in terms of the family of conformal maps $(g_t)$ which are obtained by solving the chordal Loewner equation
\begin{equation}
\label{eqn:loewner_equation}
\partial_t g_t(z) = \frac{2}{g_t(z) - W_t},\quad g_0(z) = z \in \h
\end{equation}
with $W = \sqrt{\kappa}B$ and $B$ a standard Brownian motion.  We write $K_t = \{z \in \h : \tau_z \leq t\}$ where $\tau_z = \sup\{t \geq 0 : \im(g_t(z)) > 0\}$.  Then $g_t$ is the unique conformal transformation from $\h_t = \h \setminus K_t$ onto $\h$ satisfying $\lim_{z \to \infty}|g_t(z) - z| = 0$.  It is shown in \cite{rs2005basic} for $\kappa \neq 8$ that the family of hulls $(K_t)$ is generated by a continuous curve which means there exists a curve $\eta \colon [0,\infty) \to \closure{\h}$ so that $\h_t$ is equal to the unbounded component of $\h \setminus \eta([0,t])$ for each $t \geq 0$.  The analogous result for $\kappa = 8$ was proved in \cite{lsw2004lerw} as a consequence of the convergence of the uniform spanning tree Peano curve to $\SLE_8$ (see also \cite{am2022sle8} for another proof which does not use discrete models).

An $\SLE_{\kappa}$ connecting boundary points $x$ and $y$ of an arbitrary simply connected domain $D$ is defined as the image of an $\SLE_{\kappa}$ in $\h$ from $0$ to $\infty$ under a conformal transformation $\varphi : \h \rightarrow D$ sending $0$ to $x$ and $\infty$ to $y$.  Moreover, it was shown in \cite{rs2005basic} that $\SLE_{\kappa}$ is a.s.\ a simple curve for $\kappa \in (0,4]$ intersecting the domain boundary only at its starting and ending points while for $\kappa \geq 8$ the curve is a.s.\ space-filling.  For $\kappa \in (4,8)$,  $\SLE_\kappa$ intersects itself and the domain boundary but is not space-filling.

The $\SLE_\kappa(\ul{\rho})$ processes are variants of $\SLE_\kappa$ in which one keeps track of extra marked points \cite[Section~8.3]{lsw2003confres} called force points.  Suppose that we have vectors $\underline{\rho}^L = (\rho^{\ell,  L},\ldots,\rho^{1,L})$ and $\underline{\rho}^R = (\rho^{1,R},\ldots,\rho^{r,R})$ (which give the weights of the force points to the left and right of $0$, respectively) and $\underline{x}^L = (x^{\ell,  L}<\ldots<x^{1,L} \leq 0)$ and $\underline{x}^R = (0\leq x^{1,R}<\ldots<x^{r,R})$ (which give the locations of the force points to the left and right of $0$).  Then an $\SLE_\kappa(\ul{\rho}) = \SLE_\kappa(\ul{\rho}^L;\ul{\rho}^R)$ process is defined in the same way as ordinary $\SLE_\kappa$ except with~$W$ taken to solve
\begin{align*}
dW_t = \sqrt{\kappa}dB_t + \sum_{q \in \{L,R\}}\sum_i \frac{\rho^{i,q}dt}{W_t - V_t^{i,q}},\quad
dV_t^{i,q} = \frac{2dt}{V_t^{i,q} - W_t},\quad V_0^{i,q} = x_i^q,\quad \text{for}\quad q \in \{L,R\}.
\end{align*}
It was shown in \cite{ms2016imag1} that there is a unique solution to this SDE up to the first $t$ such that $\sum_{i : V_t^{i,L} = W_t}\rho^{i,L} \leq -2$  or $\sum_{i : V_t^{i,R} = W_t}\rho^{i,R} \leq -2$.  This time $t$ is the so-called \emph{continuation threshold}.  Moreover, it was shown in \cite{ms2016imag1} that the $\SLE_\kappa(\ul{\rho})$ processes are generated by a continuous curve up to the continuation threshold.  The continuity of the (single force point) $\SLE_\kappa(\rho)$ processes with $\rho \in (-2-\kappa/2,-2)$ was proved in \cite{msw2017cleperc,ms2019gfflightcone}.  This gives the full range of $\rho$ values in which an $\SLE_\kappa(\rho)$ process can be defined.

\subsubsection{Radial and whole-plane $\SLE$}
We will also need to consider the radial and whole-plane $\SLE_\kappa(\rho)$ processes.  They are defined in terms of the radial form of the Loewner equation
\begin{equation}
\label{eqn:radial_loewner}
\partial_t g_t(z) = g_t(z) \frac{W_t + g_t(z)}{W_t-g_t(z)},\quad g_0(z) = z
\end{equation}
where $W \colon [0,\infty) \to \partial \D$ is a continuous function.  In the case of radial $\SLE_\kappa$, one takes $W = e^{i \sqrt{\kappa} B}$ where $B$ is a standard Brownian motion.  Let
\[ \Psi(z,w) = -z \frac{z+w}{z-w} \quad\text{and}\quad \wt{\Psi}(z,w) = \frac{\Psi(z,w) + \Psi(1/\ol{z},w)}{2}.\]
Fix $\rho \in \R$.  For radial $\SLE_\kappa(\rho)$, one instead takes $W$ to be the solution to the SDE
\begin{equation}
\label{eqn:radial_sle_kappa_rho}
dW_t = \left( -\frac{\kappa}{2} W_t + \frac{\rho}{2} \wt{\Psi}(O_t,W_t) \right) dt + i \sqrt{\kappa} W_t dB_t,\quad dO_t = \Psi(W_t,O_t) dt.
\end{equation}
By a comparison to a Bessel process, the SDE~\eqref{eqn:radial_sle_kappa_rho} has a solution which is defined for all $t \geq 0$ provided $\rho > -2$.  For each $z \in \D$ we let $\tau_z = \sup\{ t \geq 0 : |g_t(z)| < 1\}$ and $K_t = \{z \in \D : \tau_z \leq t\}$.  Then $g_t$ is the unique conformal transformation $\D \setminus K_t \to \D$ with $g_t(0) = 0$ and $g_t'(0) > 0$.  In fact, due to the normalization in~\eqref{eqn:radial_loewner} one has that $g_t'(0) = e^t$ for each $t \geq 0$.

Whole-plane $\SLE_\kappa$ and $\SLE_\kappa(\rho)$ are defined by solving~\eqref{eqn:radial_loewner} for all $t \in \R$ where in the former case $W = e^{i \sqrt{\kappa} B_t}$ for a two-sided Brownian motion $B$ and in the latter case $W$ is taken to be a stationary solution to~\eqref{eqn:radial_sle_kappa_rho} defined for all $t \in \R$.  In this case, for each $t \in \R$ we have that $g_t$ is the unique conformal transformation from $\C \setminus K_t$ to $\C \setminus \closure{\D}$ which fixes and has positive derivative at $\infty$ and $(K_t)$ is an increasing family of compact hulls in $\C$ starting from $0$.  That radial and whole-plane $\SLE_\kappa(\rho)$ processes correspond to a continuous curve for $\rho > -2$ follows from the continuity of the chordal $\SLE_\kappa(\rho)$ processes (see \cite[Section~2]{ms2017ig4}).

\subsubsection{Two-sided whole-plane $\SLE$}
Finally, we will also consider whole-plane two-sided $\SLE_\kappa$ from~$\infty$ to~$\infty$ through~$0$.  This path can be constructed by first sampling a whole-plane $\SLE_\kappa(2)$ process $\eta_1$ from $0$ to $\infty$, then a chordal $\SLE_\kappa$ process $\eta_2$ in $\C \setminus \eta_1$ from $0$ to $\infty$, and then taking the concatenation of the time-reversal of $\eta_1$ together with $\eta_2$.  Roughly speaking, this process describes the local behavior of an $\SLE_\kappa$ near an interior point $z$ when it is conditioned to pass through $z$.

\subsection{Natural parameterization for $\SLE$}
\label{subsec:natural_parameterization}

Implicit in the definition of $\SLE$ in the chordal, radial, and whole-plane cases respectively using~\eqref{eqn:loewner_equation} and~\eqref{eqn:radial_loewner} is the so-called capacity time parameterization.  This time parameterization is natural from the perspective of the Loewner equation in its chordal and radial form.  Another time parameterization for $\SLE$ is the so-called \emph{natural time parameterization}, which is conjectured to be the time parameterization which arises when considering $\SLE$ as the scaling limit of an interface of a discrete model in which the curve is parameterized by the number of edges it traverses.  An indirect construction of the natural parameterization of $\SLE_\kappa$ was given first for $\kappa < 4(7-\sqrt{33})$ in~\cite{ls2011natural} and then for all $\kappa < 8$ in~\cite{lz2013natural}.  In the case that $\kappa \geq 8$, the natural parameterization corresponds to parameterizing the curve according to the amount of Lebesgue measure it fills (this is the natural definition in this case as such $\SLE$ curves are space-filling).  In~\cite{lr2015minkowski}, a direct construction was given, where the authors proved that the natural parameterization of $\SLE_\kappa$ is in fact equal to (a constant times) the $d_\kappa = (1+\kappa/8)$-dimensional Minkowski content of $\SLE_\kappa$ (which they proved a.s.\ exists). Any curve $\eta$ whose law is locally absolutely continuous with respect to the law of an $\SLE$ process has a corresponding natural parameterization, which we will henceforth denote by $\nmeasure{\eta}$.  We view $\nmeasure{\eta}$ as measure on the Borel $\sigma$-algebra of subsets of $\C$ where $\nmeasure{\eta}(A) = \Leb_1(\eta^{-1}(A))$ when $\eta$ has the natural parameterization.

One of the important properties of the natural parameterization is that it is \emph{conformally covariant}.  This means that if $\eta$ is an $\SLE_\kappa$-type curve in $D$, $\varphi \colon D \to \wt{D}$ is a conformal map, and $\wt{\eta} = \varphi(\eta)$ then $\nmeasure{\wt{\eta}}(\varphi(A)) = \int_A |\varphi'(z)|^{d_\kappa} d\nmeasure{\eta}(z)$ for every Borel set $A \subseteq \C$.

Let us now recall the definition of Minkowski content.  Consider a set $S \subseteq \C$ and $d \in (0,2)$.  For $r > 0$ we set
\begin{align*}
\cont_d (S;r) &=e^{r(2-d)} \Leb_2(\{z : \dist(z,S) \leq e^{-r}\})
=e^{r(2-d)}\int_{\C}\one_{\{\dist(z,S) \leq e^{-r}\}}d\Leb_2(z).
\end{align*}
The $d$-dimensional Minkowski content of $S$ is defined as
\begin{align*}
\cont_d(S) = \lim_{r \to \infty} \cont_d(S;r),
\end{align*}
provided that the limit exists.

Suppose that $\eta$ is a two-sided whole-plane $\SLE_\kappa$ process from $\infty$ to $\infty$ through $0$ with the natural parameterization with time normalized so that $\eta(0) = 0$.  It was proved in \cite{zhan2019holder} that the law of $\eta$ is shift invariant in the sense that for each $t \in \R$ the law of $\eta^t = \eta(\cdot + t) - \eta(t)$ has the same law as $\eta$.  This fact will be important for us in Section~\ref{sec:initial_estimates}.

\subsection{Gaussian free fields}
\label{subsec:gff}

Let $D \subseteq \C$ be a simply connected domain with harmonically non-trivial boundary and let $H_0(D)$ be the Hilbert space closure of $C_0^{\infty}(D)$ with respect to the Dirichlet inner product $(f,g)_{\nabla} = \frac{1}{2\pi}\int_D \nabla f(z) \cdot \nabla g(z) d\Leb_2(z)$.  Then  the zero boundary Gaussian free field (GFF) $h$ is the random distribution defined by 
\begin{equation}
\label{eqn:gff_definition}
h = \sum_{n \geq 1} \alpha_n \phi_n,
\end{equation}
where $(\phi_n)_{n \geq 1}$ is an orthonormal basis of $H_0(D)$ with respect to $(\cdot,\cdot)_{\nabla}$ and $(\alpha_n)_{n \geq 1}$ is a sequence of i.i.d.\ $N(0,1)$ random variables.  The convergence in~\eqref{eqn:gff_definition} holds a.s.\ in $H^{-\epsilon}(D)$ for each $\epsilon > 0$ \cite{she2007gff}.  

Other variants of the GFF are defined using a series expansion as in~\eqref{eqn:gff_definition}.  For example, the free boundary GFF is defined in the same way except that we replace $H_0(D)$ by the closure $H(D)$ with respect to $(\cdot,\cdot)_{\nabla}$ of the space of functions $f \in C^{\infty}(D)$ with $\|f \|_\nabla < \infty$ such that $\int_D f(z) d\Leb_2(z) = 0$.  Note that in this way,  the free boundary GFF is defined in the space of distributions modulo additive constant.  However,  we can fix the additive constant by fixing the value of the field when acting on a given test function whose mean is non-zero.  Also, the whole-plane GFF is defined in the same way as the free boundary GFF,  but with the orthonormal basis in~\eqref{eqn:gff_definition} being that of the Hilbert space closure with respect to $(\cdot,\cdot)_{\nabla}$ of the set of $f \in C_0^{\infty}(\C)$ satisfying $\int_{\C}f(z) d\Leb_2(z) = 0$.  Then the field is defined modulo additive constant but we can fix its additive constant as described above.

\subsection{Liouville quantum gravity}
\label{subsec:lqg}

\subsubsection{LQG measure and surfaces}

Fix $\gamma \in (0,2]$. As mentioned in Section~\ref{sec:intro}, a $\gamma$-LQG surface is formally a random two-dimensional Riemannian manifold with metric tensor
\[ e^{\gamma h(z)} (dx^2 + dy^2)\]
where $dx^2 + dy^2$ is the Euclidean metric, $h$ is some form of the GFF on a planar domain $D$, and $\gamma \in (0,2]$.  Since $h$ is a distribution and not a function, this expression requires interpretation.  For $\gamma \in (0,2)$, the associated $\gamma$-LQG area measure with respect to $h$ on $D$ can be defined as 
\begin{equation}
\label{eqn:liouville_area-measure}
\qmeasure{h}(dz) = \lim_{\epsilon \to 0} \epsilon^{\gamma^2 / 2}e^{\gamma h_{\epsilon}(z)} d\Leb_2(z),
\end{equation}
where $h_{\epsilon}(z)$ denotes the average of $h$ on $\partial B(z,\epsilon)$ for $0 < \epsilon < \dist(z,\partial D)$ (see \cite{ds2011lqg}).  In order to obtain a non-trivial limit in the case $\gamma=2$, one instead considers \cite{drsv2014critical2,drsv2014critical1}
\begin{equation}
\label{eqn:liouville_area-measure_critical}
\qmeasure{h}(dz) = \lim_{\epsilon \to 0} (\log \epsilon^{-1})^{1/2} \epsilon^{2}e^{2 h_{\epsilon}(z)} d\Leb_2(z).
\end{equation}
In the case that $h$ is (some form of) a GFF with free boundary conditions, the $\gamma$-LQG boundary length measure is defined in an analogous way.  We note that it is also possible to define the measure for $\gamma=2$ as a renormalized limit as $\gamma \uparrow 2$ of the LQG measure for $\gamma \in (0,2)$ \cite{aps2019critical}.

Suppose that $h$ is an instance of (some form of) the GFF on a domain $D \subseteq \C$ and $\qmeasure{h}$ is as above.  Suppose that $\wt{D} \subseteq \C$ is another domain, $\psi : \wt{D} \to D$ is a conformal map, and set
\begin{equation}
\label{eqn:coordinate_change}
\wt{h} = h \circ \psi + Q\log |\psi'| \quad\text{where}\quad Q = \frac{2}{\gamma} + \frac{\gamma}{2}.
\end{equation}
Then we have that $\qmeasure{\wt{h}}(A) = \qmeasure{h}(\psi(A))$ for all $A \subseteq D$ Borel.  This leads to the definition of a \emph{quantum surface}, which is an equivalence class of pairs $(D,h)$, where $D \subseteq \C$ is a domain and $h \in H^{-1}(D)$ is a random distribution on $D$, where two such pairs are equivalent if there exists a conformal map $\psi \colon \wt{D} \to D$ so that $h$, $\wt{h}$ are related as in~\eqref{eqn:coordinate_change}.  We can also consider quantum surfaces $(D,h,x_1,\ldots,x_n)$,  $(\wt{D},\wt{h},\wt{x}_1,\ldots,\wt{x}_n)$ with marked points.  In this case, they are considered to be equivalent if they satisfy~\eqref{eqn:coordinate_change} and $\psi(\wt{x}_j) = x_j$ for all $j=1,\ldots,n$.

\subsubsection{Quantum wedges and cones}

We now proceed to give the definition of a \emph{quantum wedge} and a \emph{quantum cone}, which are the two types of quantum surfaces which will be relevant for this work.  The former (resp.\ latter) is a doubly marked surface which is naturally parameterized by either~$\h$ or the infinite strip $\strip = \R \times (0,\pi)$ (resp.\ $\C$ or the infinite cylinder $\cyl = \R \times [0,2\pi)$, with the top and bottom identified).  Quantum wedges and cones are each in fact one parameter families of quantum surfaces.  One convenient choice of parameterization is the \emph{weight} because for wedges it is compatible with the welding operation.  In this paper, we will only consider the case $\gamma=2$ and the case of weight $\tfrac{\gamma^2}{2}=2$ quantum wedges and weight $4$ quantum cones.

\begin{definition}
\label{def:weight2wedge}
A \emph{quantum wedge of weight $\frac{\gamma^2}{2}$} is the doubly marked quantum surface $(\strip,h,-\infty,+\infty)$ whose law can be sampled from as follows.  Let $X_t$ for $t \in \R$ be the process which for $t \leq 0$ is equal to $-X_{-2t}^1$ where $X^1$ has the law of a $\BES^3$ and for $t \geq 0$ is equal to $X_{2t}^2$ where $X^2$ is a standard Brownian motion where $X_0^1 = X_0^2 = 0$ and $X^1,X^2$ are independent.  Then for each $t \in \R$ the average of $h$ on the vertical line $\{ t \} \times (0, \pi)$ is equal to $X_t$ and the law of $h_2 = h - X_{\re(\cdot)}$ is that of the projection of a free boundary GFF onto the space of functions in $H(\strip)$ which have mean zero on vertical lines where $X$, $h_2$ are taken to be independent.
\end{definition}

We note that embeddings of a quantum surface parameterized by $\strip$ with marked points at $-\infty$, $+\infty$ differ only by a horizontal translation.  The particular choice of embedding in Definition~\ref{def:weight2wedge} is the so-called \emph{circle average embedding} of a quantum wedge.  Applying the conformal map $\strip \to \h$ given by $z \mapsto \exp(z)$ gives the law of a quantum wedge parameterized by $\h$ instead of $\strip$.

\begin{definition}
\label{def:weight4cone}
Fix $W > 0$.  A \emph{quantum cone of weight $W$} is the doubly marked quantum surface $(\cyl,h,-\infty,+\infty)$ whose law can be sampled from as follows.  Let $a = W/(2\gamma)$ and let $X_t : \R \to \R$ be the process $B_t + at$ where $B_t$ is a standard two-sided Brownian motion conditioned such that $B_t + a t > 0$ for all $t > 0$.  Then for each $t \in \R$ the average of $h$ on the vertical line $\{t\} \times (0, 2\pi)$ is equal to $X_t$ and the law of $h_2 = h - X_{\re(\cdot)}$ is that of the projection of a whole-plane GFF onto the space of functions in $H(\cyl)$ which have mean zero on vertical lines where $X$, $h_2$ are taken to be independent.
\end{definition}

We note that embeddings of a quantum surface parameterized by $\cyl$ with marked points at $-\infty$, $+\infty$ differ only by a horizontal translation.  The particular choice of embedding in Definition~\ref{def:weight4cone} is the so-called \emph{circle average embedding} of a quantum cone.  Applying the conformal map $\cyl \to \C$ given by $z \mapsto \exp(z)$ gives the law of a quantum cone parameterized by~$\C$ instead of~$\cyl$.

A number of results regarding the cutting and welding of quantum wedges and cones were proved in \cite{she2016zipper,dms2021mating} for $\gamma \in (0,2)$ and for $\gamma=2$ in \cite{hp2021welding}.  In this work, we will need to make use of an extension of one of the results in \cite{dms2021mating} to the case $\gamma=2$ which was proved in \cite{kms2021regularity}.

\begin{theorem}
\label{thm:critical_cone_cutting}
Suppose that $(\C,h,0,\infty)$ is a weight $4$ quantum cone and $\eta$ is a two-sided whole-plane $\SLE_4$ from $\infty$ to $\infty$ through $0$.  Then the quantum surfaces parameterized by the two components of $\C \setminus \eta$ and marked by $0$ and $\infty$ are independent quantum wedges of weight $2$.
\end{theorem}

\subsection{Level lines}
\label{subsec:level_lines}

Let $\lambda = \pi/2$ and suppose that $h$ is an instance of the GFF on $\h$ with boundary conditions given by $-\lambda$ (resp.\ $\lambda$) on $\R_-$ (resp.\ $\R_+$).  Even though $h$ is a distribution and not a function, it is shown in~\cite{schramm2013contour} that it is possible to make sense of a zero level line $\eta$ of $h$ from $0$ to $\infty$ and its law is that of a chordal $\SLE_4$.  Moreover, $\eta$ is a path-valued function of $h$ which is characterized by the property that for each $t \geq 0$  the conditional law of $h$ given $\eta|_{[0,t]}$ is that of a GFF in $\h \setminus \eta([0,t])$ with boundary conditions given by $-\lambda$ (resp.\ $\lambda$) on the left (resp.\ right) side of $\eta([0,t])$ and $\R_-$ (resp.\ $\R_+$).

In~\cite{wang2017level},  the methods introduced in~\cite{schramm2013contour} were generalized to the setting in which the GFF has piecewise constant boundary data which changes values at most finitely many times.  More precisely, suppose that we have fixed force points $(\underline{x}^L;\underline{x}^R)$ and weights $(\underline{\rho}^L;\underline{\rho}^R)$ as in the definition of $\SLE_4(\ul{\rho}^L;\ul{\rho}^R)$.  Let $h$ be a GFF on $\h$ with boundary conditions given by $-\lambda (1 + \sum_{i=0}^j \rho^{i,L})$ for $x \in [x^{j+1,L},x^{j,L})$ and $\lambda ( 1 + \sum_{i=0}^j \rho^{i,R})$ for $x \in [x^{j,R},x^{j+1,R})$ where $\rho^{0,L} = \rho^{0,R} = 0$,  $x^{0,L} = 0^-$,  $x^{\ell + 1,L} = -\infty$,  $x^{0,R} = 0^+$, and $x^{r+1,R} = \infty$.  Then the zero level line of~$h$ from~$0$ to~$\infty$ has the law of an $\SLE_4(\ul{\rho}^L; \ul{\rho}^R)$ in~$\h$ from~$0$ to~$\infty$.  As above, $\eta$ is a path valued function of the field which is characterized by the property that for each $t \geq 0$  the conditional law of $h$ given $\eta|_{[0,t]}$ is that of a GFF in $\h \setminus \eta([0,t])$ with boundary conditions given by $-\lambda$ (resp.\ $\lambda$) on the left (resp.\ right) side of $\eta([0,t])$ and the same boundary values as $h$ on $\partial \h$.

When $(h,\eta)$ are as above, we refer to $\eta$ as the height $0$ level line of~$h$ from~$0$ to $\infty$.  More generally, for each $u \in \R$ the height $u$ level line of $h$ is given by the $0$ height level line of $h-u$ from $0$ to $\infty$. Note that then the boundary data for $h-u$ along such a level line is $-\lambda$ (resp.\ $\lambda$) on its left (resp.\ right) side hence the boundary data for $h$ is given by $u-\lambda$ (resp.\ $u+\lambda$) on its left (resp.\ right) side.  The results of \cite{wang2017level} moreover describe the interaction of GFF level lines starting from different points and serves to generalize the flow line interaction rules established in \cite{ms2016imag1}. (We note that, in~\cite{wang2017level}, they call the level line of $h+u$ the height $u$ level line, rather than the height $-u$ level line and hence the following interaction result is mirrored, compared to~\cite{wang2017level}.)

\begin{theorem}
\label{thm:interaction_rules}
Suppose that $h$ is a GFF on $\h$ whose boundary value is piecewise constant.  For each $u \in \R$ and each $x \in \partial \h$,  we let $\eta_u^x$ be the level line of $h$ with height $u$ starting from $x$.  Fix $x_1 \leq x_2$.
\begin{enumerate}
\item If $u_2 > u_1$  then $\eta_{u_2}^{x_2}$ a.s.\ stays to the right of $\eta_{u_1}^{x_1}$.
\item If $u_2 = u_1$ then $\eta_{u_2}^{x_2}$ may intersect $\eta_{u_1}^{x_1}$ at the boundary and, upon intersecting,  the two curves merge and do not subsequently separate.
\item If $u_2 - u_1 \geq 2 \lambda$ then $\eta_{u_1}^{x_1}$, $\eta_{u_2}^{x_2}$ a.s.\ do not intersect (except at possibly their starting point).
\end{enumerate}
\end{theorem}

Finally in \cite{wang2017level},  the authors proved the following about the reversibility of the level lines of the GFF.

\begin{theorem}
\label{thm:level_lines_reversibility}
Suppose that $h$ is a GFF on $\h$ whose boundary value is piecewise constant.  Let $\eta$ be the level line of $h$ starting from $0$ targeted at $\infty$ and let $\wt{\eta}$ be the level line of $-h$ starting from $\infty$ targeted at $0$.  Then,  on the event that the two paths do not hit the continuation threshold before they reach the target points,  the two paths $\wt{\eta}$ and $\eta$ are equal (viewed as sets) a.s.
\end{theorem}

\section{Estimates for $\SLE_4$}
\label{sec:initial_estimates}

This section is devoted to the proof of two regularity estimates for $\SLE_4$.

\begin{lemma}
\label{lem:good_fraction_whole_plane}
Suppose that we have a two-sided whole-plane $\SLE_4$ process $\eta^w$ which has the natural parameterization and is normalized so that $\eta^w(0) = 0$.  For each $a ,\delta, p \in (0,1)$ there exists $C > 0$ so that the following is true.  Let $V$ be the component of $\C \setminus \eta^w$ which contains $1$, let $\CW$ be a Whitney square decomposition of $V$, and let $Q_0 \in \CW$ be the square which contains $1$.  Let~$\CT$ be the set of times $t \in [0,1]$ so that the hyperbolic geodesic from $1$ to $\eta^w(t)$ in $V$ satisfies the property that for every $Q \in \CW$ which it hits we have that $\disthyp^V(\cen(Q),1) \leq C\len(Q)^{-a}$.  Then $\p[ \Leb_1(\CT) \geq 1-\delta] \geq p$.
\end{lemma}
\begin{proof}
\noindent{\it Step 1. Setup.}  We assume that we have the notation and setup described in the lemma statement.  We also let $U$ be uniform in $[0,1]$ independently of $\eta^w$ and let $(\C,h,0,\infty)$ be a quantum cone of weight $4$ with the circle average embedding which is independent of $(\eta^w ,  U)$.  We then consider the translation map $\psi_{U} \colon \C \to \C$ given by $z \mapsto z+\eta^w(U)$ and set $h^U = h \circ \psi_{U}^{-1}$ (note that $(\psi_U^{-1})' \equiv 1$).  Then the quantum surface $(\C,h^U,  \eta^w (U) , \infty)$ is a quantum cone of weight $4$ with the circle average embedding since the quantum surface $(\C,h,0,\infty)$ has the law of a quantum cone of weight $4$ with the circle average embedding and it is independent of $(\eta^w, U)$.  We also let $V_1$ (resp.\ $V_2$) be the connected component of $\C \setminus \eta^w$ lying to the right (resp.\ left) side of $\eta^w$, and for $j=1,2$ we consider the conformal transformation $\phi_{U}^j : V_j \rightarrow \strip$ such that $\phi_{U}^j(\eta^w(U)) = -\infty$,  $\phi_{U}^j(\infty) = +\infty$ and set $\wt{h}^{U,j} = h^U \circ (\phi_{U}^{j})^{-1} + Q\log |((\phi_{U}^{j})^{-1})'|$.  We pick the third parameter in the definition of $\phi_{U}^j$ so that the quantum surface $(\strip, \wt{h}^{U,j}, -\infty,+\infty)$ has the circle average embedding.  Since $(\eta^w(s+U) - \eta^w(U))_{s \in \R}$ has the same law as $\eta^w$, it follows from \cite[Theorem~3.1]{kms2021regularity} that the surfaces
$(\strip,  \wt{h}^{U,1} -\infty,+\infty)$ and $(\strip, \wt{h}^{U,2},-\infty,+\infty)$ are independent and they both have the law of a quantum wedge of weight $2$ parameterized by $\strip$ with the circle average embedding.  Note that $\wt{h}^{U,j}$ can be decomposed as $\wt{h}^{U,j} = h_1^j + h_2^j$ where $h_1^j \in \mathcal{H}_1(\strip)$ and $h_2^j \in \mathcal{H}_2(\strip)$.  We will abuse notation and write~$h_1^j(t)$ for the common value of $h_1^j$ on the vertical segment $\{t \} \times (0,\pi)$.  We also have that $(h_1^j(t))_{t \leq 0}$ (resp.\ $(h_1^j(t))_{t \geq 0}$) has the same law as $(-X^j_{-2t})_{t \leq 0}$ (resp.\ $(B^j_{2t})_{t \geq 0}$) where~$X^j$ (resp.\  $B^j$) is a $\BES^3$ process (resp.\ standard Brownian motion) starting from~$0$ and $X^j$ and $B^j$ are independent of each other.  Moreover~$h_2^j$ is independent of~$h_1^j$ and has the same law as the corresponding projection to $\mathcal{H}_2(\strip)$ of a free boundary GFF on~$\strip$.

\noindent{\it Step 2. Initial estimates.}  Next we fix $q \in (0,1)$ small (to be chosen and depending only on $a$, $\delta$, and $p$).  Let $V_{j_0}$ for $j_0 \in \{1,2\}$ be the connected component of $\C \setminus \eta^w$ containing $1$. Since $\phi_U^{j_0}(1) \in \strip$,  we can find $R \in \N$ and $r \in (0,\pi / 2)$ so that with $K = [-R,R] \times [r, \pi - r]$ we have that $\p[\phi_U^{j_0}(1) \in K] \geq 1-q/100$.  For $k \in \N_0$,  we consider the rectangles $W_k^- = [-k-1,-k] \times [r,\pi - r]$ and $W_k^+ = [k,k+1] \times [r,\pi - r]$.  Let $z_k^-$ (resp.\  $z_k^+$) be the center of $W_k^-$ (resp.\  $W_k^+$).  Let also $\phi \coloneqq (\phi_{U}^{j_0})^{-1}$.  Note that~\cite[Theorem~3.21]{lawler2008conformally} implies that there exist finite universal constants $c_1,  c_2 >0$ depending only on $r$ such that
\begin{equation}
\label{eq:derivative_bound}
c_1 |\phi'(z_k^-)| \leq |\phi'(z)| \leq c_2 |\phi'(z_k^-)| \quad \text{for all} \quad z \in W_k^- \quad \text{and} \quad k \in \N_0
\end{equation}
and a similar statement holds with $z_k^-$ and $W_k^-$ replaced by $z_k^+$ and $W_k^+$.  We also claim that we can find finite constants $c_3,  c_4 >0$ depending only on $r$ such that for all $k \in \N_0$ we have 
\begin{equation}
\label{eq:bound_on_the_side_length}
c_3 |\phi'(z_k^-)| \leq \len(Q) \leq c_4 |\phi'(z_k^-)| \quad \text{for}\quad Q \in \CW,\ Q \cap \phi(W_k^-) \neq \emptyset 
\end{equation}
and a similar statement holds if $Q \cap \phi(W_k^+) \neq \emptyset$ when we replace $z_k^-$ and $W_k^-$ by $z_k^+$ and $W_k^+$.  Indeed, let $Q \in \CW$ be such that $Q \cap \phi(W_k^-) \neq \emptyset$ and fix $w \in Q \cap \phi(W_k^-)$.  Then by the Koebe-1/4 theorem and~\eqref{eq:derivative_bound} we have that
\begin{align*}
\dist(Q,\partial V_{j_0}) \leq \dist(w , \partial V_{j_0}) \leq 4 c_2 r |\phi'(z_k^-)|.
\end{align*}
Moreover, $\sqrt{2}\len(Q) = \diam (Q) \leq \dist(Q,  \partial V_{j_0})$ and thus $\len(Q) \leq 2\sqrt{2} c_2 r |\phi'(z_k^-)|$.  On the other hand,  we have that
\begin{align*}
\dist(w, \partial V_{j_0}) \leq \dist(Q , \partial V_{j_0}) + \diam (Q)\leq 5 \diam (Q) = 5 \sqrt{2}\len(Q)
\end{align*}
and by the Koebe-1/4 theorem,
\begin{align*}
\dist(w,  \partial V_{j_0}) \geq \frac{1}{4}\dist(\phi_U^{j_0}(w) ,  \partial \strip) |\phi'(\phi_U^{j_0}(w))| \geq \frac{c_1 r}{4} |\phi'(z_k^-)|.
\end{align*}
Consequently, $\len(Q) \asymp |\phi'(z_k^-)|$ with implicit constants depending only on $r$.  This proves~\eqref{eq:bound_on_the_side_length} whenever $Q \cap \phi(W_k^-) \neq \emptyset$ and similarly whenever $Q \cap \phi(W_k^+) \neq \emptyset$.  Moreover,  there exist finite constants $c_5,  c_6 > 0$ depending only on $r$ such that 
\begin{equation}\label{eq:bound_on_diameter}
c_5 |\phi'(z_k^-)| \leq \diam(\phi(W_k^-)) \leq c_6 |\phi'(z_k^-)|
\end{equation}
and similarly with $z_k^-$ and $W_k^-$ replaced by $z_k^+$ and $W_k^+$.  To see this,  we first note that~\cite[Corollary~3.25]{lawler2008conformally} implies that there exists a finite constant $c_*>0$ depending only on $r$ such that $B(\phi(z_k^-),c_*|\phi'(z_k^-)|) \subseteq \phi(W_k^-)$ and so $\diam(\phi(W_k^-)) \gtrsim |\phi'(z_k^-)|$.  Also~\eqref{eq:bound_on_the_side_length} implies that $\diam(\phi(W_k^-)) \lesssim |\phi'(z_k^-)|$ where the implicit constant depends only on $r$.

\noindent{\it Step 3. Quantum area lower bounds.} Next, we focus on giving lower bounds on $\qmeasure{\wt{h}^{U,j}}(W_k^\pm)$ for $k \in \N_0$ and $j =1,2$. By considering a fine grid of points in $B(\eta^w(U),1)$ and applying~\cite[Lemma~4.3]{kms2021regularity} it follows that there exist $\epsilon_0 \in (0,1)$ and $u>0$ such that with probability at least $1-q/100$ we have that 
\begin{equation}\label{eq:lower_bound_on_quantum_mass}
\qmeasure{h^U}(B(w,\epsilon)) \geq \epsilon^{u}\quad \text{for all} \quad w \in B(\eta^w(U) ,  1/2)\quad \text{and all} \quad \epsilon \in (0,\epsilon_0].
\end{equation}
Also, we have that $\dist(\eta^w(U) ,  \phi(W_k^-)) \to 0$ as $k \to \infty$ and so combining with~\eqref{eq:bound_on_diameter} implies that there exists $k_0 \in \N_0$ such that with probability at least $1-q/100$ we have that $c_*|\phi'(z_k^-)| \leq \epsilon_0$ for all $k \geq k_0$.  Moreover, since $\qmeasure{\wt{h}^{U,j}}(W_k^-) > 0$ and $\qmeasure{\wt{h}^{U,j}}(W_k^+) > 0$ for all $k \in \N_0$ and $j=1,2$ a.s.,  we can find a finite constant $C_1 > 0$ such that with probability at least $1-q/100$ we have that $|\phi'(z_k^-)|^u \leq C_1 \qmeasure{\wt{h}^{U,j_0}}(W_k^-)$ for all $0 \leq k \leq k_0$ and $|\phi'(z_k^+)|^u \leq C_1 \qmeasure{\wt{h}^{U,j_0}}(W_k^+)$ for all $0 \leq k \leq R$.  Combining,  we obtain that we can find constants $u, C_1 >0$ such that with probability at least $1-q/3$ we have that $\phi_U^{j_0}(1) \in K$,
\begin{align*}
|\phi'(z_k^-)|^u &\leq C_1 \qmeasure{h^U}(B(\phi(z_k^-),c_*|\phi'(z_k^-)|))
 \leq C_1 \qmeasure{h^U}(\phi(W_k^-)) = C_1 \qmeasure{\wt{h}^{U,j_0}}(W_k^-) \quad \text{for all} \quad k \in \N_0
\end{align*}
and
\begin{align*}
|\phi'(z_k^+)|^{u} &\leq C_1 \qmeasure{h^U}(B(\phi(z_k^+),c_*|\phi'(z_k^+)|)) \leq C_1 \qmeasure{h^U}(\phi(W_k^+)) = C_1 \qmeasure{\wt{h}^{U,j_0}}(W_k^+) \quad \text{for all} \quad 0 \leq k \leq R.
\end{align*}

\noindent{\it Step 4. Quantum area upper bounds.} Now, we turn to give some high probability upper bounds on $\qmeasure{\wt{h}^{U,j}}(W_k^\pm)$.  We note that the ideas are similar to those of the proof of~\cite[Lemma~A.1]{kms2021regularity}. First,  we observe that
\begin{align*}
\qmeasure{\wt{h}^{U,j}}(W_k^-) &\leq \exp\left(2\sup_{-k-1\leq t \leq -k}  h_1^j(t) \right)\qmeasure{h_2^j}(W_k^-)
   = \exp\left(-2\inf_{2k \leq t \leq 2(k+1)} X_t^j  \right)\qmeasure{h_2^j}(W_k^-)\\
&=\exp\left(-2\inf_{2k \leq t \leq 2(k+1)} (X_t^j - X_{2k}^j) \right) \exp(-2X_{2k}^j) \qmeasure{h_2^j}(W_k^-).
\end{align*}
By the Markov property of $X^j$ it follows that $(X_t^j - X_{2k}^j)_{t \geq 2k}$ and $X_{2k}^j$ are independent and that $-\inf_{2k \leq t \leq 2(k+1)} (X_t^j - X_{2k}^j)$ is stochastically dominated by $\wt{S}^j = \sup_{0 \leq t \leq 2} \wt{B}_t^j$ where $\wt{B}_t^j$ is a standard Brownian motion with $\wt{B}_0^j = 0$. Consequently, it follows by~\cite[Lemma A.4]{kms2021regularity} that
\begin{align*}
\E\!\left[ e^{2s\wt{S}^j} \qmeasure{h_2^j}(W_k^-)^s \right] = \E\!\left[ e^{2s\wt{S}^j}\qmeasure{h_2^j}(W_0)^s \right] < \infty \quad\text{for each}\quad s \in (0,1/2).
\end{align*}
Thus Markov's inequality implies that 
\begin{align*}
\p\!\left[ e^{2\wt{S}^j}\qmeasure{h_2^j}(W_k^-) \geq k^{2/s} \right] \lesssim k^{-2},
\end{align*}
where the implicit constant is independent of $k$.

Next,  to treat the term $e^{2X_{2k}^j}$,  we first note that the transition density for $X^j$ is given by $p_t(0,y) = \sqrt{\frac{2}{\pi}}t^{-3/2}y^2 e^{-y^2 / 2t}$ \cite[Chapter~XI]{revuz2013continuous}.  Fix $b \in (0,1/6)$.  Since $X^j$ satisfies Brownian scaling we have that
\begin{align*}
\p\!\left[ X_{2k}^j \leq k^b / 2 \right] = \p\!\left[ X_1^j \leq k^{b-1/2}/ 2^{3/2}\right] \lesssim k^{3(b-1/2)},
\end{align*}
where the implicit constant is independent of $k$.  Therefore we obtain that
\begin{align*}
\p \!\left[ \qmeasure{\wt{h}^{U,j}}(W_k^-) \geq k^{2/s}e^{-k^b} \right] &\leq \p\!\left[ e^{2\wt{S}^j}\qmeasure{h_2^j}(W_k^-) \geq k^{2/s} \right] + \p\!\left[ X_{2k}^j \leq k^b  / 2\right]
 \lesssim k^{-2} + k^{3(b-1/2)}
\end{align*}
and so as $3(b-1/2) < -1$ we have that
\begin{align*}
\sum_{k \geq 1} \p\!\left[ \qmeasure{\wt{h}^{U,j}}(W_k^-) \geq k^{2/s}e^{-k^b} \right] < \infty.
\end{align*}
Thus we can find $k_0 \in \N$ such that with probability at least $1-q/100$ we have that $\qmeasure{\wt{h}^{U,j}}(W_k^-) \leq k^{2/s}e^{-k^b}$ for all $k \geq k_0$ and $j=1,2$.  Again,  by possibly decreasing $b>0$, there exists a constant $C_2 > 0$ such that with probability at least $1-q/3$ it holds that $\qmeasure{\wt{h}^{U,j}}(W_k^-) \leq C_2 e^{-k^b}$ for all $k \in \N_0$ and $\qmeasure{\wt{h}^{U,j}}(W_k^+) \leq C_2 e^{-k^b}$ for all $0 \leq k \leq R$ and $j=1,2$.

\noindent{\it Step 5. Conclusion of the proof.}
Combining everything,  we obtain that we can find constants $r,R,u, C_1,  C_2 > 0$ and $b \in (0,1/6)$ such that with probability at least $1-q$ the following hold:
\begin{enumerate}[(i)]
\item \label{it:bound_on_location} $\phi_U^{j_0}(1) \in K$,
\item \label{it:bound_on_mass_from_derivative}$|\phi'(z_k^-)|^{u} \leq C_1 \qmeasure{\wt{h}^{U,j_0}}(W_k^-)$ for all $k \in \N_0$ and $|\phi'(z_k^+)|^{u} \leq C_1 \qmeasure{\wt{h}^{U,j_0}}(W_k^+)$ for all $0 \leq k \leq R$, and
\item \label{it:upper_bound_on_mass}$\qmeasure{\wt{h}^{U,j}}(W_k^-) \leq C_2 e^{-k^b}$ for all $k \in \N_0$ and $\qmeasure{\wt{h}^{U,j}}(W_k^+) \leq C_2 e^{-k^b}$ for all $0 \leq k \leq R$ and $j=1,2$.
\end{enumerate}
From now on we assume that we are working on the event that~\eqref{it:bound_on_location},  \eqref{it:bound_on_mass_from_derivative} and~\eqref{it:upper_bound_on_mass} hold.

Now we are ready to show that every $Q \in \CW$ which intersects the hyperbolic geodesic from $1$ to $\eta^w(U)$ satisfies $\disthyp^{V_{j_0}}(\cen(Q),1) \leq C \len(Q)^{-a}$ for some constant $C > 0$.  We set $z = \phi_U^{j_0}(1) \in K$. If $\im(z) \neq \frac{\pi}{2}$, then $e^{z} \notin i\R$, and so the hyperbolic geodesic in $\h$ between $e^z$ and the origin is given by the arc in $\h$ connecting $e^z$ to the origin of the circle passing through $e^z$ and the origin and whose center lies in $\R$.  If $\im(z) = \frac{\pi}{2}$,  then $e^z \in i\R$ and so the hyperbolic geodesic in $\h$ between $e^z$ and the origin is given by the segment in $i\R$ connecting them.  In either case, we let $\wt{\gamma}_z$ be the image of the said geodesic under $w \mapsto \log(w)$.  It follows that $\wt{\gamma}_z$ is contained in $(-\infty,R] \times [r,\pi-r]$ and that the image $\gamma_U$ of $\wt{\gamma}_z$ under $\phi$ is the hyperbolic geodesic in $V_{j_0}$ from $1$ to $\eta^w(U)$.

From the definition of the hyperbolic metric it is not difficult to see that there exists a constant $C_r > 0$ depending only on $r$ such that for all $k,m \in \N_0$,  $\wt{z} \in W_k^-$ and $\wt{w} \in W_m^-$,  we have that $\disthyp^{\strip}(\wt{z},\wt{w}) \leq C_r (1 + |k-m|)$ (e.g., by considering the hyperbolic length of the segment from~$\wt{z}$ to~$\wt{w}$).  The same also holds with the $W_k^-$'s replaced by the $W_k^+$'s.  Let $Q$ be a square in $\CW$ that $\gamma_U$ intersects and fix $w \in Q \cap \gamma_U$.  Then, either $\wt{z} =\phi_U^{j_0}(w) \in W_k^-$ for some $k \in \N_0$ or $\wt{z} \in W_k^+$ for some $0 \leq k \leq R$.  Suppose that the former holds.  Then
\begin{align*}
\disthyp^{V_{j_0}}(w,1) &= \disthyp^{\strip}(\wt{z},z) \lesssim k.
\end{align*}
Combining~\eqref{eq:bound_on_the_side_length} with~\eqref{it:bound_on_mass_from_derivative} and~\eqref{it:upper_bound_on_mass} we obtain that $e^{k^b} \lesssim \len(Q)^{-u}$.  This implies that if we set $v = \frac{a}{u}$,  then $k \lesssim \len(Q)^{-a}$ since $k \lesssim e^{v k^b}$ and so 
\begin{align*}
\disthyp^{V_{j_0}}(\cen(Q),1) \lesssim k \lesssim \len(Q)^{-a}
\end{align*}
with the implicit constants depending only on $r$,  $R$,  $a$, and $q$.  If $\wt{z} \in W_k^+$ for some $0 \leq k \leq R$,  then a similar analysis holds and so $\disthyp^{V_{j_0}}(\cen(Q),1) \leq C \len(Q)^{-a}$.

To finish the proof of the lemma,  for $t \in [0,1]$ we let $A_t$ be the event that if $\gamma_t$ is the hyperbolic geodesic in $V_{j_0}$ connecting $1$ to $\eta^w(t)$,  then for every $Q$ in $\CW$ such that $\gamma_t \cap Q \neq \emptyset$,  we have that $\disthyp^{V_{j_0}}(\cen(Q),1) \leq C \len(Q)^{-a}$. Then the analysis of the previous paragraphs implies that
\begin{align*}
\E[ \Leb_1(\{t \in [0,1] : A_t^c \text{ occurs}\})] \leq q.
\end{align*}
Markov's inequality therefore implies that
\begin{align*}
\p[ \Leb_1(\CT) < 1-\delta] = \p[ \Leb_1(\{t \in [0,1] : A_t^c \text{ occurs}\})\geq \delta] \leq \frac{q}{\delta}
\end{align*}
and so the claim of the lemma follows by taking $q$ to be sufficiently small.
\end{proof}

\begin{lemma}
\label{lem:natural_good_whole_plane}
Fix $\kappa \in (0,8)$ and let $d_\kappa = 1 + \frac{\kappa}{8}$.  Suppose that $\eta$ is a chordal $\SLE_\kappa$ in $\h$ from $0$ to $\infty$ with the natural parameterization.  For each compact set $K \subseteq \h$, $a > 0$, and $p \in (0,1)$ there exists $C > 0$ such that the following is true.  With probability at least $p$ for every Borel set $A \subseteq K$ we have that $\nmeasure{\eta}(A) = \Leb_1(\eta^{-1}(A)) \leq C \diam(A)^{d_\kappa-a}$.
\end{lemma}
\begin{proof}
Let $R \subseteq \h$ be a rectangle such that $K \subseteq R$ and $d = \dist(K,\partial R) > 0$.  The proof follows closely that of \cite[Theorem~1.2]{rz2017higher}.  For a domain $D \subseteq \h$ and $r > 0$,  we set 
\begin{align*}
\cont_{d_{\kappa}}(\eta \cap D ; r) = r^{d_{\kappa}-2} \Leb_2(\{z \in D : \dist(z,\eta) \leq r\}).
\end{align*}
Then it is shown in \cite{lr2015minkowski} that the limit 
\begin{align}
\label{eqn:cont_limit}
\cont_{d_{\kappa}}(\eta \cap D) = \lim_{r \to 0} \cont_{d_{\kappa}}(\eta \cap D ; r)
\end{align}
exists a.s.\  and it is the natural length of $\eta$ in the domain $D$.

Let $\alpha_{\kappa} = \frac{8}{\kappa}-1$.  For $y \geq 0$,  we define $P_y \colon \R_+ \to \R_+$ by
\begin{equation}
P_y(x) =
\begin{cases} 
y^{\alpha_{\kappa} - (2-d_{\kappa})} x^{2 - d_{\kappa}} \quad&\text{if}\quad x \leq y,\\
x^{\alpha_{\kappa}} \quad&\text{if}\quad x \geq y.
\end{cases}
\end{equation}
Then it is shown in \cite[Theorem~1.1]{rz2017higher} that if $z_0, \ldots, z_n \in \closure{\h}$ are distinct and such that $z_0 = 0$,  $y_k = \im(z_k) \geq 0$ and $\ell_k = \dist(z_k ,  \{z_j : 1 \leq j < k\})$ for $1\leq k \leq n$,  then we can find a constant $C_n < \infty$ depending only on $\kappa$ and $n \in \N$ such that 
\begin{align}
\label{eqn:rz_bound}
\p [ \dist(z_k , \eta) \leq r_k ,\  1\leq k \leq n] \leq C_n \prod_{k=1}^n \frac{P_{y_k}(r_k \wedge \ell_k)}{P_{y_k}(\ell_k)} \quad\text{for all}\quad r_1,\ldots,r_n > 0.
\end{align}
We fix $n \in \N$ and $A \subseteq R$ Borel such that $\diam(A) < \dist(R,\partial \h)$.  First we will find an upper bound for 
\begin{align}
\label{eqn:cont_bound}
r^{n(d_{\kappa}-2)}\p[ \dist(z_k,\eta) < r,\ 1 \leq k \leq n] \quad\text{for}\quad z_1,\ldots,z_n \in A.
\end{align}
By~\eqref{eqn:rz_bound}, we have that~\eqref{eqn:cont_bound} is bounded by 
\begin{align}
\label{eqn:cont_bound2}
r^{n(d_{\kappa}-2)}C_n \prod_{k=1}^n \frac{P_{y_k}(r \wedge \ell_k)}{P_{y_k}(\ell_k)}.
\end{align}
Suppose that there exist $i_1,\ldots,i_k$ such that $r \leq \ell_{i_j}$ for all $1\leq j \leq k$ and $r > \ell_j$ for all $j \in \{1,\ldots,n\} \setminus \{i_1,\ldots,i_k\}$.  Then since
\[ \frac{P_{y}(x_1)}{P_{y}(x_2)} \leq \frac{x_1^{2-d_{\kappa}}}{x_2^{2-d_{\kappa}}} \quad\text{for all}\quad y \geq 0 \quad\text{and}\quad 0 \leq x_1 < x_2\]
we obtain that~\eqref{eqn:cont_bound2} is bounded by
\begin{align*}
C_n r^{n(d_{\kappa}-2)}\prod_{j=1}^k \frac{r^{2-d_{\kappa}}}{\ell_{i_j}^{2-d_{\kappa}}} \leq C_n r^{n(d_{\kappa}-2)}\prod_{i=1}^n \frac{r^{2-d_{\kappa}}}{\ell_i^{2-d_{\kappa}}} = C_n \prod_{i=1}^n \ell_i^{d_{\kappa}-2}.
\end{align*}
If either $r \leq \ell_j$ for all $1\leq j \leq n$ or $r > \ell_j$ for all $1 \leq j \leq n$,  then a similar bound holds.  Set 
\begin{align*}
g(z_1,\ldots,z_n) = \prod_{k=1}^n \ell_k^{d_{\kappa}-2} = \prod_{k=1}^n \min\{|z_k-z_0|,\ldots,|z_k-z_{k-1}|\}^{d_{\kappa}-2}.
\end{align*}
Since $\diam(A) < \dist(R,\partial \h)$,  we obtain that $\ell_k = \min\{|z_k-z_1|,\ldots,|z_k-z_{k-1}|\}$.  Hence for all $1 \leq k \leq n$,  we have that
\begin{align*}
\int_A \min\{|z_k-z_0|,\ldots,|z_k - z_{k-1}|\}^{d_{\kappa}-2} d\Leb_2(z_k)&=\int_A \min\{|z_k - z_1|,\ldots,|z_k-z_{k-1}|\}^{d_{\kappa}-2} d\Leb_2(z_k)\\
&\leq \sum_{j=1}^{k-1}\int_A |z_k-z_j|^{d_{\kappa}-2} d\Leb_2(z_k) \\
&\leq n \int_{|z| \leq \diam(A)}|z|^{d_{\kappa}-2} d \Leb_2(z) \lesssim \diam(A)^{d_{\kappa}}
\end{align*}
where the implicit constant depends only on $n$ and $\kappa$.  This implies that
\begin{align*}
\E[ \cont_{d_\kappa}(\eta \cap A ; r)^n]&=r^{n(d_{\kappa}-2)}\E\left[\left(\int_A \one_{\dist(z,\eta) < r} d\Leb_2(z) \right)^n\right]\\
&=\int_{A^n}r^{n(d_{\kappa}-2)}\p[ \dist(z_j,\eta) < r,\ 1 \leq j \leq n] d\Leb_2(z_1)\cdots d\Leb_2(z_n) \\
&\leq C_n \int_{A^n}g(z_1,\ldots,z_n) d\Leb_2(z_1)\cdots d\Leb_2(z_n)\\
&\lesssim \diam(A)^{nd_\kappa}
\end{align*}
where the implicit constant depends only on $n$ and $\kappa$.  By applying Fatou's lemma and using~\eqref{eqn:cont_limit} we obtain that
\begin{equation}\label{eq:moment_bound_minkowski_content}
\E[\cont_{d_\kappa}(\eta \cap A)^n] \lesssim \diam(A)^{nd_\kappa}.
\end{equation}
Note that a similar analysis gives that $\E[ \cont_{d_\kappa}(\eta \cap R)^n] < \infty$.

Next for $m \in \Z$ we let $\CS_m$ be the set of dyadic squares contained in $R$ with side length $2^{-m}$.  Set $\CS_m = \{S_{m,1},\dots,S_{m,n_m}\}$ and note that $n_m \leq M 2^{2m}$ for all $m \in \Z$,  where $M = \Leb_2(R)$.  Also,  we fix $\epsilon \in (0,1)$ small and for $m \in \Z$,  $1 \leq j \leq n_m$,  we set $X_{m,j} = \cont_{d_{\kappa}}(\eta \cap S_{m,j})$ and let $C_m$ be the event that $X_{m,j} \geq 2^{-m(1-\epsilon)d_{\kappa}}$ for some $j \in \{1,\dots,n_m\}$.  Then combining~\eqref{eq:moment_bound_minkowski_content} with $n_m \leq M 2^{2m}$ and $\diam(S_{m,j}) < \dist(R,\partial \h)$ for all $m$ sufficiently large,  we obtain that
\begin{align*}
\p[C_m] \lesssim 2^{m(2-\epsilon n d_\kappa)}
\end{align*}
for all $m$ large enough,  where the implicit constant depends only on $n$,  $\kappa$ and $R$.  Hence if we take $n$ large enough such that $2-\epsilon n d_{\kappa} < 0$,  we have that $\sum_{m=1}^{\infty}\p[C_m] < \infty$ and so there exists $m_0$ large such that $\p[ \cup_{m=m_0}^{\infty}C_m] < \frac{p}{2}$.  Moreover there exists $N \in \Z_-$ such that $\CS_m = \emptyset$ for all $m \le N$ and since $\E[ \cont_{d_{\kappa}}(\eta \cap R)^n] < \infty$,  we can find a constant $\wt{C} < \infty$ depending only on $n$,  $\kappa$,  $p$ and $R$ such that with probability at least $1-p$ we have that $X_{m,j} \leq \wt{C} \diam(S_{m,j})^{d_{\kappa}-\epsilon d_{\kappa}}$ for all $m \geq N$ and all $1 \leq j \leq n_m$.  From now on,  we assume that we are working on this event.

Finally,  to complete the proof of the lemma,  we let $A \subseteq K$ be a Borel set and we fix $c_1 > 0$ small (to be chosen independently of $A$).  Let $m \in \Z$ be the unique integer such that $2^{-m-1} < c_1 \diam(A) \leq 2^{-m}$.  Let also $S_1,\dots,S_{k_m}$ be the dyadic squares of side length $2^{-m}$ intersecting $A$.  Then $\diam(S_j) = \sqrt{2}2^{-m} \leq c_1 2^{3/2} \diam(A)$ for all $1 \leq j \leq k_m$.  Also $A \subseteq B(x,\diam(A))$ for a fixed $x \in A$ and so $\bigcup_{j=1}^{k_m} S_j \subseteq B(x,(1+c_1 2^{3/2})\diam(A))$.  Thus we have that
\begin{align*}
k_m 2^{-2m} =\Leb_2\!\left(\bigcup_{j=1}^{k_m}S_j \right) \leq \pi (1 + c_1 2^{3/2})^2 \diam(A)^2 \leq \pi(2^{3/2} + c_1^{-1})^2 2^{-2m}
\end{align*}
and so $k_m \leq \pi(2^{3/2}+c_1^{-1})^2$.  Moreover by taking $c_1 > 0$ small enough (depending only on $R$),  we have that $S_j \subseteq R$ and so $S_j \in \mathcal{S}_m$ for all $1\leq j \leq k_m$.  Therefore we have that
\begin{align*}
\cont_{d_{\kappa}}(S_j)\leq \wt{C} \diam(S_j)^{d_{\kappa}(1-\epsilon)} \lesssim 2^{-m d_{\kappa}(1-\epsilon)}
\end{align*}
and so
\begin{align*}
\cont_{d_{\kappa}}(\eta \cap A) \leq \sum_{j=1}^{k_m} \cont_{d_{\kappa}}(S_j) \lesssim 2^{-m(1-\epsilon)d_{\kappa}} \lesssim \diam(A)^{(1-\epsilon)d_{\kappa}}
\end{align*}
where the implicit constants depend only on $n$,  $\kappa$,  $p$ and $R$.  Since $\epsilon \in (0,1)$ was arbitrary,  this completes the proof of the lemma.
\end{proof}

\section{First crossing lemmas}
\label{sec:first_crossing}

The purpose of this section is to collect a number of regularity estimates for the first crossing of an $\SLE_4$ across an annulus.  We will work with a conformal annulus rather than a Euclidean annulus as this will simplify the proofs in a number of places.  We begin in Section~\ref{subsec:first_crossing_setup} by describing the setup and notation.  Next, in Section~\ref{subsec:extremal_length} we will show that the extremal length between the left and right sides of the conformal rectangle formed by cutting the conformal annulus with the first crossing of the $\SLE_4$ with the left and right sides of the $\SLE_4$ taken to the left and right sides of the rectangle is unlikely to be too large (Lemma~\ref{lem:extremal_length_good}).  We will also prove some additional estimates (Lemmas~\ref{lem:band_mapping}, \ref{lem:first_crossing_number_of_squares}) which control the regularity properties of the map which takes the cut annulus to the rectangle as well as the Lebesgue measure of a small neighborhood of the first crossing.  Finally, in Section~\ref{subsec:first_crossing_regularity} (Lemma~\ref{lem:first_excursion_good}) we will establish a version of Lemmas~\ref{lem:good_fraction_whole_plane}, \ref{lem:natural_good_whole_plane} which hold for the first crossing of the $\SLE_4$ across an annulus.

\subsection{Setup and notation}
\label{subsec:first_crossing_setup}

\begin{figure}[ht!]
\begin{center}
\includegraphics[scale=.9]{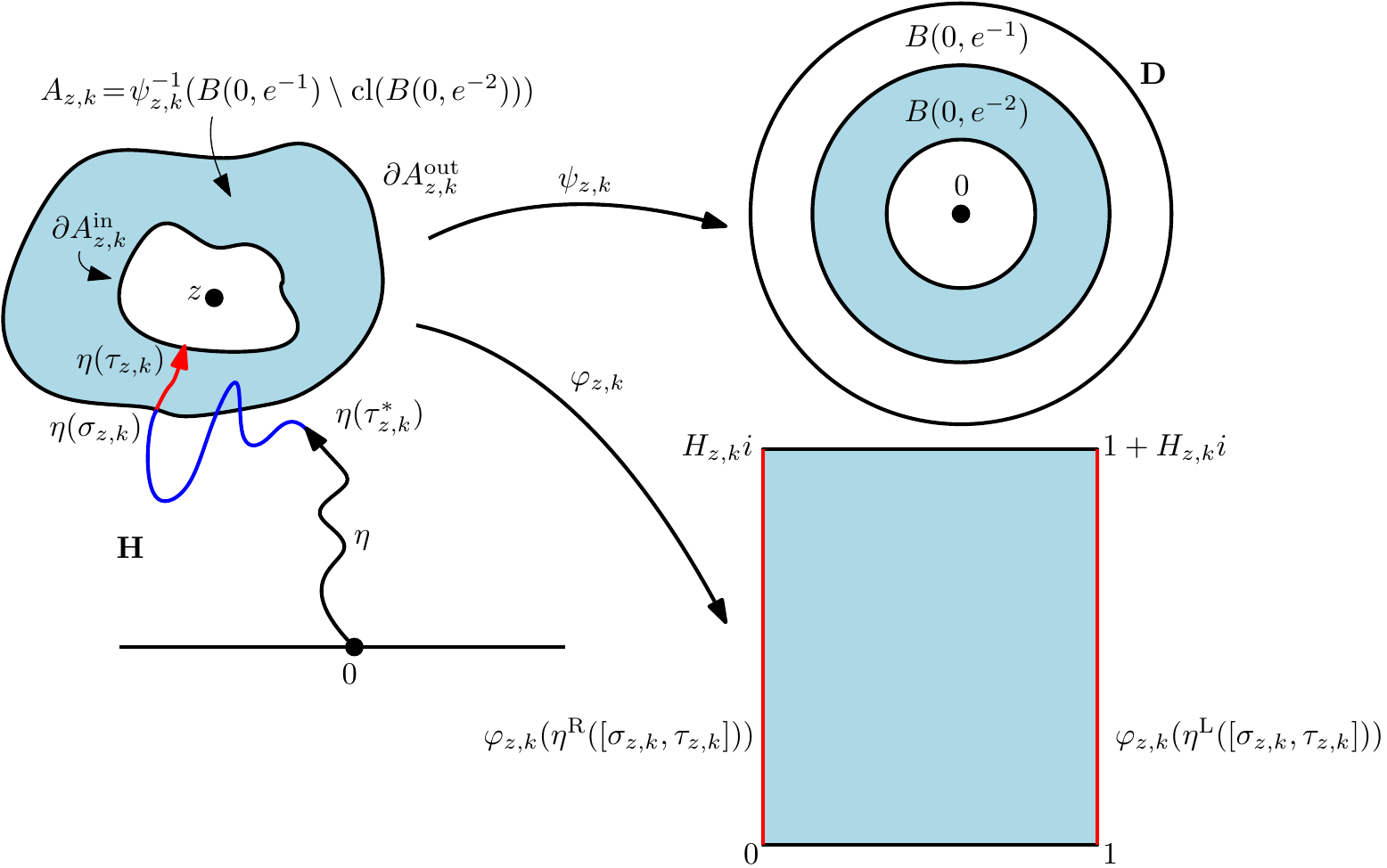}
\end{center}
\caption{\label{fig:first_crossing_setup} Illustration of the setup and notation described in Section~\ref{subsec:first_crossing_setup}.  Not shown is $A_{z,k}^*$, which is the component of $A_{z,k} \setminus \eta([0,\tau_{z,k}])$ with $\bIn A_{z,k}$ on its boundary.}
\end{figure}

See Figure~\ref{fig:first_crossing_setup} for an illustration of the setup and notation. Suppose that $\eta$ is an $\SLE_4$ in $\h$ from $0$ to $\infty$.  For each $t \geq 0$ we let $\CG_t = \sigma(\eta(s) : s \leq t)$ and $\h_t = \h \setminus \eta([0,t])$.  For each $k \in \N$ we let $\tau_{z,k}^* = \inf\{t \geq 0 : \confrad(z, \h_t) \leq e^{-4k}\}$.  On $\tau_{z,k}^* < \infty$, we let $\psi_{z,k}:\h_{\tau_{z,k}^*} \to \D$ be the unique conformal map satisfying $\psi_{z,k}(z) = 0$ and $\psi_{z,k}(\eta(\tau_{z,k}^*)) = 1$.  We let $A_{z,k} = \psi_{z,k}^{-1}( B(0,e^{-1}) \setminus \closure{B(0,e^{-2})})$, $\pin A_{z,k} = \psi_{z,k}^{-1}( \partial B(0,e^{-2}))$, $\pout A_{z,k} = \psi_{z,k}^{-1}( \partial B(0,e^{-1}))$, and $\tau_{z,k} = \inf\{t \geq 0 : \eta(t) \in \pin A_{z,k}\}$.  We let $A_{z,k}^*$ be the component of $A_{z,k} \setminus \eta([0,\tau_{z,k}])$ with $\pin A_{z,k}$ on its boundary.  Let $\sigma_{z,k}$ be the largest time $t$ before $\tau_{z,k}$ that $\eta(t) \in \pout A_{z,k}$.  Let $H_{z,k}$ be the unique positive real number so that there exists a conformal transformation $\varphi_{z,k} \colon A_{z,k}^* \to (0,1) \times (0,H_{z,k})$ which takes $\eta^\Right([\sigma_{z,k},\tau_{z,k}])$ to $[0, H_{z,k} i]$ and $\eta^\Left([\sigma_{z,k},\tau_{z,k}])$ to $1+[0,H_{z,k} i]$.  Let also $(g_t)$ be the Loewner flow for $\eta$, $W$ its Loewner driving function, $Z_t = g_t(z) - W_t = X_t + i Y_t - W_t$ and
\[ \Delta_t = |g_t'(z)|,\quad \Upsilon_t = \frac{Y_t}{|g_t'(z)|},\quad \Theta_t = \arg Z_t, \quad S_t = \sin \Theta_t.\]

\subsection{Extremal length of an annulus}
\label{subsec:extremal_length}

\begin{lemma}
\label{lem:hit_good}
For every $p_0 \in (0,1)$ there exists $\delta_0 > 0$ so that the following is true.  For each $k$ we let $\wt{\tau}_{z,k} = \inf\{t \geq \tau_{z,k-1} : S_t \leq \delta_0\}$.  Then
\[ \p[ \tau_{z,k+1} < \infty \giv \CG_{\wt{\tau}_{z,k}}] \one_{\wt{\tau}_{z,k} < \infty} \one_{\wt{\tau}_{z,k} \leq \tau_{z,k}} \leq p_0 \one_{\wt{\tau}_{z,k} < \infty} \one_{\wt{\tau}_{z,k} \leq \tau_{z,k}}.\]
\end{lemma}
\begin{proof}
We suppose that we are working on the event $\wt{\tau}_{z,k} < \infty$ and $\wt{\tau}_{z,k} \leq \tau_{z,k}$.  Let $\psi = (g_{\wt{\tau}_{z,k}} - W_{\wt{\tau}_{z,k}}) / |Z_{\wt{\tau}_{z,k}}|$ be the unique conformal transformation from $\h \setminus \eta([0,\wt{\tau}_{z,k}])$ to $\h$ which takes $\eta(\wt{\tau}_{z,k})$ to $0$, fixes $\infty$ and is normalized so the modulus of the image of $z$ is equal to $1$.  Note that $S_{\wt{\tau}_{z,k}}$ is proportional to $\im(\psi(z))$.  Thus as $S_{\wt{\tau}_{z,k}} \leq \delta_0$ we have that $\im(\psi(z)) \lesssim \delta_0$.  Since $\wt{\eta} = \psi(\eta|_{[\wt{\tau}_{z,k},\infty)})$ is an $\SLE_4$ in $\h$ from $0$ to $\infty$ and an $\SLE_4$ a.s.\ does not intersect the domain boundary (except at its starting point) and is transient it follows that the probability that $\wt{\eta}$ hits $\psi(B(z,e^{-4(k+1)}))$ tends to $0$ as $\delta_0 \to 0$, which proves the result.
\end{proof}

\begin{lemma}
\label{lem:extremal_length_good}
For every $p_0 \in (0,1)$ and $\delta_0 > 0$ there exists $H_0 > 0$ so that
\[ \p\left[ \tau_{z,k+1} < \infty,\ \inf_{\tau_{z,k} \leq t \leq \tau_{z,k+1}} S_t \geq \delta_0,\ H_{z,k+1} \leq H_0 \giv \CG_{\tau_{z,k}} \right] \one_{\tau_{z,k} < \infty} \leq p_0 \one_{\tau_{z,k} < \infty}.\]
\end{lemma}
We will first collect some preliminary facts before giving the proof of Lemma~\ref{lem:extremal_length_good}.  For each $r \in \R$ we let
\[ a = \frac{2}{\kappa},\quad \lambda = \frac{r^2}{2a} + r\left(1-\frac{1}{2a}\right),\quad \xi = \frac{r^2}{4a}.\]
Recall from~\cite[Section~6.1]{vl2012almostsure} (see also \cite{sw2005coordinate}) that
\[ M_t = S_t^{-r} \Upsilon_t^{\xi+r} \Delta_t^{\lambda+r}\]
is a continuous local martingale for $\SLE_4$.  Making the choice
\[ r = 1- \frac{8}{\kappa} = -1\]
we see that
\[ M_t = S_t \Upsilon_t^{-1/2}\]
is a continuous local martingale for $\SLE_4$.  Moreover, it follows from~\cite[Theorem~6]{sw2005coordinate} that if we weight the law of an $\SLE_4$ by $M_t$ then the resulting process is an $\SLE_4(-4)$ up to time $t$.  Recall from \cite[Theorem~3]{sw2005coordinate} that a chordal $\SLE_4(-4)$ process has the same law as a radial $\SLE_4(2)$ process (up to time change).  Note that $\Upsilon_t = \confrad(z, \h_t) / 2$.

\begin{proof}[Proof of Lemma~\ref{lem:extremal_length_good}]
Fix $p_0 \in (0,1)$, $\delta_0 > 0$, and $H_0 > 0$.  Let $E_{z,k+1}$ be the event that $\tau_{z,k+1} < \infty$, $\inf_{\tau_{z,k} \leq t \leq \tau_{z,k+1}} S_t \geq \delta_0$, and $H_{z,k+1} \leq H_0$.  We have that
\begin{align*}
\p[ E_{z,k+1}  \giv \CG_{\tau_{z,k}}] \one_{\tau_{z,k} < \infty}
&= \E\left[  \left(\frac{M_{\tau_{z,k+1}}}{M_{\tau_{z,k}}} \right)^{-1} \left(\frac{M_{\tau_{z,k+1}}}{M_{\tau_{z,k}}} \right) \one_{E_{z,k+1}} \giv \CG_{\tau_{z,k}} \right] \one_{\tau_{z,k} < \infty}\\
&= S_{\tau_{z,k}} \E_{z,k}^*\left[ S_{\tau_{z,k+1}}^{-1} \left(\frac{\Upsilon_{\tau_{z,k+1}}}{\Upsilon_{\tau_{z,k}}}\right)^{1/2} \one_{E_{z,k+1}} \right] \one_{\tau_{z,k} < \infty}
\end{align*}
where $\p_{z,k}^*$ denotes the law of a radial $\SLE_4(2)$ process in $\h \setminus \eta([0,\tau_{z,k}])$ from $\eta(\tau_{z,k})$ to $z$, until time $\tau_{z,k+1}$ and with force point located at $\infty$, and~$\E_{z,k}^*$ denotes the expectation with respect to~$\p_{z,k}^*$. Note that $\p_{z,k}^*[\tau_{z,k+1} < \infty] = 1$.  We further note that $\Upsilon_{\tau_{z,k+1}}/\Upsilon_{\tau_{z,k}} \leq 1$ and $S_{\tau_{z,k+1}}^{-1} \leq \delta_0^{-1}$ on $E_{z,k+1}$.  It therefore suffices to bound
\begin{align*}
	S_{\tau_{z,k}} \delta_0^{-1} \p_{z,k}^*[E_{z,k+1}] \one_{\tau_{z,k} < \infty} \leq \delta_0^{-1} \p_{z,k}^*[E_{z,k+1}] \one_{\tau_{z,k} < \infty}.
\end{align*}
On $\tau_{z,k} < \infty$, we have that
\begin{align*}
4k+2 - \log 4 &= - \log \confrad(z, \h_{\tau_{z,k}^*}) - \log(4 e^{-2})\\
&\leq -\log \confrad(z, \h_{\tau_{z,k}}) \quad\text{(by Koebe-$1/4$)}\\
&\leq - \log \confrad(z, \h_{\tau_{z,k}^*}) - \log(e^{-2}/4) \quad\text{(by Koebe-$1/4$)}\\
&= 4k+2 + \log 4.
\end{align*}
This implies that
\[ \log \confrad(z,\h_{\tau_{z,k}}) - \log \confrad(z, \h_{\tau_{z,k+1}^*}) \geq 1/2.\]
Thus if we let $\wh{\tau}_{z,k} = \inf\{t \geq 0 : \confrad(z, \h_t) \leq e^{-4k-1/2}\}$ we have that $\tau_{z,k} \leq \wh{\tau}_{z,k} \leq \tau_{z,k+1}^*$.  Let $\theta_{z,k+1} \in [0,2\pi)$ be such that $\psi_{z,k+1}(\infty) = e^{ i \theta_{z,k+1}}$.  Note that under $\p^*_{z,k}$ the curve $\wh{\eta} = \psi_{z,k}(\eta|_{[\tau_{z,k}^*,\tau_z)})$ has the law of a radial $\SLE_4(2)$ in $\D$ from $1$ to $0$ with the force point located at $\psi_{z,k}(\infty)$,  where $\tau_z$ is the first time that $\eta$ hits $z$.  Let $(\wh{W},\wh{O})$ be the driving pair for $\wh{\eta}$ and set $\wh{\theta}_t = \arg(\wh{O}_t) - \arg(\wh{W}_t)$.  We parameterize $\wh{\eta}$ by log conformal radius as seen from $0$.  Let also $\wh{\psi}_{z,k}$ be the unique conformal transformation mapping $\h_{\wh{\tau}_{z,k}}$ onto $\D$ such that $\wh{\psi}_{z,k}(z) = 0$ and $\wh{\psi}_{z,k}(\eta(\wh{\tau}_{z,k})) = 1$ and let $\wh{\theta}_{z,k} \in [0,2\pi)$ be such that $\wh{\psi}_{z,k}(\infty) = e^{i\wh{\theta}_{z,k}}$.  By the definitions of $\wh{\tau}_{z,k}$ and $\tau_{z,k+1}^*$ and the Markov property of $\wh{\theta}$ under $\p^*_{z,k}$ it follows that the conditional law of $\theta_{z,k+1}$ given $\wh{\theta}_{z,k}$ is that of $\wh{\theta}_2$.  Let~$f$ be the density of the law of $\theta_{z,k+1}$ under $\p^*_{z,k}$ with respect to Lebesgue measure.  By \cite[Proposition~4.4]{lawler2019bessel}, it follows that there exists a universal constant $C<\infty$ such that 
\begin{align}
\label{eqn:bound_on_density}
C^{-1}\sin(y/2)^2 \leq f(y) \leq C \sin(y/2)^2 \quad\text{for all}\quad y \in (0,2\pi).
\end{align}

For each $\theta \in (0,2\pi)$ we let $\p_\theta^*$ denote the law of a radial $\SLE_4(2)$ process $\wt{\eta}$ in $\D$ from $1$ to $0$ where the force point at time $0$ is located at $e^{i \theta}$.  Let $\wt{\tau} = \inf\{t \geq 0 : \wt{\eta}(t) \in \partial B(0,e^{-2})\}$, let $\wt{\sigma}$ be the largest time $t$ before $\wt{\tau}$ that $\wt{\eta}(t) \in \partial B(0,e^{-1})$, and let $\wt{A}^*$ be the component of $B(0,e^{-1}) \setminus \closure{B(0,e^{-2})}$ with $\partial B(0,e^{-2})$ on its boundary.  Finally, let $\wt{H} > 0$ be the unique positive real number so that there exists a conformal map $\wt{\varphi} \colon \wt{A}^* \to (0,1) \times (0,\wt{H})$ which takes the right (resp.\ left) side of $\wt{\eta}([\wt{\sigma},\wt{\tau}])$ to $[0, \wt{H} i]$ (resp.\ $1 + [0, \wt{H} i]$).    Then it suffices to show that
\[ \E_{z,k}^*[ \p_{\theta_{z,k+1}}^*[ \wt{H} \leq H_0] ] \one_{\tau_{z,k} < \infty} \to 0 \quad\text{uniformly as}\quad H_0 \to 0.\]
To prove the above,  we note that~\eqref{eqn:bound_on_density} implies that 
\begin{align*}
\E^*_{z,k}[ \p_{\theta_{z,k+1}}^*[\wt{H} \leq H_0]] = \int_0^{2\pi} \p_{\theta}^*[\wt{H} \leq H_0] f(\theta) d\theta
\leq C \int_0^{2\pi}\p_{\theta}^*[\wt{H} \leq H_0]\sin(\theta/2)^2 d\theta.
\end{align*}
Also $\wt{H}$ is a.s.\ positive as a radial $\SLE_4(2)$ is a.s.\ a simple curve which implies that $\p_{\theta}^*[\wt{H} \leq H_0] \to 0$ as $H_0 \to 0$.  The claim then follows by applying the dominated convergence theorem.  This completes the proof of the lemma.
\end{proof}

\begin{lemma}
\label{lem:band_mapping}
For every $p_0 \in (0,1)$, $\delta_0, H_0 > 0$,  and $\wt{\xi}_1 > 0$ there exist $b \in (1,2),  \xi_1 \in (0,(e^{-b}-e^{-2})/2)$, $\xi_0 \in (0,H_0/2)$ (depending only on $p_0,\delta_0$ and $H_0$) and $\wt{\xi}_0 \in (0,H_0/2)$ (depending only on $p_0,\delta_0, H_0$ and $\wt{\xi}_1$) so that the following is true.  Let $F_{z,k+1}^1$ be the event that
\begin{enumerate}[(i)]
\item\label{it:middle_circle_good} $\varphi_{z,k+1}(A^*_{z,k+1} \cap\psi_{z,k+1}^{-1}( B(0,e^{-b}-\xi_1) \setminus B(0,e^{-b}-2\xi_1))) \subseteq (0,1) \times (\xi_0, H_0- \xi_0)$ and
\item\label{it:away_from_the_boundary_good} 
$\varphi_{z,k+1}(\{ w : \dist(w, \partial A_{z,k+1}^*) \geq \wt{\xi}_1 \diam(A_{z,k+1}),  \psi_{z,k+1}(w)\in B(0,e^{-b}-\xi_1) \setminus B(0,e^{-b}-2\xi_1)\}) 
\subseteq (\wt{\xi}_0,1-\wt{\xi}_0) \times (\wt{\xi}_0, H_0-\wt{\xi}_0).$
\end{enumerate}
Then
\[ \p\left[ \tau_{z,k+1} < \infty,\ \inf_{\tau_{z,k} \leq t \leq \tau_{z,k+1}} S_t \geq \delta_0,\ (F_{z,k+1}^1)^c \giv \CG_{\tau_{z,k}} \right] \one_{\tau_{z,k} < \infty} \leq p_0 \one_{\tau_{z,k} < \infty}.\]
\end{lemma}
\begin{proof}
On $\tau_{z,k+1} < \infty$, let $Q_{z,k+1} = \varphi_{z,k+1}(A^*_{z,k+1} \cap \psi_{z,k+1}^{-1}( B(0,e^{-b}-\xi_1) \setminus B(0,e^{-b}-2\xi_1)))$.  Note that $\varphi_{z,k+1}$ is a homeomorphism on $\tau_{z,k+1} < \infty$.  Hence it follows that the probability of $Q_{z,k+1} \not\subseteq (0,1) \times (0,H_0/2)$ tends to $0$ by first taking $b \to 1$ and then $\xi_1 \to 0$ uniformly given $\CG_{\tau_{z,k}}$ on $\tau_{z,k} < \infty$. Indeed, this follows from the same argument used to prove Lemma~\ref{lem:extremal_length_good}.  Similarly, for fixed $b \in (1,2),  \xi_1 \in (0,(e^{-b}-e^{-2})/2)$,  the probability of $Q_{z,k+1} \subseteq (0,1) \times (0,H_0/2)$,  $Q_{z,k+1} \not\subseteq (0,1) \times (\xi_0,H_0-\xi_0)$ tends to $0$ as $\xi_0 \to 0$ uniformly given $\CG_{\tau_{z,k}}$ on $\tau_{z,k} < \infty$.  This proves~\eqref{it:middle_circle_good}.  The case of~\eqref{it:away_from_the_boundary_good} follows from the same argument.
\end{proof}

\begin{lemma}
\label{lem:first_crossing_number_of_squares}
For every $p_0 \in (0,1)$, $a \in (0,\frac{1}{2}),\delta_0 > 0$ there exists $M_0 \geq 1$ so that the following is true.  Let $F_{z,k+1}^2$ be the event that for every $m \in \N$ such that $2^{-m} \leq \diam(A_{z,k+1})$ the Lebesgue measure of the $2^{-m}$-neighborhood of $\eta([\sigma_{z,k+1},\tau_{z,k+1}])$ is at most $M_0 2^{m(a-1/2)}\diam(A_{z,k+1})^{a+3/2}$.  Then we have that
\[ \p\left[ \tau_{z,k+1} < \infty,\ \inf_{\tau_{z,k} \leq t \leq \tau_{z,k+1}} S_t \geq \delta_0,\ (F_{z,k+1}^2)^c \giv \CG_{\tau_{z,k}} \right] \one_{\tau_{z,k} < \infty} \leq p_0 \one_{\tau_{z,k} < \infty}.\]	
\end{lemma}
\begin{proof}
Let $\gamma$ be an $\SLE_4$ in $\h$ from $0$ to $\infty$ and fix $K \subseteq \h$ compact.  For each $m \in \N$ we let~$N_m$ be the number of points $z \in  2^{-m} \Z^2$ with $\dist(z,K) \leq 2^{-m}$  such that $\gamma \cap B(z, 2^{-m}) \neq \emptyset$.  Then \cite[Proposition~4]{beffara2008dimension} implies that $\E[N_m] \lesssim 2^{3 m/2}$ and so $\p[N_m \geq 2^{(3/2+a)m}] \lesssim 2^{-a m}$ where the implicit constant depends only $K$.  Let $\Lambda_m$ be the Lebesgue measure of the $2^{-m}$-neighborhood of $\gamma \cap K$.  Then we have that $\p[ \Lambda_m \geq 2^{(-1/2+a)m}] \lesssim 2^{-a m}$ where the implicit constant tends only on $K$.  This implies that $\p[ \Lambda_\ell \leq 2^{(-1/2+a)\ell}\ \forall \ell \geq m] \to 1$ as $m \to \infty$.  We let $f_t = g_t - W_t$ be the centered Loewner flow of $\eta$ and set $h_t(\cdot) = f_t(\cdot)/ \im(f_t(z))$.  Then we can make the choice of $K \subseteq \h$ (depending only on $\delta_0$) so that on~$\tau_{z,k+1} < \infty$ and $\inf_{\tau_{z,k} \leq t \leq \tau_{z,k+1}} S_t \geq \delta_0$ we have that $h_{\tau_{z,k}}(A_{z,k+1}) \subseteq K$.  The result then follows since $h_{\tau_{z,k}}(\eta|_{[\tau_{z,k},\infty)})$ has the law of an $\SLE_4$ in $\h$ from $0$ to $\infty$ given $\CG_{\tau_{z,k}}$ on $\tau_{z,k} < \infty$ and $|(h_{\tau_{z,k}}^{-1})'|$ on $K$ is comparable to $\diam(A_{z,{k+1}})$ (with constants depending only on $\delta_0$) on $\tau_{z,k+1} < \infty$.
\end{proof}

\subsection{Regularity of the first crossing}
\label{subsec:first_crossing_regularity}

\begin{figure}[ht!]
\begin{center}
\includegraphics[scale=.9]{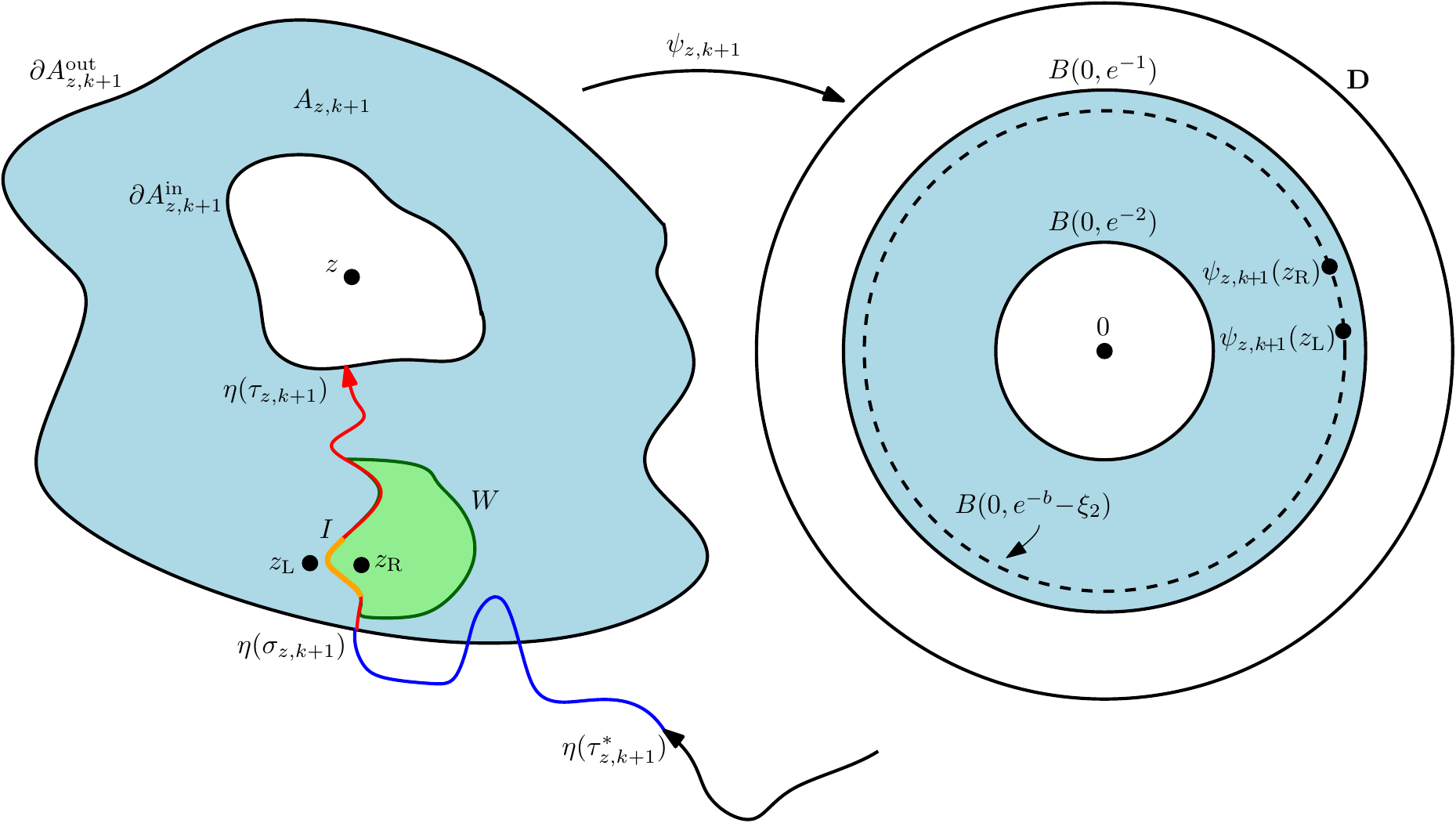}
\end{center}
\caption{\label{fig:first_crossing_good} Illustration of the setup for Lemma~\ref{lem:first_excursion_good}.  Here, $b \in (1,2)$ and $\xi_2 > 0$ is small so that $B(0,e^{-b} - \xi_2) \subseteq B(0,e^{-1}) \setminus \closure{B(0,e^{-2})}$ as shown.  The orange arc of $\eta([\sigma_{z,k+1},\tau_{z,k+1}])$ is~$I$ and $W$ is shown in the case that the component of $B(z_\Right, \delta_1^{1/4} \diam(A_{z,k+1})) \setminus \eta([0,\tau_{z,k+1}])$ containing $z_\Right$ is contained in $W$.}
\end{figure}

The purpose of this subsection is to establish a version of Lemmas~\ref{lem:good_fraction_whole_plane} and~\ref{lem:natural_good_whole_plane} which holds for the first crossing of an $\SLE_4$ across the conformal annulus~$A_{z,k+1}$ (Lemma~\ref{lem:first_excursion_good}).  As in Section~\ref{subsec:extremal_length}, we will need to truncate on the event that $S_t$ is not too small so that we can make the comparison between the first crossing and a radial $\SLE_4(2)$ in a uniform manner.  Some care will be needed in both the setup and the proof because, in contrast to Lemma~\ref{lem:good_fraction_whole_plane}, we will need to find a single arc along the first crossing of the radial $\SLE_4(2)$ and points~$z_\Left$ and~$z_\Right$ which are respectively close to the left and right sides such that the fraction (measured using the natural parameterization) of points along the crossing so the hyperbolic geodesic only passes through Whitney squares $Q$ satisfying the condition of Lemma~\ref{lem:good_fraction_whole_plane} is at least $3/4$.  Furthermore, we will need the statement to hold for the hyperbolic metric with respect to a general simply connected domain~$W$ which contains a neighborhood of~$z_\Left$ (resp.\ $z_\Right$) which is much larger than the distance of~$z_\Left$ (resp.\ $z_\Right$) to the arc.

We now turn to give the precise version of the main statement; see Figure~\ref{fig:first_crossing_good} for an illustration.  We fix $\delta_0,\delta_1, \delta_2, a \in (0,1)$, $b \in (1,2)$, and $\xi_2 \in (0,(e^{-b}-e^{-2})/2)$.  For $M_0>1$ large we let $C_{z,k+1}$ be the event that there exists a segment $I$ of $\eta([\sigma_{z,k+1},\tau_{z,k+1}])$ which satisfies the following.  Let $I^\Left$ (resp.\ $I^\Right$) be the left (resp.\ right) side of $I$.
\begin{enumerate}[(i)]
\item\label{it:x_y_pos} There exist $z_\Left,z_\Right \in A_{z,k+1}^*$ such that $\delta_2 \leq \frac{\dist(z_{q},I^{q})}{\diam(A_{z,k+1})} \leq \delta_1$ with $\psi_{z,k+1}(z_{q}) \in \partial B(0,e^{-b}-\xi_2)$ for $q \in \{\Left,\Right\}$ and satisfy the following property.  Suppose that $W \subseteq A_{z,k+1}^*$ is any simply connected domain such that the component of $B(z_\Right,\delta_1^{1/4} \diam(A_{z,k+1})) \setminus \eta([0,\tau_{z,k+1}])$ containing $z_\Right$ is contained in $W$ and $\CW$ is a Whitney square decomposition of $W$.  Let $G^+$ be the set of points $w \in I$ such that each square $Q \in \CW$ which is intersected by the hyperbolic geodesic in $W$ from $z_\Right$ to $w$ satisfies $\disthyp^{W}(z_\Right,\cen(Q)) \leq M_0 (\len(Q)/\diam(A_{z,k+1}))^{-a}$.  Then we have that
\[ \nmeasure{\eta}(G^+) \geq \frac{3}{4} \nmeasure{\eta}(I).\]
The same moreover holds with $z_\Left$ in place of $z_\Right$ in which case we denote by $G^-$ the set of points in $I$.
\item\label{it:lower_bound_natural_measure} $\nmeasure{\eta}(I) \geq M_0^{-1} \diam(A_{z,k+1})^{3/2}$.
\item\label{it:measure_diameter_bound} $\nmeasure{\eta}(Y \cap I) \leq M_0 \diam(Y)^{3/2-a}$ for all $Y$ Borel.
\end{enumerate}

The reason for the precise form of condition~\eqref{it:x_y_pos} is that in the proof of Theorem~\ref{thm:sle_4_removable} we will need that $\varphi_{z,k+1}(z_\Left)$ and $\varphi_{z,k+1}(z_\Right)$ are respectively close to the right and left sides of the rectangle $(0,1) \times (0,H_{z,k+1})$ and we will need both to be close to the bottom of the rectangle.  In particular, if $C_{z,k+1}$ and the event $F_{z,k+1}^1$ from Lemma~\ref{lem:band_mapping} both occur then $\varphi_{z,k+1}(z_\Left)$, $\varphi_{z,k+1}(z_\Right)$ will be mapped into the desired areas of the rectangle. 

\begin{lemma}
\label{lem:first_excursion_good}
Fix $\delta_0, p_0,a \in (0,1)$,  $b \in (1,2)$, and $\xi_2 \in (0,(e^{-b}-e^{-2})/2)$ and suppose that we have the above setup.  Then for all $\delta_1 \in (0,1)$ sufficiently small there exists $M_0 > 1$ and $\delta_2 \in (0,1)$ such that
\begin{align*}
\p[ \tau_{z,k+1} < \infty,  \inf_{\tau_{z,k} \leq t \leq \tau_{z,k+1}}S_t \geq \delta_0,  C_{z,k+1}^c  \giv  \CG_{\tau_{z,k}}] \one_{B_{z,k}} \leq p_0 \one_{B_{z,k}}
\end{align*}
where $B_{z,k} = \{\tau_{z,k} < \infty,\ S_{\tau_{z,k}} \geq \delta_0\}$.
\end{lemma}
Before we give the proof of Lemma~\ref{lem:first_excursion_good}, we state and prove the following upper bound for the probability that a radial $\SLE_4(2)$ in $\D$ gets close to a point which holds uniformly in the location of its force point.

\begin{lemma}
\label{lem:radial_sle_hits_a_ball}
Fix $p \in (0,1)$,  $0 < r_1 < r_2 < 1$, $\theta_0 \in (0,\pi)$, and let $\eta$ be a radial $\SLE_4(2)$ in $\D$ from $1$ and to $0$ with its force point located at $e^{i \theta}$  for $\theta \in [\theta_0 ,  2\pi - \theta_0]$.  Then there exists $\epsilon_0 \in (0,1)$ depending only on $p$, $r_1$, $r_2$, and $\theta_0$ such that for every $\epsilon \in (0,\epsilon_0)$ and every $w \in \D$ with $r_1 \leq |w| \leq r_2$ we have that
\begin{align*}
\p[ \eta \cap B(w,\epsilon) \neq \emptyset] \leq p.
\end{align*}
\end{lemma}
\begin{proof}
\noindent{\it Step 1. Result for fixed $\theta$.}  Fix $\theta \in [\theta_0,2\pi-\theta_0]$ and suppose that $p$, $r_1$, $r_2$ are as in the statement of the lemma.  We claim that there exists $\epsilon_0 \in (0,1)$ depending only on $p$, $r_1$, $r_2$, and $\theta$ such that for every $\epsilon \in (0,\epsilon_0)$ and every $w \in \D$ with $r_1 \leq |w| \leq r_2$ we have that $\p[ \eta \cap B(w,\epsilon) \neq \emptyset] \leq p$.  Indeed, suppose that this does not hold.  Then there exists a sequence $(w_k)$ in $\D$ with $r_1 \leq |w_k| \leq r_2$ for every~$k$ and a sequence $(\epsilon_k)$ of positive reals with $\epsilon_k \to 0$ as $k \to \infty$ such that $\p[ \eta \cap B(w_k,\epsilon_k) \neq \emptyset] \geq p$ for every $k$.  By passing to a subsequence if necessary, we may assume that $(w_k)$ converges to a limit $w_0$.  Fix $\epsilon > 0$.  Then as $B(w_0,\epsilon)$ contains $B(w_k,\epsilon_k)$ for all $k$ large enough, it follows that $\p[ \eta \cap B(w_0,\epsilon) \neq \emptyset] \geq p$.  Since $\epsilon > 0$ was arbitrary, we obtain that $\p[w \in \eta] \geq p$.  Since $\lim_{t \to \infty} \eta(t) = 0$ a.s., it follows that there exists $T \in (0,\infty)$ so that $\p[ w \in \eta([0,T]) ] \geq p/2$.  This is a contradiction since the law of $\eta|_{[0,T]}$ is absolutely continuous with respect to that of an $\SLE_4$ in $\D$ from $-1$ to $1$ \cite[Theorem~3]{sw2005coordinate}.

\noindent{\it Step 2. Uniformity in $\theta \in [\theta_0,2\pi-\theta_0]$.}  We will now show that we can take the value of $\epsilon_0 \in (0,1)$ from the previous step to be uniform in $\theta \in [\theta_0,2\pi-\theta_0]$.  Let $\eta$ be an $\SLE_4(2)$, $(W,O)$ the driving pair for $\eta$, and $\theta_t = \arg W_t - \arg O_t$.  Fix $r_1' \in (0,r_1)$ and $r_2' \in (r_2,1)$.  We claim that we can find $\delta \in (0,1)$ sufficiently small depending only on $r_1,r_2,r_1',r_2'$ such that $g_t(\{\frac{r_1+r_1'}{2} \leq |z| \leq \frac{r_2+r_2'}{2}\}) \subset \{r_1' \leq |z| \leq r_2'\}$ for each $t \in [0,\delta]$.  Indeed we set $\D_t = \D \setminus \eta([0,t])$ and recall from \cite[Section~3.5]{lawler2008conformally} that $\log(g_{\delta}'(0)) = \delta = -\E_0[\log(|B_{\sigma}|)]$ where $B = (B_t)$ is a Brownian motion starting from $0$ and $\sigma$ is the first time that it exits $\D_{\delta}$.  Let $r \in (0,1)$ be such that $\eta([0,\delta]) \cap B(0,1-r) \neq \emptyset$ and let $s$ be the last time before $\delta$ that $\eta$ hits $\partial B(0,1-r/2)$.  Then we have that $\eta((s,\delta]) \subseteq B(0,1-r/2)$ and $\diam(\eta((s,\delta])) \geq r/2$ and so \cite[Exercise~2.7]{lawler2008conformally} implies that with probability at least $c_1 r$,  $B$ hits $\eta((s,\delta])$ before exiting $B(0,1-r/2)$,  where $c_1>0$ is a universal constant.  It follows that $\delta \geq -c_1 r \log(1-r/2)$.  Let $r(\delta)$ be the largest $r \in (0,1)$ such that $\delta \geq -c_1 r \log(1-r/2)$.  Note that $\eta([0,\delta]) \subseteq \D \setminus \closure {B(0,1-2r(\delta))}$ and $r(\delta) \to 0$ as $\delta \to 0$.  Thus~\cite[Proposition~3.57]{lawler2008conformally} implies that $|\log(g_t(z)/z)| \leq c r(\delta) (1 - \log(1-|z|))$ for each $z \in \closure{B(0,1-4r(\delta))}$ and each $t \in [0,\delta]$,  where $c < \infty$ is a universal constant.  The claim then follows by taking $\delta$ sufficiently small since $|\log(|g_t(z)|/|z|)| \leq |\log(g_t(z)/z)|$.

 Fix $n \in \N$ and assume that $\theta_0 = \theta_1 < \ldots < \theta_n = 2\pi-\theta_0$ are equally spaced.  Let $\tau = \inf\{t \geq 0 : \theta_t \in \{\theta_1,\ldots, \theta_n\} \}$.  Note that we can choose $n$ sufficiently large such that $\p[\tau \geq \delta] \leq p/2$.  Indeed \cite[Section~2]{kms2021regularity} implies that the law of $\theta|_{[0,\delta]}$ stopped at the first time it exits $[\theta_0 / 2,  2\pi - \theta_0 / 2]$ is absolutely continuous with respect to the law of a standard Brownian motion $B$ starting from $\theta$ in the time interval $[0,\delta]$ and stopped at the first time it exits $[\theta_0 / 2,  2\pi - \theta_0 / 2]$.  Let $\wt{\p}_{\theta}$ be the law of $B$ and $\wt{\tau}$ the corresponding stopping time.  Then the claim follows since $\wt{\p}_{\theta}[\wt{\tau} \geq \delta] \to 0$ as $n \to \infty$ uniformly in $\theta$ and the Radon-Nikodym derivative is bounded from above and below by constants depending only on $\theta_0$ and $\delta$. We assume that $n$ is sufficiently large so that $\p[\tau \geq \delta] \leq p/2$.  On $\{\tau \leq \delta\}$, we have that $g_\tau( \{ \frac{r_1+r_1'}{2} \leq |z| \leq \frac{r_2+r_2'}{2}\}) \subset \{r_1' \leq |z| \leq r_2'\}$.  Moreover, by Step 1 there exists $\epsilon_0 \in (0,1)$ such that for every $w \in \D$ with $r_1' \leq |w| \leq r_2'$ and $\epsilon \in (0,\epsilon_0)$ the probability that $g_{\tau}(\eta([\tau,\infty)))$ hits $B(w,\epsilon)$ is at most $p/2$ and $B(w,\epsilon) \subset \{\frac{r_1+r_1'}{2}\leq |z| \leq \frac{r_2+r_2'}{2}\}$.  Combining with~\cite[Theorem~3.21]{lawler2008conformally} (and possibly decreasing $\epsilon_0 > 0$ to account for the distortion of $g_\tau$) implies assertion of the lemma.
\end{proof}

\begin{figure}[ht!]
\begin{center}
\includegraphics[scale=0.9]{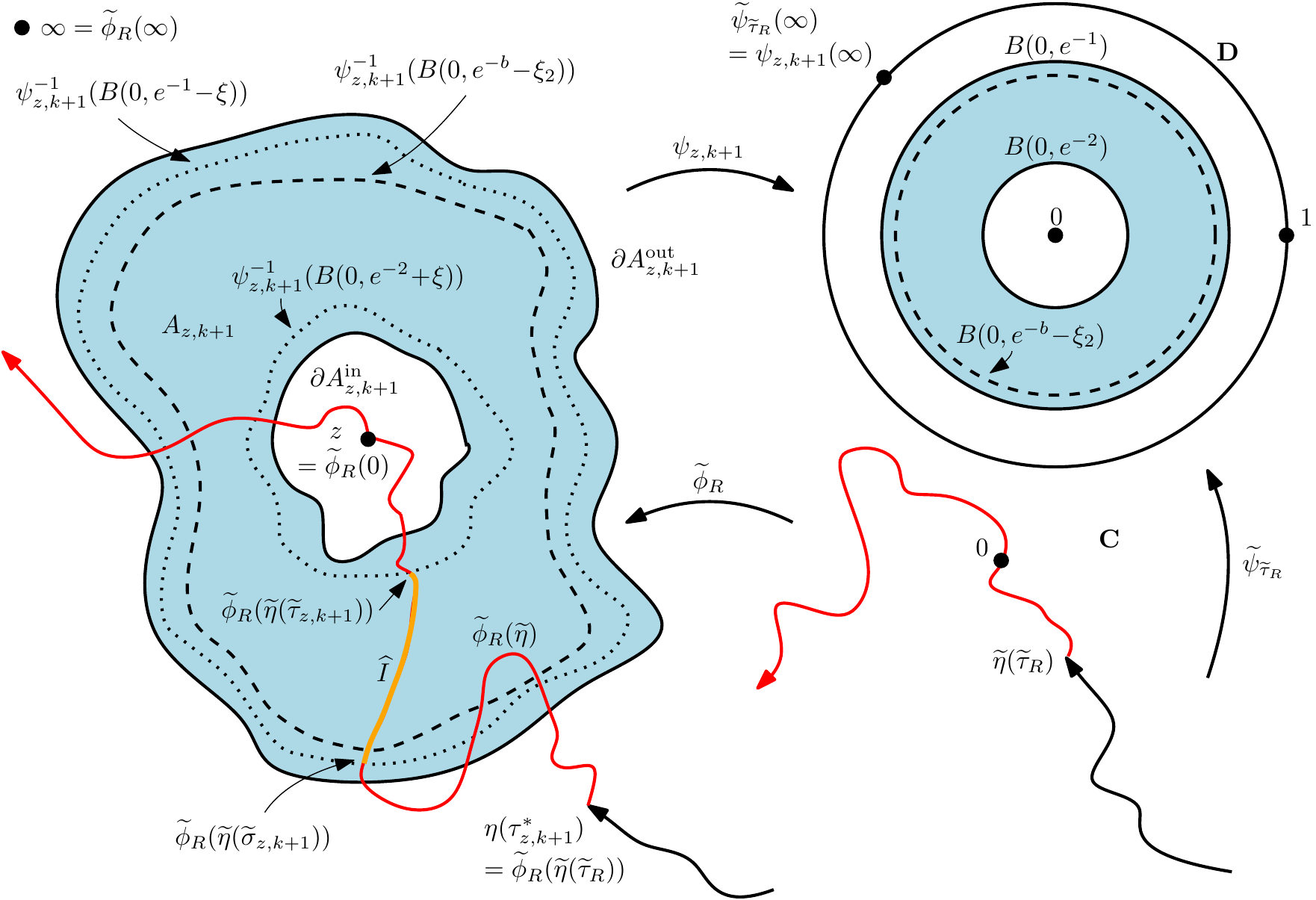}	
\end{center}
\caption{\label{fig:first_crossing_proof} Illustration of the proof of Lemma~\ref{lem:first_excursion_good}.  The conformal map $\psi_{z,k+1}$, $\wt{\psi}_{\wt{\tau}_R}$ are chosen so that $1 = \psi_{z,k+1}(\eta(\tau_{z,k+1}^*)) = \wt{\psi}_{\wt{\tau}_R}(\wt{\eta}(\wt{\tau}_R))$.  In particular, $\wt{\phi}_R = \psi_{z,k+1}^{-1} \circ \wt{\psi}_{\wt{\tau}_R}$.}
\end{figure}

As in the proof of Lemma~\ref{lem:extremal_length_good}, we have that the law of the first crossing of $\eta$ across $A_{z,k+1}$ is absolutely continuous with respect to the law of a radial $\SLE_4(2)$ curve.  This will allow us to work with the conformal image of part of a two-sided whole-plane $\SLE_4$ curve from~$\infty$ to~$\infty$ through~$0$ (i.e., the setting of Lemma~\ref{lem:good_fraction_whole_plane}) in place of an $\SLE_4$.  Some care will be needed because we need to find a pair of points close to the left and right sides of the crossing and a single arc so that the condition of Lemma~\ref{lem:good_fraction_whole_plane} holds for both points simultaneously.  We will also need to control the part of the curve after it completes the crossing to argue that it does not affect the hyperbolic geodesics near the distinguished arc.  See Figure~\ref{fig:first_crossing_proof} for an illustration of the setup of the proof.

The first step in proving Lemma~\ref{lem:first_excursion_good} is to show that we can transfer the result from the case of a whole-plane $\SLE_4$ to the radial $\SLE_4$. From radial $\SLE_4$ it is then easy to deduce the result for chordal $\SLE_4$ with an absolute continuity argument (see the last proof of this section). Recall the setup and above and consider the following.  Let $\wt{\eta}_1$ be a whole-plane $\SLE_4(2)$ from $\infty$ to $0$ and given $\wt{\eta}_1$ we let $\wt{\eta}_2$ be a chordal $\SLE_4$ in $\C \setminus \wt{\eta}_1$ from $0$ to $\infty$.  Then the concatenation $\wt{\eta}$ of $\wt{\eta}_1$ and $\wt{\eta}_2$ is a two-sided whole-plane $\SLE_4$ from $\infty$ to $\infty$ passing through $0$ and we parameterize it according to the natural parameterization with $\wt{\eta}(0) = 0$.  We assume that $\wt{\eta}$ is independent of $\eta$ so that $\wt{\eta}$ and $\eta$ are on the same probability space.  For each $t < 0$ we let~$\wt{\psi}_t$ be the unique conformal transformation mapping $\C \setminus \wt{\eta}((-\infty,t])$ to $\D$ with $\wt{\psi}_t(0) = 0$ and $\wt{\psi}_t(\wt{\eta}(t)) = 1$.  For each $R> 0$ we also consider the stopping times $\wt{\sigma}_R = \inf\{t \in \R : \wt{\eta}(t) \in \partial B(0,R)\}$ and $\wt{\tau}_R = \inf\{t \geq \wt{\sigma}_R : \wt{\psi}_t(\infty) = \psi_{z,k+1}(\infty)\}$. 

\begin{lemma}
Fix $p_0 \in (0,1)$.  There exists $R_0 = R_0(p_0) > 1$ such that for all $R \geq R_0$,
\begin{align}
\label{eqn:tau_sigma_lbd}
	\p[ \wt{\tau}_R \leq \wt{\sigma}_1 \giv   \CG_{\tau_{z,k}}] \one_{B_{z,k}} \geq \left(1-\frac{p_0}{100} \right) \one_{B_{z,k}}.
\end{align}
\end{lemma}
\begin{proof}
Note that if $B_{z,k}$ occurs, we have that $|\psi_{z,k+1}(\infty) - 1| \gtrsim \delta_0$ where the implicit constant is universal.  

Fix $R >1$.  Then $\wh{\eta} = \wt{\psi}_{\wt{\sigma}_R}(\wt{\eta}|_{[\wt{\sigma}_R,0]})$ is a radial $\SLE_4(2)$ in $\D$ from $1$ to $0$ with the force point located at $\wt{\psi}_{\wt{\sigma}_R}(\infty) \in \partial \D$.  We parameterize $\wh{\eta}$ by log conformal radius as seen from~$0$, let $(\wh{W},\wh{O})$ be its driving pair, and set $\wh{\theta}_t = \arg(\wh{O}_t) - \arg(\wh{W}_t)$.  If $\{\wt{\tau}_R > \wt{\sigma}_1\}$ occurs, then $\wh{\theta}$ does not hit $\arg(\psi_{z,k+1}(\infty))$ in the time interval $[0,\log R+c_1]$ where $c_1 \in \R$ is a universal constant. Therefore it suffices to show that the latter event occurs with probability tending to $0$ as $R \to \infty$.  As explained at the beginning of the proof, there exists $\theta_0 \in [0,\pi/4]$ depending only on $\delta_0$ such that $\arg(\psi_{z,k+1}(\infty)) \in [\theta_0,2\pi - \theta_0]$.  Let $E_R$ be the event that there exists $t \in [0,\log R + c_1 -2]$ such that $\wh{\theta}|_{[t,t+2]}$ hits both $(0,\theta_0]$ and $[2\pi-\theta_0,2\pi)$ and note that if $E_R$ occurs then $\wh{\theta}$ hits $\arg(\psi_{z,k+1}(\infty))$ before time $\log R + c_1$.  Set $\wt{\theta}_t = \wh{\theta}_t / 2$ and let $q_t$ be the transition density for $\wt{\theta}$.  As in~\eqref{eqn:bound_on_density} in the proof of Lemma~\ref{lem:extremal_length_good}, \cite[Proposition~4.4]{lawler2019bessel} implies that there exists $q_0 > 0$ depending only on $\theta_0$ so that if $q_t(\wt{\theta}_0,A) = \int_A q_t(\wt{\theta_0},s) ds$, then
\begin{equation}
\label{eqn:transition_lbd}
 q_1(\wt{\theta}_0, (0,\theta_0/2]) \geq q_0 \quad\text{and}\quad  q_1(\wt{\theta}_0, [\pi-\theta_0/2,\pi)) \geq q_0 \quad\text{for all}\quad \wt{\theta}_0 \in [0,\pi].
\end{equation}
Note that~\eqref{eqn:transition_lbd} implies that
\begin{equation}
\label{eqn:transition_lbd_markov}
\p[\wh{\theta}|_{[0,2]} \ \text{hits both}\ (0,\theta_0]\ \text{and}\ [2\pi-\theta_0,2\pi) ] \geq q_0^2 \quad\text{uniformly in}\quad \wh{\theta}_0 \in [0,2\pi].
\end{equation}
Applying the Markov property $\tfrac{1}{2}(\log R+c_1)-1$ times for $\wh{\theta}$ implies that there exists $q > 0$ depending only on $\theta_0$ so that $\p[E_R] = 1- O(R^{-q})$.  This proves~\eqref{eqn:tau_sigma_lbd}.
\end{proof}

Next, we shall show that with high probability, the crossing of $A_{z,k+1}$ is separated from the rest of the curve. We begin by proving that a regularity event is very likely to occur, provided that we choose the parameters correctly.
\begin{lemma}\label{lem:lower_bound_F1}
Fix $p_0 \in (0,1)$, $R \geq R_0$ and $\epsilon > 0$.  There exist $A,C > 0$ large and $\delta \in (0,1)$ such that the following holds. Define the events $F_0 = \{ \wt{\tau}_R \leq \wt{\sigma}_1,\ \wt{\tau}_R \geq -A,\ |\wt{\eta}(s) - \wt{\eta}(t)| \leq C |s-t|^{1/2}\ \forall s,t \in [-A,A]\}$ and $F_1 = F_0 \cap \{ \dist(\wt{\eta}([-A,s]),\wt{\eta}([s+\epsilon,\infty))) \geq \delta \ \text{for each} \ s \in [-A,0] \}$. Then,
\begin{align}\label{eq:lower_bound_F1}
	\p[ F_1 \giv \CG_{\tau_{z,k}}] \one_{B_{z,k}} \geq \left(1-\frac{p_0}{25}\right) \one_{B_{z,k}}.
\end{align}
\end{lemma}
\begin{proof}
We assume that $\wt{\eta}$ has the natural parameterization with time normalized so that $\wt{\eta}(0) = 0$.  Then \cite[Corollary~4.7]{zhan2019holder}, the independence of $\wt{\eta}$ and $\eta$, and~\eqref{eqn:tau_sigma_lbd} imply that there exist constants $A,C>0$ such that
\[ \p[ F_0 \giv \CG_{\tau_{z,k}}] \one_{B_{z,k}} \geq \left(1- \frac{p_0}{50}\right) \one_{B_{z,k}}.\]

Moreover, since $\wt{\eta}$ is a simple curve and $\wt{\eta}_1,\wt{\eta}_2$ are transient \cite{ms2017ig4}, it follows that we can find $\delta \in (0,1)$ such that~\eqref{eq:lower_bound_F1} holds.
\end{proof}

Before stating the next lemma, we define the following random times. Let $\wt{\phi}_R$ be the unique conformal map from $\C \setminus \wt{\eta}((-\infty,\wt{\tau}_R])$ to $\h_{\tau_{z,k+1}^*}$ with $\wt{\phi}_R(0) = z$ and $\wt{\phi}_R(\wt{\eta}(\wt{\tau}_R)) = \eta(\tau_{z,k+1}^*)$ and for a small parameter $\xi \in (0,1)$, define
\begin{align*}
\wt{\tau}_{z,k+1} &= \inf\{ t \in \R : \wt{\phi}_R(\wt{\eta}(t)) \in \psi_{z,k+1}^{-1}( B(0,e^{-2}+\xi))\},\\
\wt{\sigma}_{z,k+1} &= \sup\{t \leq \wt{\tau}_{z,k+1} : \wt{\phi}_R(\wt{\eta}(t)) \notin \psi_{z,k+1}^{-1}(B(0,e^{-1}-\xi))\}, \quad\text{and}\\
\wh{\tau}_{z,k+1} &= \inf\{t \geq \wt{\tau}_{z,k+1} : |\wt{\phi}_R(\wt{\eta}(t)) - \wt{\phi}_R(\wt{\eta}(\wt{\tau}_{z,k+1}))| \geq \xi \diam(A_{z,k+1})\}.
\end{align*}
Moreover, we let $\wh{I} = \wt{\phi}_R(\wt{\eta}([\wt{\sigma}_{z,k+1},\wt{\tau}_{z,k+1}]))$. It is from this segment that we will pick the segment~$I$ of the event $C_{z,k+1}$. The next lemma shows that on the event $F_1$, anything but the immediate future of the curve stays away from the set $\wh{I}$.

\begin{lemma}
\label{lem:separation_past_future}
Consider the setup of the previous lemma. There exists $\epsilon,\zeta > 0$ small such that on $F_1$ we have that $\dist(\wt{\phi}_R(\wt{\eta}([\wt{\sigma}_{z,k+1},\wt{\tau}_{z,k+1}])),\wt{\phi}_R(\wt{\eta}([\wh{\tau}_{z,k+1},\infty)))) \geq \zeta \diam(A_{z,k+1})$.
\end{lemma}
\begin{proof}
Note that $\wt{\phi}_R(\wt{\eta}|_{[\wt{\tau}_R,0]})$ is a radial $\SLE_4(2)$ in~$\h_{\tau_{z,k+1}^*}$ from $\eta(\tau_{z,k+1}^*)$ to $z$ with the force point located at $\infty$.  (Note that $\wt{\phi}_R = \psi_{z,k+1}^{-1} \circ \wt{\psi}_{\wt{\tau}_R}$.)

Suppose that we are working on $F_0$. We fix $\xi \in (0,1)$ small (depending only on the implicit constants that we have considered so far). Note that~\cite[Corollary~3.19 and Theorem~3.21]{lawler2008conformally} imply that there exist universal constants $c_1,c_2>0$ and a constant $C_R >1$ depending only on $R$ such that $\psi^{-1}_{z,k+1}(B(0,e^{-1}+\xi)) \subseteq B(z,c_1 \diam(A_{z,k+1}))$,  $B(z,c_2 \diam(A_{z,k+1}))\subseteq \psi_{z,k+1}^{-1}(B(0,e^{-2}-2\xi))$ when $\xi \in (0,1)$ is sufficiently small and $C_R^{-1}  \leq |(\wt{\phi}_R^{-1})'(w)| \diam(A_{z,k+1}) \leq C_R$ for each $w \in A_{z,k+1}$.  Combining this with \cite[Theorem~3.21]{lawler2008conformally} we obtain that if $\xi > 0$ is sufficiently small,  then we have that $|\wh{\tau}_{z,k+1}-\wt{\tau}_{z,k+1}|\geq (\xi / 2CC_R)^2$.  Choosing $\epsilon > 0$ sufficiently small, it follows that $\wt{\phi}_R(\wt{\eta})([\wh{\tau}_{z,k+1},\infty)) \subseteq \wt{\phi}_R(\wt{\eta}([\wt{\tau}_{z,k+1}+\epsilon,\infty)))$ and $\wt{\phi}_R(\wt{\eta}([\wt{\tau}_R,\wt{\tau}_{z,k+1}])) \subseteq \wt{\phi}_R(\wt{\eta}([-A,0]))$. Thus, assuming that we are on $F_1$, it follows that $\dist(\wt{\phi}_R(\wt{\eta}([\wt{\sigma}_{z,k+1},\wt{\tau}_{z,k+1}])),\wt{\phi}_R(\wt{\eta}([\wh{\tau}_{z,k+1},\infty)))) \geq \zeta \diam(A_{z,k+1})$, proving the lemma.
\end{proof}

We now turn our attention to the construction of~$I$ and the two points $z_\Left,z_\Right$ which are close to $I$. We fix $\wt{\delta}_1 \in (0,1)$ small and $\xi_2 \in (0,(e^{-b}-e^{-2})/2)$ (both to be chosen) such that $2\pi/\wt{\delta}_1 \in \N$.  We divide $\partial B(0,e^{-b}-\xi_2)$ into open and pairwise disjoint arcs $J_1,\ldots,J_{n(\wt{\delta}_1)}$ of length $\wt{\delta}_1$ such that $\partial B(0,e^{-b}-\xi_2) = \cup_{m=1}^{n(\wt{\delta}_1)}\closure{J_m}$.  Let $J$ be the arc in $\{J_1,\ldots,J_{n(\wt{\delta}_1)}\}$ so that its image under $\psi_{z,k+1}^{-1}$ is last hit by $\wh{I}$.  Let also $z_\Left$ (resp.\  $z_\Right$) be the two endpoints of $\psi_{z,k+1}^{-1}(J)$ (where naturally $z_\Right$ is the image of the first point of $J$ if it is ordered clockwise).  Moreover, we let $w_0$ be the last point (chronologically) on $\psi_{z,k+1}^{-1}(J)$ that $\wh{I}$ hits and let $w_1$ be the first point on $\partial B(w_0,\delta_1\diam(A_{z,k+1}))$ that $\wh{I}$ hits after it hits $w_0$. Finally, we let $t_{z,k+1}^j$, $j=0,1$ be such that $\wt{\phi}_R(\wt{\eta}(t_{z,k+1}^j)) = w_j$ and define $I = \wt{\phi}_R(\wt{\eta}([t_{z,k+1}^0,t_{z,k+1}^1]))$.

In the following lemma we provide the distance-diameter ratio bounds of~\eqref{it:x_y_pos}.
\begin{lemma}\label{lem:lower_bound_F2}
Fix $p_0 \in (0,1)$, $\delta_1 \in (0,1)$ and the parameters $R,A,C,\delta,\epsilon$ as above. Then for all $\wt{\delta}_1,\delta_2 \in (0,1)$ small enough we have that if $F_2 = F_1 \cap \{\dist(\{z_\Left,z_\Right\},\wt{\phi}_R(\wt{\eta})) \geq \delta_2 \diam(A_{z,k+1}) \ \cap \{ \dist(\{z_\Left,z_\Right\},I) \leq \delta_1 \diam(A_{z,k+1}) \}$, then
\begin{align*}
	\p[ F_2 \giv \CG_{\tau_{z,k}}] \one_{B_{z,k}} \geq (1-p_0/20) \one_{B_{z,k}}.
\end{align*}
\end{lemma}
\begin{proof}
By construction we have that $\dist(z_\Left,I^\Left) \lesssim \wt{\delta}_1\diam(A_{z,k+1})$ and $\dist(z_\Right,I^\Right) \lesssim \wt{\delta}_1\diam(A_{z,k+1})$ with universal implicit constant, where $I^\Left$ and $I^\Right$ denote the left and right side of $I$, respectively.  Thus, if we fix $\delta_1 \in (0,1)$ then we can pick $\wt{\delta}_1 \in (0,1)$ to be sufficiently small so that $\dist(z_\Left,I^\Left) \leq \delta_1\diam(A_{z,k+1})$ and $\dist(z_\Right,I^\Right) \leq \delta_1\diam(A_{z,k+1})$.

Next, noting that an argument similar to that of the proof of Lemma~\ref{lem:separation_past_future} implies that we can choose $\delta_1 > 0$ so small that $\dist(\psi_{z,k+1}^{-1}(J),\wt{\phi}_R(\wt{\eta}([\wt{\tau}_{z,k+1},\infty)))) \geq \delta_1 \diam(A_{z,k+1})$. From now on, we assume that $\delta_1$ is taken to be that small. We next claim that we can choose $\delta_2 > 0$ sufficiently small so that it is likely that we also have that $\dist(\{z_\Left,z_\Right\},\wt{\phi}_R(\wt{\eta}([\wt{\tau}_R,\wt{\tau}_{z,k+1}]))) \geq \delta_2 \diam(A_{z,k+1})$.  Indeed, let $x_{m-1}$ and $x_m$ be the endpoints of $\psi_{z,k+1}^{-1}(J_m)$ in clockwise order for $m=1,\ldots,n(\wt{\delta}_1)$,  where $x_0 = z_\Left$ and $x_{n(\wt{\delta}_1)} = z_\Right$.  Then Lemma~\ref{lem:radial_sle_hits_a_ball} applied to $\psi_{z,k+1} \circ \wt{\phi}_R(\wt{\eta}_1|_{[\wt{\tau}_R,0]})$ (note $\psi_{z,k+1} \circ \wt{\phi}_R = \wt{\psi}_{\wt{\tau}_R}$) implies that we can find $\delta_2 \in (0,1)$ sufficiently small such that conditional on $\CG_{\tau_{z,k}}$ and on the event $B_{z,k}$ we have that $\wt{\phi}_R(\wt{\eta}([\wt{\tau}_R,0])) \cap B(x_m ,  \delta_2\diam(A_{z,k+1})) = \emptyset$ for each $m=0,\ldots,n(\wt{\delta}_1)$ with probability at least $1-p_0/100$.  In particular,  we have that $\dist(\{z_\Left,z_\Right\},\wt{\phi}_R(\wt{\eta})) \geq \delta_2 \diam(A_{z,k+1})$. Combining this with Lemma~\ref{lem:lower_bound_F1} gives the result.
\end{proof}

The next lemma states that $I$ is likely to satisfy conditions~\eqref{it:lower_bound_natural_measure} and~\eqref{it:measure_diameter_bound}.

\begin{lemma}\label{lem:lower_bound_F3}
Fix $a,p_0 \in (0,1)$ and let $R,A,C,\delta,\epsilon,\delta_1,\delta_2$ be as above. There exists $M_0^* \geq 1$ such that for all $M_0 \geq M_0^*$, the following holds. Define the event
\begin{align*}
	F_3 = F_2 \cap \{ \nmeasure{\wt{\phi}_R(\wt{\eta})}(I) \geq M_0^{-1} \diam(A_{z,k+1})^{3/2} \} \cap \{ \nmeasure{\wt{\phi}_R(\wt{\eta})}(Y \cap I) \leq M_0 \diam(Y)^{3/2 - a},  \forall Y \ \text{Borel} \}.
\end{align*}
Then,
\begin{align*}
	\p[ F_3 \giv \CG_{\tau_{z,k}}] \one_{B_{z,k}} \geq (1-p_0/10) \one_{B_{z,k}}.
\end{align*}
\end{lemma}
\begin{proof}
By an argument to the one used in Lemma~\ref{lem:separation_past_future} and noting that $\nmeasure{\wt{\phi}_R(\wt{\eta})}(dw) = |\wt{\phi}_R'(\wt{w})|^{3/2} \nmeasure{\wt{\eta}}(d\wt{w})$,  we obtain that the conditional probability of $F_1$ and the event that
\begin{equation}
\label{eqn:good_interval_nat_lbd}
\nmeasure{\wt{\phi}_R(\wt{\eta})}(I) \gtrsim \delta_1^2 \diam(A_{z,k+1})^{3/2}	
\end{equation}
given $\CG_{\tau_{z,k}}$ and on $B_{z,k}$ is at least $1-p_0/24$,  where the implicit constant depends only on~$C$ and~$C_R$ (recall the proof of Lemma~\ref{lem:separation_past_future}).  Thus, choosing $M_0$ large enough so that $M_0^{-1}$ is at most $\delta_1^2$ times the implicit constant we need only care about the event about the upper bound. 

Moreover, by arguing as in the proof of Lemma~\ref{lem:first_crossing_number_of_squares} and using Lemma~\ref{lem:natural_good_whole_plane}, we have (by possibly increasing the value of $M_0$) that on $F_2 \cap \{ \nmeasure{\wt{\phi}_R(\wt{\eta})}(I) \geq M_0^{-1} \diam(A_{z,k+1})^{3/2} \}$, the inequality $\nmeasure{\wt{\phi}_R(\wt{\eta})}(Y \cap I) \leq M_0 \diam(Y)^{3/2 - a}$ holds for all Borel sets $Y$. Thus the result follows.
\end{proof}

For the next lemma, we consider the following. Let $\wt{\h}_\Left$ (resp.\ $\wt{\h}_\Right$) be the connected component of $\h \setminus (\eta([0,\tau_{z,k+1}^*]) \cup \wt{\phi}_R(\wt{\eta}))$ to the left (resp.\ right) of $\eta([0,\tau_{z,k+1}^*]) \cup \wt{\phi}_R(\wt{\eta})$ and let $\wt{\CW}_q$ be a Whitney square decomposition of $\wt{\h}_q$, $q \in \{\Left,\Right\}$, chosen in some measurable way. 

\begin{lemma}\label{lem:lower_bound_F4}
Fix $a,p_0 \in (0,1)$ and let $R,A,C,\delta,\epsilon,\delta_1,\delta_2$ be as above. There exists $M_0 > 0$ such that the following holds for all $M \geq M_0$. Let $F_4$ be the intersection of $F_3$ and the event that there exists a set $I_0 \subset I$ such that $\nmeasure{\wt{\phi}_R(\wt{\eta})}(I \setminus I_0) \leq \nmeasure{\wt{\phi}_R(\wt{\eta})}(I)/100$ such that for $q \in \{\Left,\Right\}$ and for each $w \in I_0$ and each $\wt{Q}$ which is intersected by $\gamma_{z_q,w}^{\wt{\h}_q}$ (recall Section~\ref{subsec:whitney_hyperbolic}), the following holds
\begin{align}\label{eq:good_square_half_plane}
	\disthyp^{\wt{\h}_q}(z_q,\cen(\wt{Q})) \leq M (\len(\wt{Q})/\diam(A_{z,k+1}))^{-a}.
\end{align}
Then
\begin{align*}
	\p[ F_4 \giv \CG_{\tau_{z,k}}] \one_{B_{z,k}} \geq (1-p_0/8) \one_{B_{z,k}}.
\end{align*}
\end{lemma}

\begin{proof}
Assume that we are working on $F_3$. We can find $r>0$ large (depending only on $R$) and by decreasing the value of $\delta_2 \in (0,1)$ if necessary we have that $\wt{\phi}_R^{-1}(B(z,c_1 \diam(A_{z,k+1}))) \subseteq B(0,r)$ and $\dist(\wt{\phi}_R^{-1}(z_\Left),\wt{\eta}) \geq C_R^{-1}\delta_2$,  $\dist(\wt{\phi}_R^{-1}(z_\Right),\wt{\eta}) \geq C_R^{-1} \delta_2$ (where $c_1$ and $C_R$ are as in the proof of Lemma~\ref{lem:separation_past_future}).  

We let $\C_\Left$ (resp.\ $\C_\Right$) denote the connected of $\C \setminus \wt{\eta}$ which is to the left (resp.\ right) of $\wt{\eta}$ and let $\CW_q$, $q \in \{\Left,\Right\}$, be a Whitney square decomposition of $\C_q$ chosen in some fixed measurable way. We fix some $\Ff \in (0,1)$ (to be determined later) and $p_1 \in (0,1)$ and note that the proof of Lemma~\ref{lem:good_fraction_whole_plane} combined with the translation invariance property of $\wt{\eta}$ when parameterized according to the natural parameterization implies that there exists $M>0$ such that conditional on $\CG_{\tau_{z,k}}$ and on the event $B_{z,k}$, with probability at least $1-p_1$ the following holds for sufficiently small $\delta_2$.  Fix a point $z'$ in $B(0,r) \setminus B(0,r_0)$ for some $r_0 \in (0,r)$.  Then $\dist(z',\wt{\eta}) \geq C_R^{-1} \delta_2$, and if we let $q \in \{ \Left,\Right \}$ be such that $z' \in \C_q$ and $X$ be the set of $t \in [-A,0]$ such that
\begin{align}\label{eq:good_squares_whole_plane}
	\disthyp^{\C_q}(\cen(Q),z') \leq M \len(Q)^{-a},
\end{align}
for each $Q \in \CW_q$ which is intersected by the hyperbolic geodesic in $\C_q$ from $z'$ to $\wt{\eta}(t)$. Then $\Leb_1(X) \geq (1-\Ff)A$. Thus, picking $p_1 > 0$ extremely small, $M > 0$ appropriately large, considering a fine grid of points $(z_j)$ in $B(0,r)$ and taking a union bound, it holds that conditional on $\CG_{\tau_{z,k}}$ and on the event $B_{z,k}$, we have that with probability at least $1-p_0/100$, we can find points $z_\Left' \in \C_\Left \cap B(0,r)$ and $z_\Right' \in \C_\Right \cap B(0,r)$ such that $\dist(\{z_\Left',z_\Right'\},\wt{\eta}) \geq C_R^{-1} \delta_2$, so that if $X_q$ is the set of $t \in [-A,0]$ such that for each $Q \in \CW_q$ intersected by $\gamma_{z_q',\wt{\eta}(t)}^{\C_q}$,~\eqref{eq:good_squares_whole_plane} holds  with $z' = z_q'$, then $\Leb_1(X_q) \geq (1-\Ff)A$. 

Thus, letting $X = X_\Left \cap X_\Right$, we may choose $\Ff \in (0,1)$ so that $\nmeasure{\wt{\phi}_R(\wt{\eta})}(I \cap \wt{\phi}_R( \wt{\eta}([-A,0] \setminus X))) \leq \nmeasure{\wt{\phi}_R(\wt{\eta})}(I) / 100$ and applying Lemma~\ref{lem:hyperbolic_geodesics_close}, we have that~\eqref{eq:good_squares_whole_plane} holds with $z' = z_\Left$ and $z' = z_\Right$ by possibly increasing $M$ by a universal constant. Hence we are done if we can show that~\eqref{eq:good_square_half_plane} follows from~\eqref{eq:good_squares_whole_plane}, so that we can set $I_0 = I \cap \wt{\phi}_R(\wt{\eta}( X))$.

Suppose that the intersection of $F_3$ and the above event occurs (the conditional probability of this given $\CG_{z,k}$ and the event $B_{z,k}$ is at least $1-p_0/8$).  Note that $z_\Right \in \wt{\h}_\Right$ since in order for $z_\Right$ to lie in $\wt{\h}_\Left$,  the curve $\wt{\phi}_R(\wt{\eta})([\wt{\tau}_{z,k+1},\infty))$ must separate $z_\Right$ from the right side of $\wh{I}$ and the latter can only happen when $\wt{\phi}_R(\wt{\eta})([\wt{\tau}_{z,k+1},\infty))$ comes within distance $\delta_1\diam(A_{z,k+1})$ of $z_\Right$ since $\dist(z_\Right,\wh{I}) \leq \delta_1\diam(A_{z,k+1})$ but the choice of~$\delta_1$ implies that this cannot happen.  Similarly, we have that $z_\Left \in \wt{\h}_\Left$. Fix $\wt{w}_0 \in I \cap \wt{\phi}_R(\wt{\eta}(X))$ and let $\wt{Q} \in \wt{\CW}_\Right$,  $\wt{z} \in \wt{Q}$ be such that $\gamma_{z_\Right,\wt{w}_0}^{\wt{\h}_\Right}(s) = \wt{z}$ for some $s\geq0$.  Note that $\diam(\gamma_{z_\Right,\wt{w}_0}^{\wt{\h}_\Right}) \leq c_1 \delta_1 \diam(A_{z,k+1})$ for some universal constant $c_1>0$ and so by taking $\delta_1 > 0$ sufficiently small we have that $\gamma_{z_\Right,\wt{w}_0}^{\wt{\h}_\Right} \subseteq A_{z,k+1}$.  This implies that $ \len(\wt{Q}) \leq \diam(A_{z,k+1})$ and so we can assume that $\disthyp^{\wt{\h}_\Right}(z_\Right,\wt{Q}) \geq 1$.  Let $Q \in \CW_\Right$ be such that $\wt{w} = \wt{\phi}_R^{-1}(\wt{z}) \in Q$.  Then we have that $\wt{w} \in \gamma_{\wt{\phi}_R^{-1}(z_\Right),w_0}^{\C_\Right}$ where $w_0 = \wt{\phi}_R^{-1}(\wt{w}_0) \in \wt{\eta}(X)$ and so $\disthyp^{\C_\Right}(\wt{\phi}_R^{-1}(z_\Right),\cen(Q)) \leq M\len(Q)^{-a}$.  This implies that
\begin{align*}
	\disthyp^{\wt{\h}_\Right}(z_\Right,\wt{Q}) &\leq \disthyp^{\wt{\h}_\Right}(z_\Right,\wt{z}) = \disthyp^{\C_\Right}(\wt{\phi}_R^{-1}(z_\Right),\wt{w}) \\
	&\leq \disthyp^{\C_\Right}(\wt{\phi}_R^{-1}(z_\Right),\cen(Q))+1 \leq M\len(Q)^{-a} + 1 
\end{align*}
since $\disthyp^{\C_\Right}(\wt{x},\wt{y}) \leq 1$ for each $\wt{x},\wt{y} \in Q$.  It follows that $\disthyp^{\wt{\h}_\Right}(z_\Right,\cen(\wt{Q})) \lesssim \len(Q)^{-a}$.  Moreover~\eqref{eq:derivative_diameter_relation} implies that $\len(\wt{Q})/\len(Q) \asymp \diam(A_{z,k+1})$.  It follows that
\[ \disthyp^{\wt{\h}_\Right}(z_\Right,\cen(\wt{Q}))\lesssim ( \len(\wt{Q})/\diam(A_{z,k+1}))^{-a}\]
where the implicit constant depends only on the previously chosen constants. Analogously, the same holds with $z_\Left$ and $\wt{\h}_\Left$ in place of $z_\Right$ and $\wt{\h}_\Right$ and thus the proof is done.
\end{proof}

The final ingredient in the proof of Lemma~\ref{lem:first_excursion_good} will be the following, which transfers the result of Lemma~\ref{lem:lower_bound_F4} to domains $W$ as in~\eqref{it:x_y_pos}.
\begin{lemma}\label{lem:change_of_domains}
Fix $a,p_0 \in (0,1)$ and let $R,A,C,\delta,\epsilon,\delta_1,\delta_2$ be as above. There exists $M_0 > 0$ such that if $F_5$ is the intersection of $F_4$ and the event that~\eqref{eq:good_square_half_plane} holds with any $W$ as in~\eqref{it:x_y_pos} in place of $\wt{\h}_q$. Then,
\begin{align*}
	\p[ F_5 \giv \CG_{\tau_{z,k}}] \one_{B_{z,k}} \geq (1-p_0/5) \one_{B_{z,k}}.
\end{align*}
\end{lemma}
\begin{proof}
The entirety of this proof is dedicated to showing that with sufficiently high probability, we are in the setting to use Lemma~\ref{lem:hyperbolic_geodesics_close}. 
We fix a simply connected domain $W \subseteq A^*_{z,k+1}$ as in condition~\eqref{it:x_y_pos} in the case of $z_\Right$ and suppose that $\CW$ is a Whitney square decomposition of $W$.  We will now apply Lemma~\ref{lem:hyperbolic_geodesics_close} in order to argue that after possibly increasing the value of $M$ (independently of $W$), the hyperbolic geodesics in $W$ from $z_\Right$ to $I$ intersect only squares $Q$ satisfying $\disthyp^W(\cen(Q),z_\Right) \leq M (\len(Q) / \diam(A_{z,k+1}))^{-a}$.  First, we introduce one more event for~$\wt{\eta}$.  We let $\CS$ be the set of squares whose top left corner is in $(\delta_1^{1/2}/4)\Z^2$,  have side length equal to $\delta_1^{1/2}$ and are contained in $B(0,r)$.  Note that $\CS = \{S_1,\ldots,S_m\}$ where~$m$ depends only on~$r$ and~$\delta_1$.  For $1 \leq j \leq m$ we let $U^j_1,\ldots,U^j_{k_j}$ be the connected components of $S_j \setminus \wt{\eta}$ which contain a point whose distance to $\partial S_j \setminus \wt{\eta}$ is at least $C_R^{-1}\delta_2$ and their boundaries also contain an arc of $\wt{\eta}$ of diameter $C_R^{-1}\delta_2$.  Note that for each $1 \leq j \leq m$, $k_j$ is almost surely bounded by a constant depending only on $R,\delta_1,\delta_2$.  For $p \in (0,1)$ fixed,  we let $\wt{F}_5$ be the event that $F_4$ occurs and the following holds.  For each $1 \leq j \leq m$ and $1\leq i \leq k_j$, there is a point in $U_i^j$ from which the harmonic measure in $U_i^j$ of any arc of $\wt{\eta}$ contained in $\partial U^j_i$ of diameter at least $C_R^{-1}\delta_2$, is at least $p$.  We claim that for $p \in (0,1)$ sufficiently small, we have that $\p[\wt{F}_5 \giv  \CG_{\tau_{z,k}}]\one_{B_{z,k}} \geq (1-p_0/7)\one_{B_{z,k}}$. Indeed, we consider conformal transformations $\psi^j_i \colon U_i^j \to \D$ for each $i,j$ chosen in some fixed measurable way.  Then the claim follows by combining the fact that the $\psi^j_i$ extend to homeomorphisms mapping $\closure{U^j_i}$ onto $\closure{\D}$ with the explicit form of the Poisson kernel in $\D$.  Next we let $1 \leq j \leq m$ be such that $\wt{\phi}_R^{-1}(z_\Right) \in S_j$ and $\dist(\wt{\phi}_R^{-1}(z_\Right),\partial S_j) \geq \delta_1^{1/2}/4$.  Let~$\wt{U}$ be the connected component of $S_j \setminus \wt{\eta}$ containing $\wt{\phi}_R^{-1}(z_\Right)$.  Note that $\dist(\wt{\phi}_R^{-1}(z_\Right),\wt{\eta})\geq C_R^{-1}\delta_2$ and so $\dist(\wt{\phi}_R^{-1}(z_\Right),\partial \wt{U}) \geq C_R^{-1}\delta_2$ for $\delta_2 > 0$ sufficiently small.  We set $\wt{J} = \partial \wt{U} \setminus \wt{\eta}$ and $\wh{J} = \wt{\phi}_R(\wt{J})$.  For $\delta_1 > 0$ sufficiently small,  we have that $\wh{J} \subseteq A^*_{z,k+1}$,  $\dist(z_\Right,\wh{J}) \asymp \delta_1^{1/2}\diam(A_{z,k+1})$,  $I \subseteq \wt{\phi}_R(\wt{\eta}\cap \partial \wt{U})$ and that $\wh{J}$ separates $z_\Right$ from $\infty$.  By possibly taking $\delta_2$ to be smaller,  we have that for every arc $K \subseteq \wh{I}$ such that $\diam(K) \geq \delta_2 \diam(A_{z,k+1})$,  it holds that $\diam(\wt{\phi}_R^{-1}(K)) \geq C_R^{-1} \delta_2$.  We fix $w \in I$ such that~$\gamma_{z_\Right,w}^{\wt{\h}_\Right}$ intersects only squares $Q$ satisfying $\disthyp^{\wt{\h}_\Right}(\cen(Q),z_\Right) \leq M (\len(Q)/\diam(A_{z,k+1}))^{-a}$ and we let $\phi_w \colon \h_{\tau_{z,k+1}} \to \h$ be the conformal transformation such that~$\phi_w(w) = 0$ and~$\phi_w(z_\Right) = i$.  We claim that for $c \in (0,1)$ sufficiently small,  we have that $B(0,c) \cap \phi_w(\wh{J}) = \emptyset$.  Indeed,  suppose that $B(0,c) \cap \phi_w(\wh{J}) \neq \emptyset$ and let $\wt{x}$ and $\wt{y}$ be the endpoints of $\phi_w(\wh{J})$, assume by symmetry that $\wt{y}$ is the smallest positive of the two (so that either $\wt{x} < 0 < \wt{y}$ or $0 < \wt{y} < \wt{x}$) and let $\wt{X} = [-\wt{x},0]$ if $\wt{x} < 0$ and $\wt{X} = (-\infty,0] \cup [\wt{x},\infty)$ if $\wt{x} > 0$. Let also $\wt{W}$ be the connected component of $\h \setminus \phi_w(\wh{J})$ which contains $i$.  Since $\wt{F}_5$ occurs and the clockwise arc of $\partial \wt{\phi}_R(\wt{U})$ connecting $\wt{\phi}_R(\partial \wt{U} \setminus \wt{J})$ to $w$ has diameter at least $\delta_1 \diam(A_{z,k+1})$,  it follows that the probability that a Brownian motion starting from $i$ exits $\wt{W}$ in $[0,\wt{y}]$ is at least $p$ and the same holds for $\wt{X}$. Furthermore, since $B(0,c) \cap \phi_w(\wh{J}) \neq \emptyset$, it follows that  either $\wt{X}$ or $[0,\wt{y}]$ is separated from $i$ by $\phi_w(\wt{J}) \cup \partial B(0,c)$. Consequently, in order for a Brownian motion from $i$ to exit $W$ in that particular set, it has to first intersect $B(0,c)$ without exiting $\h$. The probability of this, however, tends to $0$ as $c \to 0$, which contradicts the fact that the probability of exiting in said set is at least $p$. Moreover, we have that $\dist(\phi^{-1}_w(ic/4),\eta) \gtrsim \delta_2 \diam(A_{z,k+1})$ where the implicit constant depends only on $c$ and $\delta_1$ and so we can find $\zeta > 0$ sufficiently small (depending only on $c$ and $\delta_1$) such that $\dist(\wt{z},\wt{\eta}) \geq \zeta$,  where $\wt{z} = \wt{\phi}^{-1}_R(\phi^{-1}_w(ic/4))$.  By possibly taking $\zeta > 0$ to be smaller,  we have that $\dist(\wt{z},\partial \wt{U}) \geq \zeta$.

Now we let $\wt{U}_1,\ldots,\wt{U}_n$ be the connected components of $S_k \setminus \wt{\eta}$ for some $1 \leq k \leq m$ which contain some point whose distance from $\partial S_k \setminus \wt{\eta}$ is at least $\zeta$.  Note that $n < \infty$ a.s. Let $\wt{F}_6$ be the event that $\wt{F}_5$ occurs and the following holds.  For each $1 \leq i \leq n$ and each $x_1,x_2 \in \wt{U}_i$ such that $\dist(\{x_1,x_2\},\partial \wt{U}_i) \geq \zeta$,  we have that $\disthyp^{\wt{U}_i}(x_1,x_2) \leq C$. Arguing as before, we can choose $C > 0$ (depending only on the above implicit constants) such that $\p[\wt{F}_6  \giv   \CG_{\tau_{z,k}}]\one_{B_{z,k}} \geq (1-p_0/5)\one_{B_{z,k}}$.  Suppose that we are working on $\wt{F}_6$.  Then the conformal invariance of the hyperbolic metric implies that $\disthyp^{\wt{W}}(i,ic/4) \leq C$.  Consider the conformal automorphism of $\h$ given by $\wt{\psi}(w) = 2w/c$.  Then $\inrad(\wt{\psi}(\wt{W})) \geq 2$ and so Lemma~\ref{lem:hyperbolic_geodesics_close} implies that there exists a constant $\wt{C}>0$ depending only on $C$ such that $\disthyp^{\phi_w(\wt{\h}_\Right)}(\gamma_{i,0}^{\phi_w(\wt{\h}_\Right)}(t),\gamma_{ic/4,0}^{\h}(t)) \leq \wt{C}$ for each $t \geq 0$.  Since $\wt{\h}_\Right \subseteq \h_{\tau_{z,k+1}}$,  it follows that $\disthyp^{\h_{\tau_{z,k+1}}}(\gamma_{z_\Right,w}^{\wt{\h}_\Right}(t),\gamma^{\h_{\tau_{z,k+1}}}_{\phi^{-1}_w(ic/4),w}(t)) \leq \wt{C}$ for each $t \geq 0$.  Let $\wh{\CW}$ be a Whitney square decomposition of $\h_{\tau_{z,k+1}}$ and fix $\wh{Q} \in \wh{\CW}$ such that $\gamma_{\phi^{-1}_w(ic/4),w}^{\h_{\tau_{z,k+1}}}(t) \in \wh{Q}$ for some $t \geq 0$.  Note that $\dist(\phi^{-1}_w(ic/4),I) \lesssim \delta_1^{1/2} \diam(A_{z,k+1})$ and so for $\delta_1$ small enough we obtain that $\dist(\wh{Q},\partial \h_{\tau_{z,k+1}}) = \dist(\wh{Q},\partial \wt{\h}_\Right)$.  Let $\wt{Q} \in \wt{\CW}_R$ be such that $\gamma_{z_\Right,w}^{\wt{\h}_\Right}(t) \in \wt{Q}$.  Then we have that $\disthyp^{\h_{\tau_{z,k+1}}}(\wh{Q},\wt{Q}) \leq \wt{C}$ and so $\len(\wh{Q}) \asymp \len(\wt{Q})$,  where the implicit constants depend only on $\wt{C}$.  Also,  $\len(\wh{Q}) \leq \diam(A_{z,k+1})$.  Therefore,  arguing as the proof of Lemma~\ref{lem:lower_bound_F4},  we obtain that $\disthyp^{\h_{\tau_{z,k+1}}}(\phi_w^{-1}(ic/4),w) \lesssim M(\len(\wh{Q})/\diam(A_{z,k+1}))^{-a}$.  Note that $\wt{W} \subseteq \phi_w(W)$ and so $\disthyp^{\phi_w(W)}(i,ic/4) \leq C$.  Hence arguing as above by applying Lemma~\ref{lem:hyperbolic_geodesics_close} with $D = \h_{\tau_{z,k+1}}, \wt{D} = W,  z_0 = z_\Right$ and $w_0 = \phi_w^{-1}(ic/4)$,  we obtain that $\disthyp^{W}(z_\Right,\cen(Q)) \lesssim M(\len(Q)/\diam(A_{z,k+1}))^{-a}$ for each $Q \in \CW_{z,k+1}$ such that $\gamma_{z_\Right,w}^{W} \cap Q \neq \emptyset$.
\end{proof}

We now conclude this section with the proof of Lemma~\ref{lem:first_excursion_good}. This will be done by comparing the laws of radial and chordal $\SLE$.
\begin{proof}[Proof of Lemma~\ref{lem:first_excursion_good}]
Let us now compare the laws of radial $\SLE_4(2)$ and chordal $\SLE_4$.  Let $\delta_0,\delta_1,\delta_2, p_0,a \in (0,1)$,  $b \in (1,2)$, and  $\xi_2 \in (0,(e^{-b}-e^{2})/2)$ be fixed and conditional on $\CG_{\tau_{z,k}}$ we let $\p_{z,k}^*$ be the law of a radial $\SLE_4(2)$ process $\wh{\eta}$ in $\h \setminus \eta([0,\tau_{z,k}])$ from $\eta(\tau_{z,k})$ to $z$ with the force point located at $\infty$ and stopped at the first time $\wh{\tau}_{z,k+1}$ that it hits $\bIn A_{z,k+1}$.  Let also $\E_{z,k}^*$ denote expectation with respect to $\p_{z,k}^*$.  Then by Lemmas~\ref{lem:lower_bound_F1}-\ref{lem:change_of_domains}, (replacing $p_0$ with $\delta_0 p_0$) implies that we can find $M>1$ large and $\delta_1 \in (0,1)$ small such that $\p_{z,k}^*[ C_{z,k+1}^c] \one_{B_{z,k}} \leq p_0 \delta_0 \one_{B_{z,k}}$.  Let $M_t$ be as in the proof of Lemma~\ref{lem:extremal_length_good} and let $E_{z,k+1}$ is the event that $\tau_{z,k+1} < \infty$,  $\inf_{\tau_{z,k}\leq t \leq \tau_{z,k+1}} S_t \geq \delta_0$ and $C_{z,k+1}^c$ holds.  As in the proof of Lemma~\ref{lem:extremal_length_good}, we have that
\begin{align*}
\p[ E_{z,k+1}\,\giv \CG_{\tau_{z,k}}] \one_{B_{z,k}} &= \E\left [ \left( \frac{M_{\tau_{z,k+1}}}{M_{\tau_{z,k}}}\right)^{-1} \left( \frac{M_{\tau_{z,k+1}}}{M_{\tau_{z,k}}}\right) \one_{E_{z,k+1}}\,  \giv \CG_{\tau_{z,k}}\right] \one_{B_{z,k}}\\
&= S_{\tau_{z,k}} \E_{z,k}^*\left[ S_{\tau_{z,k+1}}^{-1} \left( \frac{\Upsilon_{\tau_{z,k+1}}}{\Upsilon_{\tau_{z,k}}}\right)^{1/2} \one_{E_{z,k+1}}\right] \one_{B_{z,k}}\\
&\leq \delta_0^{-1} \p_{z,k}^*[E_{z,k+1}] \one_{B_{z,k}} \leq \delta_0^{-1}\p_{z,k}^*[C_{z,k+1}^c]\one_{B_{z,k}} \leq p_0 \one_{B_{z,k}},
\end{align*}
which completes the proof.
\end{proof}

\section{Exploration of rectangle crossings}
\label{sec:rectangle_exploration}

\subsection{Definition of the exploration}
\label{subsec:exploration_definition}

\begin{figure}[ht!]
\begin{center}
\includegraphics[scale=0.9]{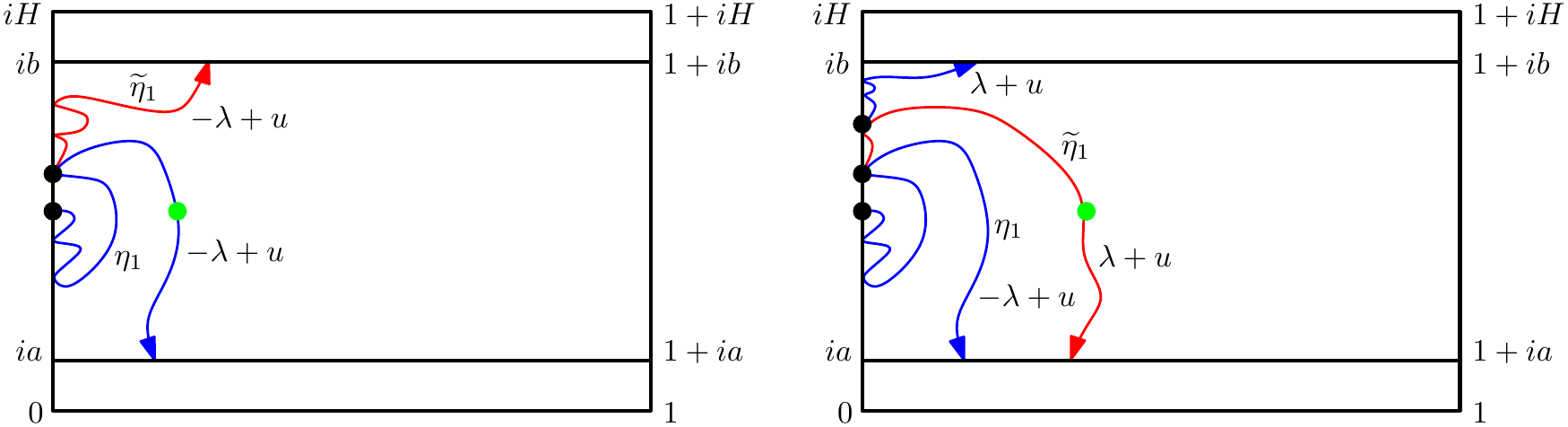}	
\end{center}
\caption{\label{fig:exploration_illustration}  Illustration of the first stage of the exploration in the case that $\eta_1$ exits~$\CR_{a,b}$ in $\pdown \CR_{a,b}$.  Note that the boundary data for $h$ along the left side of $\eta_1$ is $-\lambda+u$.  {\bf Left:} Shown in red is $\wt{\eta}_1$ in the case that it exits $\CR_{a,b}$ in $\pup \CR_{a,b}$.  Since $\wt{\eta}_1$ is a level line of $-h+u$, the boundary data for $h$ along its right side is $-\lambda+u$.  Altogether, the exploration has discovered a crossing $L_1$ so that the field boundary data on its right side (going from $\pdown \CR_{a,b}$ to $\pup \CR_{a,b}$) is $-\lambda+u$.  {\bf Right:} Shown in red is $\wt{\eta}_1$ in the case that it exits~$\CR_{a,b}$ in $\pdown \CR_{a,b}$ and the next level line in the exploration exits~$\CR_{a,b}$ in $\pup \CR_{a,b}$.  Altogether, the exploration has discovered a crossing $L_1$ so that the field boundary data on its right side (going from $\pdown \CR_{a,b}$ to $\pup \CR_{a,b}$) is $\lambda+u$. In both cases, the starting point of the part of the exploration which will discover $L_2$ is shown as a green dot.}
\end{figure}

Fix $H > 0$ and let $\CR = (0,1) \times (0,H)$.  Fix $0 < a < b < H$ and let $\CR_{a,b} = (0,1) \times (a,b)$.  We let $\pdown \CR$, $\pup \CR$, $\pleft \CR$ and $\pright \CR$ be the lower, upper, left and right boundary of $\CR$, respectively.  We use the same notation with $\CR_{a,b}$ in place of $\CR$.  Let $h$ be a GFF in~$\CR$ with boundary conditions given by $\lambda$ on $\pleft \CR$, $-\lambda$ on $\pright \CR$, and $0$ on $\pup \CR$ and $\pdown \CR$.  We shall consider an exploration of level lines of $h$ which form crossings of $\CR_{a,b}$; see Figure~\ref{fig:exploration_illustration} for an illustration.  Fix $u>0$ small and let $\eta_1$ be the level line of $h$ of height $u$ (i.e., the level line of $h-u$) started from the midpoint $i(a+b)/2$ of $\pleft \CR_{a,b}$.  Then, there are two possible outcomes:
\begin{enumerate}[(i)]
	\item $\eta_1$ exits $\CR_{a,b}$ in $\pdown \CR_{a,b}$. If this occurs, then we start to explore the level line $\wt{\eta}_1$ of $-h+u$ from the point of $\pleft \CR_{a,b}$ which has the largest imaginary part among the points which have been visited by the exploration. If $\wt{\eta}_1$ hits $\pup \CR_{a,b}$ before $\pdown \CR_{a,b}$ or $\pright \CR_{a,b}$, then the first stage of the exploration is concluded and we continue the exploration by iterating as described below. If $\wt{\eta}_1$ hits $\pdown \CR_{a,b}$ before $\pup \CR_{a,b}$, then we explore the level line of height $u$ from the point of $\pleft \CR_{a,b}$ which has the largest imaginary part among the points which have been visited by the exploration at this point in time. Repeat this step until the newly generated level line of height $u$ hits $\pup \CR_{a,b}$ before $\pdown \CR_{a,b}$ and $\pright \CR_{a,b}$, which concludes the first stage of the exploration. \label{item:option_2}
	\item $\eta_1$ exits $\CR_{a,b}$ in $\pup \CR_{a,b}$. If this occurs, then we start to explore the level line $\wt{\eta}_1$ of $-h+u$ from the point of $\pleft \CR_{a,b}$ which has the smallest imaginary part among the points which have been visited by the exploration. If $\wt{\eta}_1$ hits $\pdown \CR_{a,b}$ before $\pright \CR_{a,b}$ and $\pup \CR_{a,b}$, then the first stage of the exploration is concluded and next we iterate the exploration as described below. If $\wt{\eta}_1$ hits $\pup \CR_{a,b}$ before $\pdown \CR_{a,b}$ and $\pright \CR_{a,b}$, then we explore the level line of height $u$ from the point of $\pleft \CR_{a,b}$ which has the smallest imaginary part among the points which have been visited by the exploration at this point in time. Repeat this step until the newly generated level line of height $u$ hits $\pdown \CR_{a,b}$ before $\pup \CR_{a,b}$ and $\pright \CR_{a,b}$, concluding the first stage of the exploration. \label{item:option_3}
\end{enumerate}

Let $K_1$ be the union of $\pleft \CR_{a,b}$ and the set discovered by the first stage of the exploration.  The exploration a.s.\ does not discover a level line which hits $\pright \CR_{a,b}$ (as this is only possible if the height is in $(-2\lambda,0)$) so we let $U_1$ be the component of $\CR_{a,b} \setminus K_1$ which has $\pright \CR_{a,b}$ as part of its boundary and we let $L_1 = \closure{\partial U_1 \setminus \partial \CR_{a,b}}$.  Note that $L_1$ consists of two level lines, one of $h-u$ and one of $h+u$, from $\pleft \CR_{a,b}$ which exit $\pup \CR_{a,b}$ and $\pdown \CR_{a,b}$.  This implies that the boundary conditions of $h$ along $L_1$ are either $-\lambda + u$ or $\lambda+u$.  We say that $L_1$ is a crossing of $\CR_{a,b}$ of height $u$.

We inductively define successive stages of the exploration after $j$ steps as follows, assuming that it has not yet discovered a level line which hits $\pright \CR_{a,b}$ (which is not possible in the first stage of the iteration, but may happen later, as we will explain).  Let $K_j$ be the union of $\pleft \CR_{a,b}$ and the exploration after $j$ steps, $L_j$ the $j$th crossing, $u_j$ its height, and $U_j$ the component of $\CR_{a,b} \setminus K_j$ with $\pright \CR_{a,b}$ on its boundary.  For the iteration step, we make the following definition.
\begin{itemize}
	\item If the boundary conditions of $h$ in $U_j$ on $L_j$ are equal to $\lambda + u_j$, then let $\eta_{j+1}$ be the level line of $h$ of height $u_{j+1}=u_j+u$ starting from the leftmost intersection of $L_j$ with the line $\{z : \im(z) = (a+b)/2\}$.
	\item If the boundary conditions of $h$ in $U_j$ on $L_j$ are equal to $-\lambda + u_j$, then let $\eta_{j+1}$ be the level line of $h$ of height $u_{j+1}=u_j-u$ starting from the leftmost intersection of $L_j$ with the line $\{z : \im(z) = (a+b)/2\}$.
\end{itemize}

We repeat the steps described above, replacing $\pleft \CR_{a,b}$ with $L_j$ and $\eta_1$ by $\eta_{j+1}$ as follows. If $u_{j+1} \notin (-2\lambda,0)$, then $\eta_{j+1}$ a.s.\ does not hit $\pright \CR_{a,b}$, so we proceed as above. In the choice of where to start the next level line, one replaces the point of intersection with the largest (resp.\ smallest) imaginary part by last point with respect to the upward (resp.\ downward) exploration of $L_j$ which was intersected by the previous level line. However, if $u_{j+1} \in (-2\lambda,0)$, then $\eta_{j+1}$, or any of the subsequent level lines in the exploration steps described above, may exit $\CR_{a,b}$ in $\pright \CR_{a,b}$ and if one of them does, then we stop the exploration at that point.  From this exploration we obtain a sequence of crossings $(L_j)$ of $\CR_{a,b}$ with the property that $L_j \cap L_{j+1} \neq \emptyset$ for each $j$.

The purpose of the remainder of this section is to establish two estimates for the exploration that we have defined.  In Section~\ref{subsec:max_height}, we will prove a tail bound for $\sup_j |u_j|$.  Then, in Section~\ref{subsec:number_of_crossings}, we will show that the number of crossings that the exploration discovers is a.s.\ finite.  (One can see that the exploration a.s.\ discovers at most finitely many level lines at each stage using the same argument used to prove Lemma~\ref{lem:finitely_many_crossings}.)

\subsection{Maximum height}
\label{subsec:max_height}

\begin{lemma}
\label{lem:rectangle_maximum_crossing_height}
For each $H > 0$ and $0 < a < b < H$ there exist constants $c_0, M_0 > 0$ so that
\[ \p\!\left[ \sup_{j \geq 1} |u_j| \geq M \right] \leq \exp(-c_0 M^{1/3}\log(M)) \quad\text{for all}\quad M \geq M_0.\]
\end{lemma}

The strategy to prove Lemma~\ref{lem:rectangle_maximum_crossing_height} is as follows. We will fix $M \in \N$ and for each $1 \leq m \leq M-1$ let $v_m = \tfrac{m}{M} + \tfrac{i(a+b)}{2}$.  We will then consider the conditional expectation of $h$ evaluated at the points~$v_m$ in~$\CR$ given the exploration parameterized by $\log$-conformal radius as seen from~$v_m$.  By \cite[Proposition~6.5]{ms2016imag1}, such a conditional expectation evolves as a standard Brownian motion and its value is dominated by the boundary values of $h$ along those crossings which are close to~$v_m$ and before disconnecting $v_m$ from $\pleft \CR_{a,b}$. Each time the exploration discovers a new crossing of $\CR_{a,b}$ which is close to $v_m$ and before disconnecting~$v_m$ from $\pleft \CR_{a,b}$ this gives rise to a change in the boundary values of the GFF hence of the aforementioned Brownian motion.  In particular, in order to see a crossing with a large height there must be one of the points $v_m$ near which crossings are observed where the height increases dramatically.  On this event, the associated Brownian motion will change values by a large amount in a small amount of time and the probability of this is bounded by the standard tail estimate for the maximum of a Brownian motion.  We do not believe that the stretched exponential decay that we obtain in the statement of Lemma~\ref{lem:rectangle_maximum_crossing_height} is optimal, but it will suffice for our purposes.

Fix a point $z_0 \in \CR \setminus \closure{\CR_{a,b}}$ and let $(K_s)_{s \geq 0}$ denote the exploration parameterized by $\log$-conformal radius as seen from $z_0$. The particular choice of time parameterization here is not important, as we mentioned just above we shall reparameterize it according to $\log$-conformal radius as seen from a number of different points.  More precisely,  we set $\sigma = \log(\confrad(z_0,\CR)) - \log(\confrad(z_0,\wh{D}))$ where $K = \closure{\cup_{j\geq 1} K_j}$ and $\wh{D}$ is the connected component of $\CR \setminus K$ containing $z_0$. We continuously parameterize the union of $\pleft \CR_{a,b}$ with the level lines discovered in $\CR_{a,b}$ by $(\wt{K}_s)_{0 \leq s <1}$ in the order that they are discovered.  Then for all $0 \leq t < \sigma$,  we set 
\begin{align*}
\wt{s}(t) = \inf\{s \geq 0 : t = \log(\confrad(z_0,\CR)) - \log(\confrad(z_0,\wt{D}_s)) \}
\end{align*}
where $\wt{D}_s$ is the connected component of $\CR \setminus \wt{K}_s$ containing $z_0$.  We set $K_t = \wt{K}_{\wt{s}(t)}$ for $0 \leq t < \sigma$.  Observe that under the above parameterization we have that
$t = \log(\confrad(z_0,\CR)) - \log(\confrad(z_0,D_t))$
where $D_t  = \wt{D}_{\wt{s}(t)}$. 

In proving Lemma~\ref{lem:rectangle_maximum_crossing_height}, we consider the following setup.  Fix $1 \leq m \leq M-1$ and let $\sigma_m$ be the first time $t$ that $K_t$ disconnects $v_m$ from $\pright \CR_{a,b}$  or $\sigma_m = \sigma$ if that does not occur.  We set $s_m = \lim_{t \to \sigma_m}( \log(\confrad(v_m ,  \CR)) - \log(\confrad(v_m ,  D_t)))$.  We also set $s_m(t) = \inf\{ s \geq 0 : \log(\confrad(v_m , \CR)) - \log(\confrad(v_m , D_s)) = t\}$ for $0 \leq t \leq s_m$.  Thus, writing $D^m_t = D_{s_m(t)}$ for $0 \leq t \leq s_m$ we have that $t = \log(\confrad(v_m,\CR)) - \log(\confrad(v_m,D^m_t))$.  Furthermore, we write $K^m_t = K_{s_m(t)}$.  We let $\zeta_m$ be the first time $t$ that $K^m_t$ disconnects $v_m$ from $\pright \CR_{a,b}$ and set $\xi_m = \zeta_m \wedge s_m$.  Finally, we let~$T_m$ be the first time $t$ that $K^m_t$ discovers a crossing $L_j$ with height outside of $[-M , M]$ such that $\dist(v_m ,  L_j ) \in [M^{-100} ,  2 M^{-1}]$ and all the previously discovered crossings have height in $[-M , M]$.

We begin with the following lemma.
\begin{lemma}
\label{lem:maximal_crossing_tail}
Consider the setting above and for each $1 \leq i \leq M-1$, let $L_{j(i)}$ denote the first crossing that separates $v_i$ from $\pright \CR_{a,b}$ (if it exists).  Let $E_M$ be the event that there is no crossing~$L_k$, $k < j(i)$ with height outside $[-M,M]$ such that $\dist(v_i,L_k) \geq M^{-100}$. There exist constants $c_0 > 0$ and $M_0 \geq 0$ such that $\p[ E_M^c ] \leq \exp(-c_0 M^2 / \log(M))$ for all $M \geq M_0$.
\end{lemma}
\begin{proof}
{\noindent \it Step 1.  It is unlikely that the first crossing with large height is of distance in $[M^{-100},2M^{-1}]$ from some $v_i$.} Suppose that the event $\{T_m < \xi_m\}$ occurs and let $L^*$ be the crossing observed at time $T_m$.  Observe that there exists a constant $p > 0$ depending only on $H$, $a$, and $b$ such that if we start a Brownian motion from $v_m$,  then with probability at least $p$ it does not exit $D^m_t$ in $\pdown \CR_{a,b} \cup \pup \CR_{a,b}$ for all $0 \leq t < \xi_m$.  For any simply connected domain $D$ and $z \in D$, we have by the Schwarz lemma and the Koebe-$1/4$ theorem that $\dist(z,\partial D) \leq \confrad(z ,  D) \leq 4 \dist(z,\partial D)$. It follows that $\confrad(v_m,D_{T_m}^m) \geq M^{-100}$ and $\confrad(v_m,\CR) \leq 2$. Since $T_m = \log(\confrad(v_m,\CR)) - \log(\confrad(v_m,D^m_{T_m}))$ it follows that $T_m \leq \log(2) + 100\log(M) \leq 200 \log(M)$ for all $M \geq 2$. We set $\CG_t^m = \sigma (K^m_{s \wedge s_m} : 0 \leq s \leq t)$ and $\Fh_t^m(v_m) = \E [ h(v_m) \,\giv \, \CG_t^m]$.  Note that $\Fh^m_0$ has boundary values given by $\lambda$ (resp.\  $-\lambda$) on $\pleft \CR$ (resp.\  $\pright \CR$) and zero on $\pdown \CR \cup \pup \CR$.  In particular the boundary values are bounded uniformly and since $\Fh^m_0$ is harmonic in $\CR$ it follows that there exists universal constant $c_1>0$ such that $\Fh^m_0(v_m) \in [-c_1,c_1]$.  Next,  we claim that $\Fh^m_{T_m}(v_m) - \Fh_0^m(v_m) \notin [-cM , cM]$ where $c \in (0,1)$ is a universal constant.  Indeed, there is a universal constant $p_0 > 0$ such that with probability at least~$p_0$, a Brownian motion started from $v_m$ exits $D_{T_m}^m$ in $L^*$. To see this, note that a Brownian motion has a positive probability (independent of $M,a,b$) of hitting $\{\re(z) = v_m - 3/M \}$ before exiting the $M^{-1}$-neighborhood of $[v_m-3/M,v_m]$.  Upon doing so, there is a universal constant $p_* > 0$ such that the Brownian motion travels clockwise around $B(v_m,2/M)$ before exiting $B(v_m,4/M)$ and before hitting $[v_m-i4/M,v_m-i2/M]$ in the counterclockwise direction. By symmetry the corresponding holds in the counterclockwise direction. At least one of these events amount to hitting $L^*$ before exiting $D_{T_m}^m$. Moreover, the Beurling estimate implies that the probability that the same Brownian motion exits $D^m_{T_m} \cap \CR_{a,b}$ in $\pdown \CR_{a,b} \cup \pup \CR_{a,b}$ is $O(M^{-1/2})$ where the implicit constant depends only on $H$, $a$ and $b$.  Therefore the crossing $L^*$ gives a contribution to $|\Fh_{T_m}^m(v_m)|$ of at least $p_0(M - \lambda)$ while the crossings discovered before $L^*$ give a contribution of at most~$O(M^{1/2})$.  Similarly,  the parts of $\partial D_{T_m}^m$ lying outside of $\CR_{a,b}$ give a contribution of order $O(M^{-1/2})$. Combining it follows that $|\Fh_{T_m}^m(v_m) - \Fh_0^m(v_m)| \geq p_0M / 2$ for all $M$ sufficiently large and this proves the claim.  Note that $M^{-1} \geq M^{-100}$ and that $\Fh_t^m(v_m) - \Fh_0^m(v_m)$ evolves as a standard Brownian motion up to time $s_m$~\cite[Proposition~6.5]{ms2016imag1}.  Hence, for a standard Brownian motion $B$ the probability that $T_m< \xi_m$ is at most $\p_0 [ \sup_{0 \leq t \leq 200 \log(M)} |B_t| > cM] $ and the latter is at most $\exp(-c_1 M^2 / \log(M))$ for some constant $c_1 > 0$.  Summing over $1 \leq m \leq M-1$,  we obtain that the probability that there exists $1 \leq m \leq M-1$ such that we can find a crossing with distance in $[M^{-100} ,  2M^{-1}]$ of $v_m$ which does not disconnect~$v_m$ from $\pright \CR_{a,b}$ and with height outside of $[-M ,  M]$ and such that all of the previously discovered crossings have height in $[-M ,  M]$ is at most $M\exp(-c_1 M^2 / \log(M))  \leq \exp(-c_2 M^2/\log(M))$. 

{\noindent \it Step 2.  It is unlikely that there exists a crossing with large height and with distance at least $M^{-100}$ of some $v_i$.} Suppose that there exist $1 \leq i \leq M-1$ such that there is a crossing $L'$ with height outside of $[-M ,  M]$, $\dist(v_i,L') \in [M^{-100},2M^{-1}]$ and which does not disconnect $v_i$ from $\pright \CR_{a,b}$, let $1 \leq k \leq M-1$ be the smallest such $i$ and let $L_*^k$ be the first crossing satisfying the above with~$v_k$.  Let $L$ be the leftmost crossing which is discovered by the exploration in $\CR_{a,b}$ with height outside of $[-M ,  M]$.  We want to show that $L = L_*^k$.  If this is not the case,  then $L$ has to be to the left of $L_*^k$ and $\dist(v_k ,  L) \geq M^{-100}$.  Let $v_m$ be the leftmost point of the $v_{\ell}$'s such that $L$ is to the left of $v_m$ and $\dist(v_m ,  L) \geq M^{-100}$.  Then $m \leq k$ and $\dist(v_m,L) \leq 2M^{-1}$ (otherwise $v_{m-1}$ would be to the right of $L$ and $\dist(v_{m-1},L) > M^{-100}$, contradicting the choice of $m$). Moreover, the height of $L$ lies outside of $[-M ,  M]$ and thus $m=k$ and all of the crossings discovered before $L$ have height in $[-M ,  M]$.  This shows that $L = L_*^k$.  Combining,  it follows that off an event with probability at most $\exp(-c_2 M^2 / \log(M))$ the following holds.  For each $1 \leq i \leq M-1$,  if we explore towards $v_i$,  then we cannot find a crossing with height outside of $[-M ,  M]$ which does not disconnect $v_i$ from $\pright \CR_{a,b}$ and whose distance from $v_i$ lies between $M^{-100}$ and $2 M^{-1}$.  
\end{proof}

Next we shall see that if a crossing gets very close to a point $v_m$, then the next crossing is very likely to disconnect $v_m$ from $\pright \CR_{a,b}$. We introduce the following notation. Let $n_1$ be the first $j$ such that~$L_j$ is within distance $M^{-100}$ of some $v_k$, $k \leq M/2$ which it does not disconnect from $\pright \CR_{a,b}$ and let~$m_1$ be the smallest $k$ such that this occurs. Furthermore, we let $N_1$ be the first $j \geq n_1+1$ such that~$L_j$ disconnects $v_{m_1}$ from $\pright \CR_{a,b}$. Next, we define inductively the sequences $(n_\ell)_{\ell \geq 1}$, $(m_\ell)_{\ell \geq 1}$ and $(N_\ell)_{\ell \geq 1}$ as follows. We let $n_{\ell+1}$ be the first $j \geq N_{\ell}$ such that~$L_j$ is within distance $M^{-100}$ of some $v_k$, $k \leq M/2$ which it does not disconnect from $\pright \CR_{a,b}$, $m_{\ell+1}$ be the smallest such $k$ and $N_{\ell+1}$ be the first $j \geq n_{\ell+1}+1$ such that $L_j$ disconnects $v_{m_{\ell+1}}$ from $\pright \CR_{a,b}$. Finally, we let $\CF_k$ be the $\sigma$-algebra generated by the level lines which make up $K_1,\ldots,K_k$. Then we have the following result.

\begin{lemma}
\label{lem:next_crossing_separates}
There exist a constant $c>0$ such that the following is true. Let $F_\ell$ be the event that $n_\ell < \infty$ and $\max_{1 \leq j \leq n_\ell} |u_j| \leq M^3$. Then
\[ \p[ N_\ell - n_\ell \geq M \giv \CF_{n_\ell}] \one_{F_\ell} \leq \exp(-c M \log M)\one_{F_\ell}.\]
\end{lemma}
\begin{proof}
{\noindent \it Step 1. Setup.} We will prove that there exist constants $a,c > 0$ so that
\[ \p[ N_\ell - n_\ell \geq 2 \giv \CF_{n_\ell}] \one_{F_\ell} \leq c M^{-a} \one_{F_\ell}.\]
The result will then follow by iteratively applying this bound. The strategy of the proof is to show that by generating level lines of the GFF starting from $L_{n_{\ell}}$ sufficiently close to $v_{m_{\ell}}$, we are very likely to find one level line which cuts off $v_{m_{\ell}}$ from $\pright \CR_{a,b}$ so that the next crossing cannot travel between $L_{n_{\ell}}$ and $v_{m_{\ell}}$.

We let $D_{n_{\ell}}$ be the connected component of $\CR \setminus \cup_{i=1}^{n_{\ell}} K_i$ containing $v_{m_{\ell}}$ and let $x_{\ell}$ (resp.\  $y_{\ell}$) be the point on $\pup \CR_{a,b}$ (resp.\  $\pdown \CR_{a,b}$) which is visited by $L_{n_{\ell}}$.  We then let $\psi_{\ell}$ be the conformal transformation mapping $D_{n_{\ell}}$ onto $\h$ such that $\psi_{\ell}(x_{\ell}) = -1$,  $\psi_{\ell}(y_{\ell}) = 1$ and $\psi_{\ell}(1) = \infty$.  Note that $\psi_{\ell}$ is $\CF_{n_{\ell}}$-measurable and that the field $h \circ \psi_{\ell}^{-1}$ on $\h$ can be expressed as $h^0 + \Fh_{\ell}$,  where $h^0$ is a zero boundary GFF on $\h$ and $\Fh_{\ell}$ is a function which is harmonic on $\h$ and has piecewise constant boundary conditions that change only a finite number of times.  Next we observe that $\dist(v_{m_{\ell}}, \pup \CR_{a,b} \cup \pdown \CR_{a,b}) = \tfrac{b-a}{2}$ and $\dist(v_{m_{\ell}} , \pright \CR_{a,b} ) \geq  \tfrac{1}{M}$.  Hence the Beurling estimate implies that the probability that a Brownian motion starting from $v_{m_{\ell}}$ exits $D_{n_{\ell}}$ either in $\cup_{i=1}^{n_{\ell}-1} K_j$ or in parts of $\partial D_{n_{\ell}}$ lying outside of $\CR_{a,b}$ is $O(M^{-50})$.   Those parts give a contribution to $|\Fh_{\ell} \circ \psi_{\ell}(v_{m_{\ell}})|$ at most $O(M^{-50}(M^{3} + \lambda)) = O(M^{-47})$.  Similarly, the probability that the Brownian motion exits $D_{n_\ell}$ in $\pright \CR_{a,b}$ is $O(M^{-99/2})$ so that its contribution to $|\Fh_{\ell} \circ \psi_{\ell}(v_{m_{\ell}})|$ is $O(M^{-99/2})$. It follows that $\Fh_{\ell}(\psi_{\ell}(v_{m_{\ell}})) = u_{n_{\ell}} + O(1)$.  Note that $h$ has boundary values on the right side of $L_{n_{\ell}}$ given by either $\lambda + u_{n_{\ell}}$ or $-\lambda + u_{n_{\ell}}$.  In the former case,  we have that $u_{n_{\ell}+1} = u_{n_{\ell}} + u$ while in the latter case we have that $u_{n_{\ell}+1} = u_{n_{\ell}} - u$.  Without loss of generality,  we can assume that the former case holds.  Then $h \circ \psi_{\ell}^{-1}$ has boundary values given by $\lambda + u_{n_{\ell}}$ on $[-1,1]$.  Set $\psi_{\ell}(v_{m_{\ell}}) = a_{\ell} + ib_{\ell}$ and $d_{\ell} = \dist(\psi_{\ell}(v_{m_{\ell}}),  \R \setminus [-1,1])$.  The above imply that the probability that a Brownian motion starting from $\psi_{\ell}(v_{m_{\ell}})$ exits $\h$ in $\partial \h \setminus [-1,1]$ is at most $O(M^{-50})$.  This implies that $\arg(\psi_{\ell}(v_{m_{\ell}}) + 1) = O(M^{-50})$ and $\arg(\psi_{\ell}(v_{m_{\ell}}) -1) = \pi - O(M^{-50})$ which in turn implies that $b_{\ell}/d_{\ell} = O(M^{-50})$.  

{\noindent \it Step 2. Half-annuli and separation of $\psi_\ell(v_{m_\ell})$.} Let $B_{k,\ell} = \h \cap B(a_\ell,M b_\ell e^{-k})$ and consider the half-annuli $A_{k,\ell} = \h \cap ( B(a_{\ell} ,  Mb_{\ell} e^{-k}) \setminus \closure{B(a_{\ell} ,  Mb_{\ell} e^{-k}/2)})$ for $0 \leq k \leq N$ where $N = \lfloor \log(M) \rfloor - 1$.  Note that $A_{k,\ell}$ disconnects $\psi_{\ell}(v_{m_{\ell}})$ from $\infty$ for all $0 \leq k \leq N$.  For each $0 \leq k \leq N$ we let $
\eta_k$ be the level line of $h \circ \psi_{\ell}^{-1}$ of height $u_{n_{\ell}+1} = u_{n_{\ell}} + u$ starting from the midpoint $c_k$ of the interval of $\partial A_k \cap \R$ to the left of $a_{\ell}$.  We draw each $\eta_k$ up until the first time $\tau_k$ that it exits $A_{k,\ell}$ and let $P_k$ be the event that $\eta_k$ disconnects $\psi_{\ell}(v_{m_{\ell}})$ from $\infty$ before time $\tau_k$.  Let also $\CG_{k,n_{\ell}}$ be the $\sigma$-algebra generated by~$\CF_{n_{\ell}}$ and all of $\eta_0,\ldots,\eta_k$ and let~$D_{k,\ell}$ be the unbounded connected component of $\h \setminus \cup_{i=0}^k \eta_i$. 

In order to see that the event $\cup_{k=0}^N P_k$ is likely to occur, we shall compare with a GFF with nicer boundary data. Let $\wh{h}_\ell$ be a GFF on $\h$ with boundary data $\lambda+u_{n_\ell}$ and for each $k \leq N$, let $\wh{\eta}_k$ be the height $u_{n_{\ell}+1}$ level line from $c_k$ stopped upon exiting $A_{k,\ell}$.  Furthermore, denote by $\wh{\mu}_{k,\ell}$ the law of $\wh{h}_\ell$ restricted to~$B_{k,\ell}$ and let $\wh{P}_k$ be the event that $\wh{\eta}_k$ disconnects $\psi_\ell(v_{m_\ell})$ from $\infty$ before exiting~$A_{k,\ell}$. Note that $\wh{P}_k$ is determined by the restriction of $\wh{h}_\ell$ to $B_{k,\ell}$. Since $\wh{\eta}_k \sim \SLE_4(-2+u/\lambda,-u/\lambda)$, that is, both force points have weight in $(-2,0)$ it follows by scale invariance and ~\cite[Lemma~2.5]{mw2017intersections} that there is some $p>0$ such that $\wh{\mu}_{k,\ell}(\wh{P}_k) \geq p$ for all $k$.

Thus, if we for each $k$ let $\mu_{k,\ell}$ denote the conditional law of $h \circ \psi_\ell^{-1}$ restricted to $B_{k,\ell}$ given $\CG_{k,n_\ell}$ and prove that there exist constants $C > 0$ and $q>0$ such that conditional on $\cap_{j=1}^k P_j^c$,
\begin{align}\label{eq:dominate_measure}
	\wh{\mu}_{k+1,\ell}(\wh{P}_{k+1})^q \leq C \mu_{k+1,\ell}(P_{k+1}),
\end{align}
then it follows that $\sum_{k=1}^N \one_{P_k}$ stochastically dominates a geometric random variable with success probability $p_0 = p^q / C > 0$ and hence the result follows since level lines of the same height cannot cross each other.

{\noindent \it Step 3. Absolute continuity and the proof of~\eqref{eq:dominate_measure}.}
In this step we shall prove that~\eqref{eq:dominate_measure} holds. We will begin by showing that on $\cap_{i=0}^k P_i^c$, the restrictions of $h \circ \psi_\ell^{-1}$ and $\wh{h}_\ell$ to $\h \cap B(a_\ell, M b_\ell e^{-k-1})$ can be coupled together so that their zero-boundary and harmonic parts are equal and independent, respectively. 

Suppose that we are working on the event $\cap_{i=0}^k P_i^c$ for $0 \leq k \leq N$ fixed.  Then the Markov property implies that the restriction of $h \circ \psi_{\ell}^{-1}$ to $D_{k,\ell}$ can be expressed as $h^{0,k} + \Fh_{k,\ell}$ conditional on $\CG_{k,n_{\ell}}$,  where $h^{0,k}$ is a zero-boundary GFF on $D_{k,\ell}$ and $\Fh_{k,\ell}$ is harmonic in $D_{k,\ell}$, $\CG_{k,n_{\ell}}$-measurable, and independent of $h^{0,k}$.  Set $B_{k,\ell}^* = \h \cap B(a_{\ell} ,  M b_\ell e^{-k}(\tfrac{1}{4} + \tfrac{1}{2e}))$ and note that $B_{k+1,\ell} \subseteq B_{k,\ell}^*$ and $B_{k,\ell}^* \cap A_{k,\ell} = \emptyset$. Similarly,  $h^{0,k}$ can be expressed as $\wt{h}^{0,k} + \wt{\Fh}_{k,\ell}$ where $\wt{h}^{0,k}$ is a zero boundary GFF on $B_{k+1,\ell}^*$ and $\wt{\Fh}_{k,\ell}$ is harmonic in $B_{k+1,\ell}^*$.   In the same way, $\wh{h}_\ell$ can be decomposed as $\wh{h}^0 + \lambda + u_{n_\ell}$ where $\wh{h}^0$ is a zero-boundary GFF on $\h$.  Also $\wh{h}^0$ can be expressed as $\wh{h}^{0,k} + \wh{\Fh}_{k,\ell}$ where $\wh{h}^{0,k}$ is a zero boundary GFF on $B_{k+1,\ell}$ and $\wh{\Fh}_{k,\ell}$ is harmonic on $B_{k+1,\ell}$.  Hence we can couple $h \circ \psi_{\ell}^{-1}$ and $\wh{h}_{\ell}$ on the same probability space such that $\wt{h}^{0,k} = \wh{h}^{0,k}$ and that $\wh{\Fh}_{k,\ell}$,  $\wt{\Fh}_{k,\ell}$ are independent.  Let $g \in C_0^{\infty}(B(0,1))$ be fixed and such that $|g(z)| \leq 1$, $|\nabla g(z)| \leq C$ for all $z$ (where $C>0$ is universal) and $g|_{B(0,e^{-1}(1/4 + 1/2e)^{-1})}  \equiv 1$.  Then we set $g_{k,\ell}(w) = g(M^{-1}b_{\ell}^{-1} e^k (\tfrac{1}{4} + \tfrac{1}{2e})^{-1} (w-a_{\ell}))$ and so $g_{k,\ell} \in C_0^{\infty}(B_{k,\ell}^*)$ and $g_{k,\ell}|_{B_{k+1,\ell}} \equiv 1$.  It follows from the proof of~\cite[Proposition~3.4]{ms2016imag1} that conditional on $\CG_{k,n_{\ell}}$ and on the event $\cap_{i=0}^k P_i^c$,  the law of the restriction of $h \circ \psi_{\ell}^{-1}$ to $B_{k+1,\ell}$ weighted by $\CZ^{-1}\exp((\wt{h}^{0,k},f_{k,\ell})_{\nabla})$ with $\CZ = \exp(\|f_{k,\ell}\|_{\nabla}^2 / 2)$ is equal to the law of $\wh{h}_{\ell}$ restricted to $B_{k+1,\ell}$,  where $f_{k,\ell} = g_{k,\ell} ((\wh{\Fh}_{k,\ell} - \wt{\Fh}_{k,\ell}) - (\Fh_{k,\ell} - \lambda - u_{n_\ell}))$.  

With the Radon-Nikodym derivative between laws of $h \circ \psi_\ell^{-1}$ and $\wh{h}_\ell$ restricted to $B_{k+1,\ell}$ at our hands, we shall prove~\eqref{eq:dominate_measure}. We fix $p>1$ (to be taken very close to $1$) and let $q>1$ be such that $\tfrac{1}{p}+\tfrac{1}{q} = 1$.  Then H\"{o}lder's inequality implies that
\begin{align*}
\wh{\mu}_{k+1,\ell}(\wh{P}_{k+1}) \leq \E_{\mu_{k+1,\ell}}[ \exp((p^2 - p)\|f_{k,\ell}\|_{\nabla}^2 / 2)]^{\frac{1}{p}} \mu_{k+1,\ell}(P_{k+1})^{\frac{1}{q}}.
\end{align*}
We claim that $\E_{\mu_{k+1,\ell}}[ \exp((p^2-p)\|f_{k,\ell}\|_{\nabla}^2 / 2)] \lesssim 1$ if $p$ is chosen sufficiently close to $1$ where the implicit constant depends only on $p$.  Indeed,  first we note that $g_{k,\ell}(\wh{\Fh}_{k,\ell} - \wt{\Fh}_{k,\ell})$ and $g_{k,\ell}(\Fh_{k,\ell} - \lambda - u_{n_\ell})$ are independent and that $\|f_{k,\ell}\|_{\nabla}^2 \leq  2\|g_{k,\ell}(\wh{\Fh}_{k,\ell} - \wt{\Fh}_{k,\ell})\|_{\nabla}^2 + 2\|g_{k,\ell}(\Fh_{k,\ell} - \lambda - u_{n_\ell})\|_{\nabla}^2$.  It follows from the argument used to prove~\cite[Lemma~4.1]{mq2020geodesics} that $\E_{\mu_{k+1,\ell}}[\exp((p^2 - p)\|g_{k,\ell}(\wh{\Fh}_{k,\ell} - \wt{\Fh}_{k,\ell})\|_{\nabla}^2)] \lesssim 1$ if $p$ is chosen sufficiently close to $1$,  where the implicit constant depends only on $p$ .  We now focus on the term $\|g_{k,\ell}(\Fh_{k,\ell} - \lambda - u_{n_\ell})\|_{\nabla}^2$.  Observe that 
\begin{align*}
	\|g_{k,\ell}(\Fh_{k,\ell} - \lambda - u_{n_\ell})\|_\nabla^2 \lesssim \| \nabla\Fh_{k,\ell}\|_{2L^2(B_{k,\ell}^*)}^2 + \sup_{w \in B_{k,\ell}^*}|\Fh_{k,\ell}(w) - \lambda - u_{n_\ell}|^2,
\end{align*}
since $\|\nabla g_{k,\ell}\|^2 \lesssim 1$.  We begin by bounding the supremum.  Note that the boundary values of $\Fh_{k,\ell}$ on $\cup_{i=0}^k \eta_i$ are $u_{n_{\ell}} + O(1)$.  Fix $w \in B(a_{\ell},4Mb_{\ell}e^{-k}/9)$. Observe that $\dist(w,\partial \h ) \lesssim Mb_{\ell}$ and so $\dist(w,\partial \h) \leq \dist(w, \partial \h \setminus [-1,1]) O(M^{-49})$.  The Beurling estimate implies that $\harm{w}{\partial \h \setminus[-1,1]}{D_{k,\ell}} = O(M^{-49 / 2})$ and so this gives a contribution to $\Fh_{k,\ell}(w)$ of order $O(M^{-43 / 2})$ from the images under $\psi_{\ell}$ of the parts of $\partial D_{n_{\ell}}$ in $\CR_{a,b}$ except for the right side of $L_{n_{\ell}}$.  Overall,  it follows that $\Fh_{k,\ell}(w) - \lambda - u_{n_\ell} = O(1)$ for $w \in \h \cap B(a_{\ell} ,  4Mb_{\ell}e^{-k} / 9)$,  where the implicit constant is universal.  As for the other term,  we note that $\Fh_{k,\ell} - \lambda - u_{n_\ell}$ can be extended by the reflection principle to a harmonic function in $B(a_{\ell} ,  4Mb_{\ell} e^{-k} / 9)$ since $\Fh_{k,\ell} - \lambda - u_{n_\ell} \equiv 0$ in $B(a_{\ell} ,  4Mb_{\ell} e^{-k} / 9) \cap \partial \h$.  Moreover,  standard estimates for harmonic functions imply that $|\nabla \Fh_{k,\ell}(w)| \lesssim M^{-1}b_{\ell}^{-1} e^k$ for all $w \in B_{k,\ell}^*$,  where the implicit constant is universal.  Therefore we have that $\|\nabla \Fh_{k,\ell} \|_{L^2(B_{k,\ell}^*)}^2 \lesssim 1$.  Therefore,  there exists a universal constant $C > 1$ such that $\wh{\mu}_{k+1,\ell}(\wh{P}_{k+1}) \leq C \mu_{k+1,\ell}(P_{k+1})$ conditional on $\CG_{k,n_{\ell}}$ and on the event $\cap_{i=0}^k P_i^c$.  Thus we have proved that~\eqref{eq:dominate_measure} holds and thus by the discussion in Step 2, the proof is complete.
\end{proof}

In the following lemma we will make sure that the crossings have bounded height, even when they get within distance $M^{-100}$ of the points $(v_m)$.

\begin{lemma}\label{lem:large_heights_with_small_prob}
There exist constants $c>0$ and $M_0 \in \N$ such that the following holds. Let $G_M$ be the event that there exists a crossing of height outside of $[-M^3,M^3]$ discovered before $v_{\lfloor M/2 \rfloor}$ is disconnected from $\pright \CR_{a,b}$ (or before the exploration terminates if $v_{\lfloor M/2 \rfloor}$ is never disconnected). Then $\p[ G_M ] \leq \exp(-cM\log(M))$ for all $M \geq M_0$.
\end{lemma}
\begin{proof}
\noindent{\it Step 1.  Setup.}  Let $E_M$ be the event of Lemma~\ref{lem:maximal_crossing_tail} and suppose that we are working on $E_M \cap G_M$.  Let $J$ be the first $i$ such that $n_i < \infty$,  $N_i-n_i \geq M$ and the exploration does not stop after $M$ steps of the discovery of $L_{n_i}$.  In what follows, we are first going to show that $G_M \cap E_M \subseteq \{J<\infty\} \cap E_M$ (Step 2) and then combine this with Lemma~\ref{lem:maximal_crossing_tail} to prove the desired upper bound on $\p[G_M]$ (Step 3).

\noindent{\it Step 2.  $G_M \cap E_M \subseteq \{J<\infty\} \cap E_M$.}  We first consider the case that there is some $L_{j^*}$ separating $v_1, \dots, v_{\lfloor M/2 \rfloor - 1}$ but not $v_{\lfloor M/2 \rfloor}$ from $\pright \CR_{a,b}$.  Let $L_j$ be a crossing of height outside of $[-M^3,M^3]$ which is discovered before $v_{\lfloor M/2 \rfloor}$ is disconnected from $\pright \CR_{a,b}$ (or before the exploration terminates if $v_{\lfloor M/2 \rfloor}$ is never disconnected). First, we we will show that $n_1 \neq \infty$.  Indeed,  suppose that $n_1 = \infty$.  Let $v_i$ be the leftmost point that is not disconnected from $\pright \CR_{a,b}$ by $L_j$.  Then $1 \leq i \leq M/2$ and since $n_1 = \infty$,  we must have that $\dist(v_i ,  L_j) > M^{-100}$.  But then,  since $E_M$ occurs, $L_j$ must have height in $[-M,M]$ which is a contradiction.  Thus $n_1 < \infty$ and $\ell = \max\{i\in \N : n_i < \infty\}$ is hence well-defined.  Note that by the definition of the $n_i$'s,  we must have that $1 \leq \ell \leq M/2$.  

Suppose that $m_{\ell} < \lfloor M/2 \rfloor$.  Then we must have that $N_{\ell} < \infty$ (since $v_{m_\ell}$ is disconnected from $\pright \CR_{a,b}$ by $L_{j^*}$),  $n_{\ell+1} = \infty$ and that $v_{\lfloor M/2 \rfloor}$ is not disconnected by $L_{N_{\ell}}$.  Since $E_M$ occurs,  all of the crossings $L_1,\ldots,L_{n_1-1}$ have heights in $[-M,M]$ and so $|u_i| \leq M$ for each $1 \leq i \leq n_1 - 1$ and $|u_{n_1}| \leq |u_{n_1-1}| + u\leq M+u$,  since every time that we discover a new crossing the height either decreases or increases by $u$.  It follows that $|u_i| \leq M+u +u\sum_{k=1}^{\ell}(N_k - n_k)$ for each $1 \leq i \leq N_{\ell}$.  Suppose that $N_k - n_k \leq M$ for each $1 \leq k \leq \ell$.  Then $|u_i| \leq M+u+uM^2 \leq M^3$ for each $1\leq i \leq N_{\ell}$.  Since $n_{\ell+1} = \infty$ and $E_M$ occurs,  it follows that every crossing discovered after $L_{N_{\ell}}$ and up until $v_{\lfloor M/2 \rfloor}$ is disconnected from $\pright \CR_{a,b}$ (or up until the end of the exploration if no such time exists) has height in $[-M^3, M^3]$,  but that is a contradiction.  Therefore there exists $1 \leq k \leq \ell$ such that $N_k - n_k > M$ and $N_k < \infty$.  Therefore in this case we have that $J < \infty$.

Suppose that $m_{\ell} = \lfloor M/2 \rfloor$.  If $N_k - n_k \leq M$ for each $1 \leq k \leq \ell$,  then $|u_i| \leq M+u+u\sum_{k=1}^{\ell}(N_k - n_k) \leq M+u+uM^2 \leq M^3$ for each $1 \leq i \leq N_{\ell}$.  This is a contradiction because then $v_{\lfloor M/2 \rfloor}$ will be disconnected from $\pright \CR_{a,b}$ by the time $L_{N_{\ell}}$ is discovered hence the exploration will not discover a crossing with height outside $[-M^3,M^3]$ before disconnecting $v_{\lfloor M/2 \rfloor}$.  Hence there exists $1 \leq k \leq \ell$ such that $N_k - n_k > M$. Suppose that $N_k - n_k \leq M$ for each $1 \leq k \leq \ell-1$.  Then either we have that $N_{\ell}-n_{\ell} > M$ and $N_{\ell} < \infty$ or $N_{\ell} = \infty$.  Suppose also that $N_{\ell} = \infty$ and the exploration stops after at most $M$ steps after $L_{n_{\ell}}$ is discovered.  Again we have that $|u_i| \leq M+u+uM^2$ for each $1 \leq i \leq N_{\ell-1}$.  If $N_{\ell-1} = n_{\ell}$,  then it follows that every crossing discovered after $L_{n_{\ell}}$ has height whose absolute value is at most $M+u+uM^2+uM \leq M^3$,  but that is a contradiction.  If $N_{\ell-1} < n_{\ell}$,  then every crossing discovered after $L_{N_{\ell-1}}$ and before $L_{n_{\ell}}$ has height in $[-M,M]$ and so $L_{n_{\ell}}$ has height in $[-M-u,M+u]$.  Consequently,  every crossing discovered after $L_{n_{\ell}}$ has height in $[-M-u(M+1),M+u(M+1)]$,  and so again we have a contradiction.  Therefore in this case we have that $J < \infty$.

We next consider the case that there does not exist a crossing $L_j^*$ as described at the beginning of the step. We let $v_n$ be the leftmost of the $v_i$'s for $1 \leq i \leq \lfloor M/2 \rfloor$ which is not disconnected from $\pright \CR_{a,b}$ in the step just before the step at which $v_{\lfloor M/2 \rfloor}$ is disconnected from $\pright \CR_{a,b}$ if such disconnection occurs or the leftmost point of the $v_i$'s for $1 \leq i \leq \lfloor M/2 \rfloor$ which is not disconnected from $\pright \CR_{a,b}$ up until the end of the exploration if no such disconnection occurs.  It follows that $m_{\ell} \leq n$ since $m_{\ell} < \lfloor M/2 \rfloor$.  Therefore a similar argument to the one given in the previous paragraphs implies that $J<\infty$ if $G_M \cap E_M$ occurs.

Combining everything we obtain that in every case $G_M \cap E_M \subseteq \{J<\infty\} \cap E_M$.

\noindent{\it Step 3.  Upper bound on $\p[G_M]$.} By Lemma~\ref{lem:maximal_crossing_tail}, there exists $c_1 > 0$ such that $\p[E_M^c] \leq \exp(-c_1M^2/ \log(M))$ and thus
\begin{align*}
\p[G_M] \leq \exp(-c_1M^2 / \log(M)) + \p[\{J<\infty\}\cap E_M]
\end{align*}
for all $M$ sufficiently large.  Let $B_1$ be the event that $n_1 < \infty$ and $|u_i| \leq M$ for each $1\leq i \leq n_1$ such that $\dist(v_{m_1} ,  L_i) \geq M^{-100}$.  Let also $B_{\ell}$ be the event that $n_{\ell} <\infty$,  $N_i-n_i \leq M$ for each $1\leq i \leq \ell - 1$ and $|u_i| \leq M$ for each $1\leq i \leq n_{\ell}$ such that $\dist(v_{m_{\ell}},L_i)\geq M^{-100}$ for $2\leq \ell \leq M-1$.  Writing $N_{\ell}-n_{\ell} \geq M$ to mean that either $N_{\ell} < \infty$ and $N_{\ell}-n_{\ell} \geq M$ or $N_{\ell} = \infty$ and the exploration has not ended after $M$ steps of the discovery of $L_{n_{\ell}}$, we have that
\begin{align*}
\p[\{J<\infty\}\cap E_M]&\leq \sum_{m=1}^{M-1}\p[\{J=m\}\cap E_M]\\
&\leq\sum_{\ell = 1}^{M-1}\E[ \p[ N_{\ell} - n_{\ell}\geq M\,\giv \, \CF_{n_{\ell}}]\one_{B_{\ell}}].
\end{align*}
For each $1\leq \ell\leq M-1$ we let $C_{\ell}$ be the event that $n_{\ell} < \infty$ and $|u_i| \leq M^3$ for each $1 \leq i \leq n_{\ell}$.  Observe that the $B_{\ell}$'s and $C_{\ell}$'s are all $\CF_{n_{\ell}}$-measurable and $B_{\ell} \subseteq C_{\ell}$ for each $1 \leq \ell \leq M-1$.  Therefore it follows that
\begin{align*}
\p[G_M] \leq \exp(-c_1M^2 / \log(M)) + \sum_{\ell = 1}^{M-1} \E[ \p[N_{\ell} - n_{\ell} \geq M \, \giv \CF_{n_{\ell}}] \one_{C_{\ell}}].
\end{align*}
But Lemma~\ref{lem:next_crossing_separates} implies that $\p[N_{\ell}-n_{\ell}\geq M \, \giv \,  \CF_{n_{\ell}}] \one_{C_{\ell}} \leq \exp(-c_2 M\log(M))\one_{C_{\ell}}$ for all $M$ sufficiently large,  where $c_2 > 0$ is a fixed constant.  Combining we obtain that $\p[G_M] \leq \exp(-c_1 M^2 / \log(M)) + M\exp(-c_2 M \log(M))$ which concludes the proof.
\end{proof}

The final lemma we need in order to prove Lemma~\ref{lem:rectangle_maximum_crossing_height} is the following, which states that we cannot have an infinite number of crossings in the exploration without disconnecting the points $v_i$ from $\pright \CR_{a,b}$ in $\CR_{a,b}$.

\begin{lemma}
\label{lem:finitely_many_crossings}
Fix $M \in \N$ and for each $1 \leq i \leq M-1$ let $C_{i,M}$ be the event that either $v_i$ is disconnected from $\pright \CR_{a,b}$ in $\CR_{a,b}$ by some crossing of the exploration or the exploration hits $\pright \CR_{a,b}$ after finitely many crossings.  Then we have that $\p[ C_{i,M}^c] = 0$.
\end{lemma}
\begin{proof}
\noindent{\it Step 1.  Setup.}
Fix $1 \leq i \leq M-1$.  For each $j \in \N$,  we let $K_j$ be as in Section~\ref{subsec:exploration_definition} and let $D_j$ be the connected component of $\CR \setminus K_j$ whose boundary contains $\pright \CR_{a,b}$.  We set $K = \closure{\cup_{j\geq 1} K_j}$ and let $D$ be the connected component of $\CR \setminus K$ whose boundary contains $\pright \CR_{a,b}$.  Note that $K_j$ and $K$ are local sets for $h$ which are also determined by $h$.  Let $\Fh$ and $\Fh_j$ be the associated harmonic functions. Let $c_0,M_0>0$ be the constants of Lemma~\ref{lem:maximal_crossing_tail} and $a>0$ be the constant from Lemma~\ref{lem:next_crossing_separates}.  For $m \in \N_0$ we let $J_m$ be the first $j \in \N$ such that $L_j$ does not disconnect $v_i$ from $\pright \CR_{a,b}$ and $\dist(v_i ,  L_j) \leq m^{-100}$.  Let also $N$ be the minimum of the number of crossings discovered by the exploration and the first $j$ such that $L_j$ disconnects $v_{i}$ from $\pright \CR_{a,b}$. 

\noindent{\it Step 2.  $v_{i}$ is disconnected from $\pright \CR_{a,b}$ whenever the exploration discovers infinitely many crossings and gets arbitrarily close to $v_{i}$.}
 We claim that $\p[ N=\infty ,  v_i \in K] = 0$.  Indeed,  first we note that $\p[ N=\infty ,  v_i \in K] = \lim_{m \to \infty} \p[ N=\infty ,  J_m < \infty]$.  Also Lemma~\ref{lem:next_crossing_separates} implies that
\begin{align*}
\p[ N=\infty  \giv \CF_{J_m}] \one_{J_m < \infty} \one_{\max_{1 \leq \ell \leq J_m} |u_{\ell}| \leq m+u} \leq O(m^{-a}).
\end{align*}
Let $E_m$ be as in Lemma~\ref{lem:maximal_crossing_tail} and note that if $E_{m}$ occurs,  then we have that $\max_{1\leq \ell \leq J_m}|u_{\ell}| \leq m+u$ and this implies that $\p[ E_m, N=\infty , J_m<\infty] \leq O(m^{-a})$.  Combining with Lemma~\ref{lem:maximal_crossing_tail} we obtain that $\p[ N=\infty ,  J_m < \infty] \leq O(m^{-a}) + \exp(-c_0 m^2 / \log(m))$. Thus the claim follows by taking $m \to \infty$,  and so it suffices to show that $\p[ N=\infty ,  v_i \notin K] = 0$.

\noindent{\it Step 3.  $v_i$ is disconnected from $\pright \CR_{a,b}$ whenever the exploration discovers infinitely many crossings and does not accumulate at $v_{i}$.}

\noindent{\it Step 3a.  The accumulation points of the exploration do not disconnect $v_i$ from $\partial \CR_{a,b}$.}
Fix $m \in \N_0$.  We claim that $\p[ N=\infty ,  J_m = \infty,  E_m] = 0$.  Indeed,  suppose that we are working on the event $E_m \cap \{N=\infty\} \cap \{J_m = \infty\}$.  Suppose first that $v_{i} \in D$ and let $U$ be the connected component of $\CR_{a,b} \setminus K$ containing $v_i$.  For each $j \in \N$,  we let $U_j$ be the connected component of $\CR_{a,b} \setminus K_j$ containing $v_{i}$.  We assume that $\partial U \nsubseteq K$ and fix $z_0 \in \partial U \setminus K$.  We let $\phi$ (resp.\  $\phi_j$) be the unique conformal transformation mapping $U$ (resp.\  $U_j$) onto $\D$ such that $\phi(v_{i}) = 0$ (resp.\  $\phi_j(v_{i}) = 0$) and $\phi(z_0) = 1$ (resp.\  $\phi_j(z_0) = 1$).  Let $x_j$ (resp.\  $y_j$) be the rightmost intersection point of $K_j$ with $\CR_{a,b}^{\Down}$ (resp.\  $ \CR_{a,b}^{\Up}$).  Note that for each $j \in \N$, the harmonic measure of of the clockwise (resp.\  counterclockwise) arc of $\partial \CR_{a,b}$ from $z_0$ to $x_j$ (resp.\  $y_j$) in $U_j$ as seen from $v_{i}$ is at least the harmonic measure of the clockwise (resp.\  counterclockwise) arc of $\partial \CR_{a,b}$ from $z_0$ to $x_{j+1}$ (resp.\  $y_{j+1}$) in $U_{j+1}$ as seen from $v_{i}$.  This implies that $I_j \subseteq I_{j+1}$ where $I_j = \phi_j(K_j^\Right)$ and so $\length(I\setminus I_j) \to 0$ as $j \to \infty$ where $I = \cup_{j\geq 1}I_j$.  We claim that $\Fh_j \to \Fh$ as $j \to \infty$ locally uniformly in $D$.  Indeed,  let $(q_m)_{m \geq 1}$ be an enumeration of the points in $\CR$ with rational coordinates.  Then the martingale convergence theorem implies that $\Fh_j(q_{\ell}) \to \Fh(q_{\ell})$ as $j \to \infty$ a.s.\  on the event that $q_{\ell} \notin K$ for each $\ell \in \N$.  The claim then follows since the $\Fh_j$'s and $\Fh$ are harmonic functions and their boundary values lie in $[-m,m]$ (since $E_{m}$ occurs) hence are equicontinuous on compact sets.

Next we claim that $\phi_j^{-1} \to \phi^{-1}$ as $j \to \infty$ locally uniformly in $\closure{\D} \setminus \closure{I}$.  Indeed, let $x_0$ be the lower left corner of $\CR_{a,b}$ and let $x_1$ (resp.\  $x_2$) be the lower (resp.\  upper) right corner of $\CR_{a,b}$.  Let $F$ be the unique conformal transformation mapping $\h$ onto $\CR_{a,b}$ such that $F(\infty) = x_0,  F(-1) = x_1,  F(1) = x_2$ and set $\psi_j = \phi_j^{-1}$.  Note that $\psi_j = F \circ \wt{\psi}_j$ where $\wt{\psi}_j$ is the unique conformal transformation mapping $\D$ onto $\h \setminus F^{-1}(K_j)$ such that $\wt{\psi}_j(0) = F^{-1}(v_{i})$ and $\wt{\psi}_j(1) = F^{-1}(z_0)$.  Note that there exists $\epsilon > 0$ such that $B(z_0 ,  \epsilon) \cap \CR_{a,b} \subseteq U$ and $\harm{v_{i}}{\partial U_j \cap B(z_0,\epsilon)}{U_j} \geq p = \harm{v_{i}}{\partial U \cap B(z_0,\epsilon)}{U} >0$ and so $\length(\phi_j(\partial U_j \cap B(z_0,\epsilon))) \geq 2\pi p$ for each $j \in \N$.  Since $\phi_j(\partial U_j \cap B(z_0,\epsilon))$ is an arc of $\partial \D$ such that $I_j \cap \phi_j(\partial U_j \cap B(z_0,\epsilon)) = \emptyset$,  it follows that $1 \notin \closure{I}$.  Thus $\wt{\psi}_j(\partial \D \setminus I)$ is a non-trivial interval of $\R \cap \partial \h \setminus F^{-1}(K_j)$ such that $\dist(\wt{\psi}_j(\partial \D \setminus I) ,  F^{-1}(K_j)) > 0$.  Hence the reflection principle implies that $\wt{\psi}_j$ extends to a conformal mapping defined on $\D_I^*$,  where $\D_I^*$ is the reflection of $\partial \D$ with respect to $\partial \D \setminus \closure{I}$.  Fix a set $X$ which is compact in $\closure{\D} \setminus \closure{I}$.  Note that there exists $R>0$ such that $\h \setminus F^{-1}(K_j) \subseteq B(0,R)$ for each $j$ and so standard estimates for conformal maps imply that $\sup_{w \in X}|\wt{\psi}_j'(w)| \lesssim 1$ and $\sup_{w \in \D}|\wt{\psi}_j(w)| \leq R$ where the implicit constant does not depend on $j$.  Thus by the Arzela-Ascoli theorem we obtain that if we fix a subsequence $(\wt{\psi}_{j'})_{j \geq 1}$,  then we can find a further subsequence $(\wt{\psi}_{j''})_{j \geq 1}$ and a conformal map $\wt{\psi}$ such that $\wt{\psi}_{j''} \to \wt{\psi}$ as $j \to \infty$ locally uniformly in $\closure{\D} \setminus \closure{I}$.  Note that $|\wt{\psi}_j'(0)| \geq m^{-100}|(F^{-1})'(v_i)|/4$ for each $j$ which implies that $\wt{\psi}'(0) \neq 0$ and so $\wt{\psi}$ is a univalent function.  Note that $F$ is a homeomorphism away from $\infty$ and so $\psi_{j''} \to \psi$ as $j \to \infty$ locally uniformly in $\closure{\D} \setminus \closure{I}$ where $\psi = F \circ \wt{\psi}$.  Note that $\psi(\D) \subseteq U$.  To see this,  fix $w \in \D$ and $r \in (0,1)$ such that $w \in B(0,r)$.  Then standard estimates for conformal maps imply that $|\psi_j'(z)| \gtrsim 1$ for each $z \in B(0,r)$ where the implicit constant does not depend on $j$ or $z$.  It follows that $\dist(\psi_j([0,w]) ,  \partial U_j) \gtrsim 1$ and so $\dist(\psi([0,w]) ,  \partial U) \gtrsim 1$.  Since $U$ is connected,  $\psi(0) = v_{i} \in U$ and $\psi([0,w])$ is a path connecting $v_{i}$ to $\psi(w)$,  it follows that $\psi(w) \in U$ and so $\psi(\D) \subseteq U$.  Note that $\psi(1) = z_0$ and so to prove the claim we need to show that $U \subseteq \psi(\D)$ since the subsequence $(\psi_{j'})_{j \geq 1}$ was arbitrary.  To show the latter,  by applying again as before the Arzela-Ascoli theorem,  we can assume that $\phi_{j''} \to \wt{\phi}$ as $j \to \infty$ locally uniformly in $U$ for some univalent function $\wt{\phi}$ taking values in $\D$.  Then the claim follows since $\psi_{j''}(\phi_{j''}(w)) = w$ and so taking $j \to \infty$ gives that $\psi(\wt{\phi}(w)) = w$ for each $w \in U$.

Now we observe that $\phi_{j+1}^{-1}(I_j) \subseteq K_{j+1}^\Right$ since $I_j \subseteq I_{j+1}$ and so $\Fh_{j+1} \circ \phi_{j+1}^{-1} - \Fh_j \circ \phi_j^{-1}$ has boundary values on $I_j$ given by $d_j$,  where $d_j$ is the difference between the boundary values of $\Fh_{j+1}$ on $K_{j+1}^\Right$ and $\Fh_j$ on $K_j^\Right$.  Thus we have that
\begin{align}\label{eqn:basic_convergence}
\Fh_{j+1} \circ \phi_{j+1}^{-1}(0) - \Fh_j \circ \phi_j^{-1}(0) = \frac{d_j \length(I_j)}{2\pi} + \frac{1}{2\pi} \int_{\partial \D \setminus I_j}(\Fh_{j+1} \circ \phi_{j+1}^{-1}(w) - \Fh_j \circ \phi_j^{-1}(w))dw.
\end{align}
Note that $\Fh_{j+1} \circ \phi_{j+1}^{-1}(w) - \Fh_j \circ \phi_j^{-1}(w) \to 0$ as $j \to \infty$ for Lebesgue a.e.\ $w \in \partial \D \setminus \closure{I}$ since $\psi_j \to \psi$ as $j \to \infty$ locally uniformly in $\closure{\D} \setminus \closure{I}$ and $\Fh_j \to \Fh$ as $j \to \infty$ locally uniformly in~$D$.  Hence the second term of the right side of~\eqref{eqn:basic_convergence} tends to $0$ as $j \to \infty$ since $\|\Fh_j\|_{\infty} \leq m$ and $\length(I \setminus I_j) \to 0$ as $j \to \infty$.  Moreover we have that $\liminf_{j \to \infty} |d_j| \length(I_j)/2\pi \geq c\length(I) / 2\pi > 0$ where $c = \min\{u,  2\lambda - u\}$.  Hence~\eqref{eqn:basic_convergence} implies that $\liminf_{j \to \infty}|\Fh_{j+1} \circ \phi_{j+1}^{-1}(w) - \Fh_j \circ \phi_j^{-1}(w)| = \liminf_{j \to \infty}|\Fh_{j+1}(v_{i}) - \Fh_j(v_{i})| >0$.  This is a contradiction since $\Fh_{j+1}(v_{i}) - \Fh_j(v_{i}) \to 0$ as $j \to \infty$.

\noindent{\it Step 3b.  The limiting points of the exploration fill in open neighborhoods of $\pright \CR_{a,b}$.}
Now suppose that either $K$ disconnects $v_{i}$ from $\pright \CR_{a,b}$ or $v_{i}$ is not disconnected but $\partial U \subseteq K$.  In both cases we let $\phi_j$ be the unique conformal transformation mapping $U_j$ onto $\D$ such that $\phi_j(v_{i}) = 0$ and $\phi_j(x_j) = -i$.  Again if we set $I_j = \phi_j(K_j^\Right)$,  we have that $I_j \subseteq I_{j+1}$ for each $j \in \N$ and $\length(I_j) \to 2\pi$ as $j \to \infty$.  It follows that~\eqref{eqn:basic_convergence} still holds and the second term in the right side of~\eqref{eqn:basic_convergence} tends to $0$ as $j \to \infty$.  Hence we have that $\liminf_{j \to \infty}|\Fh_{j+1}(v_{i}) - \Fh_j(v_{i})| >0$.  Again the martingale convergence theorem implies that $\Fh_j(v_{i})$ converges as $j \to \infty$ a.s.\ in both cases,  and so we obtain a contradiction.  Therefore we have that $\p[ N=\infty ,  J_m=\infty ,  E_{m} ] =0$.

\noindent{\it Step 3c.  Conclusion of the proof.}
Finally, the previous analysis implies that 
\begin{align*}
\p[ N= \infty,  J_m=\infty] \leq \p[ E_m^c] \leq \exp(-c_0 m \log(m)).
\end{align*}
Since the event $v_i \notin K$ implies that $J_m = \infty$ for all $m$ sufficiently large, taking a limit as $m \to \infty$ in the above implies that $\p[ N = \infty, v_i \notin K] = 0$.
\end{proof}

\begin{proof}[Proof of Lemma~\ref{lem:rectangle_maximum_crossing_height}]
We consider the case that $M$ is even. By Lemma~\ref{lem:large_heights_with_small_prob} we see that if we explore the crossings of~$\CR_{a,b}$ from left to right, we are very unlikely to find any crossing with height outside of $[-M^3,M^3]$ before either disconnecting $v_{\lfloor M/2 \rfloor}$ from $\pright \CR_{a,b}$ or terminating the exploration. By symmetry, we may define the analogous exploration from right to left, and we obtain the same bound on the probability that this new exploration discovers a crossing with height outside $[-M^3,M^3]$ before disconnecting $v_{\lfloor M/2 \rfloor}$ from $\pleft \CR_{a,b}$. Denote this event by $G_M^*$. Let $C_M$ be the event from Lemma~\ref{lem:finitely_many_crossings} in the case that $i=\lfloor M/2 \rfloor$. Also,  by arguing in the same way and using Lemma~\ref{lem:finitely_many_crossings}, we obtain that a.s.\ the exploration from right to left discovers finitely many crossings in the case that it does not disconnect $v_{\lfloor M/2 \rfloor}$ from $\pleft \CR_{a,b}$.  Denote this event by $C_M^*$.

Suppose that we are working on the event $G_M^c \cap C_M\cap (G_M^*)^c \cap C_M^*$.  If the left to right exploration does not disconnect $v_{\lfloor M/2 \rfloor}$ from $\pright \CR_{a,b}$,  then all of the crossings of that exploration have heights in $[-M^3,M^3]$ (since $G_M^c$ occurs).  Suppose that the left to right exploration disconnects $v_{\lfloor M/2 \rfloor}$ from $\pright \CR_{a,b}$.  
Note that the right to left exploration either disconnects $v_{\lfloor M/2 \rfloor}$ from $\pleft \CR_{a,b}$ or it discovers finitely many crossings (since $C_M^*$ occurs).  Therefore,  since both of the explorations are connected,  if we are in the first case,  then every crossing discovered by the left to right exploration after it discovers the crossing disconnecting $v_{\lfloor M/2 \rfloor}$ from $\pright \CR_{a,b}$ disconnects the left and right sides of~$\CR_{a,b}$ and so the only way that the right to left exploration can disconnect $v_{\lfloor M/2 \rfloor}$ from $\pleft \CR_{a,b}$ is by hitting the crossing.  Similarly,  if we are in the second case,  then the only way that the right to left exploration can hit $\pleft \CR_{a,b}$ is by hitting every crossing discovered by the left to right exploration.  It follows that in both cases,  every crossing discovered by the left to right exploration after it discovers the crossing disconnecting $v_{\lfloor M/2 \rfloor}$ from $\pright \CR_{a,b}$ is hit by a crossing of the right to left exploration discovered before the latter disconnects $v_{\lfloor M/2 \rfloor}$ from $\pleft \CR_{a,b}$.  We note that two level lines can intersect only if their height difference is less than $2 \lambda$ (Theorem~\ref{thm:interaction_rules}).  Moreover every crossing discovered by the right to left exploration before disconnecting $v_{\lfloor M/2 \rfloor}$ has height in $[-M^3,M^3]$ (since $(G_M^*)^c$ occurs).  It follows that every crossing discovered by the left to right exploration after it disconnects $v_{\lfloor M/2 \rfloor}$ from $\pright \CR_{a,b}$ has height in $[-M^3-2\lambda,  M^3 + \lambda]$ and so the same is true for every crossing discovered from that exploration.  Consequently there exists a constant $c>0$ such that off an event with probability at most $\exp(-cM\log(M))$ every crossing in the left to right exploration has height in $[-M^3,M^3]$.  This completes the proof of the lemma.
\end{proof}

\subsection{Number of crossings}
\label{subsec:number_of_crossings}

\begin{lemma}
\label{lem:rectangle_crossing_bound}
For each $H > 0$ and $0 < a < b < H$ we have that the exploration a.s.\ discovers only finitely many crossings.
\end{lemma}

We shall now bound the probability that we discover a crossing $L$ which separates $v_{M-1}$ from $\pright \CR_{a,b}$ but does not intersect $\pright \CR_{a,b}$. With this at hand we can use Lemma~\ref{lem:finitely_many_crossings} to get that the number of crossings discovered up until $v_{M-1}$ is disconnected from $\pright \CR_{a,b}$ or the exploration hits $\pright \CR_{a,b}$ is finite.  In the former case,  the crossing disconnecting $v_{M-1}$ from $\pright \CR_{a,b}$ is the last crossing of the exploration, which will conclude the proof of Lemma~\ref{lem:rectangle_crossing_bound}.

\begin{lemma}
\label{lem:terminate_upon_separation}
Let $J_M$ be the first $j$ so that $K_j$ disconnects $v_{M-1}$ from $\pright \CR_{a,b}$ in $\CR_{a,b}$. There exists a constant $a>0$ such that
\begin{align*}
	\p[J_M < \infty, L_{J_M} \cap \pright \CR_{a,b} = \emptyset, E_M] = O(M^{-a}).
\end{align*}
\end{lemma}
\begin{proof}
The proof is essentially the same as that of Lemma~\ref{lem:next_crossing_separates}. We let $v_* = 1 + i(a+b)/2$, $N_* = \lceil 1 - \log_2 (a+b) \rceil$, $N^* = \lfloor \log M  \rfloor$ and $A_k = B(v_*,2^{-k}) \setminus B(v_*,\tfrac{3}{4} \cdot 2^{-k})$. Then for all $N_* \leq k \leq N^*$ (and all $M$ sufficiently large) $A_k$ separates $v_{M-1}$ from $\partial \CR_{a,b} \setminus \pright \CR_{a,b}$ and $\dist(A_k,A_{k+1}) = 2^{-k-2}$. For such $k$, we let $x_k$ be the midpoint of the lower interval of $\partial A_k \cap \pright \CR_{a,b}$ and $\eta_k$ be the level line of $h$ of height $\wt{u} \in (-2\lambda,0)$ starting from $x_k$. Moreover, we run $\eta_k$ up until the first time $\tau_k$ that it exits~$A_k$ and note that it may intersect $\pright \CR_{a,b}$ away from~$x_k$. Let $\wt{E}_k$ be the event that~$\eta_k$ separates~$v_{M-1}$ from $\partial \CR_{a,b} \setminus \pright \CR_{a,b}$ before exiting $A_k$ and $\wt{E} = \cup_{k = N_*}^{N^*} \wt{E}_k$. By the same strategy as Lemma~\ref{lem:next_crossing_separates} we have that there exists $a>0$ such that
\begin{align*}
	\p[ \wt{E}^c\cap E_M] = O(M^{-a}).
\end{align*}
Since two level lines can never cross (Theorem~\ref{thm:interaction_rules}), we have that on $\wt{E} \cap \{{J_M} < \infty\} \cap E_M$, the level line $L_{J_M}$ must hit $\pright \CR_{a,b}$ in order to pass into the regions of $\CR_{a,b}$ which are cut off by $\cup_k \eta_k$. Thus, on the event $\wt{E}\cap \{{J_M}<\infty\} \cap E_M$, $L_{J_M}$ must intersect $\pright \CR_{a,b}$ in order to separate $v_{M-1}$ from $\pright \CR_{a,b}$. Thus the lemma is proved.
\end{proof}

\begin{proof}[Proof of Lemma~\ref{lem:rectangle_crossing_bound}]
Let $J_M$ be as in the statement of Lemma~\ref{lem:terminate_upon_separation}.  Lemma~\ref{lem:finitely_many_crossings} implies that a.s.\ either $J_M < \infty$ or the exploration hits $\pright \CR_{a,b}$ after discovering a finite number of crossings and before disconnecting $v_{M-1}$ from $\pright \CR_{a,b}$ in $\CR_{a,b}$.  Note that on the event $J_M < \infty$ and $L_{J_M} \cap \pright \CR_{a,b} \neq \emptyset$ the exploration hits $\pright \CR_{a,b}$ after discovering only finitely many crossings.  The result thus follows by combining this observation with Lemmas~\ref{lem:maximal_crossing_tail} and~\ref{lem:terminate_upon_separation} and taking a limit as $M \to \infty$.
\end{proof}

\section{Exploration of annulus crossings}
\label{sec:annulus_exploration}

\subsection{Setup and definition of the exploration}
\label{subsec:annulus_exploration_definition}

\begin{figure}[ht!]
\begin{center}
\includegraphics[scale=0.9]{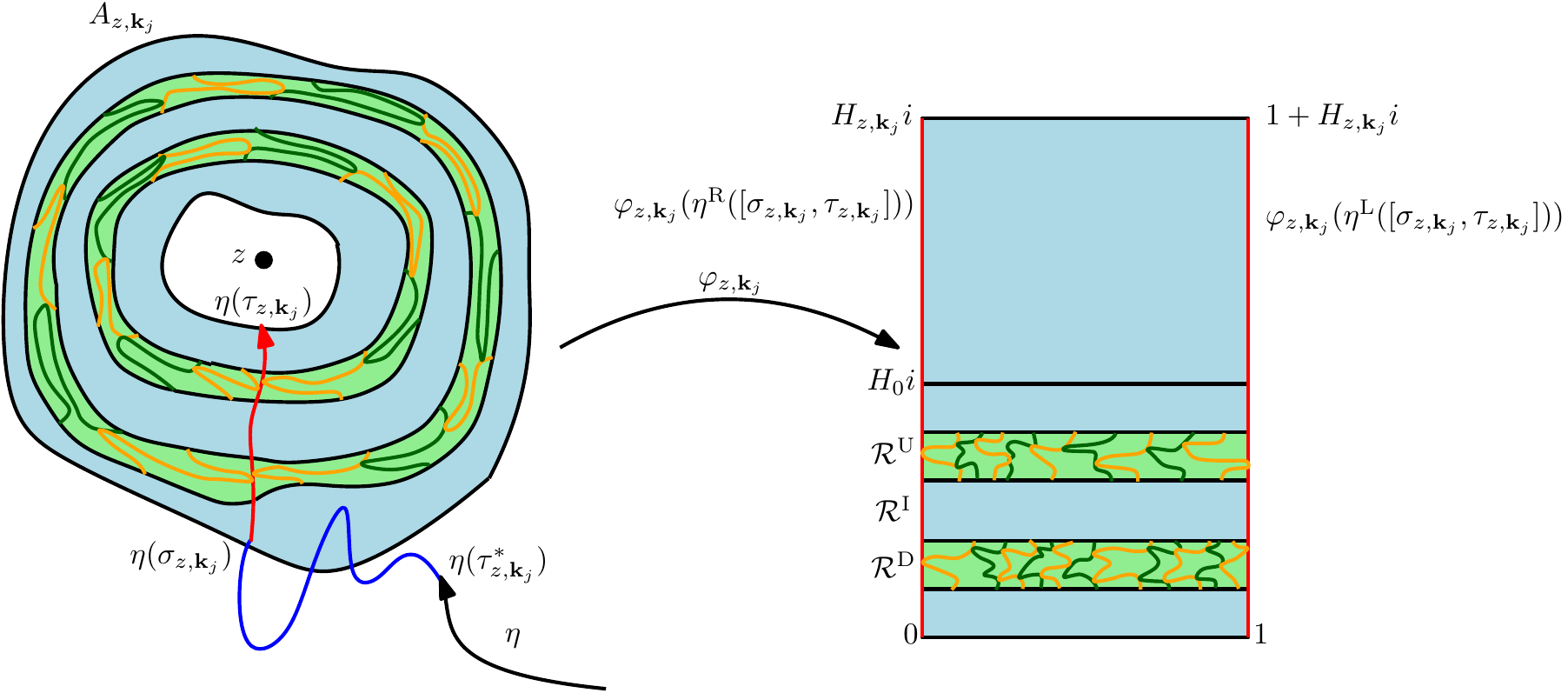}	
\end{center}
\caption{\label{fig:annulus_exploration} Illustration of the setup for the exploration of the crossings across the annulus $A_{z,\kay_j}$.  Note that by the definition of $\kay_j$ we have that $H_{z,\kay_j} \geq H_0$ so that the rectangles $\CR^\Down$, $\CR^\Inner$, $\CR^\Up$ are contained in $(0,1) \times (0,H_{z,\kay_j})$.  The crossings that the exploration in $\CR^\Down$, $\CR^\Up$ discovers are colored alternating in dark green and orange. We remark that the exploration discovers many strands which are not part of a crossing, but for the sake of convenience only the crossings are shown.}
\end{figure}

Throughout, we let $h$ be a GFF on $\h$ with boundary conditions $\lambda$ (resp.\ $-\lambda$) on $\R_+$ (resp.\ $\R_-$) and let $\eta$ be the level line of $\h$ from $0$ to $\infty$.  Then $\eta$ is an $\SLE_4$ in $\h$ from $0$ to $\infty$.  We denote by $\eta^\Left$ (resp.\ $\eta^\Right$) the collection of prime ends corresponding to $\eta$ in the boundary of the component of $\h \setminus \eta$ which is to the left (resp.\ right) of $\eta$.  See Figure~\ref{fig:annulus_exploration} for an illustration of the setup (in the context of Section~\ref{subsec:iteration_exploration} described just below).

We define the annuli and stopping times for $\eta$ as in Section~\ref{subsec:first_crossing_setup}.  Fix $H_0 > 0$ and $\xi \in (0,H_0/10)$.  Let $\CR^\Down = (0,1) \times (2 \xi, 3 \xi)$, $\CR^\Up = (0,1) \times (H_0-3 \xi,H_0-2 \xi)$, and $\CR^\Inner = (0,1) \times (3 \xi,H_0 - 3\xi)$.  As in Section~\ref{sec:rectangle_exploration}, for $q \in \{\Down,\Up\}$, we let $\pdown \CR^q$, $\pup \CR^q$, $\pleft \CR^q$ and $\pright \CR^q$ be the lower, upper, left and right boundary of $\CR^q$, respectively.  For each $z \in \h$ and $k \in \N$ with $H_{z,k} \geq H_0$ we perform the exploration as described in Section~\ref{sec:rectangle_exploration} using the GFF $h \circ \varphi_{z,k}^{-1}$ in each of $\CR^\Down$ and $\CR^\Up$.  For $q \in \{\Down,\Up\}$, we let $(L_j^q)$, $(u_j^q)$, respectively, be the crossings and heights discovered by the exploration.

Part of this reason that we have defined the exploration in this way comes from the following lemma.
\begin{lemma}
\label{lem:eta_crossing_bound}
On $\tau_{z,k} < \infty$ and $H_{z,k} \geq H_0$, the number of times that $\eta$ crosses $A_{z,k}$ is at most the minimal number of $j$ such that $u_j^q = 0$ for $q \in \{\Down,\Up\}$.
\end{lemma}
\begin{proof}
Suppose that $\eta$ makes a crossing $L$ across $A_{z,k}$.  Then by Theorem~\ref{thm:interaction_rules} the level lines which make up the exploration in $\varphi_{z,k}^{-1}(\CR^q)$ for $q \in \{\Down,\Up\}$ cannot cross $L$.
\end{proof}

\subsection{Iterating the exploration across annuli}
\label{subsec:iteration_exploration}

We are now going to describe how to iterate the exploration from Section~\ref{subsec:annulus_exploration_definition} across successive annuli.  Suppose that we have fixed $z \in \h$.  We let $\kay_0$ be the smallest $k \in \N$ so that $A_{z,k} \subseteq \h$.  We fix $a \in (0,1)$,  $H_0, \delta_0 > 0$,  $b \in (1,2)$,  $\xi_1 \in (0,(e^{-b}-e^{-2})/2)$,  $\xi_0 \in (0,H_0/2)$,  $\wt{\xi}_1>0$,  $\delta_1 \in (0,1)$,  $\wt{\xi}_0 \in (0,H_0/2)$, $\xi_2 >0$ and $M_0>1$.  We also fix $\xi \in (0,H_0/10)$ such that $\xi < \xi_0/3$ and perform the explorations in $\CR^{\Up}$ and $\CR^{\Down}$ as in Section~\ref{subsec:annulus_exploration_definition}.  We let $\kay_1$ be the smallest $\ell \geq \kay_0$ so that
\begin{enumerate}[(i)]
\item\label{it:good_annulus1} $\tau_{z,\ell} < \infty$, $H_{z,\ell} \geq H_0$, $\inf_{\tau_{z,\ell-1} \leq t \leq \tau_{z,\ell}}S_t \geq \delta_0$, and
\item\label{it:good_annulus2} the events $F_{z,\ell}^1,F_{z,\ell}^2$,  and $C_{z,\ell}$ of Lemmas~\ref{lem:band_mapping}, \ref{lem:first_crossing_number_of_squares},  and~\ref{lem:first_excursion_good} respectively occur with the above choice of constants,  where $\xi_2 = 3\xi_1/2$ and $\wt{\xi}_1 = \delta_2$ with $\delta_2$ as in Lemma~\ref{lem:first_excursion_good}.
\end{enumerate}
If there is no such $\ell$, then we let $\kay_1 = \infty$.  On the event $\tau_{z,\kay_1} < \infty$, we perform the exploration as in Section~\ref{subsec:annulus_exploration_definition} in $A_{z,\kay_1}$.  Let $M_1 \in \N$ be the smallest $n \in \N$ so that the height of all of the crossings discovered in $A_{z,\kay_1}$ is in $[-n,n]$.  Suppose that we have defined $\kay_1,\ldots,\kay_j$ and $M_1,\ldots,M_j$.  If $\kay_j = \infty$, we set $\kay_{j+1} = \infty$ and $M_{j+1} = 0$.  If $\kay_j < \infty$, we set $\kay_j^* = \lceil 2\log_2 M_j \rceil = \lceil 2 \log M_j / \log 2 \rceil$ and let $\kay_{j+1}$ be the smallest $\ell \geq \kay_j + \kay_j^*+4$ so that~\eqref{it:good_annulus1}, \eqref{it:good_annulus2} hold.  If there is no such $\ell$, we set $\kay_{j+1} = \infty$.  If $\kay_{j+1} < \infty$, we perform the exploration from Section~\ref{subsec:annulus_exploration_definition} in $A_{z,\kay_{j+1}}$ and let $M_{j+1}$ be the smallest $n \in \N$ so that the height of all of the crossings discovered in $A_{z,\kay_{j+1}}$ is in $[-n,n]$.  We let $\CF_{z,j}$ be the $\sigma$-algebra generated by $\eta|_{[0,\tau_{z,\kay_j}]}$ and the exploration in each of $A_{z,\kay_\ell}$ for $1 \leq \ell \leq j$.  We also let~$\CF_{z,j}^-$ be the $\sigma$-algebra generated by $\eta|_{[0,\tau_{z,\kay_j}]}$ and the exploration in each of $A_{z,\kay_\ell}$ for $1 \leq \ell \leq j-1$.

The purpose of the following lemma is to illustrate why we have defined the successive scales in the manner that we have.

\begin{lemma}
\label{lem:exploration_renewal_times}
There exists a constant $c_0 > 0$ so that the following is true.  On $\tau_{z,\kay_j} < \infty$, we let $\Fh_{z,j}$ be the conditional expectation of $h$ given $\CF_{z,j}^-$.  Then
\begin{equation}
\label{eqn:harmonic_maximum_bound}
\sup_{w \in A^*_{z,\kay_j}} |\Fh_{z,j}(w)| \leq c_0.
\end{equation}
Let $\Fh_0$ be the harmonic function on $(0,1) \times (0,H_0)$ with boundary conditions equal to $0$ on $[0,1] \times \{0,H_0\}$ and $\lambda$ (resp.\ $-\lambda$) on $\{0\} \times [0,H_0]$ (resp.\ $\{1\} \times [0,H_0]$).  Then for each $\delta \in (0,H_0/2)$ there exists a constant $c_1 > 0$ so that with $\wt{\Fh}_{z,j} = \Fh_{z,j} \circ \varphi_{z,\kay_j}^{-1}$ and $\wh{\Fh}_{z,j} = \wt{\Fh}_{z,j} - \Fh_0$ we have
\begin{equation}
\label{eqn:harmonic_dirichlet_bound}
\int_{(0,1) \times (\delta,H_0-\delta)} | \nabla \wh{\Fh}_{z,j}(w)|^2 d\Leb_2(w) \leq c_1.	
\end{equation}

\end{lemma}
\begin{proof}
Given $\CF_{z,j}^-$ and $\tau_{z,\kay_j} < \infty$, let $D_{z,j}$ be the component containing $z$ in the complement of $\h$ of the union of $\eta([0,\tau_{z,\kay_j}])$ and the level lines discovered by the exploration in $A_{z,\kay_i}$ for $1 \leq i \leq j-1$.  Let $\Fh^+$ be the function which is harmonic in $D_{z,j}$ with boundary values given by $-\lambda$ (resp.\ $\lambda$) on the part of $\eta^\Left([0,\tau_{z,\kay_j}])$ (resp.\ $\eta^\Right([0,\tau_{z,\kay_j}])$) in $\partial D_{z,j}$ and $M_{j-1}+\lambda$ on the rest of $\partial D_{z,j}$.  We define $\Fh^-$ in the same way except we take the boundary values on the rest of $\partial D_{z,j}$ to be given by $-M_{j-1}-\lambda$.  Note that the exploration in $A_{z,\kay_{j-1}}$ disconnects $D_{z,j}$ from the exploration in $A_{z,\kay_i}$ for $1 \leq i \leq j-2$.  Also,  combining \cite[Theorem~3.21]{lawler2008conformally} with the fact that $|\psi_{z,\kay_{j-1}}'(z)| = e^{4\kay_{j-1}}$,  we obtain that there exist universal constants $c_1,c_2>0$ such that $\dist(z,\partial^{\text{in}}A_{z,\kay_{j-1}}) \geq c_1 e^{-4\kay_{j-1}}$ and $A_{z,\kay_j} \subseteq B(z,c_2 e^{-4\kay_j}) \subseteq B(z,c_1 e^{-4\kay_{j-1}})$.  It follows from the definition of $M_{j-1}$ that
\begin{align*}
	\Fh^-(w) \leq \Fh_{z,j}(w) \leq \Fh^+(w) \quad\text{for all}\quad w \in B(z,c_2 e^{-4\kay_j}) \setminus \eta([0,\tau_{z,\kay_j}]).
\end{align*}
Note that that if $\E^w$ denotes the expectation with respect to the law of a Brownian motion $B$, started from $w$, which is independent of $h$ and $\eta$, then conditionally on $\CF_{z,j}^-$,  we have that $\Fh_{z,j}(w) = \E^w[\Fh_{z,j}(B_{\tau}) \giv \CF_{z,j}^-]$, where $\tau$ is the first time that $B$ exits $D_{z,j}$. By the Beurling estimate, there exists a constant $c_0 > 0$ such that for all $w \in B(z,c_2 e^{-4\kay_j})$ the probability that a Brownian motion started from $w$ exits $D_{z,j}$ without first hitting $\eta([0,\tau_{z,\kay_j}]))$ is at most $c_0 e^{-2\kay^*_{j-1}} \leq c_0 M_{j-1}^{-1}$.  Consequently, the contribution to~$\Fh^+(w)$ coming from the boundary values of~$\Fh^+$ on the part of $\partial D_{z,j}$ not in $\eta([0,\tau_{z,\kay_j}]))$ is at most $c_0(1+ \lambda/M_{j-1}) \leq c_0(1+\lambda)$ whereas the contribution from the boundary values on $\eta([0,\tau_{z,\kay_j}])$ is at most $\lambda$.  Altogether, this gives the desired upper bound for $\Fh^+$.   Applying the same argument for $\Fh^-$ gives the analogous lower bound for $\Fh^-$, which completes the proof of~\eqref{eqn:harmonic_maximum_bound}.
 
We will now deduce~\eqref{eqn:harmonic_dirichlet_bound} from~\eqref{eqn:harmonic_maximum_bound}.  Fix $\delta \in (0,\frac{H_0}{2})$.  We note that $\wh{\Fh}_{z,j}$ is harmonic in $(0,1) \times (0,H_0)$ and its boundary values on $ \{0,1\} \times (0, H_0)$ are equal to zero and, by~\eqref{eqn:harmonic_maximum_bound}, its boundary values on $(0,1)$ and $(0,1) + H_0 i$ are at most $c_0$.  We can extend $\wh{\Fh}_{z,j}$ to be a harmonic function on $\wt{\CR} = (-1,2) \times (0,H_0)$ by reflection as follows.  In $(-1,0] \times (0,H_0)$ we set $\wh{\Fh}_{z,j}(x+iy) =-\wh{\Fh}_{z,j}(-x+iy)$ and in $[1,2) \times (0,H_0)$ we set $\wh{\Fh}_{z,j}(x+iy) = - \wh{\Fh}_{z,j}(2-x+iy)$.  Then the extension of $\wh{\Fh}_{z,j}$ (which we also refer to as $\wh{\Fh}_{z,j}$) is harmonic in $\wt{\CR}$ with boundary values given by zero on $\{-1,2\} \times (0,H_0)$ and bounded by $c_0$ on $[-1,2] \times \{0,H_0\}$.  Note that 
\begin{align*}
|\nabla \wh{\Fh}_{z,j}(w)| \lesssim \frac{\| \wh{\Fh}_{z,j} \|_\infty}{\dist(w,\partial \wt{\CR})} \quad \text{for all} \quad w \in \wt{\CR}
\end{align*}
where the implicit constant is universal.  In particular,  $|\nabla \wh{\Fh}_{z,\kay_j}(w)| \lesssim c_0 \delta^{-1}$ for each $w \in (0,1) \times (\delta,H_0-\delta)$.  This completes the proof of~\eqref{eqn:harmonic_dirichlet_bound}.
\end{proof}

Let $\CR = (0,1) \times (0,H_0)$, $\CR^\inn$ be the interior of $\closure{\CR^\Down \cup \CR^\Inner \cup \CR^\Up}$, and let $\p_{H_0}$ denote the law of a GFF on $\CR$ with boundary conditions equal to $\lambda$ (resp.\ $-\lambda$) on $\pleft \CR$ (resp.\ $\pright \CR$) and by $0$ on $\pdown \CR$ and $\pup \CR$.  Finally, let $\p_{H_0}^\inn$ denote the law of the restriction to $\CR^\inn$ of a sample from $\p_{H_0}$.

\begin{lemma}
\label{lem:rn_formula}
There exists a constant $c_0 < \infty$ depending only on $H_0$, $\xi$ and a universal constant $p > 1$ so that the following is true.  Let $\CZ$ be the Radon-Nikodym derivative of the conditional law of $h \circ \varphi^{-1}_{z,\kay_j}|_{\CR^\inn}$ given $\CF_{z,j}^-$ and on $\tau_{z,\kay_j} < \infty$ with respect to $\p_{H_0}^\inn$.  Then $\E[ \CZ^p \giv \CF_{z,j}^-] \one_{\tau_{z,\kay_j} <\infty} \leq c_0$.
\end{lemma}
\begin{proof}
Throughout, we shall assume that we are working on the event that $\tau_{z,\kay_j} <\infty$.  As in the proof of Lemma~\ref{lem:exploration_renewal_times}, let $D_{z,j}$ be the component containing $z$ of the complement in $\h$ of the union of $\eta([0,\tau_{z,\kay_j}])$ and the level lines discovered in the exploration of $A_{z,\kay_i}$ for $1 \leq i \leq j-1$. First we note that conditional on $\CF_{z,j}^-$ we can write $h|_{D_{z,j}} = h^0 + \Fh_{z,j}^0$ where $h^0$ is a GFF in $D_{z,j}$ with boundary values given by $-\lambda$ (resp.\ $\lambda$) on the part of the left (resp.\ right) side of $\eta$ contained in $\partial D_{z,j}$ and~$0$ on the remaining parts of $\partial D_{z,j}$ and $\Fh_{z,j}^0$ is a harmonic function in $D_{z,j}$ which is measurable with respect to $\CF_{z,j}^-$ with $0$ boundary values on the part of the left (resp.\ right) side of $\eta$ in $\partial D_{z,j}$.  Similarly,  we can write $h^0 = h^1 + \Fh_{z,j}^1$ where $h^1$ is a GFF in $\varphi_{z,\kay_j}^{-1}(\CR)$ with boundary values given by $-\lambda$ (resp.\ $\lambda$) on the part of the left (resp.\ right) side of $\eta$ contained in $\partial \varphi_{z,\kay_j}^{-1}(\CR)$ and $0$ on the remaining part of $\partial \varphi_{z,\kay_j}^{-1}(\CR)$ and $\Fh_{z,j}^1$ is harmonic in $\varphi_{z,\kay_j}^{-1}(\CR)$ and conditionally independent of $h^1$ given $\CF_{z,j}^-$.  Note that~$\Fh_{z,j}^1$ has zero boundary values on the left and right sides of $\eta$.

For each $\epsilon > 0$ set $\CR_\epsilon = (0,1) \times (\epsilon,H_0-\epsilon)$; note that $\CR^\inn = \CR_{2\xi}$.  Let $\phi \in C_0^{\infty}(\CR_\xi)$ be such that $\phi|_{\CR^\inn} \equiv 1$.  Let $\phi_{z,j} = \phi \circ \varphi_{z,\kay_j}$ and
\begin{align*}
\CZ = \E\!\left[ \exp\left((h^1 ,  \phi_{z,j}(\Fh_{z,j}^0+\Fh_{z,j}^1))_{\nabla} - \frac{1}{2}\|\phi_{z,j}(\Fh_{z,j}^0 + \Fh_{z,j}^1)\|^2_{\nabla}\right)   \, \middle| \,  h^1|_{\varphi^{-1}_{z,\kay_j}(\CR^\inn)},   \CF_{z,j}^-\right].
\end{align*}
Note that if we weight the conditional law of $h^1|_{\varphi^{-1}_{z,\kay_j}(\CR^\inn)}$ given $\CF_{z,j}^-$ by $\CZ$ then we obtain the conditional law of $h|_{\varphi^{-1}_{z,\kay_j}(\CR^\inn)}$ given $\CF_{z,j}^-$.  Note also that the field $h^1 \circ \varphi^{-1}_{z,\kay_j}|_{\CR}$ has the law of a GFF in~$\CR$ with boundary values given by $\lambda$ (resp.\  $-\lambda$) on $\pleft \CR$ (resp.\ $\pright \CR$) and $0$ on $\pup \CR$, $\pdown \CR$.  Fix $p > 1$.  We need to give an upper bound on $\E[ \CZ^p  \giv \CF_{z,j}^-]$.  By Jensen's inequality,  we have that
\begin{align*}
\E[ \CZ^p  \giv \CF_{z,j}^-]  \leq \E[ \exp((p^2-p)/2 \|\phi_{z,j}(\Fh_{z,j}^0 + \Fh_{z,j}^1)\|^2_{\nabla})  \giv  \CF_{z,j}^-].
\end{align*}
Here, we used that $\Fh_{z,j}^0$ and $\Fh_{z,j}^1$ are conditionally independent given $\CF_{z,j}^-$ and that $h^1$ is conditionally independent of $(\Fh_{z,j}^0 ,  \Fh_{z,j}^1)$ given $\CF_{z,j}^-$.  Next,  we have that
\begin{align*}
\|\phi(\Fh_{z,j}^1 \circ \varphi^{-1}_{z,\kay_j})\|_\nabla \lesssim \|\phi\|_\infty \|\Fh_{z,j}^1 \circ \varphi^{-1}_{z,\kay_j}\|_\nabla + \|\Fh_{z,j}^1 \circ \varphi^{-1}_{z,\kay_j}\|_\infty \|\phi\|_\nabla
\end{align*}
where the norms are computed on $\CR_\xi$ and the implicit constants are universal.  First,  we focus on the term $I = \|\Fh_{z,j}^1 \circ \varphi^{-1}_{z,\kay_j}\|_\infty$.  Arguing as in the proof of Lemma~\ref{lem:exploration_renewal_times}, we can extend $\Fh_{z,j}^1 \circ \varphi^{-1}_{z,\kay_j}|_\CR$ to a function $\wh{\Fh}^1$ which is harmonic in $(-1,2) \times (0,H_0)$.  For each $\epsilon > 0$ we let $\wt{\CR}_\epsilon = (-1,2) \times (\epsilon,H_0-\epsilon)$.  Let $\wt{\Fp}$ be the Poisson kernel in $\wt{\CR}_{\xi/2}$ and let $\ol{\Leb}_1$ denote Lebesgue measure on $\partial \wt{\CR}_{\xi/2}$ normalized to have unit mass.  For each $w \in \CR_{\xi}$ we have that $\wh{\Fh}^1(w) = \int_{\partial \wt{\CR}_{\xi/2}} \wt{\Fp}(w,u) \wh{\Fh}^1(u) d \ol{\Leb}_1(u)$ and so $|\wh{\Fh}^1(w)| \leq \int_{\partial \wt{\CR}_{\xi/2}} \wt{\Fp}(w,u) |\wh{\Fh}^1(u)|d \ol{\Leb}_1(u)$.  Also,  Harnack's inequality implies that there exist $w_0 \in \wt{\CR}_{\xi}$ and $c_0 < \infty$ depending only on $H_0$ and $\xi$ such that $\wt{\Fp}(w,u) \leq c_0 \wt{\Fp}(w_0,u)$ for each $w \in \CR_{\xi}$ and $u \in \partial \wt{\CR}_{\xi/2}$.  It follows that
\begin{align*}
I \leq c_0 \int_{\partial \wt{\CR}_{\xi/2}} \wt{\Fp}(w_0,u) |\wh{\Fh}^1(u)| d\ol{\Leb}_1(u) \leq c_1 \int_{\partial \wt{\CR}_{\xi/2}} |\wh{\Fh}^1(u)|d \ol{\Leb}_1(u)
\end{align*}
where $c_1 < \infty$ depends only on $H_0$ and $\xi$.  Hence Jensen's inequality implies that
\begin{align*}
\E[ \exp(aI^2) \giv   \CF_{z,j}^-] \leq \int_{\partial \wt{\CR}_{\xi/2}} \E[ \exp(ac_1^2 |\wh{\Fh}^1(u)|^2)  \giv  \CF_{z,j}^-] d \ol{\Leb}_1(u).
\end{align*}
Fix $u \in \partial \wt{\CR}_{\xi/2}$.  Then there exists $v \in \CR_{\xi/2}$ such that either $\wh{\Fh}^1(u) = \Fh^1_{z,j}(\varphi^{-1}_{z,\kay_j}(v))$ or $\wh{\Fh}^1(u) =- \Fh^1_{z,j}(\varphi^{-1}_{z,\kay_j}(v))$.  Let $\Fg_{z,j}^0$ (resp.\  $\Fg_{z,j}^1$) be the harmonic function in $D_{z,j}$ (resp.\  $\varphi_{z,\kay_j}^{-1}(\CR)$) with boundary values given by $-\lambda$ (resp.\  $\lambda$) on the part of the left (resp.\ right) side of $\eta$ contained in $\partial D_{z,j}$ (resp.\ $\partial \varphi_{z,\kay_j}^{-1}(\CR)$) and zero on the remaining parts of $\partial D_{z,j}$ (resp.\  $\partial \varphi_{z,\kay_j}^{-1}(\CR)$).  In particular, the boundary values of $\Fg_{z,j}^0$ and $\Fg_{z,j}^1$ are in $[-\lambda,\lambda]$.  Note that
\[ \mu := \E[ \Fh_{z,j}^1(\varphi^{-1}_{z,\kay_j}(v)) \giv   \CF_{z,j}^-] = \Fg_{z,j}^0(\varphi^{-1}_{z,\kay_j}(v)) - \Fg_{z,j}^1(\varphi^{-1}_{z,\kay_j}(v)).\]
We also have that
\[ \sigma^2 := \var(\Fh_{z,j}^1(\varphi^{-1}_{z,\kay_j}(v))   \giv \CF_{z,j}^-) = \log(\confrad(\varphi^{-1}_{z,\kay_j}(v),D_{z,j})) - \log(\confrad(\varphi^{-1}_{z,\kay_j}(v),\varphi_{z,\kay_j}^{-1}(\CR))).\]
By the Koebe-1/4 theorem, $\confrad(\varphi_{z,\kay_j}^{-1}(v),D_{z,j}) \asymp \dist(\varphi_{z,\kay_j}^{-1}(v),\partial D_{z,j})$ and $\confrad(\varphi_{z,\kay_j}^{-1}(v),\varphi_{z,\kay_j}^{-1}(\CR)) \asymp \dist(\varphi_{z,\kay_j}^{-1}(v),\partial \varphi_{z,\kay_j}^{-1}(\CR))$, where the implicit constants are universal.  Moreover, we claim that
\[ \dist(\varphi_{z,\kay_j}^{-1}(v),\partial D_{z,j}) \asymp \dist(\varphi_{z,\kay_j}^{-1}(v),\partial \varphi_{z,\kay_j}^{-1}(\CR)),\]
where the implicit constants depend only on $\xi$ and $H_0$.  Indeed,  first we note that there exists $q \in (0,1)$ depending only on $\xi$ and $H_0$ such that if we start a Brownian motion from $v$,  then with probability at least $q$ it exits $\CR$ either in the left or the right side of $\CR$.  Conformal invariance implies that if we start a Brownian motion from $\varphi_{z,\kay_j}^{-1}(v)$,  then with probability at least $q$ it exits $\varphi_{z,\kay_j}^{-1}(\CR)$ in $\eta([0,\tau_{z,\kay_j}])$.  Thus the Beurling estimate implies that $\dist(\varphi_{z,\kay_j}^{-1}(v),\eta([0,\tau_{z,\kay_j}])) \lesssim \dist(\varphi_{z,\kay_j}^{-1}(v),\partial \varphi_{z,\kay_j}^{-1}(\CR))$.  Also note there exists a constant $c_2 \geq 1$ so that $\dist(\varphi_{z,\kay_j}^{-1}(v),A_{z,\kay_{j-1}}) \geq c_2 e^{-4\kay_j} \geq \dist(\varphi_{z,\kay_j}^{-1}(v),\eta([0,\tau_{z,\kay_j}]))$ and so
$\dist(\varphi_{z,\kay_j}^{-1}(v),\partial D_{z,j}) = \dist(\varphi_{z,\kay_j}^{-1}(v),\eta([0,\tau_{z,\kay_j}])) \lesssim \dist(\varphi_{z,\kay_j}^{-1}(v),\partial \varphi_{z,\kay_j}^{-1}(\CR))$.  On the other hand,  we have that $\dist(\varphi_{z,\kay_j}^{-1}(v),\partial \varphi_{z,\kay_j}^{-1}(\CR)) \leq \dist(\varphi_{z,\kay_j}^{-1}(v),\eta([0,\tau_{z,\kay_j}])) = \dist(\varphi_{z,\kay_j}^{-1}(v),\partial D_{z,j})$ and so this proves the claim.  It follows that $\mu = O(1)$ and $\sigma^2 = O(1)$ where the implicit constants depend only on $H_0$ and $\xi$.  We thus have for $a > 0$ sufficiently small (depending only on $H_0$ and $\xi$) that $\E[ \exp(aI^2) \giv  \CF_{z,j}^-] = O(1)$.  Therefore arguing as in the proof of Lemma~\ref{lem:exploration_renewal_times} and observing that $\|\phi\|_\nabla \lesssim 1$,  $\|\phi\|_{\infty} \lesssim 1$ with the implicit constants depending only on $H_0$ and $\xi$,  we obtain (possibly decreasing $a > 0$) that
\begin{align}
\label{eqn:square_exp_bound1}
\E[ \exp(a \| \phi (\Fh_{z,j}^1 \circ \varphi^{-1}_{z,\kay_j}) \|^2_{\nabla}) \giv  \CF_{z,j}^-] = O(1).
\end{align}
Furthermore,  arguing as in the proof of Lemma~\ref{lem:exploration_renewal_times} gives that $\Fh_{z,j}^0 \circ \varphi^{-1}_{z,\kay_j}$ extends to a function~$\wt{\Fh}$ which is harmonic in $(-1,2) \times (0,H_0)$ and such that $\sup_{w \in (-1,2) \times (0,H_0)} |\wt{\Fh}(w)| \lesssim 1$ where the implicit constant is universal since $\Fh_{z,j}^0 + \Fg_{z,j}^0$ has the same boundary values as $h$ on $\partial D_{z,j}$.  It follows that $\|\phi (\Fh_{z,j}^0 \circ \varphi^{-1}_{z,\kay_j})\|_{\nabla} \lesssim 1$ where the implicit constant depends only on $H_0$ and $\xi$.  The proof is then complete by combining with~\eqref{eqn:square_exp_bound1} and the conformal invariance of $(\cdot,\cdot)_{\nabla}$.
\end{proof}

\subsection{Crossing bounds}
\label{subsec:annulus_crossing_bounds}

\begin{lemma}
\label{lem:annulus_maximum_crossing_height}
For each $H_0 > 0$ and $\xi \in (0,H_0/10)$ there exist constants $c_0, M_0 > 0$ so that the following is true.  On $\tau_{z,\kay_{j+1}} < \infty$, let $u^*$ be the maximum height of a crossing discovered by the exploration in $A_{z,\kay_{j+1}}^*$, otherwise let $u^* = 0$.  Then
\[ \p[ u^* \geq M  \giv \CF_{z,\tau_{z,\kay_j}}] \one_{\tau_{z,\kay_j} < \infty} \leq \exp(-c_0 M^{1/3}\log(M)) \quad\text{for all}\quad M \geq M_0.\]
\end{lemma}
\begin{proof}
This follows by combining Lemma~\ref{lem:rectangle_maximum_crossing_height} with Lemma~\ref{lem:rn_formula}.
\end{proof}

\begin{lemma}
\label{lem:annulus_crossing_bound}
Suppose that $H_0 > 0$ and $\xi \in (0,H_0/10)$.  On $\tau_{z,\kay_j} < \infty$, it is a.s.\ the case that the exploration in $A_{z,\kay_j}$ discovers only finitely many crossings.
\end{lemma}
\begin{proof}
This follows by combining Lemma~\ref{lem:rectangle_crossing_bound} with Lemma~\ref{lem:rn_formula}.
\end{proof}

\section{$\SLE_4$ satisfies the hypotheses of the removability theorem}
\label{sec:sle4_is_removable}

The purpose of this section is to complete the proof of Theorem~\ref{thm:sle_4_removable} by showing that $\SLE_4$ satisfies the hypotheses of Theorem~\ref{thm:removability_of_X}.  Throughout, we will assume that we have the same setup and use the same notation as in Sections~\ref{subsec:annulus_exploration_definition} and~\ref{subsec:iteration_exploration}.

\begin{figure}[ht!]
\begin{center}
\includegraphics{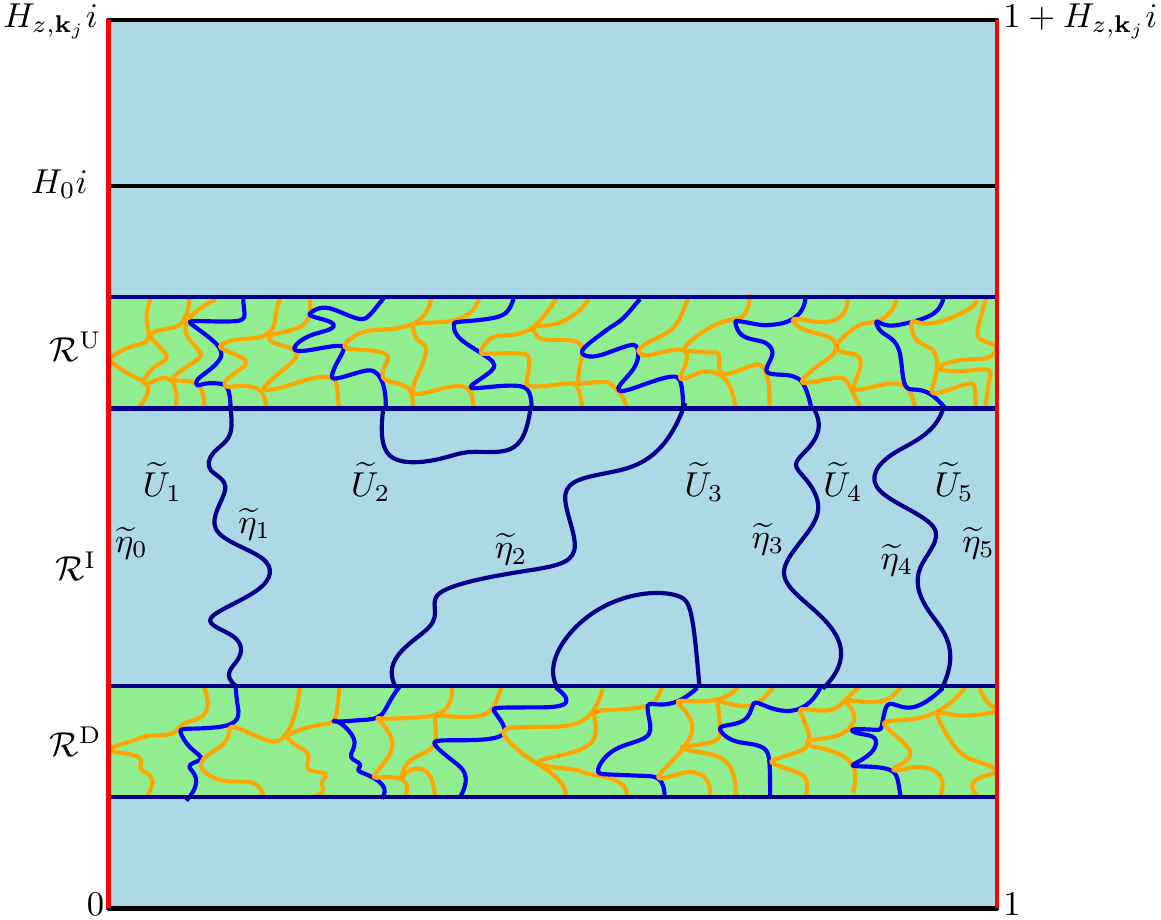}
\end{center}
\caption{\label{fig:sle4_removable_proof} Illustration of the setup for the proof of Theorem~\ref{thm:sle_4_removable} in the setting of the rectangle.  Shown is the rectangle $\CR_{z,\kay_{j}} = (0,1) \times (0,H_{z,\kay_j})$ and the explorations of $\CR^\Down$ and $\CR^\Up$.  The crossings with height zero are shown in blue and the other crossings are shown in orange.  The continuations of these zero height crossings into $\CR^\Inner$ are shown in dark blue and the ones which connect $\CR^\Down$ to $\CR^\Up$ are $\wt{\eta}_1,\ldots,\wt{\eta}_n$ (with $n=4$ in the illustration); the left and right sides of $\CR^\Inner$ are $\wt{\eta}_0$ and $\wt{\eta}_{n+1}$.  Those zero height crossings which do not cross $\CR^\Inner$ are unlabeled.  The conditional law of each $\wt{\eta}_i$ given all of the other paths is absolutely continuous with respect to the law of an $\SLE_4$ curve. As in the case of Figure~\ref{fig:annulus_exploration}, the explorations discover other strands which are not crossings, but only crossings are shown.}
\end{figure}

\subsection{Setup and main statement}
\label{sec:sle4_removable_setup}

See Figure~\ref{fig:sle4_removable_proof} for an illustration of the setup and the notation.  Suppose $j \in \N$ and $z \in \h$ are such that $\tau_{z,\kay_j} < \infty$.  Let $\wt{\eta}_i^q$ for $i = 1,\dots,n_q$ and $q \in \{\Down,\Up\}$ be the height-$0$ crossings of $\CR^q$ which are discovered by the exploration and write $\wt{\eta}_0^q$ (resp.\ $\wt{\eta}_{n_q+1}^q$) for the left (resp.\ right) side of $\CR^q$. Each of these crossings will continue into $\CR^\Inner$ and either form a crossing of~$\CR^\Inner$ or make an excursion into $\CR^\Inner$ from $\CR^q$ back to itself.  We let $\wt{\eta}_1,\dots,\wt{\eta}_n$ be the aforementioned crossings of $\CR^\Inner$, ordered from left to right, $\wt{\eta}_{n+2},\dots,\wt{\eta}_{n+m}$ be the height-$0$ level line excursions into $\CR^\Inner$ from either $\CR^\Down$ or $\CR^\Up$ back to itself, and let $\wt{\eta}_0$ (resp.\ $\wt{\eta}_{n+1}$) be the left (resp.\ right) side of $\CR^\Inner$.  For each $0 \leq i \leq n+1$ we let $\eta_i = \varphi_{z,\kay_j}^{-1}(\wt{\eta}_i)$.  We let $\wt{K}$ be the union of the exploration in~$\CR^\Down$ and~$\CR^\Up$.  For each $1 \leq i \leq n+1$ we let $\wt{U}_i$ be the component of $\CR^\inn \setminus  (\wt{K} \cup ( \cup_{\ell=1}^{n+m} \wt{\eta}_\ell))$ with $\wt{\eta}_{i-1}$ and $\wt{\eta}_i$ on its boundary (with the convention that $\wt{U}_{n+1+i} = \wt{U}_i$ for each $i$), $U_i = \varphi_{z,\kay_j}^{-1}(\wt{U}_i)$, and let $\CW_i$ be a Whitney square decomposition of $U_i$.  We also let $\wt{U}$ be the component of the complement in $\CR$ of the explorations in $\CR^\Down$, $\CR^\Up$ which contains $\CR^\Inner$.  We say $Q \in \CW_i$ is $(M,a)$-good with respect to $Q_i \in \CW_i$ if $\disthyp^{U_i}(\cen(Q),\cen(Q_i)) \leq M ( \len(Q)/\diam(A_{z,\kay_j}))^{-a}$ where $\len(Q)$ is the side length of~$Q$.

\begin{definition}
\label{def:good_event}
Fix $M, a > 0$.  We say that $A_{z,\kay_j}$ is $(M,a)$-good if either $\kay_j = \infty$ or $\kay_j < \infty$, $n \leq M$, and the following are true.  For each $0 \leq i \leq n$ there is a segment $\eta_i(I_i)$ of $\eta_i$ such that 
\begin{enumerate}
	\item $\nmeasure{\eta_i}(\eta_i(I_i))\geq M^{-1} \diam(A_{z,\kay_j})^{3/2}$; \label{it:natural_parameterization_lower_bound}
	\item $\nmeasure{\eta_i}(\eta_i(I_i) \cap Y) \leq M \diam(Y)^{3/2-a}$ for every $Y$ Borel; \label{it:natural_parameterization_diameter_bound}
	\item There exists a square $Q_i \in \CW_i$ such that if $G_i^-$ (resp.\ $G_i^+$) is the set of points $w \in \eta_{i-1}(I_{i-1})$ (resp.\ $w \in \eta_i(I_i)$) such that each square $Q \in \CW_i$ which is hit by the hyperbolic geodesic in~$U_i$ from $\cen(Q_i)$ to $w$ is $(M,a)$-good with respect to~$Q_i$, then
			\begin{enumerate}[(a)]
				\item $\nmeasure{\eta_i}(G_i^-)/\nmeasure{\eta_i}(\eta(I_i)) \geq 3/4$ and $\nmeasure{\eta_i}(G_i^+)/\nmeasure{\eta_i}(\eta(I_{i+1})) \geq 3/4$. \label{it:fraction_of_good_points}
				\item \label{it:not_too_many_squares} The number of elements of $\CW_i$ with side length $2^{-m}$ which are hit by a hyperbolic geodesic in $U_i$ from $\cen(Q_i)$ to a point in $G_i^+$ (resp.\ $G_i^-$) is at most  $M (2^m  \diam(A_{z,\kay_j}))^{3/2+a}$.
			\end{enumerate}
\end{enumerate}
\end{definition}

Let $E_{z,\kay_j}^{M,a}$ be the event that $A_{z,\kay_j}$ is $(M,a)$-good.

\begin{proposition}
\label{prop:good_everywhere}
Suppose that $\eta$ is an $\SLE_4$ in $\h$ from $0$ to $\infty$ and let $K \subseteq \h$ be compact.  For every $a \in (0,1)$ there a.s.\ exists $n_0 \in \N$ and $M > 1$ so that for every $n \geq n_0$ and $z \in (e^{-5n} \Z^2) \cap K$ there exists $m \geq n_0$ with $(1-a^2) n \leq m \leq n$ so that $E_{z,m}^{M,a}$ occurs.
\end{proposition}

In order to start to prove Proposition~\ref{prop:good_everywhere}, we will first in Section~\ref{subsec:regularity_rectangle} work in the rectangle and analyze the crossings  in this setting.  We will then combine these results with our other estimates in Section~\ref{subsec:proof_conclusion} to complete the proof.

\subsection{Regularity statements in the rectangle}
\label{subsec:regularity_rectangle}

We begin by formulating a version of the conditions from Definition~\ref{def:good_event} for the crossings of the rectangle in place of the annulus.  We will then show that each of these conditions is likely to occur provided we adjust the parameters appropriately.  Afterwards, we will show that these conditions imply the corresponding conditions from Definition~\ref{def:good_event} (with the exception of the first crossing, which was handled earlier in Section~\ref{sec:first_crossing}) provided the crossings of the rectangle are separated from each other and the boundary.

Fix $a,H_0 > 0$, $H \geq H_0$, and $\xi \in (0,H_0/10)$.  Let $\CR = (0,1) \times (0,H)$ and let $\p_H$ be the law of a GFF on $\CR$ with boundary conditions given by $\lambda$ on $\pleft \CR$, $-\lambda$ on $\pright \CR$, and $0$ on $\pdown \CR$ and $\pup \CR$.  We perform the exploration in $\CR^\Down$, $\CR^\Up$ as in Section~\ref{sec:annulus_exploration} and let $\wt{U}$, $\wt{U}_i$, and $\wt{\eta}_i$ be as in Section~\ref{sec:sle4_removable_setup}.  In particular, $\wt{\eta}_i$ for $1 \leq i \leq n$ are the continuations of height-$0$ crossings of the $\CR^q$ for $q \in \{\Down,\Up\}$ which cross $\CR^\Inner$.  We also let $\wt{\CW}_i$ be a Whitney square decomposition of $\wt{U}_i$ for each $1 \leq i \leq n+1$.  Suppose that we have fixed $H \geq H_0$ and $\xi \in (0,H_0/10)$.  For each $1 \leq i \leq n$, we let~$\wt{x}_i$ (resp.\ $\wt{y}_i$) be the starting (resp.\ ending) point of $\wt{\eta}_i$ and we let $\wt{\sigma}_i = \sup\{t \geq 0 : \wt{\eta}_i(t) \in B(\wt{x}_i,\xi)\}$, $\wt{\tau}_i = \inf\{t \geq 0 : \wt{\eta}_i(t) \in B(\wt{y}_i,\xi)\}$, and $\wt{I}_i = [ \wt{\sigma}_i, \wt{\tau}_i]$.  We will now argue that by adjusting the parameters it is likely that $n \leq M$ and for each $1 \leq i \leq n$ each of the following conditions hold.

\begin{enumerate}[(I)]
	\item $\nmeasure{\wt{\eta}_i}(\wt{\eta}_i(\wt{I}_i))\geq M^{-1}$; \label{it:natural_parameterization_lower_bound_rectangle}
	\item $\nmeasure{\wt{\eta}_i}(\wt{\eta}_i(\wt{I}_i) \cap Y) \leq M \diam(Y)^{3/2-a}$ for every $Y$ Borel; \label{it:natural_parameterization_diameter_bound_rectangle}
	\item \label{it:natural_measure_cube_conditions} There exists a square $\wt{Q}_i \in \wt{\CW}_i$ such that if $\wt{G}_i^-$ (resp.\ $\wt{G}_i^+$) is the set of points on $w \in \wt{\eta}_{i-1}(\wt{I}_{i-1})$ (resp.\ $w \in \wt{\eta}_i(\wt{I}_i)$) such that each square $Q \in \wt{\CW}_i$ which is intersected by the hyperbolic geodesic from $\cen(\wt{Q}_i)$ to $w$ is $(M,a)$-good with respect to $\wt{Q}_i$, then
			\begin{enumerate}[(a)]
				\item \label{it:fraction_of_good_points_rectangle} $\nmeasure{\wt{\eta}_{i-1}}(\wt{G}_i^-)/\nmeasure{\wt{\eta}_{i-1}}(\wt{\eta}_{i-1}(\wt{I}_{i-1})) \geq 3/4$ (if $i \geq 2$) and $\nmeasure{\wt{\eta}_i}(\wt{G}_i^+)/\nmeasure{\wt{\eta}_i}(\wt{\eta}_i(\wt{I}_i)) \geq 3/4$ (if $i \leq n$). 
				\item \label{it:not_too_many_squares_rectangle} 
				The number of elements of $\wt{\CW}_i$ with side length $2^{-m}$ which are hit by a hyperbolic geodesic in $\wt{U}_i$ from $\cen(\wt{Q}_i)$ to a point in $\wt{G}_i^+$ (if $i \leq n$) (resp.\ $\wt{G}_i^-$ (if $i \geq 2$)) is at most  $M 2^{(3/2+a)m}$.
			\end{enumerate}
\end{enumerate}

We begin by showing that the crossings are likely to be separated from each other.

\begin{lemma}
\label{lem:crossings_separated}
For each $H \geq H_0$ and $\xi \in (0,H_0/10)$ we have that
\begin{align}
\label{eqn:crossing_distance_positive}
	\p_H\!\left[ \min_{i \neq \ell} \dist(\wt{\eta}_i,\wt{\eta}_\ell) \leq \epsilon  \right] \to 0 \quad \text{as} \quad \epsilon \to 0.
\end{align}
\end{lemma}
\begin{proof}
Almost surely, these crossings will not intersect each other or hit $\pleft \CR \cup \pright \CR$.  Consequently, $\min_{i \neq \ell} \dist(\wt{\eta}_i^q,\wt{\eta}_\ell^q)$ for $q \in \{\Down,\Up\}$ is an a.s.\ positive random variable, which immediately implies~\eqref{eqn:crossing_distance_positive}.
\end{proof}

In the following, we will first argue that~\eqref{it:natural_parameterization_lower_bound_rectangle} and~\eqref{it:natural_parameterization_diameter_bound_rectangle} are likely to hold (Lemma~\ref{lem:natural_parameterization_bounds}), then that~\eqref{it:fraction_of_good_points_rectangle} is likely to hold (Lemma~\ref{lem:all_crossings_good}), and finally that~\eqref{it:not_too_many_squares_rectangle} is likely to hold (Lemma~\ref{lem:not_too_many_squares}).

\begin{lemma}
\label{lem:natural_parameterization_bounds}
Fix $p_0 \in (0,1)$, $a, H_0 > 0$, $H \geq H_0$, and $\xi \in (0,H_0/10)$.  Then there exists $M_0 > 1$ depending only on $p_0$, $a$, $H$, $\xi$ such that for every $M \geq M_0$ both~\eqref{it:natural_parameterization_lower_bound_rectangle} and~\eqref{it:natural_parameterization_diameter_bound_rectangle} hold with probability at least $1-p_0$ under $\p_H$.
\end{lemma}
\begin{proof}
Let $h \sim \p_H$ and fix $N \in \N$.  Assume that we have carried out the explorations in~$\CR^\Up$ and~$\CR^\Down$, let~$\wt{K}$ be their union, let $\wt{U}$ be the connected component of $\CR \setminus \wt{K}$ which contains $\CR^\Inner$, and let~$\wt{\CW}$ be a Whitney square decomposition of $\wt{U}$.  Let $F_1$ be the event that the number of crossings discovered by the exploration in $\CR^\Down$ and $\CR^\Up$ is at most $N$, the height of any such crossing is in absolute value at most $N$, and the distances between different crossings are at least $1/N$.  Lemmas~\ref{lem:rectangle_crossing_bound}, \ref{lem:rn_formula}, and~\ref{lem:crossings_separated} imply that we can choose $N$ sufficiently large (depending only on $p_0$, $a$, $H$, $\xi$) so that $\p_H[F_1] \geq 1-p_0/100$.

Let $\wt{\eta}_1,\ldots,\wt{\eta}_n$ be the continuations of the $0$-height crossings of $\CR^\Down$ into $\CR^\Inner$ which connect with a $0$-height crossing of $\CR^\Up$ and let $\wt{\eta}_{n+1},\ldots,\wt{\eta}_{n+m}$ be the remainder of the $0$-height crossings of $\CR^\Down$ or $\CR^\Up$ which do not cross $\wt{U}$.  For each $1 \leq i \leq n$, we let $\wt{K}_i$ be the union of $\wt{K}$ and $\wt{\eta}_\ell$ for $\ell \neq i$ and we let $\wt{W}_i$ be the component of $\CR \setminus \wt{K}_i$ which contains $\wt{\eta}_i$.  Then the conditional law of $\wt{\eta}_i$ given everything else is absolutely continuous with respect to that of a chordal $\SLE_4$ in $\wt{W}_i$ from $\wt{x}_i$ to $\wt{y}_i$.  Let $\wt{\CZ}_i$ be the Radon-Nikodym derivative between the conditional law of $\wt{\eta}_i$ given everything else and that of a chordal $\SLE_4$ in $\wt{W}_i$ from $\wt{x}_i$ to $\wt{y}_i$.  We note that $\wt{\CZ}_i$ is an a.s.\ finite random variable.  Let $F_2$ be the event that $F_1$ holds and $\wt{\CZ}_i \leq N$ for each $1 \leq i \leq n$.  Lemma~\ref{lem:rn_formula} implies that we can take $N$ sufficiently large (depending only on $p_0$, $a$, $H$, $\xi$) so that $\p_H[F_2] \geq 1-p_0/50$.

For each $1 \leq i \leq n$ we let $\wt{\phi}_i$ be the unique conformal transformation $\h \to \wt{W}_i$ which takes $0$ to $\wt{x}_i$, $-1$ to the leftmost point in $\partial \wt{W}_i$ contained in the line $\{x + i H_0/2 : x \in \R\}$, and $\infty$ to $\wt{y}_i$ and let $\wh{\eta}_i = \wt{\phi}_i^{-1}(\wt{\eta}_i)$.  On $F_2$, we have that $\wh{\eta}_i$ is absolutely continuous with respect to an $\SLE_4$ in $\h$ from $0$ to $\infty$ with a Radon-Nikodym derivative which is bounded by $N$.

Fix a compact set $K \subseteq \h$.  Let $F_3$ be the event that $F_2$ holds and $\wt{\phi}_i^{-1}(\wt{\eta}_i([\wt{\sigma}_i,\wt{\tau}_i])) \subseteq K$ for each $1 \leq i \leq n$.  We note that we can choose $K$ to be sufficiently large so that $\p_H[F_3] \geq 1-p_0/25$.  By Lemmas~\ref{lem:natural_good_whole_plane} and~\ref{lem:rn_formula} there exists a constant $c_0 > 0$ (depending only on $K$, $N$, $H$, and $\xi$) so that if $F_4$ is the event that $F_3$ holds and for every Borel set $Y \subseteq K$ we have that $\nmeasure{\wh{\eta}_i}(Y) \leq c_0 \diam(Y)^{3/2-a}$ then $\p[F_4] \geq 1-p_0/10$.

Note that
\[ \nmeasure{\wt{\eta}_i}(Y) = \int_{\wt{\phi}_i^{-1}(Y)} |\wt{\phi}_i'(w)|^{3/2} d\nmeasure{\wh{\eta}_i}(w) \quad\text{for all}\quad Y \quad\text{Borel}.\]
Therefore applying~\cite[Corollary~3.25]{lawler2008conformally} we obtain that there exists $M_0 > 1$ so that if $M \geq M_0$ then~\eqref{it:natural_parameterization_diameter_bound_rectangle} holds on $F_4$.  If we let $F_5$ be the event that $F_4$ holds and~\eqref{it:natural_parameterization_lower_bound_rectangle} also holds then by increasing the value of $M_0$ if necessary we have that $\p_H[F_5]\geq 1 - p_0$, which completes the proof.
\end{proof}

\begin{lemma}
\label{lem:all_crossings_good}
Fix $p_0 \in (0,1)$, $a > 0$, $H_0 > 0$, $H \geq H_0$, and $\xi \in (0,H_0/10)$.  Then there exists $M_0 > 1$ depending only on $p_0$, $a$, $H$, $\xi$ such that~\eqref{it:natural_parameterization_lower_bound_rectangle}, \eqref{it:natural_parameterization_diameter_bound_rectangle}, and~\eqref{it:fraction_of_good_points_rectangle} hold for every $M \geq M_0$ with probability at least $1-p_0$ with respect to $\p_H$.
\end{lemma}

Before we proceed to the proof of Lemma~\ref{lem:all_crossings_good}, we first record the following analog of Lemma~\ref{lem:good_fraction_whole_plane} which is stated in the context of chordal (rather than two-sided whole-plane) $\SLE_4$.

\begin{lemma}
\label{lem:sle4_good}
Suppose that $D \subseteq \C$ is a simply connected domain, $x,y \in \partial D$ are distinct, and $\eta$ is an $\SLE_4$ in $D$ from $x$ to $y$.  Let $\sigma,\tau$ be stopping times for $\eta$ such that a.s.\ $0 < \sigma < \tau$.  For each $p_0, \Ff_0 \in  (0,1)$, $a > 0$,  there exists $M > 1$ so that the following is true.  Let $D_\Left$ (resp.\ $D_\Right$) be the component of $D \setminus \eta$ which is to the left (resp.\ right) of $\eta$ and let $\CW^\Left$ (resp.\ $\CW^\Right$) be a Whitney square decomposition of~$D_\Left$ (resp.\ $D_\Right$).  There exist squares $Q_q \in \CW^q$ for $q \in \{\Left,\Right\}$ so that the following is true. For $q \in \{\Left,\Right\}$, let $G_{q,\sigma,\tau}$ be the set of $r \in [\sigma,\tau]$ such that every $Q \in \CW^q$ hit by the hyperbolic geodesic from $\cen(Q_q)$ to $\eta(r)$ in $D_q$ satisfies $\disthyp^{D_q}(\cen(Q_q),\cen(Q)) \leq M \len(Q)^{-a}$.  Then
\[ \p[ \nmeasure{\eta}(\eta(G_{q,\sigma,\tau})) \leq \Ff_0 \nmeasure{\eta}(\eta([\sigma,\tau]))] \leq p_0 \quad\text{for}\quad q \in \{\Left,\Right\}.\]
\end{lemma}

Transferring the two-sided whole-plane $\SLE_4$ result from Lemma~\ref{lem:good_fraction_whole_plane} to the setting of the first crossing of an $\SLE_4$ across an annulus was the focus of Lemma~\ref{lem:first_excursion_good}.  Since we included all of the details of the proof that in case, we will be more brief in the present setting.

\begin{proof}[Proof of Lemma~\ref{lem:sle4_good}]
We fix $\Ff_1 \in (0,1)$ sufficiently close to $1$ (to be chosen).  Consider the two-sided whole-plane $\SLE_4$ process $\eta^w$ of Lemma~\ref{lem:good_fraction_whole_plane} which we assume to have the natural parameterization and with time normalized so that $\eta^w(0) = 0$.   We let~$\C_\Left$ (resp.\ $\C_\Right$) denote the component of $\C \setminus \eta^w$ which is to the left (resp.\ right) of $\eta^w$. Moreover, we let $\CW^{\C_\Left}$ (resp.\ $\CW^{\C_\Right}$) be a Whitney square decomposition of $\C_\Left$ (resp.\ $\C_\Right$). Let $\CW^\Left$ (resp.\ $\CW^\Right$) be a Whitney square decomposition of the component $D_\Left$ (resp.\ $D_\Right$) of $D \setminus \eta$ which is to the left (resp.\ right) of $\eta$.  
Fix $z$ in the counterclockwise arc of $\partial D$ from $x$ to $y$ which is distinct from $x,y$.  Let $\phi: \C \setminus \eta^w((-\infty,0]) \to D$ be such that $\phi(0) = x$, $\phi(\infty) = y$, and the prime end associated with~$\eta^w(-1)$ on the left side of~$\eta^w$ is mapped to~$z$. Then $\eta = \phi(\eta^w(\cdot))|_{[0,\infty)}$ is a chordal $\SLE_4$ in~$D$ from~$x$ to~$y$. 
Thus we can find constants $0<t<T$ and $c_1 > 0$ such that with probability at least $1 - p_0/100$ we have that $t \leq \sigma \leq \tau \leq T$ and $\nmeasure{\eta}(\eta([\sigma,\tau])) \geq c_1$.

Next we claim that there exist constants $\wt{M}>1$ large enough and $\delta \in (0,1)$ small enough such that with probability at least $1-p_0/100$ for $q \in \{\Left,\Right\}$ there exist squares $Q^{q} \in \CW^{\C_q}$ with $\dist(Q^{q},\eta^w) \geq \delta$, $\dist(Q^{q},\eta^w)\geq \delta$ and $Q^{q} \subseteq B(0,1/\delta)$ such that if $G_q$ is the set of $r \in [t,T]$ such that the hyperbolic geodesic in~$\C_q$ from $\cen(Q^q)$ to $\eta^w(r)$ only intersects squares in $\C_q$ which are $(\wt{M},a)$-good with respect to $Q^q$, then $\Leb_1(G_q) \geq \Ff_1 (T-t)$.  Indeed,  first we note that $\eta^w$ is a.s.\  a simple curve and it does not hit fixed points.  Thus there exist $0 < \delta < \xi <1$ small enough such that with probability at least $1-p_0/300$ we have that there exists $z^q \in \xi \Z^2 \cap B(0,1/\xi)$ so that $z^q \in \C_q$ and $\dist(Q^q,\eta^w)\geq \delta,  Q^q \subseteq B(0,1/\delta)$,  where $Q^q$ is the square in $\CW^{\C_q}$ such that $z^q \in Q^q$ for $q \in \{\Left,\Right\}$.  Moreover arguing as in the proof of Lemma~\ref{lem:good_fraction_whole_plane},  we obtain that there exists $\wt{M}>1$ large such that with probability at least $1-p_0/300$ the following holds.  For each $z \in \xi \Z^2 \cap B(0,1/\xi)$ the statement of Lemma~\ref{lem:good_fraction_whole_plane} holds when $1$, $[0,1]$, and $1-\delta$ are replaced by $z$, $[t,T]$, and $\Ff_1$ respectively.  The claim then follows by taking the intersection of the above events.  By possibly taking $\delta$ to be smaller we can also assume that every hyperbolic geodesic in $\C_q$ starting from $\cen(Q^q)$ and ending at some point on $\eta^w([t,T])$ lies in $B(0,1/\delta) \cap \{w \in \C : \dist(w,\eta^w((-\infty,0]))\geq \delta\}$ and $\eta^w([0,T]) \subseteq B(0,1/\xi)$.  Let $F$ be the intersection of the above events and note that $\p[F] \geq 1 - p_0/100$.  Moreover we can find a constant $b \in (0,1)$ such that with probability at least $1-p_0/100$ we have that $b \leq |\phi'(w)| \leq b^{-1}$ for each $w \in B(0,1/\delta)$ such that $\dist(w,\eta^w((-\infty,0])) \geq \delta$.  Suppose that we are working on the event that the following hold.
\begin{enumerate}[(i)]
\item $t \leq \sigma \leq \tau \leq T$,  $\nmeasure{\eta}(\eta([\sigma,\tau])) \geq c_1$,
\item $F$ occurs,
\item $b \leq |\phi'(w)| \leq b^{-1}$ for each $w \in B(0,1/\delta)$ such that $\dist(w,\eta^w((-\infty,0])) \geq \delta$.
\end{enumerate}
Note that the above event has probability at least $1-p_0/25$.  First,  we observe that
\begin{align*}
\nmeasure{\eta}(\phi(\eta^w([t,T] \setminus G_q)))&= \int_{\eta^w([t,T] \setminus G_q)}|\phi'(w)|^{3/2}d\nmeasure{\eta}(w)\\
&\leq \frac{(1-\Ff_1)(T-t)}{c_1 b^{3/2}} \nmeasure{\eta}(\eta([\sigma,\tau])) \leq (1-\Ff_0) \nmeasure{\eta}(\eta([\sigma,\tau]))
\end{align*}
if we choose $\Ff_1$ sufficiently close to $1$.

For $q \in \{\Left,\Right\}$ we set $z_q = \cen(Q^q)$ and let $\wt{Q}^q \in \CW^q$ be the Whitney square such that $\phi(z_q) \in \closure{\wt{Q}^q} \setminus (\pup \wt{Q}^q \cup \pright \wt{Q}^q)$, where $\pup \wt{Q}^q$ and $\pright \wt{Q}^q$ denotes the upper and right boundary of $\wt{Q}^q$, respectively. Suppose that $\wt{Q} \in \CW^q$ is such that $\phi^{-1}(\wt{Q})$ is hit by a hyperbolic geodesic $\gamma_{z_q,\eta^w(r)}^{\C_q}, r \in [\sigma,\tau]$, which only intersects squares in $\CW^{\C_q}$ which are $(\wt{M},a)$-good with respect to $Q^q$.  Let $Q \in \CW^{\C_q}$ be such that $\phi(\cen(Q)) \in \wt{Q}$.  By~\eqref{eq:derivative_diameter_relation} we have that
\begin{align}
\label{eq:good_path_transformation}
	\disthyp^{D_q}(\phi(z_q),\cen(\wt{Q})) &\asymp \disthyp^{D_q}(\phi(z_q), \phi(\cen(Q))) = \disthyp^{\C_q}(z_q,\cen(Q)) \\
	&\leq \wt{M} \len(Q)^{-a} \asymp \wt{M} |\phi'(\cen(Q))|^a \len(\wt{Q})^{-a} \lesssim \wt{M} \len(\wt{Q})^{-a} \nonumber
\end{align} 

Finally,  since $\disthyp^{D_q}(z_q,\cen(\wt{Q}^q)) \leq 1$,  by applying Lemma~\ref{lem:hyperbolic_geodesics_close} as in the proofs of of Lemmas~\ref{lem:lower_bound_F4} and~\ref{lem:change_of_domains},  we obtain that~\eqref{eq:good_path_transformation} still holds if we replace $\phi(z_q)$ with $\cen(\wt{Q}^q)$.  This completes the proof.
\end{proof}

\begin{proof}[Proof of Lemma~\ref{lem:all_crossings_good}]
Assume that we have carried out the explorations in~$\CR^\Up$ and~$\CR^\Down$, let~$\wt{K}$ be their union, let $\wt{U}$ be the connected component of $\CR \setminus \wt{K}$ which contains $\CR^\Inner$, and let~$\wt{\CW}$ be a Whitney square decomposition of $\wt{U}$.  Choose $N$ large so that with probability at least $1-p_0/100$ there are at most~$N$ crossings of either $\CR^\Down$ or $\CR^\Up$.  Let $\wt{\eta}_1,\ldots,\wt{\eta}_n$ be the continuations of the $0$-height crossings of~$\CR^\Down$ into $\CR^\Inner$ which connect with a $0$-height crossing of $\CR^\Up$ and let $\wt{\eta}_{n+1},\ldots,\wt{\eta}_{n+m}$ be the remainder of the $0$-height crossings of $\CR^\Down$ or $\CR^\Up$ which do not cross $\wt{U}$.  For each $1 \leq i \leq n$, we let $\wt{K}_i$ be the union of $\wt{K}$ and $\wt{\eta}_\ell$ for $\ell \neq i$ and we let $\wt{W}_i$ be the component of $\CR \setminus \wt{K}_i$ which contains $\wt{\eta}_i$.  Then the conditional law of $\wt{\eta}_i$ given everything else is absolutely continuous with respect to that of a chordal $\SLE_4$ in $\wt{W}_i$.  Let $\wt{U}_1,\ldots,\wt{U}_{n+1}$ be the complementary components of $\wt{U} \setminus \cup_{i=1}^{n+m} \wt{\eta}_i$ which have at least one of the $\wt{\eta}_i$ for $1 \leq i \leq n$ on their boundary, ordered from left to right.  Let $\wt{\CW}_i$ be a Whitney square decomposition of $\wt{U}_i$.

We assume that we have the setup of Lemma~\ref{lem:natural_parameterization_bounds} and for each $1 \leq i \leq n$ we let $\wt{I}_i$ be the segment as constructed in Lemma~\ref{lem:natural_parameterization_bounds}.  Lemma~\ref{lem:sle4_good} implies that for each $a > 0$ there exists $M > 1$ large so that with probability at least $1-p_0/100$,  for each $1 \leq i \leq n$ we can find squares $\wt{Q}_i^\Left \in \wt{\CW}_i$ and $\wt{Q}_i^\Right \in \wt{\CW}_{i+1}$ such that~\eqref{it:fraction_of_good_points_rectangle} holds for the segment $\wt{\eta}_i(\wt{I}_i)$ and the domains $\wt{U}_i$ and $\wt{U}_{i+1}$ respectively.  Let $\wt{G}_{\Left,i}$ be the set of $r \in \wt{I}_i$ such that every square in $Q \in \wt{\CW}_i$ which is hit by a hyperbolic geodesic from $\cen(\wt{Q}_i^\Left)$ to $\wt{\eta}_i(r)$ is $(M,a)$-good.  Similarly, let $\wt{G}_{\Right,i}$ be the set of $r \in \wt{I}_i$ such that every square in $Q \in \wt{\CW}_{i+1}$ which is hit by a hyperbolic geodesic from $\cen(\wt{Q}_i^\Right)$ to $\wt{\eta}_i(r)$ is $(M,a)$-good.  Then 
\[ \frac{\nmeasure{\wt{\eta}_i}(\wt{\eta}_i(\wt{G}_{\Left,i}))}{\nmeasure{\wt{\eta}_i}(\wt{\eta}_i(\wt{I}_i))} \geq \frac{3}{4} \quad\text{and} \quad \frac{\nmeasure{\wt{\eta}_i}(\wt{\eta}_i(\wt{G}_{\Right,i}))}{\nmeasure{\wt{\eta}_i}(\wt{\eta}_i(\wt{I}_i))} \geq \frac{3}{4}.\]
Moreover we can choose $M>1$ large enough such that with probability at least $1-p_0/100$ both~\eqref{it:natural_parameterization_lower_bound_rectangle} and~\eqref{it:natural_parameterization_diameter_bound_rectangle} hold.  Also,  we can find $C>0$ large such that with probability at least $1-p_0/100$ we have that $\disthyp^{\wt{U}_{i+1}}(\cen(\wt{Q}^{\Right}_i),\cen(\wt{Q}^{\Left}_{i+1})) \leq C$ for each $1 \leq i \leq n-1$.  Fix $1 \leq i \leq n-1$ and $r \in \wt{G}_{\Right,i}$.  Let $Q \in \wt{\CW}_{i+1}$ be such that $\gamma_{\wt{y}_i,\wt{\eta}_i(r)}^{\wt{U}_{i+1}} \cap Q \neq \emptyset$ where $\wt{y}_i = \cen(\wt{Q}^{\Left}_{i+1})$.  Then by applying Lemma~\ref{lem:hyperbolic_geodesics_close} as in the proofs of of Lemmas~\ref{lem:lower_bound_F4} and~\ref{lem:change_of_domains}, we obtain that $\disthyp^{\wt{U}_{i+1}}(\wt{y}_i,\cen(Q)) \lesssim M \len(Q)^{-a}$.  This completes the proof of the lemma by considering the intersection of all of the above events.
\end{proof}

Finally, we shall see that we can choose the parameters so that with high probability,~\eqref{it:not_too_many_squares_rectangle} holds (in addition to~\eqref{it:natural_parameterization_lower_bound_rectangle}, \eqref{it:natural_parameterization_diameter_bound_rectangle}, and~\eqref{it:fraction_of_good_points_rectangle}).

\begin{lemma}
\label{lem:not_too_many_squares}
Fix $p_0 \in (0,1)$, $a, H_0 > 0$, $H \geq H_0$, and $\xi \in (0,H_0/10)$.  There exists $M_0 > 0$ depending only on $p_0$, $a$, $H_0$, and $\xi$ such that with probability at least $1-p_0$ with respect to $\p_H$ we have that~\eqref{it:natural_parameterization_lower_bound_rectangle}, \eqref{it:natural_parameterization_diameter_bound_rectangle}, and~\eqref{it:natural_measure_cube_conditions} hold.
\end{lemma}
\begin{proof}
Suppose that we have the setup and notation from the proof of Lemma~\ref{lem:natural_parameterization_bounds}.  We let $\wt{\sigma}_i^0 = \sup\{t \geq 0 : \wt{\eta}_i(t) \in B(\wt{x}_i,\xi/2)\}$ and $\wt{\tau}_i^0 = \inf\{t \geq 0 : \wt{\eta}_i(t) \in B(\wt{y}_i,\xi/2)\}$ and set $\wt{I}^0_i = [\wt{\sigma}^0_i,\wt{\tau}^0_i]$.  Then~$\wt{I}_i^0$ contains $\wt{I}_i$.  For each $\ell \in \N$ we let~$\wt{\CQ}_{i,\ell}$ be the set of squares $\wt{Q}  \in \wt{\CW}_i$ such that $\len(\wt{Q}) = 2^{-\ell}$ and $\dist(\wt{Q},\partial \wt{U}_i) = \dist(\wt{Q},\wt{\eta}_i(\wt{I}^0_i))$.  

Note that Lemmas~\ref{lem:crossings_separated} and~\ref{lem:all_crossings_good} imply that we can choose $\wt{M}_0 > 1$,  $N \in \N$, and $K \subseteq \h$ compact depending only on $p_0$, $H$, and $\xi$ such that with probability at least $1-p_0/100$ we have that $n\leq N$,  $\wt{\phi}^{-1}_i(\wt{\eta}_i(\wt{I}^0_i)) \subseteq K$,  $\dist(\wt{\eta}_i(\wt{I}^0_i)$, $\partial \wt{W}_i) \geq 1/N$ for each $1 \leq i \leq n$,  and conditions~\eqref{it:natural_parameterization_lower_bound_rectangle},~\eqref{it:natural_parameterization_diameter_bound_rectangle} and~\eqref{it:fraction_of_good_points_rectangle} hold for the segments $\wt{\eta}_i(\wt{I}_i)$.  Let $F_1$ be that event.

Fix $\wt{Q} \in \wt{\CQ}_{i,\ell}$ and let $w \in \wt{\eta}_i(\wt{I}^0_i)$ be such that $\dist(\wt{Q},\partial \wt{U}_i) = \dist(\wt{Q},w) \lesssim 2^{-\ell}$.  Note that we can find constants $c_1,c_2>0$ and $\ell_0 \in \N$ depending only on the above implicit constants such that $\wt{\phi}^{-1}_i(\wt{Q}) \subseteq B(\wt{\phi}^{-1}_i(w),c_1 2^{-\ell})$ and $B(\wt{\phi}^{-1}_i(\cen(\wt{Q})),c_2 2^{-\ell}) \subseteq \wt{\phi}^{-1}_i(\wt{Q})$ for each $\ell \geq \ell_0$.  Thus by decreasing the value of $c_2$ if necessary,  we obtain that there exists $z \in c_2 2^{-\ell} \Z^2$ such that $z \in \wt{\phi}^{-1}_i(\wt{Q})$ and $\dist(z,\wh{\eta}_i) \leq c_1 2^{-\ell}$.  Combining we obtain that $|\wt{\CQ}_{i,\ell}| \leq N_{i,\ell}$ for each $\ell \geq \ell_0$,  where $N_{i,\ell}$ is the number of points $z \in c_2 2^{-\ell}\Z^2 \cap \h \cap B(0,2R)$ such that $\wh{\eta}_i \cap B(z,c_1 2^{-\ell}) \neq \emptyset$ and $R>0$ is chosen such that $K \subseteq B(0,R)$.  Let $\gamma$ be an $\SLE_4$ in $\h$ from $0$ to $\infty$.  Let $N_m$ the number of points $z \in  c_2 2^{-m} \Z^2 \cap \h \cap B(0,2R)$ such that $\gamma \cap B(z, c_1 2^{-m}) \neq \emptyset$.  Then \cite[Proposition~4]{beffara2008dimension} implies that $\E[N_m] \lesssim 2^{3 m/2}$ and so $\p[N_m \geq 2^{(3/2+a)m}] \lesssim 2^{-a m}$ where the implicit constant depends only on $K$, $c_1$ and $c_2$.  This implies that $\p[N_\ell \leq 2^{(3/2+a)\ell}\ \forall \ell \geq m] \to 1$ as $m \to \infty$.   Note that the proof of Lemma~\ref{lem:natural_parameterization_bounds} implies the $\p_H$-probability that the Radon-Nikodym derivative of the law of $\wh{\eta}_i$ with respect to that of $\gamma$ is at most $N$ tends to $1$ as $N \to \infty$ at rate which depends only on $H_0$ and $\xi$.  Thus we have that $\p_H[N_{i,\ell} \leq 2^{(3/2+a)\ell}\ \forall \ell \geq m] \to 1$ as $m \to \infty$.  Note also that $\max_{1 \leq \ell \leq \ell_0,1 \leq i \leq n} |\wt{\CQ}_{i,\ell}|$ is an a.s.\  finite random variable and so we can choose $M_0 > 1$ large such that with probability at least $1-p_0/100$ under $\p_H$ we have that $|\wt{\CQ}_{i,\ell}| \leq M_0 2^{(3/2 + a)\ell}$ for each $1\leq i \leq n$ and each $\ell \in \N$.  Let $F_2$ be that event and note that $\p_H[F_1 \cap F_2] \geq 1-p_0/50$.  Suppose that we are working on $F_1 \cap F_2$.

Now for each $1 \leq i \leq n$,  we let $\psi_i$ be the unique conformal transformation mapping $\D$ onto $\wt{U}_i$ such that $\psi_i(0) = \cen(\wt{Q}_i)$ and $\psi'_i(0) > 0$.  Furthermore, we let $\wt{\theta}_{i,0} \in [0,2\pi)$ and $\wt{\theta}_{i,1} \in (\wt{\theta}_{i,0},\wt{\theta}_{i,0}+2\pi)$ be such that let $\psi^{-1}_i(\wt{\eta}_i(\wt{I}_i)) = \{e^{i\theta}: \theta \in [\wt{\theta}_{i,0},\wt{\theta}_{i,1}] \}$ and $S_i = \psi_i(\{re^{i\theta} : r \in [0,1],\  \theta \in [\wt{\theta}_{i,0},\wt{\theta}_{i,1}]\})$.  Since $\min_{1 \leq i \leq n} \dist(S_i ,  \partial \wt{U}_i \setminus \wt{\eta}_i(\wt{I}^0_i)) > 0$ a.s.,   it follows that there exists $\epsilon > 0$ depending only on the above constants such that with probability at least $1-p_0/100$ under $\p_H$ we have that $\dist(S_i ,  \partial \wt{U}_i \setminus \wt{\eta}_i(\wt{I}^0_i)) \geq \epsilon$ for each $1 \leq i \leq n$.  Let~$F_3$ be that event and suppose that we are working on $F_1 \cap F_2 \cap F_3$.  Let $\wt{Q} \in \wt{\CW}_i$ be such that $\dist(\wt{Q},\partial \wt{U}_i) < \dist(\wt{Q},  \wt{\eta}_i(\wt{I}^0_i))$ and $\gamma_{\cen(\wt{Q}_i),w}^{\wt{U}_i} \cap \wt{Q} \neq \emptyset$ for some $w \in \wt{\eta}_i(\wt{I}_i)$.  Then we have that $\diam(\wt{Q}) \geq \epsilon/100$.  Note that the number of $\wt{Q} \in \wt{\CW}_i$ such that $\diam(\wt{Q}) \geq \epsilon/100$ is at most $O(\epsilon^{-2})$ where the implicit constant depends only on $H_0$ and $\xi$.  Therefore,  possibly by increasing the value of $M_0$,  we obtain that~\eqref{it:not_too_many_squares_rectangle} holds if $F_1\cap F_2 \cap F_3$ occurs,  which completes the proof.
\end{proof}

\subsection{Conclusion of the proof}
\label{subsec:proof_conclusion}

Now we fix $p_0 \in (0,1)$, $\delta_0,H_0 >0$,  $a \in (0,1/2)$, and let $b \in (1,2)$, $\xi_1 \in (0,(e^{-b}-e^{-2})/2)$, and $\xi_0 \in (0,H_0/2)$ be as in~\eqref{it:middle_circle_good} of Lemma~\ref{lem:band_mapping}.  We then let $\delta_1,  \delta_2 \in (0,1)$,  $M_0>1$ be as in Lemma~\ref{lem:first_excursion_good} for the above choices of $p_0,\delta_0,a,b$ and $\xi_2 = 3\xi_1 / 2$.  We also let $\wt{\xi}_0 \in (0,H_0/2)$ be as in~\eqref{it:away_from_the_boundary_good} of Lemma~\ref{lem:band_mapping} for $\wt{\xi}_1 = \delta_2$.  We fix $\xi \in (0,H_0/10)$ such that $\xi < \xi_0 / 3$ and perform the explorations in $\CR^{\Up}$ and  $\CR^{\Down}$ and define $(\kay_j)_{j \geq 1}$ as in Section~\ref{sec:annulus_exploration}.  
\begin{lemma}
\label{lem:conditional_probability_good}
Suppose that we have the above setup.  Then for all $\delta_1 \in (0,1)$ sufficiently small there exists $M>1$ sufficiently large and $\delta_2 \in (0,1)$ sufficiently small such that
\[ \p[ (E_{z,\kay_{j+1}}^{M,a})^c \giv \CF_{z,\kay_j}] \one_{\kay_j < \infty} \leq p_0 \one_{\kay_j < \infty} \quad\text{a.s.}\]
\end{lemma}
\begin{proof}
We are going to show that the conditions of Definition~\ref{def:good_event} are likely to hold by using that the conditions in the rectangle are likely to hold and the definition of $\kay_{j+1}$ implies that on $\kay_{j+1} < \infty$ the first crossing of $A_{z,\kay_{j+1}}$ by $\eta$ is well-behaved.  Let $\phi_{z,\kay_{j+1}} = \varphi_{z,\kay_{j+1}}^{-1}$ so that $\phi_{z,\kay_{j+1}}$ conformally maps $\CR_{z,\kay_{j+1}}$ to $A_{z,\kay_{j+1}}^*$ with $\pleft \CR_{z,\kay_{j+1}}$ taken to $\eta^\Right([\sigma_{z,\kay_{j+1}},\tau_{z,\kay_{j+1}}])$ and $\pright \CR_{z,\kay_{j+1}}$ taken to $\eta^\Left([\sigma_{z,\kay_{j+1}},\tau_{z,\kay_{j+1}}])$.

\noindent{\it Step 1. First crossing.}  
First,  by the definition of $\kay_{j+1}$ on $\kay_{j+1} < \infty$ we have that the events $F_{z,\kay_{j+1}}^1$ (Lemma~\ref{lem:band_mapping}), $F_{z,\kay_{j+1}}^2$ (Lemma~\ref{lem:first_crossing_number_of_squares}), and $C_{z,\kay_{j+1}}$ (Lemma~\ref{lem:first_excursion_good}) occur.  This implies that the following is true.  

There exists an interval $I_0 \subseteq [\sigma_{z,\kay_{j+1}},\tau_{z,\kay_{j+1}}]$ and points $z_\Left,z_\Right \in A^*_{z,\kay_{j+1}}$ such that $\dist(z_\Left,\eta^{\Left}(I_0)) \leq \delta_1 \diam(A_{z,\kay_{j+1}})$,  $\dist(z_\Right,\eta^{\Right}(I_0)) \leq \delta_1 \diam(A_{z,\kay_{j+1}})$,  $\dist(\{z_\Left,z_\Right\},\partial A^*_{z,\kay_{j+1}}) \geq \delta_2 \diam(A_{z,\kay_{j+1}})$ and such that for every simply connected domain $W \subseteq A^*_{z,\kay_{j+1}}$ containing the component of $B(y,\delta_1^{1/4}\diam(A_{z,\kay_{j+1}})) \setminus \eta([0,\tau_{z,\kay_{j+1}}])$ which contains $z_\Right$ the following holds (and the same with $z_\Left$ in place of $z_\Right$).  Let $\CW$ be a Whitney square decomposition of $W$.  Then the hyperbolic geodesic in $W$ from $z_\Right$ (resp.\  $z_\Left$) to $w \in \eta^{\Right}(I_0)$ (resp.\  $w \in \eta^{\Left}(I_0)$) passes only through squares which are $(M,a)$-good with respect to $z_\Right$ (resp.\  $z_\Left$) for at least $\frac{3}{4}$-fraction of points with respect to $\nmeasure{\eta}(dw)/\nmeasure{\eta}(I_0)$.  Moreover conditions~\eqref{it:natural_parameterization_lower_bound} and~\eqref{it:natural_parameterization_diameter_bound} hold

\noindent{\it Step 2a. Condition~\eqref{it:natural_parameterization_lower_bound}.} We note that since $\eta_i = \phi_{z,\kay_{j+1}}(\wt{\eta}_i)$, we have that
\begin{align}\label{eq:conformal_covariance_natural_parameterization}
	\nmeasure{\eta_i}(Y) = \int_{\varphi_{z,\kay_{j+1}}(Y)} |\phi_{z,\kay_{j+1}}'(w)|^{3/2} d\nmeasure{\wt{\eta}_i}(w) \quad\text{for all}\quad Y \quad\text{Borel}.
\end{align}
Moreover, we note that by choosing $\epsilon > 0$ small enough and $M>1$ large enough, we have that with probability at least $1-p_0/100$, all of the crossings $\wt{\eta}_i$ are separated from each other and from $\partial^\Left \CR_{z,\kay_{j+1}}$ and $\partial^\Right \CR_{z,\kay_{j+1}}$ by at least $\epsilon$ (Lemmas~\ref{lem:rn_formula} and~\ref{lem:crossings_separated}) and $\nmeasure{\wt{\eta}_i}(\wt{\eta}_i(\wt{I}_i)) \geq M^{-1}$ for all $i$ (Lemmas~\ref{lem:rn_formula} and~\ref{lem:natural_parameterization_bounds}).  By the definition of $\kay_{j+1}$, we have that $\inf_{\sigma_{z,\kay_{j+1}} \leq t \leq \tau_{z,\kay_{j+1}}} S_t \geq \delta_0$.  Assume that we are on this event.

Let $\CR = (0,1) \times (0,H_0) \subseteq \CR_{z,\kay_{j+1}}$.  For each $\epsilon > 0$, we let $\CR^\epsilon$ be the set of points in $\CR$ at distance at least $\epsilon/2$ from $\partial \CR$. Then a Brownian motion starting from $w \in \CR^\epsilon$ has probability bounded from below (depending only on $H_0$, $\epsilon$) of exiting $\CR$ in any of the four boundary arcs.  Set $\wt{w} = \psi_{z,\kay_{j+1}}(\phi_{z,\kay_{j+1}}(w)) \in B(0,e^{-1}) \setminus \closure{B(0,e^{-2})}$ (recall Section~\ref{subsec:first_crossing_setup}).  We claim that the probability that a Brownian motion starting from $\wt{w}$ exits $B(0,e^{-1}) \setminus \closure{B(0,e^{-2})}$ in $\partial B(0,e^{-1})$ (resp.\  $\partial B(0,e^{-2})$) is bounded from below by a positive constant which depends only on $H_0$, $\epsilon$, and $b$.  Indeed,  first we note that condition~\eqref{it:middle_circle_good} of Lemma~\ref{lem:band_mapping} implies that $(\varphi_{z,\kay_{j+1}} \circ \psi_{z,\kay_{j+1}}^{-1})(\partial B(0,e^{-b}-\xi_1)) \subseteq \CR$ is a crossing from $\pleft \CR$ to $\pright \CR$ and so the probability that a Brownian motion starting from $\wt{w}$ hits $\partial B(0,e^{-b} -\xi_1)$ before exiting $B(0,e^{-1}) \setminus \closure{B(0,e^{-2})}$ is bounded from below by a positive constant depending only on $H_0$ and $\epsilon$.  The claim then follows by combining with the Markov property and the fact that for any point on $\partial B(0,e^{-b} - \xi_1)$ the probability that a Brownian motion starting from that point exits $B(0,e^{-1}) \setminus \closure{B(0,e^{-2})}$ on $\partial B(0,e^{-1})$ (resp.\ $\partial B(0,e^{-2})$) is bounded from below by a positive constant depending only on $b$ and $\xi_1$ which is uniform in the location of the starting point.  Thus we can find $q \in (0,(e^{-1}-e^{-2})/2)$ depending only on $H_0$, $\epsilon$, $b$ and $\xi_1$ such that $\wt{w} \in B(0,e^{-1}-q) \setminus \closure{B(0,e^{-2}+q)}$.  Moreover we have that $|(\psi_{z,\kay_{j+1}}^{-1})'(\wt{w})| \asymp \diam(A_{z,\kay_{j+1}})$ where the implicit constants are universal.  Also, the Koebe-$\tfrac{1}{4}$ theorem combined with the Beurling's estimate imply that $|(\psi_{z,\kay_{j+1}}^{-1})'(\wt{w})| \asymp \dist(\phi_{z,\kay_{j+1}}(w),  \partial A^*_{z,\kay_{j+1}})$ and so $\dist(\phi_{z,\kay_{j+1}}(w) ,  \partial A^*_{z,\kay_{j+1}}) \asymp \diam(A_{z,\kay_{j+1}})$.  It follows that
\begin{equation}
\label{eqn:phi_deriv_bounds}
|\phi_{z,\kay_{j+1}}'(w)| \asymp \diam(A_{z,\kay_{j+1}}) \quad\text{for all}\quad w \in \CR^\epsilon
\end{equation}
where the implicit constants depend only on $H_0$,  $\epsilon$, $b$ and $\xi_1$. Consequently, since $\wt{\eta}_i \subseteq \CR^\epsilon$, it follows from~\eqref{eq:conformal_covariance_natural_parameterization}, \eqref{eqn:phi_deriv_bounds} that $\nmeasure{\eta_i}(\eta_i(I_i)) \gtrsim M^{-1} \diam(A_{z,\kay_{j+1}})^{3/2}$ where the implicit constant depends only on $H_0$,   $\epsilon$, $b$ and $\xi_1$  and $\eta_i(I_i) = \phi_{z,\kay_{j+1}}(\wt{\eta}_i(\wt{I}_i))$.

\noindent{\it  Step 2b. Condition~\eqref{it:natural_parameterization_diameter_bound}.}  By Lemmas~\ref{lem:rn_formula},  and~\ref{lem:natural_parameterization_bounds} we can assume that in addition to the conditions of Step~2a we have that~\eqref{it:natural_parameterization_diameter_bound_rectangle} also holds with probability at least $1-p_0/100$.  Suppose that $Y \subseteq A_{z,\kay_{j+1}}$ is a Borel set such that $\varphi_{z,\kay_{j+1}}(Y) \subseteq \CR^{\epsilon}$.  Note by~\eqref{eqn:phi_deriv_bounds} that $\diam(\varphi_{z,\kay_{j+1}}(Y)) \asymp \diam(Y)/\diam(A_{z,\kay_{j+1}})$ where the implicit constants depend only on $H_0$,  $\epsilon$, $b$ and $\xi_1$.  By~\eqref{eq:conformal_covariance_natural_parameterization}, \eqref{eqn:phi_deriv_bounds}, it follows that 
\begin{align*}
	\nmeasure{\eta_i}(\eta_i(I_i) \cap Y) &= \int_{\wt{\eta}_i(\wt{I}_i) \cap \varphi_{z,\kay_{j+1}}(Y)} |\phi_{z,\kay_{j+1}}'(w)|^{3/2} d\nmeasure{\wt{\eta}_i}(w) \\
	&\lesssim \diam(A_{z,\kay_{j+1}})^{3/2} \nmeasure{\wt{\eta}_i}(\wt{\eta}_i(\wt{I}_i) \cap \varphi_{z,\kay_{j+1}}(Y)) \\
	&\lesssim \diam(A_{z,\kay_{j+1}})^{3/2-a} M \diam (\varphi_{z,\kay_{j+1}}(Y))^{3/2-a} \lesssim M \diam(Y)^{3/2-a}
\end{align*}
where the implicit constants depend only on $H_0$,  $\epsilon$, $b$ and $\xi_1$.

\noindent{\it Step 2c. Condition~\eqref{it:fraction_of_good_points}.}  We will first handle the squares in the domains $U_2,\dots,U_{n}$, that is, the domains which do not have $\eta([\sigma_{z,\kay_{j+1}},\tau_{z,\kay_{j+1}}])$ as part of their boundaries.  Combining Lemmas~\ref{lem:rn_formula} and~\ref{lem:not_too_many_squares},  we obtain that with probability at least $1-p_0/100$ condition~\eqref{it:fraction_of_good_points_rectangle} holds for $\wt{U}_i$ for each $2 \leq i \leq n$ for some large constant $\wt{M} > 1$ and with $3/4$ replaced by any fixed $\Ff \in (0,1)$ which depends on the aforementioned constants except for $\wt{M}$.  Let $G_i^- = \phi_{z,\kay_{j+1}}(\wt{G}_i^-)$ and $G_i^+ = \phi_{z,\kay_{j+1}}(\wt{G}_i^+)$.  If we choose $\Ff \in (0,1)$ sufficiently close to $1$ we have that
\begin{align*}
\frac{\nmeasure{\eta_i}(\eta_i(I_i) \setminus G_i^-)}{\nmeasure{\eta_i}(\eta_i(I_i))} = \frac{\int_{\wt{\eta}_i(\wt{I}_i) \setminus \wt{G}_i^-}|\phi_{z,\kay_{j+1}}'(w)|^{3/2}d\nmeasure{\wt{\eta}_i}(w)}{\int_{\wt{\eta}_i(\wt{I}_i)}|\phi_{z,\kay_{j+1}}'(w)|^{3/2}d\nmeasure{\wt{\eta}_i}(w)}
\lesssim \frac{\nmeasure{\wt{\eta}_i}(\wt{\eta}_i(\wt{I}_i)\setminus \wt{G}_i^-)}{\nmeasure{\wt{\eta}_i}(\wt{\eta}_i(\wt{I}_i))} \lesssim 1- \Ff < \frac{1}{4}
\end{align*}
and similarly for $G_i^+$.  We claim that condition~\eqref{it:fraction_of_good_points} holds for $U_i$ for each $2\leq i \leq n$,  by possibly taking $\wt{M}$ to be larger.  Indeed,  we fix $w \in G_i^-$ and set $\wt{z}_i = \cen(\wt{Q}_i)$.  Let $Q \in \CW_i$ be such that $\gamma_{\wh{z}_i,w}^{U_i}(t) \in Q$ for some $t \geq 0$,  where $\wh{z}_i = \phi_{z,\kay_{j+1}}(\wt{z}_i)$.  Let also $Q_i \in \CW_i$ be such that $\wh{z}_i \in \closure{Q_i} \setminus (\pright Q_i \cup \pup Q_i)$ and set $z_i = \cen(Q_i)$.  We pick $\wt{Q} \in \wt{\CW}_i$ such that $\gamma^{\wt{U}_i}_{\wt{z}_i,\varphi_{z,\kay_{j+1}}(w)}(t) \in \wt{Q}$.  Since $\varphi_{z,\kay_{j+1}}(w) \in \wt{G}_i^-$,  we have that $\disthyp^{U_i}(\wh{z}_i,\cen(Q)) \lesssim \disthyp^{\wt{U}_i}(\wt{z}_i,\cen(\wt{Q})) \lesssim \wt{M} \len(\wt{Q})^{-a}$.  Note that by~\eqref{eq:derivative_diameter_relation},
\begin{align*}
	\frac{\len(Q)}{\len(\wt{Q})} \asymp |\phi_{z,\kay_{j+1}}'(\cen(\wt{Q}))| \asymp \diam(A_{z,\kay_{j+1}})
\end{align*}
since $\wt{U}_i \subseteq \CR^{\epsilon}$ and so 
\begin{align}
\label{eq:intermediate_domains_good}
	\disthyp^{U_i}(\wh{z}_i,\cen(Q)) \lesssim \wt{M} (\len(Q)/\diam(A_{z,\kay_{j+1}}))^{-a}.
\end{align}
By applying Lemma~\ref{lem:hyperbolic_geodesics_close} as in the proofs of Lemmas~\ref{lem:lower_bound_F4} and~\ref{lem:change_of_domains},  we have that~\eqref{eq:intermediate_domains_good} still holds if we replace $\wh{z}_i$ with $z_i$.  This proves the claim.

We now turn to the case of the domain $U_1$, which has part of $\eta^\Right([\sigma_{z,\kay_{j+1}},\tau_{z,\kay_{j+1}}])$ as part of its boundary.  Let $\CW_1$ be a Whitney square decomposition of $U_1$.  We claim that there exists $Q_1 \in \CW_1$ such that the hyperbolic geodesics in $U_1$ from $z_1 = \cen(Q_1)$ satisfy~\eqref{eq:intermediate_domains_good} where $\wh{z}_i$ is replaced by $z_1$ and $w \in G_1^-$ for $\delta_1$ sufficiently small and $M>1$ sufficiently large.  Here $G_1^-$ is the set of $w \in \eta^{\Right}(I_0)$ such that $\gamma_{z_1,w}^{U_1}$ intersects only $(M,a)$-good squares in $\CW_1$.  Note that condition~\eqref{it:middle_circle_good} of Lemma~\ref{lem:band_mapping} implies that $\varphi_{z,\kay_{j+1}}(z_\Right) \in (0,1) \times (\xi_0,H_0-\xi_0)$.  Hence combining with the Beurling estimate we obtain that $\varphi_{z,\kay_{j+1}}(z_\Right) \in U_1$ for $\delta_1 > 0$ sufficiently small.  A similar argument implies that the image under $\varphi_{z,\kay_{j+1}}$ of the component of $B(z_\Right,\delta_1^{1/4} \diam(A_{z,\kay_{j+1}}))\setminus \eta([0,\tau_{z,\kay_{j+1}}])$ containing $z_\Right$ is a subset of $\wt{U}_1$,  which implies that the component of $B(z_\Right,\delta_1^{1/4}\diam(A_{z,\kay_{j+1}})) \setminus \eta([0,\tau_{z,\kay_{j+1}}])$ containing $z_\Right$ is a subset of $U_1$.  It follows that $\nmeasure{\eta}(G^{\Right})/\nmeasure{\eta}(I_0) \geq \frac{3}{4}$,  where $G^{\Right}$ is the set of points $w \in \eta^{\Right}(I_0)$ such that $\gamma_{z_\Right,w}^{U_1}$ intersects only $(M,a)$-good squares in $\CW_1$.  Note that Lemma~\ref{lem:band_mapping} implies that $\varphi_{z,\kay_{j+1}}(z_\Right) \in (\wt{\xi}_0,1-\wt{\xi}_0) \times (\wt{\xi}_0,H_0-\wt{\xi}_0)$ and so we can find $\zeta>0$ small enough depending only on the above implicit constants such that $\dist(\varphi_{z,\kay_{j+1}}(z_\Right),\partial \wt{U}_1) \geq \zeta$.  By possibly taking $\zeta$ to be smaller and applying Lemma~\ref{lem:rn_formula},  we can find a constant $C<\infty$ such that with probability at least $1-p_0/100$,  we have that $\dist(\wt{Q}_i,\partial \wt{U}_i) \geq \zeta$ and $\disthyp^{\wt{U}_i}(x_1,x_2) \leq C$ for each $x_1,x_2 \in \wt{U}_i$ such that $\dist(\{x_1,x_2\},\partial \wt{U}_i) \geq \zeta$,  for each $1\leq i \leq n$.  Suppose that we are working on that event.  Then we have that $\disthyp^{U_1}(z_\Right,z_1) = \disthyp^{\wt{U}_1}(\varphi_{z,\kay_{j+1}}(z_\Right),\varphi_{z,\kay_{j+1}}(z_1)) \leq C$ and so applying Lemma~\ref{lem:hyperbolic_geodesics_close} as in the proofs of Lemmas~\ref{lem:lower_bound_F4} and~\ref{lem:change_of_domains},  we have that $\gamma_{z_1,w}^{U_1}$ intersects only $(M,a)$-good squares in $\CW_1$ for each $w \in G^{\Right}$ (possibly by taking $M$ to larger).  A similar argument works for the domain $U_{n+1}$.

\noindent{\it Step 2d. Condition~\eqref{it:not_too_many_squares}.} 
First,  we note that if we consider $M>1$ large enough,  then with probability at least $1-p_0/100$ condition~\eqref{it:not_too_many_squares_rectangle} is satisfied for $\wt{G}_i^+$ (resp.\  $\wt{G}_i^-$) for each $1 \leq i \leq n$ (resp.\  $2\leq i \leq n+1$).  Fix $1 \leq i \leq n$ and let $Q \in \CW_i$ be such that $\len(Q) = 2^{-m}$,  $m \in \Z$,  and $\gamma_{\wh{z}_i,w}^{U_i} \cap Q \neq \emptyset$ for some $w \in G_i^+$,  where $\wh{z}_i$ is as in Step 2c.  Let also $\wt{Q} \in \wt{\CW}_i$ be such that $\wt{Q} \cap \varphi_{z,\kay_{j+1}}(Q) \neq \emptyset$.  Then~\eqref{eq:derivative_diameter_relation} and~\eqref{eqn:phi_deriv_bounds} imply that there exists $C_1 > 1$ depending only on $\epsilon$,  $H_0$ and $b$ such that $C_1^{-1}2^{-m} \diam(A_{z,\kay_{j+1}})^{-1} \leq \len(\wt{Q}) \leq C_1 2^{-m} \diam(A_{z,\kay_{j+1}})^{-1}$.  Hence combining with Lemma~\ref{lem:bound_on_squares},  we obtain that the number of $Q \in \CW_i$ such that $\len(Q) = 2^{-m}$ and $\gamma_{\wh{z}_i,w}^{U_i} \cap Q \neq \emptyset$ for some $w \in G_i^+$ is at most a constant times $M 2^{(3/2+a)m}\diam(A_{z,\kay_{j+1}})^{3/2+a}$.  We claim that the same holds if we replace $\wh{z}_i$ with $z_i$ and possibly by taking $M$ to be larger.  Indeed, since $\disthyp^{U_i}(z_i,\wt{z}_i) \leq 1$, Lemma~\ref{lem:hyperbolic_geodesics_close} implies that there exists a universal constant $C<\infty$ such that $\disthyp^{U_i}(\gamma_{z_i,w}^{U_i}(t),\gamma_{\wh{z}_i,w}^{U_i}(t)) \leq C$ for each $t \geq 0$.  This implies that there exists a universal constant $\wt{C}>1$ such that for each $Q \in \CW_i$ with $\len(Q) = 2^{-m}$ and $\gamma_{z_i,w}^{U_i} \cap Q \neq \emptyset$ for some $w \in G_i^+$,  there exists $\wh{Q} \in \CW_i$ with $\wt{C}^{-1}2^{-m} \leq \len(Q) \leq \wt{C} 2^{-m} $,  $\gamma_{\wh{z}_i,w}^{U_i} \cap \wt{Q} \neq \emptyset$ and $\Bhyp^{U_i}(\cen(\wh{Q}),\wt{C}) \cap Q \neq \emptyset$ for some $w \in G_i^+$.  The claim then follows by combining with Lemma~\ref{lem:hyperbolic_ball_covered}.  A similar argument holds for $G_i^-$ and $2\leq i \leq n+1$.

It remains to show that condition~\eqref{it:not_too_many_squares} holds for $G_{n+1}^+$ and $G_1^-$.  We will show that it holds for $G_1^-$.  A similar argument works for $G_{n+1}^+$.  Note that $\eta(I_0) \subseteq B(z_\Right,\delta_1^{1/2} \diam(A_{z,\kay_{j+1}}))$ and the component of  $B(z_\Right,\delta_1^{1/4} \diam(A_{z,\kay_{j+1}})) \setminus \eta([0,\tau_{z,\kay_{j+1}}])$ containing $z_\Right$ is a subset of $U_1$.  Hence we have that $\gamma_{z_\Right,w}^{U_1} \subseteq B(z_\Right,\delta_1^{1/3} \diam(A_{z,\kay_{j+1}}))$ for each $w \in \eta(I_0)$,  if $\delta_1$ is sufficiently small.  This implies that $\dist(Q,\partial U_1) = \dist(Q ,  \eta([\sigma_{z,\kay_{j+1}},\tau_{z,\kay_{j+1}}]))$ and so $Q$ is contained in the $2^{-m+3}$-neighborhood of $\eta([\sigma_{z,\kay_{j+1}},\tau_{z,\kay_{j+1}}])$ for every square $Q \in \CW_1$ such that $\len(Q) = 2^{-m}$ and $\gamma_{y,w}^{U_1} \cap Q \neq \emptyset$ for some $w \in G_1^-$.  Let $N_m$ be the number of such squares.  Then,  since the event $F_{z,\kay_{j+1}}^2$ occurs,  we obtain that $2^{-2m}N_m \lesssim M2^{m(a-1/2)}\diam(A_{z,\kay_{j+1}})^{3/2+a}$ and so $N_m \lesssim M2^{(3/2+a)m} \diam(A_{z,\kay_{j+1}})^{3/2+a}$.  The proof is then complete by applying Lemma~\ref{lem:hyperbolic_geodesics_close} as in the previous paragraph.
\end{proof}

\begin{proof}[Proof of Proposition~\ref{prop:good_everywhere}]
\noindent{\it Step 1. Density of scales at which $\SLE$ regularity events hold.}  Fix $a \in (0,1)$. For each $z \in \h$ and $j \in \N$, let $D_{z,j} = \{ \tau_{z,j} < \infty\}$, $G_{z,j}^1 = \{ \inf_{\tau_{z,j-1} \leq  t \leq \tau_{z,j}} S_t \geq \delta_0\}$, $G_{z,j}^2 = \{ H_{z,j} \geq H_0\}$, and
let $G_{z,j}^3$ be the event that condition~\eqref{it:middle_circle_good} of Lemma~\ref{lem:band_mapping} holds.  Let also $G_{z,j}^4$ be the event that the event $C_{z,j}$ of Lemma~\ref{lem:first_excursion_good} occurs,  $G_{z,j}^5$ be the event that condition~\eqref{it:away_from_the_boundary_good} of Lemma~\ref{lem:band_mapping} holds and let $G_{z,j}^6$ be the event that the event $F_{z,j}^2$ in the statement of Lemma~\ref{lem:first_crossing_number_of_squares} occurs.  We also set $G_{z,j} = \cap_{i=1}^6 G_{z,j}^i$.  Fix $K \subseteq \h$ compact.  For each $k \in \N$, we let $\CD_k$ be the set of squares of side length $e^{-5k}$ and with corners in $(e^{-5k} \Z)^2 + (\frac{e^{-5k}}{2},\frac{e^{-5k}}{2})$ which intersect $K$.    Fix $Q \in \CD_k$ and let $z = \cen(Q)$.  Then we have that
\begin{align}
\p[ D_{z,\ell+1}, G_{z,\ell}^c \giv \CG_{z,\ell-1} ] \one_{D_{z,\ell-1}} \leq& \big(\p[ D_{z,\ell+1}, (G_{z,\ell}^1)^c  \giv \CG_{z,\ell-1}] + \p[ D_{z,\ell+1}, G_{z,\ell}^1, (G_{z,\ell}^2)^c  \giv \CG_{z,\ell-1}] \notag\\
&\quad\quad+ \p[ D_{z,\ell+1}, G_{z,\ell}^1, G_{z,\ell}^2, \cup_{i=3}^6 (G_{z,\ell}^i)^c \giv \CG_{z,\ell-1}]\big) \one_{D_{z,\ell-1}}.	 \label{eqn:d_z_not_f_z}
\end{align}
Fix $q_0 > 0$.  By Lemma~\ref{lem:hit_good} we can make $\delta_0 > 0$ sufficiently small so that
\begin{equation}
\p[ D_{z,\ell+1}, (G_{z,\ell}^1)^c  \giv \CG_{z,\ell-1}] \one_{D_{z,\ell-1}} \leq q_0/100. \label{eqn:d_z_not_f_z1}	
\end{equation}
For the above choice of $\delta_0$,  Lemma~\ref{lem:extremal_length_good} implies that we can make $H_0 > 0$ sufficiently small so that
\begin{equation}
\p[ D_{z,\ell+1}, G_{z,\ell}^1, (G_{z,\ell}^2)^c  \giv \CG_{z,\ell-1}] \one_{D_{z,\ell-1}} \leq q_0/100.	 \label{eqn:d_z_not_f_z2}	
\end{equation}
Also part~\eqref{it:middle_circle_good} of Lemma~\ref{lem:band_mapping} implies that for the above choice of $H_0$ and $\delta_0$,  we can find $b \in (1,2)$,  $\xi_1 \in (0,(e^{-b}-e^{-2})/2)$ and $\xi_0 \in (0,H_0/2)$ (depending only on $H_0$, $\delta_0$, and $q_0$) such that
\begin{equation}
\p[ D_{z,\ell+1}, G_{z,\ell}^1, G_{z,\ell}^2, (G_{z,\ell}^3)^c  \giv \CG_{z,\ell-1}] \one_{D_{z,\ell-1}} \leq q_0/100. \label{eqn:d_z_not_fz3}
\end{equation}
Moreover Lemma~\ref{lem:first_excursion_good} implies that for the above choice of $H_0$, $\delta_0$, $a$, $b$, and $\xi_2 = 3\xi_1/2$,  we have that for $\delta_1 \in (0,1)$ sufficiently small (depending only on $H_0$, $\delta_0$, $a$, $b$, $\xi_1$, and $q_0$),  we can find $M_0>1$ sufficiently large and $\delta_2 \in (0,1)$ sufficiently small such that
\begin{equation}
\p[ D_{z,\ell+1},G_{z,\ell}^1,G_{z,\ell}^2,(G_{z,\ell}^4)^c \giv \CG_{z,\ell-1}] \one_{D_{z,\ell-1}} \leq q_0/100.    \label{eqn:d_z_not_f_z_4}
\end{equation}
Similarly part~\eqref{it:away_from_the_boundary_good} of Lemma~\ref{lem:band_mapping} implies that for the above choice of $\delta_0,H_0,\delta_1,b$ and $\xi_1$ with $\wt{\xi}_1 = \delta_2$,  we can find $\wt{\xi}_0 \in (0,H_0/2)$ such that
\begin{equation}
\p[ D_{z,\ell+1},G_{z,\ell}^1,G_{z,\ell}^2,(G_{z,\ell}^5)^c  \giv  \CG_{z,\ell-1}] \one_{D_{z,\ell-1}} \leq q_0/100.    \label{eqn:d_z_not_f_z_5}
\end{equation}
Finally Lemma~\ref{lem:first_crossing_number_of_squares} implies that by possibly taking $M_0$ to be larger,  we have that
\begin{equation}
\p[  D_{z,\ell+1}, G_{z,\ell}^1,G_{z,\ell}^2,(G_{z,\ell}^6)^c  \giv  \CG_{z,\ell-1}] \one_{D_{z,\ell-1}} \leq q_0/100.     \label{eqn:d_z_not_f_z_6}
\end{equation}
Overall, \eqref{eqn:d_z_not_f_z} and~\eqref{eqn:d_z_not_f_z1}--\eqref{eqn:d_z_not_f_z_6} imply that the number of the even values of $(1-a^2/2) k \leq \ell \leq k$ for which $D_{z,\ell+1} \cap G_{z,\ell}^c$ occurs is stochastically dominated from above by a binomial random variable with parameters $n = \lfloor a^2k/4 \rfloor$ and $p = q_0$.  Let $F_{z,k}$ be the event that the number of even values of $(1-a^2/2) k \leq j \leq k$ for which $D_{z,j+1}\cap G_{z,j}^c$ occurs is at least $a^2 k /8$.  Then \cite[Lemma~2.6]{mq2020geodesics} implies that by making $q_0 > 0$ sufficiently close to zero,  we have that $\p[ F_{z,k} ] \leq e^{-12k}$. Therefore by the Borel-Cantelli lemma there a.s.\ exists $k_0 \in \N$ so that $k \geq k_0$ implies that if $Q \in \CD_k$, $z = \cen(Q)$, and $D_{z,k}$ occurs then the number of $(1-a^2/2) k \leq j \leq k$ so that $G_{z,j}$ occurs is at least $a^2 k/8$.

\noindent{\it Step 2. Density of good scales at which the regularity of crossings hold.}  Fix $p_0 \in (0,1)$.  Then Lemma~\ref{lem:conditional_probability_good} implies that by possibly taking $\delta_1 \in (0,1)$ to be smaller and keeping all of the aforementioned parameters fixed,  we can choose $M>1$ sufficiently large and $\delta_2 \in (0,1)$ sufficiently small such that
\begin{align*}
	\p\!\left[(E_{z,\kay_j}^{M,a})^c  \giv  \CF_{z,\kay_{j-1}} \right] \one_{\kay_{j-1} < \infty} \leq p_0.
\end{align*}
We recall the convention that $(E_{z,\kay_j}^{M,a})^c = \emptyset$ if $\kay_j = \infty$.  We then have for $n, N \in \N$ that
\begin{align*}
	\p\!\left[ \bigcap_{N \leq j \leq N+n} (E_{z,\kay_j}^{M,a})^c \right] &= \E\!\left[ \p\!\left[(E_{z,\kay_{N +  n}}^{M,a})^c  \giv \CF_{z,\kay_{N+n - 1}} \right] \one_{\left\{ \bigcap_{N \leq j \leq N+n - 1} (E_{z,\kay_j}^{M,a})^c \right\}} \right] \\
	&\leq p_0 \p\!\left[ \bigcap_{N \leq j \leq N+ n - 1} (E_{z,\kay_j}^{M,a})^c \right] 
	 \leq \dots \leq p_0^n.
\end{align*}
Fix $\delta > 0$ (to be chosen) and assume that $0 < p_0 < \exp(-13/\delta)$.  Applying the above for $n = \lfloor k \delta \rfloor$ and $k \in \N$ large enough we have that
\begin{align}
\label{eq:event_occurs_for_some_j}
	\p\!\left[ \bigcap_{N \leq j \leq N + \lfloor \delta k \rfloor} (E_{z,\kay_j}^{M,a})^c \right] \leq e^{-12 k} \quad \text{for each}\,  N \in \N.
\end{align}
By the Borel-Cantelli lemma, there a.s.\ exists $n_0 \in \N$ so that $n \geq n_0$ implies that if $Q \in \CD_n$ and $z = \cen(Q)$ then for each $1 \leq N \leq n$,  there exists $j \in \N$ such that $N\leq j \leq N + \lfloor \delta n \rfloor$ and $E_{z,\kay_j}^{M,a}$ occurs.

\noindent{\it Step 3.  Gap sizes between good scales.}  Using the notation of Section~\ref{subsec:iteration_exploration}, we note that we can write
\begin{equation}
\kay_j = \sum_{\ell = 1}^j (\kay_{\ell-1}^* + \Delta_\ell+4) \label{eqn:k_j_formula}	
\end{equation}
where $\Delta_\ell = \kay_\ell - (\kay_{\ell-1}^* + \kay_{\ell-1}+4)$ gives the gap size between the scale $\kay_{\ell-1} + \kay_{\ell-1}^* + 4$ and the next scale so that the event $G_{z,j}$ occurs.  We take the convention that $\kay_j^* = 0$ whenever $\kay_j = \infty$.  Lemma~\ref{lem:annulus_maximum_crossing_height} implies that there exist constants $c_0 > 1$ and $c_1 > 0$ (depending only on $H_0$, $\xi_0$, and $\xi$) so that
\begin{align}
\label{eqn:xi_ell_tail_bound}
	\p[ \kay_\ell^* \geq t \giv \CF_{z,\kay_{\ell-1}}] \one_{\kay_{\ell-1} < \infty} \leq c_0 \exp(- c_0^{-1} e^{c_1 t}) \quad \text{for each} \quad t \geq 0.
\end{align}
Fix $\theta > 1$.  By~\eqref{eqn:xi_ell_tail_bound} we have that
\begin{align}
\label{eqn:k_ell_star_exp_moment_bound}
	\E\!\left[ e^{\theta \kay_\ell^*} \giv \CF_{z,\kay_{\ell-1}} \right]  &\leq e^{c_1^{-1} \theta \log(c_0 \theta)} + c_0 \int_{c_1^{-1} \log(c_0 \theta)}^\infty \exp(\theta t - c_0^{-1} e^{c_1 t}) dt.
\end{align}
It is easy to see that there exists a constant $c_2 > 1$ (depending only on $c_0$ and $c_1$) so that~\eqref{eqn:k_ell_star_exp_moment_bound} is at most $c_2 e^{c_2 \theta \log \theta}$.  We consequently have for any $N \in \N$ that
\begin{align}\
	\p\!\left[ \sum_{j = 1}^{\lfloor \delta n \rfloor} \kay_{N+j}^* \geq a^2 n/16 \right] &= \p\!\left[ e^{\sum_{j= 1}^{\lfloor \delta n \rfloor} \theta \kay_{N+j}^*} \geq e^{\theta a^2 n/16} \right] \leq e^{-\theta a^2 n/16} \E\!\left[ e^{\sum_{j= 1}^{\lfloor \delta n \rfloor} \theta \kay_{N+j}^*} \right] \notag\\
	&= e^{-\theta a^2 n/16} \E\!\left[ e^{\sum_{j= 1}^{\lfloor \delta n \rfloor -1} \theta \kay_{N+j}^*} \E\!\left[ e^{\theta \kay_{N+ \lfloor \delta n \rfloor}^*}  \giv \CF_{z,\kay_{N+\lfloor \delta n \rfloor-1}} \right] \right] \nonumber \\
	&\leq c_2e^{c_2 \theta \log \theta - \theta a^2 n/16} \E\!\left[ e^{\sum_{j= 1}^{\lfloor \delta n \rfloor - 1} \theta \kay_{N+j}^*} \right] \nonumber \\
	&\leq \dots \leq c_2^{\delta n}e^{n \theta (c_2 \delta \log \theta - a^2/16)}. \label{eq:all_points_in_interval1}
\end{align}
Choosing $\theta > 1$ large enough (depending only on $c_2$ and $a$) so that $\theta a^2/32 > 12 +\log (c_2)$ and then $\delta \in (0,a^2/144)$ small enough so that $c_2 \delta \log \theta < a^2/32$, we thus have that
\begin{equation}
\p\!\left[ \sum_{j = 1}^{\lfloor \delta n \rfloor} \kay_{N+j}^* \geq a^2 n/16 \right] \leq e^{-12 n} \quad \text{for each} \quad n,N \in \N.\label{eq:all_points_in_interval}	
\end{equation}
Note that $\delta$ is independent of $\delta_1$ and so we can first choose $\delta$ so that~\eqref{eq:all_points_in_interval} holds and then $p_0 \in (0,e^{-13/\delta})$ so that 
\begin{align*}
	\p\!\left[(E_{z,\kay_j}^{M,a})^c  \giv  \CF_{z,\kay_{j-1}} \right] \one_{\kay_{j-1} < \infty} \leq p_0.
\end{align*}
for $\delta_1$ small enough,  $M$ large enough and $\delta_2$ small enough. Moreover, we note for $q> 0$ large enough (depending only on $c_0$ and $c_1$) and $n \in \N$ large enough by~\eqref{eqn:xi_ell_tail_bound} and a union bound that
\begin{align}
\label{eq:gap_starting_point}
	\p[ \exists 1\leq j \leq n: \kay_j^* \geq q \log n ] &\leq n c_0 \exp(-c_0^{-1} n^{c_1 q}) \leq e^{-12 n}.
\end{align}

It thus follows from~\eqref{eq:all_points_in_interval}, \eqref{eq:gap_starting_point} and the Borel-Cantelli lemma that the following is true.  There a.s.\ exists $n_1 \in \N$ so that $n \geq n_1$ implies that for every $Q \in \CD_n$ with $z = \cen(Q)$ we have $\kay_j^* \leq q \log n$ for all $1 \leq j \leq n$ and $\sum_{j=1}^{\lfloor \delta n \rfloor} \kay_{N+j}^* \leq a^2 n/16$ for all $1 \leq N \leq n$.

\noindent{\it Step 4.  Conclusion of the proof.}  Assume that $n \geq \max(n_0,n_1,k_0)$.  Fix $Q \in \CD_n$, $z = \cen(Q)$ and suppose that $\tau_{z,n} < \infty$.  Let $J = \min\{j \in \N : \kay_j \geq n(1-a^2)\}$; note that $J \leq n/5$ as $\kay_j - \kay_{j-1} \geq \kay_{j-1}^* + 4 \geq 5$ whenever $\kay_j <\infty$ for every $j \in \N$.  As $n \geq n_1$ and $J \leq n/5$ we have that
\[ \kay_{J-1}^* \leq q \log n \quad\text{and}\quad \sum_{j=1}^{\lfloor \delta n \rfloor} \kay_{J+j-1}^* \leq a^2 n/16.\]

Let $g_1$ be the first $k \geq (1-a^2)n$ such that the event $G_{z,k}$ occurs. For each $j \geq 2$, we inductively let $g_j$ be the first $k > g_{j-1}$ such that $G_{z,k}$ occurs. Let
\[ \wt{\kay}_j = \sum_{\ell=0}^{j-1} (\kay_{J+\ell-1}^* + 4).\]
We claim that
\begin{align}
\label{eqn:g_k_lbd}
g_{\wt{\kay}_j} \geq \kay_{J+j-1} \quad\text{for each}\quad j \in \N.
\end{align}
We are going to prove~\eqref{eqn:g_k_lbd} by induction on $j \in \N$.  By definition, $g_{\wt{\kay}_1}$ is the $(\kay_{J-1}^*+4)$th scale $i$ after $(1-a^2)n$ such that $G_{z,i}$ occurs while $\kay_J$ is the first $i$ after $\kay_{J-1} + \kay_{J-1}^* + 4$ such that $G_{z,i}$ occurs.  Since $\kay_{J-1} \leq (1-a^2) n$ we have that $g_{\wt{\kay}_1} \geq \kay_J$.  Suppose there exists $j \in \N$ such that $g_{\wt{\kay}_i} \geq \kay_{J+i-1}$ for each $1 \leq i \leq j$.  Then we have that $g_{\wt{\kay}_{j+1}}$ is the $(\kay_{J+j-1}^*+4)$th scale $i$ after $g_{\wt{\kay}_j} \geq \kay_{J+j-1}$ so that $G_{z,i}$ occurs.  On the other hand, $\kay_{J+j}$ is the first $i$ after $\kay_{J+j-1}^*+4$ steps after $\kay_{J+j-1}$ so that $G_{z,i}$ occurs.  Thus it is easy to see that $g_{\wt{\kay}_{j+1}} \geq \kay_{J+j}$ hence~\eqref{eqn:g_k_lbd} holds.  We thus have that
\begin{align*}
   \kay_{J+\lfloor \delta n \rfloor }
&\leq g_{\wt{\kay}_{\lfloor \delta n \rfloor +1 }} \quad\text{(by~\eqref{eqn:g_k_lbd})}\\
 &\leq g_{ (a^2/16 + 8\delta) n + q \log n} \quad\text{(as $n \geq n_1$)}\\
& \leq n \quad \text{(as $n \geq k_0$)}.
\end{align*}
Note that $J\leq n/5 <n$ and so since $n \geq n_0$,  we obtain that there exists $j \in \N$ such that $J\leq j \leq J + \lfloor \delta n \rfloor$ and $E_{z,\kay_j}^{M,a}$ occurs.  But then $\kay_j \geq \kay_J \geq (1-a^2)n$ and $\kay_j \leq \kay_{J + \lfloor \delta n \rfloor} \leq n$.  The proof is then complete since for each $z \in (e^{-5n}\Z^2) \cap K$,  there exists $Q \in \CD_{n}$ such that $z=\cen(Q)$.
\end{proof}

\begin{proof}[Proof of Theorem~\ref{thm:sle_4_removable}]
We will show that $\eta$ satisfies the conditions of Theorem~\ref{thm:removability_of_X} a.s., which completes the proof.  We fix $a \in (0,1)$ small enough such that $\tfrac{3}{2} \in (10a ,  2 - 10a)$.  Then Proposition~\ref{prop:good_everywhere} implies that the following holds a.s.  For every compact set $K \subseteq \h$,  there exist $\wt{n}_0 \in \N$ and $M>1$ such that for each $\wt{n}\geq \wt{n}_0$ and $z \in (e^{-5(\wt{n}-1)}\Z^2) \cap K$,  there exists $\wt{m}\geq \wt{n}_0$ with $(1-a^2)(\wt{n} - 1) \leq \wt{m} \leq \wt{n}-1$ so that $E_{z,\wt{m}}^{M,a}$ occurs.

For $m \in \N$ we consider the family of sets $\CA_m = \{A_{z,\wt{m}} : z \in \h,  m = \lfloor (4\wt{m}-\log(c_1/2))/\log(2) \rfloor + 1\}$
where $c_1 \in (0,1)$ is small enough (to be chosen and depending only on $a$) such that $B(z,c_1 e^{-4m})$ lies in the bounded component of $\C \setminus A_{z,m}$ for each $z \in \h$ such that $\tau_{z,m} < \infty$.  Fix $K \subseteq \h$ compact and $\wt{n}_0 \in \N$ as above.  Let $n_0$ be large (to be chosen) and $n \geq n_0$.  Set $\wt{n} = \lfloor (\log(2)n + \log(c_1/2))/4 \rfloor$.  Then by picking $n_0$ large enough,  we have that $\wt{n}\geq \wt{n}_0$.  Fix $z \in K \cap \eta$ and let $w \in (e^{-5(\wt{n}-1)}\Z^2) \cap K$ be such that $|z-w| \leq 2 e^{-5(\wt{n}-1)}$.  Then there exists $(1-a^2)(\wt{n}-1) \leq \wt{m} \leq \wt{n}-1$ such that $E_{w,\wt{m}}^{M,a}$ occurs.  Set $m = \lfloor (4\wt{m} - \log(c_1/2))/\log(2) \rfloor + 1$ and note that $A_{w,\wt{m}} \in \CA_m$.  Also,  by possibly taking $n_0$ to be larger,  we have that $B(z,2^{-n})$ lies in the bounded component of $\C \setminus A_{w,\wt{m}}$.  Moreover we have that $(1-a^2) n \leq m \leq n$ if $c_1$ is sufficiently close to $0$ (depending only on $a$).  Furthermore we have that $\diam(A_{w,\wt{m}}) \geq 2c_1 e^{-4\wt{m}} \geq 4 \cdot 2^{-m}$ and $\diam(A_{w,\wt{m}}) \leq 2 c_2 e^{-4\wt{m}} \leq 4 c_2 2^{-m} / c_1$,  where $c_2$ is a universal constant such that $A_{x,k} \subseteq B(x,c_2 e^{-4k})$ for every $x \in \h$ such that $\tau_{x,k} < \infty$.  Thus it follows that $\diam(A_{w,\wt{m}}) \asymp 2^{-m}$ where the implicit constants depend only on $c_1$ and $c_2$.  Then it is clear that conditions~\eqref{it:components} and~\eqref{it:intersections} are satisfied with $d_i = \tfrac{3}{2}$ possibly by taking $M$ to be larger.  The proof is then complete by using that $\ol{\dim}_{\mathrm{M}}(K \cap \eta) \leq \dim_{\mathrm{M}}(\eta) = \tfrac{3}{2}$ a.s.\ \cite{rs2005basic,beffara2008dimension}, where for a set $A$, $\ol{\dim}_{\mathrm{M}}(A)$ denotes the upper Minkowski dimension of $A$.
\end{proof}

\section{Removability theorem}
\label{sec:removability_theorem}

\setlength{\leftmargini}{0cm}

\subsection{Setup and statement}
\label{subsec:removability_setup_statement}

\begin{figure}[ht!]
\begin{center}
\includegraphics[scale=0.9]{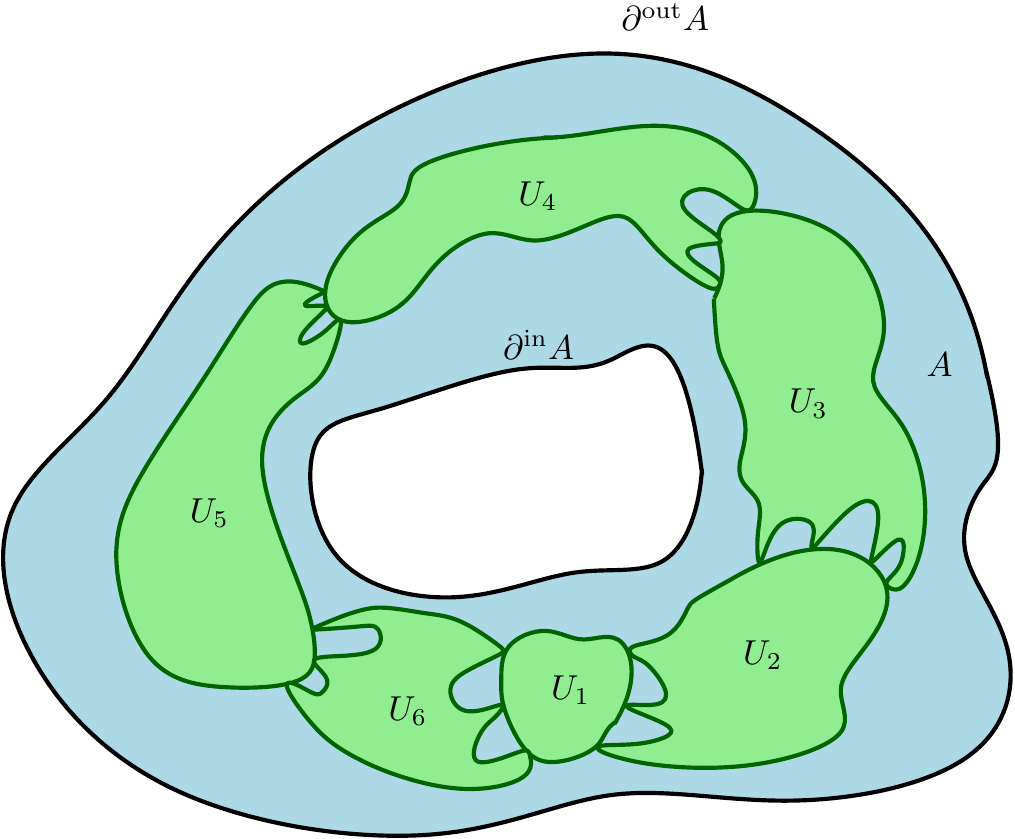}	
\end{center}
\caption{\label{fig:annulus_condition}  Illustration of the assumptions for Theorem~\ref{thm:removability_of_X}.  Shown is a topological annulus $A \in \CA_k$ where $\pin A$ (resp.\ $\pout A$) denotes its inner (resp.\ outer) boundary.  The sets $U_1,\ldots,U_m$ (with $m=6$) from~\eqref{it:components} are shown in green.  The main assumption is~\eqref{it:intersections}, which controls the behavior of $\partial U_i \cap \partial U_{i+1}$ for each $1 \leq i \leq m$ (where $U_{m+1} = U_1$).}  
\end{figure}

Suppose that $X \subseteq \C$ and that there exists $a \in (0,1)$ such that for every $K \subseteq X$ compact there exists $M>1$ such that the following is true. Let $\ol{\dim}_{\mathrm{M}}(K)$ be the upper Minkowksi dimension of $K$.  Then we have that $\ol{\dim}_{\mathrm{M}}(K) < 2$ and $0 < a < (2-\ol{\dim}_{\mathrm{M}}(K))/5$.

There exists a family $\CA = \cup_{k=1}^\infty \CA_k$ of open sets such that every $A \in \CA_k$ has the topology of an annulus with $\diam(A) \leq M 2^{-k}$.  For such $A$ we let $\bIn A$ (resp.\ $\bOut A$) be the boundary of the bounded (resp.\ unbounded) component of $\C \setminus A$.

Moreover there exists $n_0 \in \N$ so that for all $n \geq n_0$ and $z \in K$ there exists $k \in \N$ with $(1-a^2)n \leq k \leq n$ and $A \in \CA_k$ so that $B(z,2^{-n})$ is contained in the bounded component of $\C \setminus A$ and the following are true (see Figure~\ref{fig:annulus_condition}).

\begin{enumerate}[(I)]
\item\label{it:components} There exist $1 \leq m \leq M$ and pairwise disjoint, open, and simply connected subsets $U_1,\ldots,U_m$ of $A \setminus X$ such that for each $1 \leq i \leq m$ there exists $I_i \subseteq \partial U_i \cap \partial U_{i+1}$ closed (with the convention $U_{m+1} = U_1$ and $U_0 = U_m$) such that every loop $\gamma \colon \s^1 \to \closure{\cup_{i=1}^m U_i}$ which hits each of the $I_i$ exactly once and does not otherwise hit $\partial U_i \cap \partial U_{i+1}$ disconnects $\bIn A$ from $\bOut A$.
\item\label{it:intersections} There exists $d_i \in (10a,2-10a)$ and a positive and finite measure $\mu_i$ supported on $I_i$ so that the following properties hold.  Let $\ol{\mu}_i$ be $\mu_i$ normalized to be a probability measure.
	 \begin{enumerate}[(a)]
	 \item\label{it:intersection_lbd} We have that $\mu_i(I_i) \geq M^{-1} 2^{-d_i k}$.
\item\label{it:intersection_ubd} For every $Y \subseteq I_i$ Borel we have that $\mu_i(Y) \leq M \diam(Y)^{d_i-a}$.
\item\label{it:g_assumptions} For each $1 \leq i \leq m$ let $\CW_i$ be a Whitney square decomposition of $U_i$.  There exists a square $Q_i \in \CW_i$ such that the following is true.  Let $G_i^-$ be the set of $w \in I_{i-1}$ (where $I_0 = I_m$) so that if $\gamma_{i,w}$ is the hyperbolic geodesic in $U_i$, from $\cen(Q_i)$ to $w$ then for every $Q \in \CW_i$ hit by $\gamma_{i,w}$ we have that $\disthyp^{U_i}(\cen(Q_i),\cen(Q)) \leq M (2^k \len(Q))^{-a}$.  We also let $G_i^+$ be defined in the same way but with $I_i$ in place of $I_{i-1}$.  Then
\begin{enumerate}[(i)]
\item\label{it:g_meas_lbd} $\ol{\mu}_{i-1}(G_i^-) \geq 3/4$ and $\ol{\mu}_i(G_i^+) \geq 3/4$ and
\item\label{it:g_meas_ubd} The number of elements of $\CW_i$ with side length $2^{-j}$ which are hit by a hyperbolic geodesic in $U_i$ from $\cen(Q_i)$ to a point in $G_i^+$ (resp.\ $G_i^-$) is at most $M 2^{(d_{i}+a)(j-k)}$ (resp.\ $M 2^{(d_{i-1}+a)(j-k)}$). 
\end{enumerate}
\end{enumerate}
\end{enumerate}

\begin{theorem}
\label{thm:removability_of_X}
Suppose that $D \subseteq \C$ is open and $X \subseteq D$ is closed in $D$ and satisfies the above assumptions.  If $f \colon D \to \C$ is a homeomorphism onto its image and conformal on $D \setminus X$ then $f$ is conformal on $D$.
\end{theorem}

\subsection{Main lemma}

The main step in the proof of Theorem~\ref{thm:removability_of_X} is the following lemma.

\begin{lemma}
\label{lem:path_good}
Suppose that we have the setup described in Section~\ref{subsec:removability_setup_statement} and Theorem~\ref{thm:removability_of_X}.  There exists a finite and positive constant $M_0$ depending only on $a$ and $M$ so that the following is true.  Let $f \colon D \to \C$ be a homeomorphism which is conformal on $D \setminus X$.  Suppose that $A \in \CA_k$ satisfies~\eqref{it:components} and~\eqref{it:intersections} above.  There exists a path $\gamma$ in $A$ which disconnects $\bIn A$ from $\bOut A$ such that
\begin{align*}
\diam(f(\gamma)) \leq M_0 \left ( 2^{(3a - 1)k} + 2^{(1-a)k} \int_{A \setminus X} |f'(u)|^2 du \right ).
\end{align*}
\end{lemma}

\begin{proof}
This proof is in part inspired by that of~\cite[Theorem~2]{js2000removability} and calculations will be similar in Step~2.

\textit{Step~1. An integral bound.} Fix $1\leq i \leq m$.  Let $\CG_i^+$ (resp.\ $\CG_i^-$) be the set of squares in $\CW_i$ which are hit by a hyperbolic geodesic in $U_i$ from $z_i = \cen(Q_i)$ to a point of $G_i^+$ (resp.\ $G_i^-$).  For each $Q \in \CW_i$, let $\SH_i(Q)$ be the shadow cast on $\partial U_i$ by the hyperbolic geodesics $\gamma_{i,w}$ from $z_i$ to boundary points $w \in \partial U_i$ which pass through $Q$ and let $\SH_i^+(Q) = \SH_i(Q) \cap G_i^+$ and $\SH_i^-(Q) = \SH_i(Q) \cap G_i^-$.  Moreover,  we set $s_i^+(Q) = \diam(\SH_i^+(Q))$ and $s_i^-(Q) = \diam(\SH_i^-(Q))$. We fix $w \in G_i^+$.  By Lemma~\ref{lem:hyperbolic_path_length_bound}, there exists a constant $c_0 > 0$ so that
\begin{equation}
\label{eqn:f_gamma_len_bound}
\length(f(\gamma_{i,w})) \leq c_0 \sum_{Q \in \gamma_{i,w}} \len(Q) |f'|(Q).
\end{equation}

By integrating~\eqref{eqn:f_gamma_len_bound} over $I_i$ against $\ol{\mu}_i$, we see that
\begin{align}
	\int \length(f(\gamma_{i,w})) \one_{G_i^+}(w) d\ol{\mu}_i(w) 
	&\lesssim  \int \sum_{Q \in \gamma_{i,w}} \len(Q) |f'|(Q) d\ol{\mu}_i(w) \nonumber \\
	&\lesssim \sum_{Q \in \CG_i^+} \ol{\mu}_i(\SH_i^+(Q)) \len(Q) |f'|(Q) \nonumber\\
	&\lesssim 2^{-k(1-a)} \sum_{Q \in \CG_i^+} \ol{\mu}_i(\SH_i^+(Q))^2 +  2^{k(1-a)} \sum_{Q \in \CG_i^+} \len(Q)^2 |f'|(Q)^2. \label{eqn:avg_length_bound0}
\end{align}
By~\eqref{it:intersection_lbd} we have that $\mu_i(I_i) \geq M^{-1} 2^{-d_i k}$ hence $\ol{\mu}_i(\SH_i^+(Q)) \leq M 2^{d_i k} \mu_i(\SH_i^+(Q))$.  Thus~\eqref{eqn:avg_length_bound0} is at most a ($M$-dependent) constant times
\begin{align}
	&\frac{2^{-k(1-a)}}{2^{-2d_i k}} \sum_{Q \in \CG_i^+} \mu_i(\SH_i^+(Q))^2 + 2^{k(1-a)} \int_{A \setminus X} |f'(u)|^2 du \nonumber \\
	&\lesssim 2^{(2d_i -1 + a)k} \sum_{Q \in \CG_i^+} \mu_i(\SH_i^+(Q))^2 + 2^{k(1-a)} \int_{A \setminus X} |f'(u)|^2 du. \label{eqn:avg_length_bound}
\end{align}
We similarly have that
\begin{align*}
\int \length(f(\gamma_{i,w})) \one_{G_i^-}(w) d\ol{\mu}_{i-1}(w) \lesssim 2^{(2d_{i-1}-1+a)k}\sum_{Q \in \CG_i^-}\mu_{i-1}(\SH_i^-(Q))^2 + 2^{(1-a)k} \int_{A \setminus X} |f'(u)|^2 du.
\end{align*}
Note that the implicit constants depend only on $M$.  This leaves us to deal with the first term in~\eqref{eqn:avg_length_bound} (and analogously the first term on the right hand side of the above).

\textit{Step~2. Bounding $\mu_i(\SH_i^+(Q))$.}   By~\eqref{it:intersection_ubd}, we have that $\mu_i(\SH_i^+(Q)) \leq M s_i^+(Q)^{d_i-a}$.  Therefore the first term in~\eqref{eqn:avg_length_bound} is at most a constant (depending only on $M$) times
\begin{align*}
2^{(2d_i -1 + a)k}\sum_{Q \in \CG_i^+} s_i^+(Q)^{2(d_i-a)}.
\end{align*}
We shall now bound this term. Fix $Q \in \CG_i^+$ such that $s_i^+(Q) >0$ and $w_1= w_1(Q)$ and $w_2=w_2(Q)$ such that $w_1,  w_2 \in \SH_i^+(Q)$ and $|w_2 - w_1| \geq s_i^+(Q) / 2$.  Then by the definition of $\SH_i^+(Q)$,  the hyperbolic geodesics $\gamma_{i,w_1}$ and $\gamma_{i,w_2}$ intersect only squares in $\CG_i^+$.  Let $(Q_Q^{j,w_1})_{j \geq 0}$ (resp.\ $(Q_Q^{j,w_2})_{j \geq 0}$) be the sequence of squares in $\CG_i^+$, starting from $Q_i$ which $\gamma_{i,w_1}$ (resp.\ $\gamma_{i,w_2}$) hits.  Then there exist $j_1,j_2 \in \N$ such that $Q = Q_Q^{j_1,w_1} = Q_Q^{j_2,w_2}$. Writing $(\wh{Q}_Q^j)_{j \geq 0} \coloneqq (Q_Q^{j+j_1,w_1})_{j \geq 0}$ and $(\wc{Q}_Q^j)_{j \geq 0} \coloneqq (Q_Q^{j+j_2,w_2})_{j \geq 0}$,  we obtain a bi-infinite sequence in $\CG_i^+$ connecting $w_1$ and $w_2$.  Since $\diam (Q') = \sqrt{2} \len(Q')$ we have by~\eqref{eq:neighbor_squares_sidelength} that
\begin{align}
\label{eqn:good_shadow_diam_bound}
s_i^+(Q) \lesssim \sum_{j \geq 1} \len(\wh{Q}_Q^j) + \sum_{j \geq 1} \len(\wc{Q}_Q^j),
\end{align}
where the implicit constant is universal.  We will now consider the cases $d_i -a < 1$ and $d_i -a \geq 1$ separately, first focusing on the latter, and will show in both cases that
\begin{align}
\label{eqn:q_in_cg_i_plus_bound}
2^{(2d_i - 1 + a)k}\sum_{Q \in \CG_i^+} s_i^+(Q)^{2(d_i - a)} \lesssim 2^{(3a-1)k}.
\end{align}

\noindent{\it Step 2(i). $d_i -a \geq 1$.} As $s_i^+(Q) \leq \diam(A) \lesssim 2^{-k}$ we have that
\begin{align}
\sum_{Q \in \CG_i^+} s_i^+(Q)^{2(d_i - a)}
&=\sum_{Q \in \CG_i^+} s_i^+(Q)^{2(d_i - a - 1)} s_i^+(Q)^2 
 \lesssim 2^{-2(d_i - a -1)k} \sum_{Q \in \CG_i^+} s_i^+(Q)^2 \nonumber\\
&\lesssim 2^{-2(d_i - a - 1)k}\left(\sum_{Q \in \CG_i^+} \left(\sum_{j \geq 1} \len(\wh{Q}_Q^j)\right)^2 + \sum_{Q \in \CG_i^+} \left(\sum_{j \geq 1}\len(\wc{Q}_Q^j)\right)^2 \right) \quad\text{(by~\eqref{eqn:good_shadow_diam_bound})}. \label{eqn:sum_shadow_cubed_bound}
\end{align}
We will only bound the first term in~\eqref{eqn:sum_shadow_cubed_bound}; the other term in~\eqref{eqn:sum_shadow_cubed_bound} is bounded in an analogous manner.  See also~\cite[Page~275]{js2000removability}  for a similar calculation.

For $Q' \in \CW_i$, we let 
\begin{align*}
	q_i(Q') = \frac{1}{\len(Q')^2} \int_{Q'} \disthyp^{U_i}(z,z_i) d\Leb_2(z).
\end{align*}
We assume that $(\wh{Q}_Q^j)$ is ordered according to when the squares are hit by $\gamma_{i,w_1}$ after it first hits $Q$ (breaking ties using any fixed convention).
It follows from Lemma~\ref{lem:q_i_approx} that
\begin{equation}
\label{eqn:q_i_universal_bound}
q_i(\wh{Q}_Q^j) \asymp q_i(Q) + j \quad \text{for all} \quad j \in \N_0.
\end{equation}
Consequently,
\begin{align}
\sum_{Q \in \CG_i^+}\left( \sum_{j \geq 1}\len(\wh{Q}_Q^j)\right)^2
=& \sum_{Q \in \CG_i^+}\left( \sum_{j \geq 1}\len(\wh{Q}_Q^j) q_i(\wh{Q}_Q^j)^{3/4} q_i(\wh{Q}_Q^j)^{-3/4}\right)^2 \nonumber\\
\leq& \sum_{Q \in \CG_i^+}\left( \sum_{j \geq 1}\len(\wh{Q}_Q^j)^2 q_i(\wh{Q}_Q^j)^{3/2}\right) \left( \sum_{j \geq 1} q_i(\wh{Q}_Q^j)^{-3/2} \right) \quad\text{(by Cauchy-Schwarz)} \nonumber\\
\lesssim& \sum_{Q \in \CG_i^+}\left( \sum_{j \geq 1}\len(\wh{Q}_Q^j)^2 q_i(\wh{Q}_Q^j)^{3/2}\right) \left( \sum_{j \geq 1} (q_i(Q)+j)^{-3/2} \right) \quad\text{(by~\eqref{eqn:q_i_universal_bound})}\nonumber \\
\lesssim& \sum_{Q \in \CG_i^+} \left( \sum_{j \geq 1} \len(\wh{Q}_Q^j)^2 q_i(\wh{Q}_Q^j)^{3/2}\right) (q_i(Q) + 1)^{-1/2}. \label{eqn:cube_square_bound0}
\end{align}
For $\wt{Q} \in \CG_i^+$, we let $\CQ_{i,\wt{Q}}^+$ be the set of $Q \in \CG_i^+$ such that $\wt{Q} = \wh{Q}_Q^j$ for some $j \in \N_0$.  In other words,~$\CQ_{i,\wt{Q}}^+$ is the set of $Q \in \CG_i^+$ such that the sequence of squares $(\wh{Q}_Q^j)$ contains $\wt{Q}$.  Then~\eqref{eqn:cube_square_bound0} is at most a constant times
\begin{align}
\sum_{Q, \wt{Q} \in \CG_i^+} \len(\wt{Q})^2 q_i(\wt{Q})^{3/2} (q_i(Q) + 1)^{-1/2} \one_{Q \in \CQ_{i,\wt{Q}}^+} = \sum_{\wt{Q} \in \CG_i^+}\len(\wt{Q})^2 q_i(\wt{Q})^{3/2} \sum_{Q \in \CQ_{i,\wt{Q}}^+}(q_i(Q) + 1)^{-1/2}. \label{eqn:cube_square_bound}
\end{align}

We claim that
\begin{align}
\label{eqn:geo_hit_bound}
	 \sum_{Q \in \CQ_{i,\wt{Q}}^+}(q_i(Q) + 1)^{-1/2} \lesssim \sum_{j=1}^{q_i(\wt{Q})} j^{-1/2} \lesssim q_i(\wt{Q})^{1/2}
\end{align}
for every $\wt{Q} \in \CG_i^+$ distinct from $Q_i$,  where the implicit constants are universal.  Indeed,  first we let $\wt{z} = \cen(\wt{Q})$.  Lemma~\ref{lem:q_i_approx} implies that there exists a universal constant $C > 1$ such that
\begin{align*}
C^{-1}(q_i(Q) + j) \leq q_i(\wt{Q}) \leq C (q_i(Q) + j) \quad\text{for all}\quad Q \in \CQ_{i,\wt{Q}}^+ \quad\text{with}\quad \wt{Q} = \wh{Q}_Q^j.
\end{align*}
Recalling from Section~\ref{subsec:whitney_hyperbolic} that $\disthyp^{U_i}(z_i,z) \asymp d_{\CW_i}(z_i,\cen(Q))$ for all $z \in Q$, where the implicit constant is universal, we have that $\disthyp^{U_i}(z_i,z) \lesssim q_i(Q)$ for all $z \in Q$.  Hence by taking~$C$ sufficiently large,  we can assume that $\disthyp^{U_i}(z_i,z) \leq C q_i(Q)$ for all $z \in Q$.  Set
\begin{equation}
\label{eqn:n_def}
n = \lfloor C q_i(\wt{Q}) \rfloor +1
\end{equation}
and for $1\leq j \leq n$,  we let $\CF_{j,\wt{Q}}^i$ be the set of squares $Q$ in $\CW_i$ for which a hyperbolic geodesic (parameterized by hyperbolic length) from $z_i$ to some point of $\wt{Q}$ passes through $Q$ at some time $t \in [j-1,j)$.  Then we have that
\begin{align*}
\sum_{Q \in \CQ_{i,\wt{Q}}^+} (q_i(Q) + 1)^{-1/2} \leq \sum_{j=1}^n \sum_{Q \in \CF_{j,\wt{Q}}^i} (q_i(Q) + 1)^{-1/2}\lesssim \sum_{j=1}^n |\CF_{j,\wt{Q}}^i| j^{-1/2}
\end{align*}
where the implicit constants are universal.  Combining this with Lemma~\ref{lem:hyperbolic_geo_intersection} completes the proof of~\eqref{eqn:geo_hit_bound}.

Combining~\eqref{eqn:cube_square_bound} with~\eqref{eqn:geo_hit_bound} we have that
\begin{align*}
\sum_{Q \in \CG_i^+}\left( \sum_{j \geq 1}\len(\wh{Q}_Q^j)\right)^2 &\lesssim \sum_{\wt{Q} \in \CG_i^+} \len(\wt{Q})^2 q_i(\wt{Q})^{3/2} q_i(\wt{Q})^{1/2} = \sum_{\wt{Q} \in \CG_i^+} \len(\wt{Q})^2 q_i(\wt{Q})^2 \\
 &\lesssim \sum_{\wt{Q} \in \CG_i^+}\len(\wt{Q})^2 \disthyp^{U_i}(\wt{Q},z_i)^2.
\end{align*}
For the last inequality we used that $q_i(\wt{Q}) \lesssim \disthyp^{U_i}(\wt{Q},z_i)$ for every $\wt{Q} \in \wt{\CW}_i$ and some universal implicit constant.

Note that $\diam(Q) \leq \diam(A) \leq M2^{-k} \leq 2^{-k_0}$ where $k_0 = \lfloor k - \log_2(M) \rfloor$.
For $j \geq k_0$,  we let $\CN_{i,j}^+$ be the family of squares in $\CG_i^+$ of side length $2^{-j}$.  By~\eqref{it:g_meas_ubd}, we have that $|\CN_{i,j}^+| \leq M 2^{(d_i + a)(j-k)}$.  Using that $d_i +3a -2 < a$ in the final inequality, we thus have
\begin{align}
&\sum_{\wt{Q} \in \CG_i^+} \len(\wt{Q})^2 \disthyp^{U_i}(\wt{Q},z_i)^2
= \sum_{j \geq k_0} \sum_{\wt{Q} \in \CN_{i,j}^+} \len(\wt{Q})^2 \disthyp^{U_i}(\wt{Q},z_i)^2 \notag\\
&\lesssim \sum_{j \geq k_0} \underbrace{2^{-2j}}_{\len(\wt{Q})} \cdot \underbrace{2^{(d_i + a)(j-k)}}_{|\CN_{i,j}^+|} \cdot \underbrace{2^{2a(j-k)}}_{\wt{Q} \in \CG_i^+}
 \lesssim 2^{-2k} \sum_{j \geq k_0} 2^{(d_i + 3a -2)(j-k)}
 \lesssim 2^{-2k}. \label{eqn:wt_q_gi_plus_ubd}
\end{align}
Combining the above and repeating the same argument for $(\wc{Q}_Q^j)$,  it follows that
\begin{align*}
2^{(2d_i - 1 + a)k} \sum_{Q \in \CG_i^+} s_i^+(Q)^{2(d_i - a)} \lesssim 2^{(2d_i - 1 + a)k} \cdot \underbrace{2^{-2(d_i - a - 1)k}}_{\eqref{eqn:sum_shadow_cubed_bound}} \cdot \underbrace{2^{-2k}}_{\eqref{eqn:wt_q_gi_plus_ubd}} \leq 2^{(3a-1)k}\quad \text{if} \quad d_i - a \geq 1.
\end{align*}
That is, we have proved~\eqref{eqn:q_in_cg_i_plus_bound} in the case $d_i-a \geq 1$.

\noindent{\it Step 2(ii). $0 < d_i -a < 1$.}  We have that
\begin{equation}\label{eqn:sum_shadow_cubed_bound2}
\sum_{Q \in \CG_i^+} s_i^+(Q)^{2(d_i - a)} \lesssim \sum_{Q \in \CG_i^+} \left ( \left ( \sum_{j \geq 1} \len(\wh{Q}_Q^j) \right )^{2(d_i - a)} + \left ( \sum_{j \geq 1} \len(\wc{Q}_Q^j) \right )^{2(d_i - a)} \right ).
\end{equation}
Again,  we will only bound the first term in~\eqref{eqn:sum_shadow_cubed_bound2}.  We have that
\begin{align}
&\sum_{Q \in \CG_i^+} \left ( \sum_{j \geq 1} \len(\wh{Q}_Q^j) \right )^{2(d_i - a)}  = \sum_{Q \in \CG_i^+} \left ( \sum_{j \geq 1} \len(\wh{Q}_Q^j) q_i(\wh{Q}_Q^j)^{3/4} q_i(\wh{Q}_Q^j)^{-3/4} \right )^{2(d_i - a)} \nonumber \\
\leq& \sum_{Q \in \CG_i^+} \left ( \sum_{j \geq 1} \len(\wh{Q}_Q^j)^2 q_i(\wh{Q}_Q^j)^{3/2} \right )^{d_i - a} \left ( \sum_{j \geq 1} q_i(\wh{Q}_Q^j)^{-3/2} \right )^{d_i - a} \quad \text{(by Cauchy-Schwarz)} \nonumber\\
\lesssim& \sum_{Q \in \CG_i^+} \left ( \sum_{j \geq 1} \len(\wh{Q}_Q^j)^2 q_i(\wh{Q}_Q^j)^{3/2} \right )^{d_i - a} \left ( \sum_{j \geq 1} (q_i(Q) + j)^{-3/2} \right )^{d_i - a} \quad \text{(by~\eqref{eqn:q_i_universal_bound})} \label{eqn:cube_square_bound2_0}
\end{align}
Since $(x_1+\cdots+ x_n)^p \leq x_1^p+\cdots +x_n^p$ for all $n \in \N$,  $p \in (0,1)$ and $x_1,\ldots ,x_n \geq 0$, we have that~\eqref{eqn:cube_square_bound2_0} is at most a constant times
\begin{align}
 & \sum_{Q \in \CG_i^+} \left ( \sum_{j \geq 1} \len(\wh{Q}_Q^j)^{2(d_i  -a)} q_i(\wh{Q}_Q^j)^{3(d_i - a)/2} \right ) (q_i(Q) + 1)^{-(d_i - a)/2} \nonumber\\
&\lesssim \sum_{\wt{Q} \in \CG_i^+} \len(\wt{Q})^{2(d_i - a)}q_i(\wt{Q})^{3(d_i - a)/2} \sum_{Q \in \CQ_{i,\wt{Q}}^+} (q_i(Q) + 1)^{-(d_i - a) / 2}. \label{eqn:cube_square_bound2}
\end{align}

Arguing as in the case $d_i -a \geq 1$ and with $n$ as in~\eqref{eqn:n_def}, we have that
\begin{align}
\sum_{Q \in \CQ_{i,\wt{Q}}^+} (q_i(Q) + 1)^{-(d_i - a) /2} \leq& \sum_{j=1}^{n} \sum_{Q \in \CF_{j,\wt{Q}}^i} (q_i(Q) + j)^{-(d_i - a)/2} 
\lesssim \sum_{j=1}^{n} |\CF_{j,\wt{Q}}^i| j^{-(d_i - a)/2} \nonumber\\
\lesssim&  \sum_{j=1}^{n} j^{-(d_i - a) / 2}
\lesssim  q_i(\wt{Q})^{1- (d_i - a)/2}. \label{eqn:geo_hit_bound2}
\end{align}

Combining~\eqref{eqn:cube_square_bound2} with~\eqref{eqn:geo_hit_bound2},  we obtain that
\begin{align*}
\sum_{Q \in \CG_i^+} \left ( \sum_{j \geq 1} \len(\wh{Q}_Q^j) \right )^{2(d_i - a)} &\lesssim \sum_{\wt{Q} \in \CG_i^+} \len(\wt{Q})^{2(d_i - a)} q_i(\wt{Q})^{d_i - a +1} \\
&\lesssim \sum_{\wt{Q} \in \CG_i^+} \len(\wt{Q})^{2(d_i - a)} \disthyp^{U_i}(\wt{Q},z_i)^{d_i - a + 1}.
\end{align*}

Therefore if $0 < d_i - a <1$,  we have that
\begin{align*}
\sum_{\wt{Q} \in \CG_i^+} \len(\wt{Q})^{2(d_i - a)} \disthyp^{U_i}(z_i , \wt{Q})^{d_i - a +1} &= \sum_{j \geq k_0} \sum_{\wt{Q} \in \CN^+_{i,j}} \len(\wt{Q})^{2(d_i - a)} \disthyp^{U_i}(z_i,\wt{Q})^{d_i - a + 1} \\
& \lesssim \sum_{j \geq k_0} \underbrace{2^{-2j(d_i-a)}}_{\len(\wt{Q})} \underbrace{2^{(j-k)(d_i+a)}}_{| \CN_{i,j}^+|} \underbrace{2^{(d_i-a+1)a(j-k)}}_{\wt{Q} \in \CG_i^+}  \\
&\lesssim 2^{-2(d_i - a)k} \sum_{j \geq k_0} 2^{(j-k)(d_i + a + a(d_i - a + 1) - 2(d_i - a))} \lesssim 2^{-2(d_i - a)k}
\end{align*}
since $d_i + a + a(d_i - a + 1) -2d_i + 2a \leq -d_i + 5a <-a$ and the implicit constant depends only on~$M$ and~$a$.  Altogether, we have shown that
\begin{align*}
2^{(2d_i - 1 + a)k} \sum_{Q \in \CG_i^+} s_i^+(Q)^{2(d_i-a)} \lesssim 2^{(2d_i - 1 + a)k} \cdot 2^{-2(d_i - a)k} = 2^{(3a-1)k}\quad \text{if}\quad  0 < d_i - a <1.
\end{align*}
That is, we have proved~\eqref{eqn:q_in_cg_i_plus_bound} in the case $0 < d_i-a < 1$.

\textit{Step~3. Combining the above bounds.} Finally, combining the bounds of Steps~1 and~2, we obtain that there exists $\wt{M}_0 > 0$ depending only on $M$ and $a$ so that
\begin{equation}\label{eq:length_bound}
\int \text{length}(f(\gamma_{i,w}))\one_{G_i^+}(w)d\ol{\mu}_i(w)  \leq \wt{M}_0 \left( 2^{-(1-3a)k} + 2^{k(1-a)} \int_{A \setminus X} |f'(w)|^2 d\Leb_2(w) \right).
\end{equation}
A similar bound holds when we replace $G_i^+$ by $G_i^-$ and~$\ol{\mu}_i$ by~$\ol{\mu}_{i-1}$.  By applying Markov's inequality,  we obtain that the fraction of points in $I_i$ (with respect to $\ol{\mu}_i$) which satisfy~\eqref{eq:length_bound} but with $\wt{M}_0$ replaced by $16 \wt{M}_0$ is at least $11/16$ and similarly for the set $I_{i-1}$ and the measure $\ol{\mu}_{i-1}$.  By also applying the above reasoning to the domain $U_{i+1}$,  we can find $w \in I_i$ and paths $\gamma_{i,w}$ and $\gamma_{i+1,w}$ in~$U_i$ and~$U_{i+1}$ respectively which connect $w$ to $z_{i}$ and $z_{i+1}$ respectively and such that
\begin{align*}
\text{length}(f(\gamma_{j,w})) \leq 16 \wt{M}_0 \left( 2^{-(1-3a)k} + 2^{k(1-a)} \int_{A \setminus X} |f'(w)|^2 d\Leb_2(w) \right) \quad \text{for} \quad j \in \{i,i+1\}.
\end{align*}
By repeating the above procedure for all $i$,  and concatenating the paths $\gamma_{i,w}$, we obtain a path satisfying the desired bound with $M_0  = 16 M \wt{M}_0$.
\end{proof}

\subsection{Proof of Theorem~\ref{thm:removability_of_X}}

\emph{Step 1. Setup.} Let $f \colon D \to \C$ be a homeomorphism onto its image which is conformal on $D \setminus X$.  We will first show that $f$ satisfies the ACL property on $D$ and then deduce that $f$ is conformal on $D$.

We fix a rectangle $F = [x_1,x_2] \times [y_1,y_2] \subseteq D$ and assume that $K \subseteq X$ is compact and contains $F \cap X$.  Let $M, a$ be the constants for $K$ as described before the statement of Theorem~\ref{thm:removability_of_X}.  We note that the assumption $0 < a < (2-\ol{\dim}_{\mathrm{M}}(K))/5$ implies in particular that $a \in (0,2/5)$.

For $t \in [y_1,y_2]$ we let $L_t = \{s+it : s \in \R\}$.  Since~$f$ is conformal on $D \setminus X$, with $\mathrm{J}_f$ the Jacobian of~$f$, we have that $\int_E \mathrm{J}_f (z) d\Leb_2(z) = \Leb_2(f(E))$ for every Borel set $E \subseteq D \setminus X$.  Note that $\Leb_2(X)=0$.  Indeed,  since $X$ is closed in $D$,  it follows that there exists an increasing sequence of compact sets $(K_n)_{n\geq 1}$ so that $X = \cup_{n\geq 1} K_n$.  By our assumptions,  each $K_n$ has upper Minkowski dimension strictly less than $2$ hence $\Leb_2(K_n) = 0$.  Therefore $\Leb_2(X) = 0$ as a countable union of sets of Lebesgue measure $0$. Moreover, since $\Leb_2(X) = 0$, we obtain that $\mathrm{J}_f$ is defined Lebesgue a.e.\ and $\int_E \mathrm{J}_f(z) d\Leb_2(z) \leq \Leb_2(f(E))$ for every Borel set $E \subseteq D$.  In particular, $\int_F \mathrm{J}_f(z) d\Leb_2(z) < \infty$ since $\Leb_2(f(F)) < \infty$ as $f$ is continuous, and since $|f'|^2 = \mathrm{J}_f$ Lebesgue a.e., it follows that there exists $\wt{M}$ finite such that
\begin{align}
\label{eq:bound_on_derivatives}
 \int_{F}|f'(w)|^2 d \Leb_2(w) \leq \wt{M}.
\end{align}

Next, we fix $\xi \in (0,1)$ small and for $N \in \N$ and $t \in [y_1,y_2]$ we consider the following property:
\begin{enumerate}
\item[$(P_{t,N})$] There exists $\delta > 0$ such that for every collection $((r_i,s_i))_i$ of pairwise disjoint intervals in $L_t \cap F$ such that $\sum_i |s_i - r_i| < \delta$ we have that $\sum_i |f(s_i) - f(r_i)| < \frac{1}{N}$.
\end{enumerate}
We will show that the Lebesgue measure of the set of $t \in [y_1,y_2]$ such that $(P_{t,N})$ holds for every $N \in \N$ is at least $(1-\xi)(y_2 - y_1)$. 

\noindent{\it Step 2. Typical lines are good.} 

\noindent{\it Step 2a. Bound on the number of intersections of $K$ with a typical line.} For $k \in \N$,  $t \in [y_1,y_2]$,  we partition the line $L_t \cap F$ into segments of length $2^{-k}$ with dyadic endpoints. We assume (without loss of generality) that $x_2-x_1 = 2^{n'}$ for some $n' \in \Z$.  We also let $I_{1,k}^t,\dots ,  I_{n_t(k),k}^t$ be the subfamily of the above intervals that $K$ intersects and we set $I_{j,k}^t = [a_{j,k}^t,  b_{j,k}^t]$ for $j = 1,\dots,n_t(k)$.  Since the upper Minkowski dimension of $K$ is at most $2-5a$,  it follows that $\limsup_{s \to \infty} \cont_{2-4a}(K;s) = 0$.  In particular, there exists a finite constant $c_1$ such that
\begin{align*}
\Leb_2(\{z \in \C : \dist(z,K) \leq r \}) \leq c_1 r^{4a} \quad \text{for all}\quad r \in (0,1).
\end{align*}
Then we have for all $k \in \N$ that
\begin{align*}
\int_{y_1}^{y_2} 2^{-k}n_t(k) d\Leb_1(t) \leq \Leb_2(\{z \in \C : \dist(z,K) \leq 2^{-k} \}) \leq c_1 2^{-4a k}.
\end{align*}
Fix $c_2 > 0$ (independent of $k$) and set $T_k = \{t \in [y_1,y_2] : n_t(k) > c_2 2^{(1-4a)k}\}$.  Then, taking $c_2$ sufficiently large (but still independent of $k$), Markov's inequality implies that $\Leb_1(T_k) \leq c_1 / c_2 \leq \xi (y_2 - y_1) / 3$ for all $k \in \N$.

\noindent{\it Step 2b. Bound on the integral of $|f'|^2$ on a neighborhood of a typical line.}  By Fubini's theorem, we have for a constant $\wt{c} > 0$ that
\begin{align}\label{eq:bound_on_derivative_on_ngbd}
	\int_{y_1}^{y_2}	 \int_{\neigh{L_t}{M 2^{1-(1-a^2)k}} \cap F}|f'(w)|^2 d\Leb_2(w) dt \leq \wt{c} \wt{M} 2^{-(1-a^2)k}
\end{align}

Consequently, by taking $c_3$ large (independently of $k$) the Lebesgue measure of the set of $t \in [y_1,y_2]$ such that
\begin{equation}\label{eq:bound_on_derivative_on_ngbd}
	\int_{\neigh{L_t}{M 2^{1-(1-a^2)k}} }|f'(w)|^2 d\Leb_2(w) \leq c_3 2^{-(1-a/2)k}
\end{equation}
holds is at least $(y_2-y_1)(1-\xi/3)$. Thus, the (one-dimensional) Lebesgue measure of the intersection of $[y_1,y_2] \setminus T_k$ with the set of $t \in [y_1,y_2]$ such that~\eqref{eq:bound_on_derivative_on_ngbd} holds is at least $(1-2\xi / 3)(y_2-y_1)$ for all $k$ sufficiently large.

\noindent{\it Step 2c. Bound on the integral of $|f'|^2$ on a  typical line.} We similarly have for all $c_4 > 0$ sufficiently large (depending only on $\wt{M}$ and $\xi$) that
\[ \Leb_1\left(\left\{ t \in [y_1,y_2] : \int_{L_t \cap F} |f'(w)|^2 d\Leb_1(w) \geq c_4 \right\} \right) \leq \frac{\wt{M}}{c_4} < \frac{\xi}{3} (y_2-y_1).\]

\noindent{\it Step 2d. All three properties hold on a typical line.}  Combining the above, we obtain that the one-dimensional Lebesgue measure of $t \in [y_1,y_2]$ such that the following hold is at least $(1-\xi)(y_2-y_1)$:
\begin{enumerate}[(A)]
\item \label{it:good_intersections} $n_t(k) \leq c_2 2^{(1-4a)k}$,
\item \label{it:good_bound_on_ngbds} $\int_{\neigh{L_t}{M 2^{1-(1-a^2)k}} }|f'(w)|^2d\Leb_2(w) \leq  c_3 2^{-(1-a/2)k}$,  and
\item \label{it:good_bound_on_lines} $\int_{L_t \cap F} |f'(w)|^2 d\Leb_1(w) \leq c_4$.
\end{enumerate}

\emph{Step 3. Bounding sums of over subintervals.} We fix $t \in [y_1,y_2]$ and for $1 \leq j \leq n_t(k)$ we fix $z_{j,k}^t \in K \cap I_{j,k}^t$.  Then there exists $k_j \in \N$ such that $(1-a^2) k \leq k_j \leq k$ and $A(z_{j,k}^t,k_j) \in \CA_{k_j}$ and the properties of $K$ stated right above Theorem~\ref{thm:removability_of_X} hold for $A(z_{j,k}^t,k_j)$ and $B(z_{j,k}^t,2^{-k})$ is contained in the bounded component of $\C \setminus A(z_{j,k}^t,k_j)$.  Note that $z_{j,k}^t \in L_t \cap F$ and so $A(z_{j,k}^t,k_j) \subseteq \neigh{L_t}{M 2^{1-(1-a^2)k}}$ and $\bIn A(z_{j,k}^t,k_j)$ surrounds $I_{j,k}^t$ provided $k$ is sufficiently large. Throughout this step, we assume that~\eqref{it:good_intersections}, \eqref{it:good_bound_on_ngbds}, and~\eqref{it:good_bound_on_lines} hold.

Next, we fix a collection $([r_i,s_i])_i$ of pairwise disjoint (except possibly at their endpoints) subintervals of $L_t \cap F$ such that $\sum_i |s_i - r_i| < \delta$ and $|s_i - r_i| \geq 2^{-k}$ for all $i$.  Here,  $\delta > 0$ is small but fixed (to be chosen independently of $k$ and depending only on $N$, $c_1,\ldots,c_4$, and the implicit constants $M$, $a$ of the properties of $K$).  Note that each interval $I_{j,k}^t$ can contain at most one of the $[r_i,s_i]$.  By subdividing each of the $[r_i,s_i]$ if necessary (which only makes the sum $\sum_i |f(s_i)-f(r_i)|$ larger),  we can further assume that each $[r_i,s_i]$ is either contained in one of the intervals $I_{j,k}^t$ or is disjoint from all of them,  except possibly at its endpoints.  After performing this subdivision,  each $I_{j,k}^t$ can contain at most two of the $[r_i,s_i]$.  Moreover, if we have an interval $[r_i,s_i]$ which does not intersect one of the $I_{j,k}^t$,  then clearly $|f(s_i) - f(r_i)| \leq \int_{[r_i,s_i]}|f'(w)| d\Leb_1(w)$. On the other hand, if we have an interval $[r_i,s_i]$ which is contained in one of the $I_{j,k}^t$ and $\gamma_{j,k}$ is the path in $A(z_{j,k}^t,k_j)$ from Lemma~\ref{lem:path_good} then since $\gamma_{j,k}$ surrounds $I_{j,k}^t$, we have that $|f(s_i) - f(r_i)| \leq \diam(f(\gamma_{j,k}))$.  Hence, it follows that
\begin{equation}
\label{eq:basic_bound}
\sum_i |f(s_i) - f(r_i)| \leq \int_{\cup_i [r_i,s_i]} |f'(w)|d\Leb_1(w) + 2\sum_{j=1}^{n_t(k)}\diam(f(\gamma_{j,k})).
\end{equation}
By Lemma~\ref{lem:path_good}, we have that
\begin{align*}
\diam(f(\gamma_{j,k})) &\leq M_0 \left(2^{-(1-3a)k_j} + 2^{(1-a)k_j} \int_{A(z_{j,k}^t,k_j)}|f'(w)|^2 d\Leb_2(w) \right) \\
&\leq M_0 \left( 2^{-(1-a^2)(1-3a)k} + 2^{(1-a)k}\int_{A(z_{j,k}^t,k_j)}|f'(w)|^2 d\Leb_2(w) \right)
\end{align*}
for all $j$. The Cauchy-Schwarz inequality together with~\eqref{it:good_bound_on_lines} implies that the first term in the right-hand side of~\eqref{eq:basic_bound} is at most $\delta^{1/2} \left( \int_{L_t \cap F}|f'(w)|^2 d\Leb_1(w) \right)^{1/2} \leq \delta^{1/2} c_4^{1/2}$.  Together with~\eqref{it:good_intersections},  this implies that 
\begin{align*}
&\sum_i |f(s_i) - f(r_i)| \\
&\leq c_4^{1/2} \delta^{1/2} + 2M_0 \left( n_t(k) 2^{-(1-a^2)(1-3a) k} + 2^{(1-a)k} \int \sum_{j=1}^{n_t(k)} \one_{A(z_{j,k}^t,k_j)}(w) |f'(w)|^2 d\Leb_2(w) \right)\\
&\leq c_4^{1/2} \delta^{1/2} + 2M_0 \left(c_2 2^{-(a-a^2+3a^3)k} + 2^{(1-a)k}\int \sum_{j=1}^{n_t(k)} \one_{A(z_{j,k}^t,k_j)}(w) |f'(w)|^2 d\Leb_2(w) \right).
\end{align*}
Note also that if we pick a point $z \in L_t \cap F$,  then $\sum_{j=1}^{n_t(k)} \one_{A(z_{j,k}^t,k_j)}(z)$ is equal to the number of annuli $A(z_{j,k}^t,k_j)$ which contain $z$.  We fix such an annulus $A(z_{i,k}^t,k_i)$.  Noting that $A(z_{j,k}^t,k_j) \subseteq B(z_{j,k}^t , M2^{-(1-a^2)k})$ and that the balls $B(z_{i,k}^t, M 2^{-(1-a^2)k}),  B(z_{j,k}^t,  M 2^{-(1-a^2)k})$ intersect only if $|z_{i,k}^t - z_{j,k}^t| \leq M 2^{1-(1-a^2)k}$,  we have that $\sum_{j=1}^{n_t(k)}\one_{A(z_{j,k}^t,k_j)}(z)$ is at most the number of $j$ such that $B(z_{j,k}^t,M 2^{-(1-a^2)k})$ intersects $B(z_{i,k}^t,M 2^{-(1-a^2)k})$. This, in turn, is at most $M 2^{1+a^2k}+1 \leq M 2^{2+a^2k}$.  It thus follows that
\begin{align*}
\sum_i |f(s_i) - f(r_i)| \lesssim \delta^{1/2} + 2^{-(a-a^2+3a^3)k} + 2^{(1-a+a^2)k} \int_{\cup_{j=1}^{n_t(k)} A(z_{j,k}^t,k_j)} |f'(w)|^2 d\Leb_2(w).
\end{align*}
Since each set $A(z_{j,k}^t,k_j)$ is contained in $\neigh{L_t}{M 2^{1-(1-a^2)k}}$ we have by~\eqref{it:good_bound_on_ngbds} that
\begin{align*}
\sum_i |f(s_i) - f(r_i)| &\lesssim \delta^{1/2} + 2^{-(a-a^2 + 3a^3)k} + 2^{(1-a+a^2)k} \int_{\neigh{L_t}{M 2^{1-(1-a^2)k}}}|f'(w)|^2 d\Leb_2(w) \\
&\lesssim \delta^{1/2} + 2^{-(a-a^2 + 3a^3)k} + 2^{(a^2-a/2) k}
 < \frac{1}{N}
\end{align*}
for all $k \in \N$ sufficiently large and $\delta > 0$ is chosen small and independent of $k$.  Note also that the implicit constant in the above inequality is independent of $N$ and $k$. Since $a \in (0,2/5)$, it follows that both exponents are negative.

\emph{Step 4. Deducing that $f$ is ACL and hence conformal.} Let $G_{\delta,k,N}$ be the set of $t \in [y_1,y_2]$ such that $(P_{t,N})$ holds for $\delta$ and intervals $([r_i,s_i])_i$ such that $|s_i - r_i| \geq 2^{-k}$ for all $i$.  The above implies that for each $N \in \N$ and $\xi \in (0,1)$ there exists $k_0 = k_0(N) \in \N$ and $\delta_0 = \delta_0(N,\xi) > 0$ (i.e., $\delta_0$ depends on $N$, $\xi$ but $k_0$ depends only on $N$) such that for all $k \geq k_0$ and $\delta \in (0,\delta_0)$ we have that $\Leb_1(G_{\delta,k,N}) \geq (1-\xi) (y_2-y_1)$.  This implies that for all $k \geq k_0$ we have that $\Leb_1( \cup_{r \in \Q_+} G_{r,k,N}) = y_2 - y_1$.  As $\cup_{r \in \Q_+} G_{r,k,N}$ is decreasing in $k$, we have that $\Leb_1( \cap_{k=1}^\infty \cup_{r \in \Q_+} G_{r,k,N}) = y_2 - y_1$.  Note that if $(P_{t,N})$ holds for every $N$, i.e., $t \in \cap_{N=1}^\infty \cap_{k=1}^\infty \cup_{r \in \Q_+} G_{r,k,N}$ then $f$ is absolutely continuous on $L_t \cap F$.  Thus, as $\Leb_1( \cap_{N=1}^\infty \cap_{k=1}^\infty \cup_{r \in \Q_+} G_{r,k,N}) = y_2 - y_1$, we see that $f$ is absolutely continuous on $L_t \cap F$ for Lebesgue a.e.\ $t \in [y_1,y_2]$.  Finally,  by repeating the same argument for lines of the form $\{t + is : s \in \R\}$ for $t \in [x_1,x_2]$,  we obtain that Lebesgue a.e., $f$ is absolutely continuous on segments of the form $\{t + is : s \in \R\}$ for $t \in [x_1,x_2]$ or $\{it + s : s \in \R\}$ for $t \in [y_1,y_2]$.  It follows that $f$ has the ACL property on $D$.  Moreover, since $\Leb_2(X)=0$ and $f$ is $1$-quasiconformal on $D \setminus X$, it follows that $f$ is $1$-quasiconformal on $D$ and so conformal on $D$ as well.  This completes the proof of the theorem.
\qed

\appendix

\section{Whitney squares and the hyperbolic metric}
\label{app:whitney}

In this appendix, we will collect several facts about Whitney squares and the hyperbolic metric which are used in this article. 

\begin{lemma}
\label{lem:hyperbolic_ball_covered}
Let $D$ be a simply connected domain and $\CW$ a Whitney square decomposition of $D$. For $z \in D$ and $r>0$, let $\Bhyp^D(z,r) \coloneqq \{w \in D: \disthyp^D(z,w) < r \}$. Then, for each $r>0$, there exists a positive constant $M(r)$ such that for all $z \in D$, $\Bhyp^D(z,r)$ intersects at most $M(r)$ squares in $\CW$. Moreover, $M(r)$ is uniform in the choice of simply connected domain and Whitney square decomposition.
\end{lemma}
\begin{proof}
Fix $a \in D$ and let $\phi$ be the unique conformal transformation mapping $\D$ onto $D$ such that $\phi(0) = a$ and $\phi'(0)>0$.  Then the conformal invariance of the hyperbolic metric implies that $\Bhyp^D(a,r) = \phi(\Bhyp^{\D}(0,r))$.  Also we have that $\Bhyp^{\D}(0,r) = B(0,a(r))$ where $a(r) = \frac{e^{2r}-1}{e^{2r}+1}$ (see \cite{gm2005harmonic}) and so $\Bhyp^D = \phi(B(0,a(r)))$.  Moreover~\cite[Theorem~3.21]{lawler2008conformally} implies that $\Bhyp^D(a,r) \subseteq B(a,\frac{a(r)|\phi'(0)|}{(1-a(r))^2})$.  In particular we have that
\begin{equation}
\label{eqn:hyperbolic_ball_lebesgue_ubd}
\Leb_2(\Bhyp^D(a,r)) \lesssim |\phi'(0)|^2	
\end{equation}
where the implicit constant depends only on $r$.  Let $Q \in \CW$ be such that $Q \subseteq \Bhyp^D(a,r)$.  Then the Koebe-$\frac{1}{4}$ theorem combined with~\cite[Theorem~3.21]{lawler2008conformally} implies that $\dist(\cen(Q) ,  \partial D) \gtrsim |\phi'(0)|$ and so $\Leb_2(Q) = \ell(Q)^2 \gtrsim |\phi'(0)|^2$ where the implicit constant depends only on $r$.  Fix $M \in \N$ and suppose that $Q_j \subseteq \Bhyp^D(a,r)$ for $Q_j \in \CW$ and $1 \leq j \leq M$ are pairwise disjoint.  By~\eqref{eqn:hyperbolic_ball_lebesgue_ubd} we have that $\Leb_2(\cup_{j=1}^M Q_j) = \sum_{j=1}^M \ell(Q_j)^2 \lesssim |\phi'(0)|^2$ and so $M \lesssim 1$ where the implicit constant depends only on $r$.  It follows that there exists $\wt{M}(r) > 0$ depending only on $r$ such that for each $a \in D$,  the hyperbolic ball $\Bhyp^D(a,r)$ can contain at most $\wt{M}(r)$ squares in $\CW$.  Note for each $Q \in \CW$ and $z,w \in Q$ we have that $\disthyp^D(z,w) \leq 1$ (which follows from~\eqref{eq:dist_hyp_qh_comparable} and the fact that the same holds for the quasihyperbolic diameter).  The claim of the lemma then follows by setting $M(r) = \wt{M}(r+1)$ since $Q \subseteq \Bhyp^D(a,r+1)$ for every $Q \in \CW$ such that $Q \cap \Bhyp^D(a,r) \neq \emptyset$.
\end{proof}

Lemma~\ref{lem:hyperbolic_ball_covered} immediately implies the following.

\begin{lemma}
\label{lem:bound_on_squares}
There exists a universal constant $M > 0$ such that the following holds. Let $D,\wt{D}$ be simply connected domains, $\varphi : D \to \wt{D}$ conformal and let $\CW$ and $\wt{\CW}$ be a Whitney square decomposition of $D$ and $\wt{D}$, respectively. Then, for any $Q \in \CW$, $\varphi(Q)$ intersects at most $M$ Whitney squares in $\wt{\CW}$ and for any $\wt{Q} \in \wt{\CW}$, $\varphi^{-1}(\wt{Q})$ intersects at most~$M$ Whitney squares in~$\CW$.
\end{lemma}
\begin{proof}
This follows from Lemma~\ref{lem:hyperbolic_ball_covered}, since the hyperbolic diameter of any Whitney square is at most~$1$.
\end{proof}

\begin{lemma}
\label{lem:hyperbolic_path_length_bound}
There exists a constant $c_0 > 0$ so that the following is true.  Suppose that $U,V \subseteq \C$ are simply connected domains,  $f \colon U \to V$ is a conformal transformation, and $\CW$ is a Whitney square decomposition of $U$.  Then for each $z \in U$ and $w \in \partial U$ we have that
\begin{equation}
\label{eqn:f_gamma_len_bound_statement}
\length(f(\gamma_{z,w})) \leq c_0 \sum_{Q \in \gamma_{z,w}} \len(Q) |f'|(Q),
\end{equation}
where $\gamma_{z,w}$ denotes the hyperbolic geodesic in $U$ connecting $z$ to $w$, by $Q \in \gamma_{z,w}$ we mean that $\gamma_{z,w}$ intersects the square $Q$, and $|f'|(Q) = \len(Q)^{-2} \int_{Q}|f'(u)| du$.	
\end{lemma}
\begin{proof}
Suppose that $Q \in \gamma_{z,w}$.  We parameterize $\gamma_{z,w}$ by $[0,1]$ and we set
\begin{align*}
	s = \inf\{r \geq 0 : \gamma_{z,w}(r) \in Q\} \quad\text{and}\quad t = \sup\{r \geq s : \gamma_{z,w}(r) \in Q\}.
\end{align*}
Since $\disthyp^{U}(x,y) \leq 1$ for all $x,y \in Q$,  it follows for $r \in [s,t]$ that
\begin{align*}
\disthyp^{U}(\cen(Q),\gamma_{z,w}(r))
&\leq \disthyp^{U}(\cen(Q),\gamma_{z,w}(s)) + \disthyp^{U}(\gamma_{z,w}(s),\gamma_{z,w}(r))\\
&\leq \disthyp^{U}(\cen(Q),\gamma_{z,w}(s))+ \disthyp^{U}(\gamma_{z,w}(s),\gamma_{z,w}(t))\leq 2
\end{align*}
as $\gamma_{z,w}$ is a geodesic.  Thus $\gamma_{z,w}([s,t]) \subseteq \closure{\Bhyp^U(\cen(Q),2)}$ where $\Bhyp^U(z,r) = \{w \in U: \disthyp^U(z,w) < r \}$ denotes the hyperbolic ball of radius $r$ around $z$ in $U$. Let $\phi : \D \rightarrow U$ be the unique conformal transformation such that $\phi(0) = \cen(Q)$ and $\phi'(0) > 0$.  Note that $\Bhyp^{\D}(0,2) = B(0,a)$ where $a = \frac{e^4-1}{e^4+1}$ (see \cite{gm2005harmonic}) and so the conformal invariance of the hyperbolic metric implies that if $\wt{\gamma} = \phi^{-1}(\gamma_{z,w}|_{[s,t]})$, then $\wt{\gamma} \subseteq B(0,a)$.  Moreover, we have that
\begin{align*}
1 \geq \disthyp^{\D}(\wt{\gamma}(s),\wt{\gamma}(t)) = \int_s^t \frac{|\wt{\gamma}'(r)|}{1-|\wt{\gamma}(r)|^2}dr \geq \text{length}(\wt{\gamma})
\end{align*}
and so $\length(\wt{\gamma}) \leq 1$.  We let $Q_1,\dots,Q_k \in \CW$ be such that $Q_j \cap \Bhyp^U(\cen(Q),2) \neq \emptyset$ for each $1\leq j \leq k$ and $\Bhyp^U(\cen(Q),2) \subseteq \bigcup_{j=1}^m \closure{Q_j}$.  Note that Lemma~\ref{lem:hyperbolic_ball_covered} implies that $m \leq M_0$ for some universal constant~$M_0$.  Since for every square $Q_j$ we can find a subfamily of squares of the above family which is a chain and connects $Q$ to $Q_j$, it follows from~\eqref{eq:neighbor_squares_sidelength} that $\len(Q_j) \asymp \len(Q)$, where the implicit constants are universal. Consequently, $\dist(Q_j,  \partial U) \lesssim \len(Q_j) \lesssim \len(Q)$ for all $1 \leq j \leq k$ and so $\dist(\gamma_{z,w}(r),  \partial U)\lesssim \len(Q)$ for all $r \in [s,t]$. By the Koebe-$1/4$ theorem, this implies that $|(\phi^{-1})'(\gamma_{z,w}(r))| \gtrsim \len(Q)^{-1}$ for all $r \in [s,t]$. Combining with the above, we obtain that $1 \geq \length(\wt{\gamma}) \gtrsim \length(\gamma_{z,w}([s,t])) \len(Q)^{-1}$ which implies that $\length(\gamma_{z,w}([s,t])) \lesssim \len(Q)$ where the implicit constants are universal.  Finally, by~\eqref{eq:uniform_derivative_in_square} it follows that $\length(f(\gamma_{z,w} \cap Q)) \leq c_0 \len(Q) |f'|(Q)$ for some universal constant $c_0 > 0$. The claim then follows by summing over $Q \in \gamma_{z,w}$.
\end{proof}

\begin{lemma}
\label{lem:hyperbolic_geo_intersection}
There exists a constant $c_0 > 0$ so that the following is true.  Suppose that $U \subseteq \C$ is a simply connected domain, $\CW$ is a Whitney square decomposition of $U$, $Q_0,Q_1 \in \CW$, and for $j \in \N$ we let $\CF_j$ be the set of $\wt{Q} \in \CW$ for which a hyperbolic geodesic (parameterized by hyperbolic length) from $\cen(Q_0)$ to some point in $Q_1$ passes through at some time $t \in [j-1,j)$.  Then $|\CF_j| \leq c_0$.	
\end{lemma}
\begin{proof}
Let $\phi  \colon U \to \D$ be the unique conformal transformation such that $\phi(\cen(Q_0)) = 0$ and $\phi'(\cen(Q_0)) > 0$.  Fix $j \in \N$ and suppose that $\wt{Q} \in \CF_j$.  We divide $\wt{Q}$ into squares $\wt{Q}_k$ for $1 \leq k \leq 4$ with sides parallel to the coordinate axes and equal side length $\len(\wt{Q}_k) = \len(\wt{Q})/2$.  Note that
\begin{align*}
\dist(\cen(\wt{Q}_k),  \partial U) \geq \frac{\len(\wt{Q}_k)}{2} + \dist(\wt{Q},\partial U)\geq \left(1 + \frac{1}{4\sqrt{2}} \right)\! \diam (\wt{Q}),
\end{align*}
so that for $w \in \wt{Q}_k$,  we have that 
\begin{align*}
|w-\cen(\wt{Q}_k)| \leq \frac{\diam(\wt{Q}_k)}{2} = \frac{\diam(\wt{Q})}{4} \leq r_0 \dist(\cen(\wt{Q}_k), \partial U) \quad\text{where}\quad r_0 = \frac{\sqrt{2}}{1+4\sqrt{2}} \in (0,1).
\end{align*}
Hence, $\wt{Q}_k \subseteq B(\cen(\wt{Q}_k), r_0 \dist(\cen(\wt{Q}_k),\partial U))$ and thus, by \cite[Corollary 3.25]{lawler2008conformally} we have that
\begin{align*}
|\phi(w) - \phi(\cen(\wt{Q}_k))|\leq  \frac{4|w-\cen(\wt{Q}_k)|}{1-r_0^2}\frac{\dist(\phi(\cen(\wt{Q}_k)),\partial \D)}{\dist(\cen(\wt{Q}_k),\partial U)} \leq \frac{4r_0}{1-r_0^2}\dist(\phi(\cen(\wt{Q}_k)),\partial \D)
\end{align*}
for all $w \in B(\cen(\wt{Q}_k),r_0\dist(\cen(\wt{Q}_k),\partial U))$. It follows that $\phi(\wt{Q}_k)\subseteq B(\phi(\cen(\wt{Q}_k)),pd)$ where $p = \frac{4r_0}{1-r_0^2}$ and $d = \dist(\phi(\cen(\wt{Q}_k)),\partial \D)$.  Note that $p \in (0,1)$.

For $x,y \in \D$ we let $\wh{\gamma}_{x,y}$ be the hyperbolic geodesic in $\D$ from $x$ to $y$ parameterized by hyperbolic length and such that $\wh{\gamma}_{x,y}(0) = x$.  Let $\wh{z}_k = \phi(\cen(\wt{Q}_k))$.  Fix $j \in \N$ and suppose that there exists $t \in [j-1,j)$ and $\wh{z} \in B(\wh{z}_k,pd)$ such that $\disthyp^{\D}(0,\wh{z}) \geq t$.  Note that $\wh{\gamma}_{0,\wh{z}}$ is the line segment connecting $0$ to $\wh{z}$.  Let also $\wh{w} \in B(\wh{z}_k,pd)$ be such that $\disthyp^{\D}(0,\wh{w}) \geq s$ for some $s \in [j-1,j)$.  Then we have that
\begin{align*}
|\wh{\gamma}_{0,\wh{z}}(j-1) - \wh{\gamma}_{0,\wh{w}}(j-1)|\leq |\wh{z}-\wh{w}| \leq |\wh{z}-\wh{z}_k|+|\wh{z}_k-\wh{w}| \leq 2pd.
\end{align*}
Moreover,  for all $w \in B(\wh{z}_k,pd)$ it holds that $\dist(w,\partial \D) \geq (1-p)d$ and so $\dist(x,\partial \D) \geq (1-p)d$ for all $x \in I = [\wh{\gamma}_{0,\wh{z}}(j-1),\wh{\gamma}_{0,\wh{w}}(j-1)]$.  Thus,  it follows that
\begin{align*}
\disthyp^{\D}(\wh{\gamma}_{0,\wh{z}}(j-1),\wh{\gamma}_{0,\wh{w}}(j-1)) \leq \int_{I}\frac{|dz|}{1-|z|^2} \leq \frac{2p}{1-p}.
\end{align*}
Moreover,  $\disthyp^{\D}(\wh{\gamma}_{0,\wh{w}}(j-1),\wh{\gamma}_{0,\wh{w}}(r)) \leq 1$ and so $\wh{\gamma}_{0,\wh{w}}(r) \in B_k = \Bhyp^{\D}(\wh{\gamma}_{0,\wh{z}}(j-1),\frac{1+p}{1-p})$ for all $r \in [j-1,j)$.

By the conformal invariance of the hyperbolic metric we obtain the following. Fix $w \in Q_1$, pick some $t \in [j-1,j)$ and let $\wt{Q} \in \CW$ be such that $\gamma_{\cen(Q_0),w}(t) \in \wt{Q}$. Then it holds that $\gamma_{\cen(Q_0),w}(t) \in 
\phi^{-1}(B_k)$ for some $1 \leq k \leq 4$.  In particular,  $\wt{Q} \cap \phi^{-1}(B_k) \neq \emptyset$.  By Lemma~\ref{lem:hyperbolic_ball_covered},  we obtain that there exists some universal constant $M_0 > 0$ such that for each $1 \leq k \leq 4$, $\phi^{-1}(B_k)$ intersects at most $M_0$ squares in $\CW$,  and so $|\CF_j| \leq 4M_0$.
\end{proof}

\begin{lemma}
\label{lem:q_i_approx}
Suppose that $U \subseteq \C$ is a simply connected domain and $\CW$ is a Whitney square decomposition of $U$.  Suppose that $Q_0 \in \CW$ and for $Q \in \CW$ we let 
\begin{align*}
	q(Q) = \frac{1}{\len(Q)^2} \int_{Q} \disthyp^{U}(z,\cen(Q_0)) d\Leb_2(z).
\end{align*}
Suppose that $w \in \partial U$ and $\gamma$ is the hyperbolic geodesic from $\cen(Q_0)$ to $w$.  Suppose that $Q \in \CW$ is hit by $\gamma$ and let $(\wh{Q}_j)$ be the sequence of squares hit by $\gamma$ starting after it first hits $Q$ and ordered according to when they are first hit by $\gamma$ (breaking ties using any fixed convention).  Then we have that
\begin{equation}
\label{eqn:q_i_universal_bound_statement}
q(\wh{Q}_j) \asymp q(Q) + j \quad \text{for all} \quad j \in \N_0
\end{equation}
with universal implicit constants.
\end{lemma}
\begin{proof}
We will first prove that $q(\wh{Q}_j) \lesssim q(Q) + j$ for all $j \in \N_0$.  Let $\wh{Q}$ be the square with the same center as $Q$ but with half the side length.  Since $\disthyp^U(z,\cen(Q_0)) \gtrsim d_{\CW}(\cen(Q),\cen(Q_0))$ for all $z \in \wh{Q}$ where the implicit constant is universal, it follows that
\begin{equation}\label{eqn:lower_bound_on_q_i}
q(Q) \geq \frac{1}{\len(Q)^2}\int_{\wh{Q}}\disthyp^U(z,\cen(Q_0)) d\Leb_2(z) \gtrsim d_{\CW}(\cen(Q),\cen(Q_0)).
\end{equation}
By the definition of $(\wh{Q}_k)$ we can find $s<t$ such that $\gamma(s) \in Q$,  $\gamma(t) \in \wh{Q}_j$,  $\gamma([s,t]) \subseteq \closure{\bigcup_{k=0}^j \wh{Q}_k}$ and $\gamma([s,t]) \cap \wh{Q}_k \neq \emptyset$ for all $1\leq k \leq j$.  Let $\varphi : \D \rightarrow U$ be the conformal transformation such that $\varphi(0) = \gamma(s)$ and $a = \varphi^{-1}(\gamma(t)) \in (0,1)$.  Let $\CW'$ be a fixed Whitney square decomposition of $\D$ and let $(Q_k')_{0 \leq k \leq l}$ be a chain of minimal length of adjacent squares in $\CW'$ such that $0 \in Q_0'$ and $a \in Q_l'$.  Note that $\varphi([0,a]) = \gamma([s,t])$ since $\gamma|_{[s,t]}$ is the hyperbolic geodesic in~$U$ from~$\gamma(s)$ to~$\gamma(t)$, and consequently $[0,a] \subseteq \closure{\bigcup_{k=0}^j \varphi^{-1}(\wh{Q}_k)}$.  Moreover for all $0 \leq k \leq j$ there exists $z_k \in \varphi^{-1}(\wh{Q}_k)$ such that $\varphi^{-1}(\wh{Q}_k) \subseteq \Bhyp^{\D}(z_k,2)$ and so $[0,a] \subseteq \bigcup_{k=0}^j \Bhyp^{\D}(z_k,2)$.  Note that Lemma~\ref{lem:hyperbolic_ball_covered} implies that $B_{\text{hyp}}^{\D}(z_k,2)$ can intersect at most $M_0$ squares in $\CW'$ for some universal constant $M_0$ and hence $l \lesssim j$.  Suppose that $l \geq 1$.  Then $(\varphi(Q_k'))_{0 \leq k \leq l}$ is a connected family of sets connecting $Q$ to $\wh{Q}_j$ and by Lemma~\ref{lem:bound_on_squares}, each $\varphi(Q_k')$ can intersect at most $\wt{M}_0$ squares in $\CW$ for some universal constant $\wt{M}_0$.  This implies that $d_{\CW}(\cen(Q),\cen(\wh{Q}_j)) \lesssim l \lesssim j$ and so by~\eqref{eqn:lower_bound_on_q_i} we have that for all $z \in \wh{Q}_j$,
\begin{align*}
\disthyp^{U}(z,\cen(Q_0)) \lesssim d_{\CW}(\cen(Q_0),\cen(\wh{Q}_j)) \lesssim d_{\CW}(\cen(Q_0),\cen(Q)) + d_{\CW}(\cen(Q),\cen(\wh{Q}_j)) \lesssim  q(Q) + j
\end{align*}
where the implicit constants are universal and so this gives $q(\wh{Q}_j) \lesssim q(Q) + j$ for all $j \in \N_0$ in the case that $l\geq 1$.  When $l=0$,  we have that $\{\gamma(s),\gamma(t)\} \subseteq \varphi(Q_0')$ and $\varphi(Q_0')$ can be covered by at most $\wt{M}_0$ squares in $\CW$.  Therefore $d_{\CW}(\cen(Q),\cen(\wh{Q}_j)) \lesssim 1 \lesssim j$ and similarly we obtain $q(\wh{Q}_j) \lesssim q(Q) + j$ for all $j \in \N_0$.

We now turn to proving that $q(\wh{Q}_j) \gtrsim q(Q) + j$ for all $j \in \N_0$.  Let $j_1 \in \N$ be such that $Q$ is the $j_1$th square in $\CW$ that $\gamma$ intersects when the squares in $\CW$ are ordered according to when they are first hit by $\gamma$.  We note that the analysis of the previous paragraph implies that $q(\wh{Q}_j) \gtrsim j+j_1$ where the implicit constant is universal.  Moreover,  the collection of squares in $\CW$ which $\gamma$ intersects up until it first hits $Q$ is connected and connects $Q_0$ to $Q$ and thus we can find a subcollection of those squares which is a chain of squares connecting~$Q_0$ to~$Q$.  This implies that 
\begin{align*}
\disthyp^{U}(z,\cen(Q_0)) \lesssim d_{\CW}(\cen(Q_0),\cen(Q)) \lesssim j_1 \quad \text{for all}\quad z \in Q,
\end{align*}
and hence $q(Q) \lesssim j_1$,  where the implicit constant is universal.  Combining, this completes the proof of~\eqref{eqn:q_i_universal_bound_statement}.
\end{proof}

Recall that the inradius of a set $D \subseteq \h$ with $0 \in \partial D$ is defined by $\inrad(D) = \sup\{ r>0: B(0,r) \cap \h \subseteq D \}$. Furthermore, recall the notation $\gamma_{z,w}^D$ in Section~\ref{subsec:whitney_hyperbolic}.

\begin{lemma}
\label{lem:hyperbolic_geodesics_close}
Fix $C>0$ and $r_0 \geq 2$. Let $\wt{D} \subseteq D \subseteq \C$ be simply connected domains, $z_0, w_0 \in \wt{D}$, and let $\zeta$ be a prime end of $\partial D$ which is also in $\partial \wt{D}$.  Let $\phi : D \to \h$ be the unique conformal map sending $\zeta$ to $0$ and $w_0$ to $i$ and let $r =  \inrad(\phi(\wt{D}))$. Assume that $r \geq r_0$ and
\begin{equation}
\label{eqn:hyp_close_to_ball_assumptions}
\disthyp^{\wt{D}}(z_0,w_0) \leq C.
\end{equation}
There exists a constant $c_0 > 0$ depending only on $C$, $r_0$ such that
\begin{equation}
\label{eqn:hyperbolic_constant}
\disthyp^{\wt{D}}(\gamma_{z_0,\zeta}^{\wt{D}}(t), \gamma_{w_0,\zeta}^D(t)) \leq c_0 \quad\text{for all}\quad t \geq 0.
\end{equation}
\end{lemma}

In order to prove Lemma~\ref{lem:hyperbolic_geodesics_close}, we will first state and prove the following infinite length version.

\begin{lemma}
\label{lem:hyperbolic_geodesics_exponentially_close}
Assume that we have the same setup as in Lemma~\ref{lem:hyperbolic_geodesics_close}.  Then there exists a constant $\wt{C} > 0$ depending only on $r_0$ and $C$ such that
\begin{equation}
\label{eqn:hyperbolic_exponential_decay}
	\disthyp^{\wt{D}}(\gamma_{z_0,\zeta}^{\wt{D}}(t),\gamma_{w_0,\zeta}^D(\R_+)) \leq \wt{C} e^{-t}.
\end{equation}
\end{lemma}
\begin{proof}
By conformal invariance of hyperbolic distance, may assume that $D = \h$, $w_0 = i$ and $\zeta = 0$. Note that the unit speed hyperbolic geodesic from $i$ to $0$ is $\gamma_{i,0}^\h(t) = ie^{-t}$, $t \geq 0$.  Moreover, since the unique conformal automorphism $f$ of $\h$ which fixes $0$ and sends $i$ to $z_0 = x_0 + i y_0$ is given by
\begin{align*}
	f(z) = \frac{|z_0|^2 z}{x_0 z + y_0},
\end{align*}
it follows that the unit speed hyperbolic geodesic from $z_0$ to $0$ is given by
\begin{align}\label{eq:hyperbolic_geodesic_upper_half_plane}
	\gamma_{z_0,0}^\h(t) = \frac{|z_0|^2 x_0 e^{-2t} + i|z_0|^2 y_0 e^{-t}}{x_0^2 e^{-2t} + y_0^2}.
\end{align}
Note here, that the distance to the hyperbolic geodesic $\gamma_{i,0}^\h(t) = ie^{-t}$, $t \in \R$,  shrinks much faster than that to the boundary, due to the difference in exponents. Indeed, we have that
\begin{align*}
	\frac{\re(\gamma_{z_0,0}^\h(t))}{\im(\gamma_{z_0,0}^\h(t))} = \frac{x_0}{y_0} e^{-t}.
\end{align*}
Consequently, by~\eqref{eq:dist_hyp_qh_comparable},
\begin{align}\label{eq:dist_geodesic_line}
	\disthyp^\h(\gamma_{z_0,0}^\h(t),\gamma_{i,0}^\h(\R)) &\leq \distqh^\h(\gamma_{z_0,0}^\h(t),\gamma_{i,0}^\h(\R)) = \inf_{P:\gamma_{z_0,0}^\h(t) \to i \R} \int_P \frac{|dz|}{|\im(z)|} \\
	&\leq \frac{|\re(\gamma_{z_0,0}^\h(t))|}{\im(\gamma_{z_0,0}^\h(t))} = \frac{|x_0|}{y_0} e^{-t}. \nonumber
\end{align}
By mapping $z_0$ and $i$ to $\D$ with the map $\psi(z) = (i-z)/(i+z)$ it is easy to see that the condition $\disthyp^\h(z_0,i) \leq C$ gives a lower bound on $y_0$ and an upper bound on $|z_0|$.

Let $\wt{D}^*$ be the reflection of $\wt{D}$ across $(-r,r)$ and let $\varphi: \wt{D}^* \to \C$ be the unique conformal map satisfying $\varphi(\wt{D}) = \h$, $\varphi(0) = 0$, $\varphi'(0) = 1$ and $\varphi(-r_0) = -s$ and $\varphi(r_0) = s$ for some $s >0$. The proof idea is the following. We know that we have the desired exponential convergence of the image $\gamma_{\wt{z}_0,0}^\h = \varphi(\gamma_{z_0,0}^{\wt{D}})$ to the geodesic $\gamma_{i,0}^\h$, so in order to check that $\gamma_{z_0,0}^{\wt{D}}$ converges towards $\gamma_{i,0}^\h$ as well, conformally map it with $\varphi$ and then check that $\varphi$ does not perturb $\gamma_{i,0}^\h$ too much. Indeed, this is the case as $\varphi$ will locally look like the identity close to $0$.

We now formalize the ideas.   We let $H^* = \varphi(\wt{D}^*)$ and note that $\varphi_r(z) \coloneqq \varphi(rz)$ is a conformal map defined on $\D$ and $\varphi_r: \D \to \varphi_r(\D) \subseteq H^*$, $\varphi_r(0) = 0$ and $\varphi_r'(0) = r$. Thus, writing $\wt{\varphi}(z) = \tfrac{1}{r} \varphi_r(z)$, then (considering the restriction to $\D$) $\wt{\varphi}: \D \to \wt{\varphi}(\D) \subseteq \tfrac{1}{r}H^*$ is conformal and $\wt{\varphi}(0) = 0$ and $\wt{\varphi}'(0) = 1$. By~\cite[Proposition~3.26]{lawler2008conformally} there is a universal constant $C_*$ such that whenever $|z| \leq 1/2$, then $|\wt{\varphi}(z) -z| \leq C_* |z|^2$ (in fact, the optimal value is $C_* = 6$). In particular, when $|z| \leq r/2$, 
\begin{align}
\label{eqn:varphi_identity}
	|\varphi(z) - z| \leq C_* \frac{|z|^2}{r}.
\end{align}
It follows that for $s > \max(0,\log (2C^*/r))$, we have $\im(\varphi(i e^{-s})) \geq e^{-s}/2$ and consequently that
\begin{align}\label{eq:distortion_hyperbolic_geodesic}
	\disthyp^\h(\gamma_{i,0}^{\h}(s),\varphi(\gamma_{i,0}^{\h}(s))) \leq \distqh^\h(\gamma_{i,0}^{\h}(s),\varphi(\gamma_{i,0}^{\h}(s))) \leq \frac{|ie^{-s}-\varphi(ie^{-s})|}{e^{-s}/2} \leq \frac{2C_* e^{-s}}{r}.
\end{align}
We let $T_t = \argmin\{ s \in \R  : \disthyp^\h(\gamma_{\wt{z}_0,0}^\h(t),\gamma_{i,0}^\h(s)) \}$ and $\wt{z}_t = \wt{x}_t + i \wt{y}_t \coloneqq \gamma_{\wt{z}_0,0}^\h(t) = \varphi(\gamma_{z_0,0}^{\wt{D}}(t))$ and note that by conformal invariance and the triangle inequality
\begin{align}\label{eq:triangle_inequality_hyperbolic}
	\disthyp^{\wt{D}}(\gamma_{z_0,0}^{\wt{D}}(t),\gamma_{i,0}^\h(\R)) &\leq \disthyp^\h(\gamma_{\wt{z}_0,0}^\h(t),\gamma_{i,0}^\h(T_t)) + \disthyp^\h(\gamma_{i,0}^\h(T_t),\varphi(\gamma_{i,0}^\h(T_t))).
\end{align}
Moreover, we have by the distortion theorem (\cite[Theorem~3.21]{lawler2008conformally}) that there is a universal constant $C' > 0$ such that $\disthyp^{\h}(\varphi(i),i) \leq C'$. Thus, as above, the bound $\disthyp^\h(\varphi(z_0),i) \leq \disthyp^{\wt{D}}(z_0,i) + \disthyp^{\h}(\varphi(i),i) \leq C + C'$ provides an upper bound on $|\wt{z}_0|$ and a lower bound on $\wt{y}_0$ which depend only on $C$.  By~\eqref{eq:dist_geodesic_line} the first term on the right-hand side of~\eqref{eq:triangle_inequality_hyperbolic} is at most $\tfrac{|\wt{x}_0|}{\wt{y}_0} e^{-t}$.  Thus, what remains is to handle the second term.

We begin by noting that
\begin{align*}
	\log \frac{e^{-T_t}}{\wt{y}_t} = \disthyp^{\h}(\gamma_{i,0}^\h(T_t), i\wt{y}_t) \leq \disthyp^{\h}(\gamma_{i,0}^\h(T_t), \wt{z}_t)\leq \frac{|\wt{x}_0|}{\wt{y}_0} e^{-t}.
\end{align*}
It follows that
\begin{align}\label{eq:geodesic_min_time}
	T_t \geq \log(1/\wt{y}_t) - \frac{|\wt{x}_0|}{\wt{y}_0} e^{-t}.
\end{align}
Thus, by the above mentioned bounds on $|\wt{x}_0|$ and $\wt{y}_0$, the fact that $\wt{y}_t \asymp \tfrac{|z_0|^2}{\wt{y}_0} e^{-t}$ where the implicit constant depends only on $C$ (recall~\eqref{eq:hyperbolic_geodesic_upper_half_plane}) and~\eqref{eq:distortion_hyperbolic_geodesic}, it follows that there exist constants $t_0 \geq 0$ and $C_0 > 0$ depending only on $r_0$ and $C$ such that for $t \geq t_0$,  $T_t \geq \max(0,\log(2C_*/r))$ and
\begin{align*}
	\disthyp^\h(\gamma_{i,0}^\h(T_t),\varphi(\gamma_{i,0}^\h(T_t))) \leq C_0 e^{-t}.
\end{align*}
Thus, the proof is done for $t \geq t_0$,  since then $\disthyp^{\wt{D}}(\gamma_{z_0,0}^{\wt{D}}(t),\gamma_{i,0}^\h(\R)) = \disthyp^{\wt{D}}(\gamma_{z_0,0}^{\wt{D}}(t),\gamma_{i,0}^\h(\R_+))$. 

Finally, we note that if $0 \leq t \leq t_0$, then
\begin{align*}
	\disthyp^{\wt{D}}(\gamma_{z_0,0}^{\wt{D}}(t),\gamma_{i,0}^\h(\R_+)) \leq \disthyp^{\wt{D}}(\gamma_{z_0,0}^{\wt{D}}(t),i) \leq \disthyp^{\wt{D}}(z_0,i) + t \leq C + t_0,
\end{align*}
which finishes the proof.
\end{proof}

\begin{proof}[Proof of Lemma~\ref{lem:hyperbolic_geodesics_close}]
By applying the conformal map $\phi$, we can assume without loss of generality that $D = \h$, $w_0 = i$, and $\zeta = 0$.  By Lemma~\ref{lem:hyperbolic_geodesics_exponentially_close} there exists $\wt{C}>0$ depending only on $C,r_0$ such that for each $t \geq 0$ $\disthyp^{\wt{D}}(\gamma_{\wt{z}_0,0}^{\wt{D}}(t),\gamma_{i,0}^{\h}(\R_+)) \leq \wt{C}e^{-t}$.  Fix $t \geq 0$ and let $s \geq 0$ be such that $\disthyp^{\wt{D}}(\gamma_{\wt{z}_0,0}^{\wt{D}}(t),\gamma_{i,0}^{\h}(s)) \leq 2 \wt{C}e^{-t} \leq 2\wt{C}$.  Then by~\eqref{eq:dist_hyp_qh_comparable} and the fact that $\dist(\gamma_{i,0}^\h(u),\partial \wt{D}) = \dist( \gamma_{i,0}^\h(u),\partial \h)$ for all $u>0$, we have that
\begin{align}\label{eq:t_s_ubd}
t = \disthyp^{\wt{D}}(z_0,\gamma_{z_0,0}^{\wt{D}}(t)) &\leq \disthyp^{\wt{D}}(z_0,i) + \distqh^{\wt{D}}(i,\gamma_{i,0}^\h(s)) + \disthyp^{\wt{D}}(\gamma_{i,0}^\h(s),\gamma_{z_0,0}^{\wt{D}}(t)) \leq  C  + s + 2\wt{C}. 
\end{align}

Since $\disthyp^{D_2}(z_1,w_1) \leq \disthyp^{D_1}(z_1,w_1)$ whenever $D_1 \subseteq D_2 \subseteq \C$
 and $z_1,w_1 \in D_1$, it follows that
\begin{align}
s
&= \disthyp^{\h}(i, \gamma_{i,0}^\h(s))\leq \disthyp^{\wt{D}}(i,\gamma_{i,0}^{\h}(s))\notag\\ 
&\leq \disthyp^{\wt{D}}(i,z_0) + \disthyp^{\wt{D}}(z_0,\gamma_{z_0,0}^{\wt{D}}(t)) + \disthyp^{\wt{D}}(\gamma_{z_0,0}^{\wt{D}}(t),\gamma_{i,0}^{\h}(s)) \notag \\ 
&\leq C + t + 2\wt{C}.	 \label{eq:s_t_ubd}
\end{align}
Thus~\eqref{eq:t_s_ubd} and~\eqref{eq:s_t_ubd} imply that there exists a constant $c_1 > 0$ depending only on $C,r_0$ such that $|t-s| \leq c_1$.  Finally, since $\disthyp^{\wt{D}}(\gamma_{i,0}^{\h}(s),\gamma_{i,0}^{\h}(t)) \leq \distqh^{\wt{D}}(\gamma_{i,0}^{\h}(s),\gamma_{i,0}^{\h}(t)) = |t-s|$ we have that
\begin{align*}
 \disthyp^{\wt{D}}(\gamma_{z_0,0}^{\wt{D}}(t),\gamma_{i,0}^\h(t))
&\leq \disthyp^{\wt{D}}(\gamma_{z_0,0}^{\wt{D}}(t),\gamma_{i,0}^\h(s)) + \disthyp^{\wt{D}}(\gamma_{i,0}^\h(s),\gamma_{i,0}^\h(t)) \leq 2\wt{C}+c_1.
\end{align*}
\end{proof}

\bibliographystyle{abbrv}
\bibliography{references}

\end{document}